\documentclass[aap]{imsart} 

\usepackage{amssymb, mathtools, mathrsfs, xcolor, environ, comment}
\usepackage{hyperref}
\mathtoolsset{showonlyrefs}
\numberwithin{equation}{section}\allowdisplaybreaks

\usepackage[alphabetic,abbrev,msc-links,nobysame]{amsrefs}

\startlocaldefs

\theoremstyle{plain}
\newtheorem{theorem}{Theorem}[section]
\newtheorem{lemma}[theorem]{Lemma}
\newtheorem{proposition}[theorem]{Proposition}
\newtheorem{corollary}[theorem]{Corollary}
\theoremstyle{definition}
\newtheorem{definition}[theorem]{Definition}
 
\newtheorem{remark}[theorem]{Remark}

\newenvironment{old}{}{\ignorespacesafterend}\RenewEnviron{old}{} 
\newenvironment{final}{\color{black}}{\ignorespacesafterend}
\newenvironment{draft}{}{\ignorespacesafterend}\RenewEnviron{draft}{} 

\DeclarePairedDelimiter{\Parentheses}{(}{)} \newcommand{\pr}{\Parentheses*}
\DeclarePairedDelimiter{\Brackets}{[}{]} \newcommand{\brk}{\Brackets*}
\DeclarePairedDelimiter{\Braces}{\{}{\}} \newcommand{\brc}{\Braces*}
\DeclarePairedDelimiter{\Chevrons}{\langle}{\rangle} \newcommand{\chv}{\Chevrons*}
\DeclarePairedDelimiter{\Absolute}{|}{|} \newcommand{\abs}{\Absolute*}
\DeclarePairedDelimiter{\Norm}{\|}{\|} \newcommand{\norm}{\Norm*}
\DeclarePairedDelimiter{\Condition}{.}{|}
\newcommand{\cond}{\Condition*}

\newcommand{\eq}{\refstepcounter{equation}(\theequation)}
\newcounter{constant}
\newcommand{\C}[1]{\refstepcounter{constant}\label{#1}}\newcommand{\Cr}[1]{C_{\ref{#1}}}

\newcounter{step}
\newcounter{proof}
\makeatletter
\@addtoreset{step}{proof}
\makeatother
\newcommand{\firststep}{\refstepcounter{proof}\refstepcounter{step}\emph{Step \thestep. }} \newcommand{\nextstep}{\refstepcounter{step}\emph{Step \thestep. }}

\newcommand{\E}{\mathbb{E} }
\newcommand{\bbP}{\mathbb{P}}
\newcommand{\Q}{\mathbb{Q}}
\newcommand{\R}{\mathbb{R}}
\newcommand{\N}{\mathbb{N}}

\endlocaldefs %
\begin{document}
	\begin{frontmatter}	%
		
		\title{On the subcritical self-catalytic branching Brownian motions} %
		\runtitle{SBBM} %
		\begin{aug} %
			\author{\fnms{Haojie} \snm{Hou}\ead[label=e1]{houhaojie@bit.edu.cn}} %
			\and
			\author{\fnms{Zhenyao} \snm{Sun}\ead[label=e2]{zhenyao.sun@gmail.com}} %
				\address{School of Statistics and Data Science, Nankai University, Tianjin 300071, China \printead{e1}} 
			\address{School of Mathematics and Statistics, Beijing Institute of Technology, Beijing 100081, China \printead{e2}} 
		\end{aug}
		\begin{abstract}
			The self-catalytic branching Brownian motions are extensions of the classical one-dimensional branching Brownian motions by incorporating pairwise branchings catalyzed by the intersection local times of the particle pairs. These processes naturally arise as the moment duals of certain reaction-diffusion equations perturbed by multiplicative space-time white noise. For the subcritical case of the catalytic branching mechanism, we construct the SBBM allowing an infinite number of initial particles. Additionally, we establish the coming down from infinity (CDI) property for these systems and characterize  their CDI rates.
		\end{abstract}
		
	\end{frontmatter}
	
	\section{Introduction}
	
	\subsection{Motivation}
	
	The branching Brownian motion (BBM) is a classical probabilistic model describing a system of particles that move independently according to Brownian motions and branch independently into  random number of offspring at random times. This model serves as a cornerstone of modern probability theory, with applications spanning partial differential equations, statistical physics, and mathematical biology. Foundational works such as \cite{MR0400428} and \cite{MR0494541} established its dual connection to the FKPP equation, while a modern overview is provided in \cite{MRBerestycki2014Topics}.
	
	While the classical BBM has been extensively studied, more complex models have emerged to incorporate intraspecific competition. Examples include the branching coalescing Brownian motions \cite{MR0948717}, where pairs of particles merge randomly based on their pairwise intersection local times; the $N$-BBM \cite{MR3568046}, which maintains a fixed population size through a selection mechanism; the $L$-BBM \cite{MR3492932}, where particles further
	than a distance $L$ from the leading particle are deleted; and BBM with decaying mass \cite{MR4020312}, where each particle carries a mass that decays at a rate proportional to its neighboring mass field.
	
	A further extension of the branching coalescing Brownian motion is the self-catalytic branching Brownian motion (SBBM), first introduced in \cite{MR2698239} and further developed in \cite{MR1813840}.
	In this model, branching events are catalyzed by the intersection local times of particle pairs, allowing not only the competitive but also the cooperative interactions between particles.
	A defining mathematical feature of SBBM is its dual relationship with a family of stochastic reaction-diffusion equations, akin to the connection between BBM and the FKPP equation.
	This duality provides a comprehensive framework for studying both SBBMs and their dual SPDEs.
	
	For example, the foundational work \cite{MR0948717} employed branching coalescing Brownian motion to establish the weak uniqueness of the stochastic FKPP equation.
	This result was later extended in \cite{MR1813840}, which used SBBM to prove the well-posedness of a family of stochastic reaction-diffusion equations. Another important contribution is \cite{MR1339735}, where local-time coalescing Brownian motions (LCBM) were utilized to demonstrate the compact interface property of the Wright-Fisher stochastic heat equation (Wright-Fisher SHE).
	Later, an advanced small-noise asymptotic analysis of the propagation speed of the stochastic FKPP equation was established in \cite{MR2793860}, which, via duality, confirmed the Brunet-Derrida conjecture for the branching coalescing Brownian motions.
	See also \cite{MR4828440} for related results on the effect of small noise on the speed of reaction-diffusion equations with non-Lipschitz drift.
	The duality framework was also employed in \cite{MR3582808} to demonstrate the convergence of the biased voter model to the stochastic FKPP equation.
	The branching coalescing Brownian motions on metric graphs were utilized in \cite{MR4278798} to study the stochastic FKPP equation on metric graphs.
	A two-type coalescing Brownian motion system was applied in \cite{MRFanTough2023quasi} to derive the quasi-stationary distributions of the stochastic FKPP equation on the circle.
	Moreover, the analytical properties of the Wright-Fisher  SHE  were applied in \cite{MR4698025} to prove the coming down from infinity (CDI) property of LCBM.
	Building on this result, \cite{MRBMS24} showed that a system of branching coalescing Brownian motions, allowing for infinitely many offspring in its branching mechanism, has its total population reflected from infinity. This phenomenon was further used in \cite{MRBMS24} to demonstrate
	a regularization-by-noise effect of the Wright-Fisher space-time white noise.
	
	In this article, we continue the study of the SBBM by addressing the following fundamental question:
	\begin{itemize}
		\item[\eq\label{eq:Q}]
		\emph{What happens if there are infinitely many initial particles in an SBBM?}
	\end{itemize}
	This question was previously explored in \cite{MR2162813} in the context of  LCBM  with a comprehensive resolution provided in \cite{MR4698025}.
	We will show that, under the assumptions that the catalytic branching is subcritical and not parity-preserving, an SBBM model supporting infinitely many initial particles can be defined as the limit of a sequence of SBBMs with finitely many initial particles (Theorem \ref{thm:existence-of-Z-t}).
	The law of this limiting process is characterized by the initial trace---a key concept introduced in \cite{MR4698025}---and the branching mechanisms.
	Additionally, we will establish  the CDI result: a necessary and sufficient condition for the finiteness of the number of particles in any region at any time (Theorem \ref{Comming-down-finite-first-moment}); and characterize the corresponding CDI rates in terms of
	a deterministic partial differential equation (Theorem \ref{thm:CDIrate});
	thereby generalizing the results in \cite{MR4698025}.
	Our findings reveal an universal behavior: the CDI rates, while depending on the initial trace and the  mean-field effect of the pairwise interaction, are independent of the precise form of the branching mechanisms.

	\subsection{Main results} \label{sec:MR}
	
	The SBBM model has five parameters:

	\begin{itemize}
		\item[\eq\label{eq:IC}]
		The initial configuration $(x_i)_{i=1}^n$, which is a finite list in $\mathbb R$.
		\item[\eq\label{eq:OBR}]
		The ordinary branching rate  $ \beta_{\mathrm o}\geq 0$;
		\item[\eq\label{eq:OB}]
		The ordinary  offspring  law
		$(p_k)_{k=0}^\infty$, which is a probability measure on $\mathbb{Z}_+$, the space of non-negative integers;
		\item[\eq]
		The catalytic branching rate $\frac{1}{2}\beta_{\mathrm c} > 0$;
		\item[\eq\label{eq:CB}]
		The catalytic  offspring  law $(q_k)_{k=0}^\infty$, which is a probability measure on $\mathbb{Z}_+$.
	\end{itemize}
	The SBBM with the above parameters is a stochastic particle system which evolves according to the following rules \eqref{eq:R1}--\eqref{eq:R4}.
	\begin{itemize}
		\item[\eq\label{eq:R1}]
		At time $0$, there are $n$
		particles located in the real line $\mathbb R$
		whose locations are given by the initial configuration $(x_i)_{i=1}^n$.

		\item[\eq\label{eq:R2}]
		Each particle moves as independent Brownian motions unless one of the following ordinary branching or catalytic branching occurs.
		\item[\eq\label{eq:R3}]
		Each particle induces an ordinary branching according to an independent exponential clock of  rate $\beta_{\mathrm o}$.
		When an ordinary branching occurs, the corresponding particle will be killed and be replaced, at its location of death, by a random number of new particles.  This random number will be independently sampled according to the ordinary  offspring  law  $(p_k)_{k= 0}^\infty$.
		\item[\eq\label{eq:R4}]
		Each unordered pair of particles induces a catalytic branching according to an independent exponential clock of rate $ \frac{1}{2}\beta_{\mathrm c} $  with respect to their intersection local time.
		When a catalytic branching occurs, the corresponding pair of particles will both be killed and be replaced, at their mutual location of death, by a random number of new particles.
		This random number will be independently sampled according to the catalytic  offspring distribution $(q_k)_{k= 0}^\infty$.
	\end{itemize}
	Note  that, producing one  child  in an ordinary branching, or two children in a catalytic branching, does not change the configuration of the particle profile at the occurring time of that branching.
	Therefore, we can assume, without loss of generality, that $p_1 = q_2 = 0$.
	For every $z\in \mathbb R$, we define
	\begin{equation}\label{Def-Phi}
		\Phi(z) := \beta_{\mathrm o} \pr{\sum_{k=0}^\infty p_k (1-z)^k- (1-z)}
	\end{equation}
	and
	\begin{equation}\label{Def-Psi}
		\Psi(z) := \beta_{\mathrm c}\pr{\sum_{k=0}^\infty q_k (1-z)^k - (1-z)^2}
	\end{equation}
	whenever the infinite series on the right hand sides are absolutely summable.
	The functions $\Phi$, and $\Psi$, will be referred to as the ordinary, and the catalytic, branching mechanisms, respectively.
	
	We assume that
	\begin{itemize}
		\item[\eq\label{asp:FFM}]	both of the ordinary offspring law \eqref{eq:OB} and the catalytic offspring law \eqref{eq:CB} have finite first moment, i.e.,
		$\sum_{k\in \mathbb Z_+} k p_k +  \sum_{k\in \mathbb Z_+} k q_k < \infty$.
	\end{itemize}
	We say the ordinary branching is subcritical, critical, or supercritical, according to $\sum_{k\in \mathbb Z_+} kp_k$ is strictly less than, equals to,  or strictly greater than, $1$, respectively.
	We say that the catalytic branching is subcritical, critical, or supercritical, according to $\sum_{k\in \mathbb Z_+} kq_k$ is strictly less than, equals to,  or strictly greater than, $2$, respectively.
	We will only be considering the case when the catalytic branching is subcritical, that is, we assume that
	\begin{equation} \label{asp:A3}
		\sum_{k\in \mathbb Z_+} kq_k < 2.
	\end{equation}

	A priori speaking, a particle system following the rules \eqref{eq:R1}--\eqref{eq:R4} can only be defined up to its explosion time.
	In order to be more precise, let $\tau_0 := 0$, and inductively for every $k\in \mathbb Z_+$, let $\tau_{k+1}$ be the earliest occurring time of a branching after the time $\tau_k$ (if $\tau_k = \infty$, or if there is no branching occurring after the time $\tau_k<\infty$, then set $\tau_{k+1} :=\infty$ for convention.)
	Thanks to the strong Markov property of the Brownian motions, one can construct  a particle system following the rules \eqref{eq:R1}--\eqref{eq:R4} in the time interval $[0,\tau_\infty)$ where $\tau_\infty := \lim_{k\to \infty} \tau_k$
	is  called  the explosion time.
	(We omit the details of the construction since it is tedious but straightforward.)

	Our first result, whose proof is postponed in Section \ref{ss:3.1}, says that this explosion won't really happen.
	
	\begin{proposition} \label{prop:WD}
		Suppose that $\tau_\infty$ is the explosion time of an SBBM with its parameters given by \eqref{eq:IC}--\eqref{eq:CB}.
		Suppose that \eqref{asp:FFM} and \eqref{asp:A3} hold.
		Then, almost surely, $\tau_\infty = \infty$.
	\end{proposition}

	For every  $t\geq 0$, we denote by $I_t$ the collection of unique labels of the particles that are  alive at time $t$.
	(How  we label the particles is not crucial for our purpose.)
	For every $t\geq 0$ and $\alpha \in I_t$, denote by $X^{\alpha}_t$ the spatial location of the particle labeled by $\alpha$ at time $t$.
	For every $t\geq 0$ and $U\in \mathcal B(\mathbb R)$, denote by $Z_t(U) := |\{\alpha \in I_t: X_t^{\alpha} \in U\}|$ the number of alive particles at time $t$ whose locations belong to $U$.
	Here, $\mathcal B(\mathbb R)$ represents the collection of Borel subsets of $\mathbb R$, and $|A|$ represents the cardinality of a given set $A$.

	By convention, at each ordinary or catalytic branching event, the children are considered alive at the branching time, while the parent is not.
	By this convention, $(Z_t)_{t\geq 0}$ is a càdlàg process taking values in $\mathcal N := \mathcal N_{\mathbb R}$, the space of integer-valued Radon measures on $\mathbb R$.
	To be more precise, for any open subset $U$ of $\mathbb R$, we denote by $\mathcal N_{U}$ the collection of integer-valued Radon measures on $U$, equipped with the vague topology, i.e.~the coarsest topology such that the map $\mu \mapsto \mu(f)$
	from $\mathcal N_U$ to $\mathbb R$ is continuous for every $f\in \mathcal C_\mathrm c(U)$.
	Here, $\mathcal C_\mathrm c(U)$ represents the collection of compactly supported continuous functions on $U$; and $\mu(f)$ represents the integration of a measurable function $f$ against a measure $\mu$ whenever it is well-defined.
	It is known that $\mathcal N_U$ is Polish \cite{MR3642325}*{Theorem 4.2}.
	For any Borel measure $\nu$ and non-negative Borel measurable function $g$ on $U$, let $g\cdot \nu$ be the unique Borel measure on $U$ such that $(g\cdot \nu)(A) = \nu(\mathbf 1_A g)$ for any Borel subset $A$ of $U$.
	It is known (c.f.~\cite{MRYan1998Lectures}*{Theorem 6.5.8}) that there exists a sequence $(h_i)_{i\in \mathbb N}$ in $\mathcal C_{\mathrm c}^\infty(\mathbb R)$ such that the vague topology of $\mathcal N$ is compatible with the complete metric
	\begin{equation}\label{eq:dN}
		d_{\mathcal N}(\nu_0, \nu_1):=\sum_{i=1}^\infty \frac{1}{2^i}\pr{1\land \abs{\nu_0 (h_i)- \nu_1(h_i) }}, \quad \nu_0, \nu_1 \in \mathcal N.
	\end{equation}
	We will fix this sequence $(h_i)_{i\in \mathbb N}$ and treat $\mathcal N$ as the separable complete metric space $(\mathcal N, d_{\mathcal N})$ throughout this paper.
	For any $t_0\geq 0$, denote by $\mathbb D([t_0,\infty), \mathcal N)$ the space of $\mathcal N$-valued c\`adl\`ag paths indexed by $[t_0,\infty)$ equipped with the Skorokhod $J_1$-topology in the sense of \cite{MR4226142}*{Lemma A5.3}. It is known that $\mathbb{D}([t_0,\infty), \mathcal{N})$ is also Polish.
	Note that the law of the process $(Z_t)_{t\geq 0}$ induced on $\mathbb D(\mathbb R_+, \mathcal N)$
	is uniquely determined by the parameters listed in \eqref{eq:IC}--\eqref{eq:CB}.
	Any  $\mathcal N$-valued c\`adl\`ag process that shares the same law as $(Z_t)_{t\geq 0}$ will be therefore referred to as an  SBBM with respect to those parameters.
	
	Given a sequence of SBBMs $(Z^{(n)}_\cdot)_{n\in \mathbb N}$, sharing the same branching mechanisms, whose
	initial total populations $Z^{(n)}_0(\mathbb R)$ converge to $\infty$, we are interested in the asymptotic behavior of the process $Z^{(n)}_\cdot$ as $n\uparrow \infty$.
	If one can show the convergence in distribution of such sequence of processes, then it is reasonable to regard the limit as an SBBM with infinitely many initial particles, answering the proposed question \eqref{eq:Q}.
	With this purpose in mind, we introduce several notations:
	Given $\mu_0, \mu_1 \in \mathcal N$, we say $\mu_0$ is dominated by $\mu_1$, and write $\mu_0 \preceq \mu_1$, provided that there exists a $\nu \in \mathcal N$ such that $\mu_0 + \nu = \mu_1$. We say a sequence $(\mu_n)_{n\in \mathbb N}$ in $\mathcal N$ is monotonically increasing provided that $\mu_n \preceq \mu_{n+1}$ for each of $n\in \mathbb N$.
	Let $\mathcal T$ be the collection of the pair $(\Lambda, \mu)$ where $\Lambda$ is a closed subset of $\mathbb R$ and $\mu$ is a  Radon measure on the complement $\Lambda^{\mathrm c} =\mathbb R\setminus \Lambda$.
	Let $\mathcal{T}_{\mathrm{a}}$ be the collection of the pair $(\Lambda, \mu)\in \mathcal T$ where the Radon measure $\mu$ is integer-valued.
	Given a sequence $(\mu_n)_{n\in \mathbb N}$ of Radon measures on $\mathbb R$, we say it converges m-weakly to a pair $(\Lambda, \mu) \in \mathcal T$ provided that, as $n\uparrow \infty$,
	\begin{itemize}
		\item $\mu_n(U) \to \infty$ for any open $U \subset \mathbb R$ such that $U\cap \Lambda \neq \emptyset$; and
		\item the restriction $\mu_n|_{\Lambda^\mathrm{c}}$ converges vaguely to $\mu$.
	\end{itemize}
	This notion of m-weak convergence, which extends the classical vague topology to systematically characterize localized mass explosions through the paired structure $(\Lambda, \mu)$, was originally introduced by Marcus and V\'{e}ron in their foundational work on boundary traces for elliptic equations \cite{MR1658392}, and subsequently adapted to the parabolic framework in \cite{MR1697494}.
	Note that every monotonically increasing sequence $(\mu_n)_{n\in \mathbb N}$ in $\mathcal N$ converges m-weakly to a pair $(\Lambda, \mu) \in \mathcal T_a$ determined by
	\begin{align}
		& \Lambda := \brc{y\in \mathbb R: \lim_{n\to \infty}\mu_n((y-\epsilon, y+\epsilon)) = \infty, \forall \epsilon > 0},
		\\& \mu(A):= \lim_{n\to \infty} \mu_n(A), \quad \text{for any Borel subset of $A\in \Lambda^\mathrm{c}$}.
	\end{align}
	Conversely, for each pair $(\Lambda, \mu) \in \mathcal T_a$, there exists a monotonically increasing sequence $(\mu_n)_{n\in \mathbb N}$ converging m-weakly to it.
	The support of a given $(\Lambda, \mu)\in \mathcal T$ is defined by $\operatorname{supp}(\Lambda, \mu) := \Lambda \cup \operatorname{supp} (\mu)$ where $\operatorname{supp} (\mu)$ represents the support of the measure $\mu$.
	
	We will be focusing on the case where $(Z^{(n)}_\cdot)_{n\in \mathbb N}$ is a sequence of SBBMs, sharing the same branching mechanisms, whose initial values can be any monotonically increasing sequence in $\mathcal N$.
	In this case, the initial values $(Z_0^{(n)})_{n\in \mathbb N}$ converges m-weakly to a pair $(\Lambda, \mu)\in \mathcal T_a$ while it is possible that they fail to converge vaguely in $\mathcal N$ (for example, if $\mu_n = n \cdot \delta_0$, then the m-weak limit is $(\{0\}, \mathbf 0)$ where $\mathbf 0$ denotes the null measure, but the vague limit does not exist).
	Due to this reason, we will focus on establishing the convergence in distribution of the sequence $(Z^{(n)}_\cdot)_{n\in \mathbb N}$ in the path space $\mathbb{D}([t_0, \infty), \mathcal N)$ where $t_0>0$ can be arbitrarily small but away from $0$.
	If such kind of convergence holds, it gives rise to a limiting process $(Z_t)_{t>0}$ indexed by the open time interval $(0,\infty)$.
	We will show that the distribution of the limiting process $(Z_t)_{t>0}$ is uniquely characterized by the pair $(\Lambda, \mu)$ and the branching mechanisms $\Phi$ and $\Psi$.
	It is then reasonable to regard the limiting pair $(\Lambda, \mu)$ as the initial configuration, termed as the initial trace, of the limiting SBBM process $(Z_t)_{t>0}$.
	If the initial values $(Z_0^{(n)})_{n\in \mathbb N}$ converge vaguely in $\mathcal N$, we can also establish the convergence of the sequence $(Z^{(n)}_\cdot)_{n\in \mathbb N}$ in the path space $\mathbb{D}([0, \infty), \mathcal N)$.
	
	\begin{remark} \label{rem:initial_data_convergence}
		While it is natural to investigate the limiting behavior of finite particle systems when their initial configurations $(\mu_n)_{n\in \mathbb N}$ approximate the initial trace without monotonely converging to it, such a general scenario will not be considered in this paper.
		Rather than reflecting an intrinsic limitation of the underlying processes, this restriction is introduced primarily as a methodological simplification to navigate the structural discontinuity inherent in the moment duality technique.
		We will comment on the duality technique in Subsubsection~\ref{sec:PSD} and more on this restriction in Subsubsection~\ref{sec:initial_data_convergence}.
	\end{remark}
	
	To establish the above approximation results, let us assume the existence of the exponential moments of the two offspring  laws:
	\begin{itemize}
		\item[\eq\label{asp:A1}]
		There exists $R>1$ such that $\sum_{k=0}^\infty R^k p_k < \infty$  and  $\sum_{k =0}^\infty R^k q_k < \infty.$
	\end{itemize}
	This assmuption will be used in Proposition \ref{prop:D} for the establishment of the moment duality between the SBBMs and their dual SPDEs.
	Let us also assume the following:
	\begin{itemize}
		\item[\eq\label{asp:A2}]
		The catalytic offspring  law
		is  not parity-preserving, that is, there exists an odd number $k$ such that $q_k > 0$.
	\end{itemize}
	This assumption is crucial for our establishment of the distributional convergence of $(Z_t^{(n)})_{t>0}$ as $n\to\infty$.
	We will comment on this assumption in Subsubsection \ref{sec:DA}.

	\begin{theorem}\label{thm:existence-of-Z-t}
		Let the offspring parameters $\beta_{\mathrm o}$, $(p_k)_{k=0}^\infty$, $\beta_{\mathrm c}$ and $(q_k)_{k=0}^\infty$ be given as in \eqref{eq:OBR}--\eqref{eq:CB}.
		Assume that \eqref{asp:A3}, \eqref{asp:A1} and \eqref{asp:A2} hold.
		Then, for any $(\Lambda, \mu)\in \mathcal T_\mathrm{a}$ and branching mechanisms $\Phi$ and $\Psi$ defined by \eqref{Def-Phi} and \eqref{Def-Psi}, there exists a unique in law $\mathcal N$-valued c\`{a}dl\`{a}g Markov process $(Z_t)_{t>0}$ such that the following statement holds:
		for any
		\begin{itemize}
			\item[\eq\label{eq:Zt-monotone}] a monotonically increasing sequence $(\mu_n)_{n\in \mathbb N}$ in $\mathcal N$ converging m-weakly to $(\Lambda, \mu)$,
			\item[\eq\label{eq:Zt-SBBM}] a sequence of SBBMs $(Z_\cdot^{(n)})_{n\in \mathbb N}$ sharing the same branching mechanisms $\Phi$ and $\Psi$ whose initial values are given by the sequence $(\mu_n)_{n\in \mathbb N}$, and
			\item[\eq\label{eq:Zt-t0}] $t_0>0$,
		\end{itemize}
		the $\mathbb D([t_0,\infty), \mathcal N)$-valued random elements $(Z_t^{(n)} )_{t\geq t_0}$ converges in distribution to $(Z_t)_{t\geq t_0}$ as $n\to \infty$.
		If we assume in addition that $\Lambda = \emptyset$, then the above statement also holds for $t_0=0$ with $Z_0 := \mu$.
	\end{theorem}

	The proof of Theorem \ref{thm:existence-of-Z-t} is postponed to Section \ref{sec:SS5}.
	
	In light of Theorem \ref{thm:existence-of-Z-t}, any $\mathcal N$-valued c\`{a}dl\`{a}g process indexed by $(0,\infty)$  that shares the same law as the process $(Z_t)_{t> 0}$ in Theorem \ref{thm:existence-of-Z-t}
	will be referred to as an SBBM with ordinary branching mechanism $\Phi$, catalytic branching mechanism $\Psi$, and initial trace $(\Lambda, \mu)$.
	
	In the rest of this  subsection,  let $(Z_t)_{t>0}$ be such a process, whose corresponding probability measure and expectation operator  is  denoted by $\mathbb P_{(\Lambda, \mu)}$ and $\mathbb E_{( \Lambda, \mu)}$, respectively.
	We want to investigate the \textit{coming down from infinity} (CDI) property for this process.
	Given a Borel subset $U \subset \mathbb R$, we say the CDI property holds for the process $(Z_t(U))_{t>0}$ provided that:
	\begin{itemize}
		\item[\eq\label{eq:CDIP}] $Z_t(U)<\infty$ for every $t>0$ and $\lim_{t\downarrow 0} Z_t(U) = \infty$ in probability.
	\end{itemize}
	Intuitively, this means that the number of particles in the region $U$ starts from an infinite value, but becomes finite immediately after the system starts to evolve.
	This property was studied by \cite{MR2162813} for the local-time coalescing Brownian motions (LCBM) living on the circle $\mathbb T$ and by \cite{MR4698025} on the real line $\mathbb R$.
	(Note that LCBM is a particular SBBM with parameters $\beta_0 = 0$ and $q_1 = 1$.)
	According to \cite{MR4698025}*{Theorem 1.4}, in the context of the LCBM, $Z_t(U)$ is almost surely finite if and only if $U \cap \operatorname{supp}(\Lambda, \mu)$ is bounded.
	(A set $A\subset \mathbb R$ is said to be bounded if $\sup\brc{|x|: x\in A}<\infty$.)
	Furthermore, the expected number $\mathbb E[Z_t(U)]$ explodes to infinity when $t\downarrow 0$ if and only if $\bar{U}\cap \Lambda \neq \emptyset$.
	(We use $\bar A$ to denote the closure of a given set $A \subset \mathbb R$.)
	Our next  result says that these results hold for the more  general SBBM model.
	
	\begin{theorem} \label{Comming-down-finite-first-moment}
		Let the offspring parameters $\beta_{\mathrm o}$, $(p_k)_{k=0}^\infty$, $\beta_{\mathrm c}$ and $(q_k)_{k=0}^\infty$ be given as in \eqref{eq:OBR}--\eqref{eq:CB}.
		Assume that \eqref{asp:A3}, \eqref{asp:A1} and \eqref{asp:A2} hold.
		Let $(Z_t)_{t> 0}$ be an SBBM, defined via Theorem \ref{thm:existence-of-Z-t}, with ordinary branching mechanism $\Phi$ given by \eqref{Def-Phi}, catalytic branching mechanism $\Psi$ given by \eqref{Def-Psi}, and initial trace $(\Lambda, \mu)$ chosen arbitrarily in $\mathcal T_\mathrm{a}$.
		Then, the following statements hold for any open interval $U\subset \R$:
		\begin{itemize}
			\item[(i)]
			If $U\cap \operatorname{supp}(\Lambda, \mu)$ is unbounded, then $\mathbb{P}_{(\Lambda, \mu)}\pr{Z_t(U)= +\infty, \forall t>0}=1$ (CDI property \eqref{eq:CDIP} does not hold).
			\item[(ii)]
			If $U\cap \operatorname{supp}(\Lambda, \mu)$ is bounded, then $\mathbb{P}_{(\Lambda, \mu)}\pr{Z_t(U) < \infty, \forall t>0 }=1.$
		\end{itemize}
		Moreover, the following statements hold for any open interval $U\subset \R$ such that $U \cap \operatorname{supp}(\Lambda, \mu)$ is bounded:
		\begin{itemize}
			\item[(iii)]
			If $\bar{U}\cap \Lambda =\emptyset,$ then the CDI property \eqref{eq:CDIP} does \emph{not} hold for $(Z_t(U))_{t>0}$.
			\item[(iv)]
			If $\bar{U}\cap \Lambda \neq \emptyset,$ then the CDI property \eqref{eq:CDIP} holds for $(Z_t(U))_{t>0}$.
		\end{itemize}
	\end{theorem}
	
	The proof of Theorem \ref{Comming-down-finite-first-moment}  will be given in Section \ref{Main}.
	From the above theorem, we see that the statement \eqref{eq:CDIP}, i.e.~the CDI property holds for the SBBM $(Z_t)_{t\geq 0}$ in a given open interval $U$, is equivalent to the following condition:
	\begin{itemize}
		\item[\eq\label{eq:CDIC}]
		$U \cap \operatorname{supp}(\Lambda, \mu)$ is bounded and $\bar{U}\cap \Lambda \neq \emptyset$.
	\end{itemize}
	Having established the CDI property in this case, a natural question arises:
	\begin{itemize}
		\item	\emph{At what rate does $Z_t(U)$, the particle population in $U$, explode as $t \downarrow 0$?}
	\end{itemize}
	To characterize this rate, we compare the stochastic particle system to a deterministic partial differential equation.

	Let us denote by $\mathcal C^{1,2}((0,\infty)\times \mathbb R)$ the collection of real-valued functions $(h_t(x))_{t>0,x\in \mathbb R}$ which is continuously differentiable in $t$ and twice continuously differentiable in $x$.
	For every $(\tilde{\Lambda},\tilde{\mu}) \in \mathcal T$, from \cite{MR1429263}*{Theorem 4}, there exists a unique non-negative $v^{(\tilde \Lambda, \tilde \mu)} = (v_t^{(\tilde \Lambda, \tilde \mu)}(x))_{t>0, x\in \mathbb R}\in \mathcal C^{1,2}((0,\infty) \times \mathbb R)$
	satisfying the following equation:
	\begin{equation}\label{PDE}
		\begin{cases}
			\displaystyle
			\partial_t v_t^{(\tilde \Lambda, \tilde \mu)}(x) = \frac{1}{2}\partial_x^2 v_t^{(\tilde \Lambda, \tilde \mu)}(x) - \frac{\Psi'(0+)}{2} v_t^{(\tilde \Lambda, \tilde \mu)}(x)^2, \quad t>0, x\in \mathbb R;
			\\ \text{The Radon measures $v_t(x)\mathrm dx$ on $\mathbb R$ converges m-weakly to $(\tilde \Lambda, \tilde \mu)$ as $t\downarrow 0$.}
		\end{cases}
	\end{equation}
	Here,
	
	\begin{equation}\label{eq:Psi'}
		\Psi'(0+) := \beta_{\mathrm c}\pr{2-\sum_{k=0}^\infty kq_k } \in (0,\infty),
	\end{equation}
	where the condition $\Psi'(0+) \in (0,\infty)$ follows from the subcritical catalytic branching assumption~\eqref{asp:A3}, which ensures $\sum_{k=0}^\infty kq_k < 2$.
	Equation \eqref{PDE}, with $(\tilde \Lambda, \tilde \mu)$ being replaced by the initial trace $(\Lambda, \mu)$ of an SBBM, will be referred to as the CDI Profile Equation, as its solution, $(v^{(\Lambda, \mu)}_t(x))_{t\geq 0, x\in \mathbb R}$, will serve as the leading-order spatial profile against which the particle density is compared to establish the exact CDI rates.
	
	In fact, the equation \eqref{PDE} already appeared in the study of the LCBM \cite{MR4698025} where it was shown that $Z_t(U)$ is asymptotically equivalent to $\int_U v_t(x) \mathrm dx$ in the $L^1$-sense as $t\downarrow 0$.
	We generalize this result for our SBBM model. We emphasize that the ordinary branching mechanism is absent from the CDI Profile Equation \eqref{PDE}. As we will elucidate in Subsubsection \ref{sec:MFE}, this is because the ordinary branching is dominated by the catalytic branching in the short-time asymptotic regime ($t \downarrow 0$), and thus only the latter contributes to this leading-order profile.

	\begin{theorem}\label{thm:CDIrate}
		Let the offspring parameters $\beta_{\mathrm o}$, $(p_k)_{k=0}^\infty$, $\beta_{\mathrm c}$ and $(q_k)_{k=0}^\infty$ be given as in \eqref{eq:OBR}--\eqref{eq:CB}.
		Assume that \eqref{asp:A3}, \eqref{asp:A1} and \eqref{asp:A2} hold.
		Let $(Z_t)_{t> 0}$ be an SBBM, defined via Theorem \ref{thm:existence-of-Z-t}, with ordinary branching mechanism $\Phi$ given by \eqref{Def-Phi}, catalytic branching mechanism $\Psi$ given by \eqref{Def-Psi}, and initial trace $(\Lambda, \mu)$ chosen arbitrarily from $\mathcal T_\mathrm{a}$.
		Suppose that $U\subset \mathbb R$ is an open interval satisfying the condition \eqref{eq:CDIC}. Then
		\[
		\pr{\int_U v^{(\Lambda, \mu)}_{t}(x)\mathrm{d}  x }^{-1}  Z_t(U)\stackrel{t\downarrow 0}{\longrightarrow}1, \quad \mbox{in } L^1  \text{~w.r.t.~} \mathbb{P}_{(\Lambda, \mu)}.
		\]
	\end{theorem}

	The proof of Theorem \ref{thm:CDIrate} will be given in Section \ref{Main}.
	We note here that the CDI rate
	\[
	t\mapsto \int_U v_t^{(\Lambda, \mu)}(x)\mathrm{d} x, \quad t>0
	\]
	depends on the initial trace $(\Lambda, \mu)$ and the constant $\Psi'(0+)$, which captures the mean-field effect of the self-catalytic branching, but is independent of the ordinary branching mechanism $\Phi$ and the precise form of the catalytic branching mechanism $\Psi$.
	We comment on universal behaviour in Subsubsection \ref{sec:MFE}.
	
	\subsection{Perspectives}
	\subsubsection{Discussion of assumptions} \label{sec:DA}

	The subcritical assumption \eqref{asp:A3} on the catalytic branching is crucial for our results. In the supercritical regime, where \(\sum_{k\in \mathbb{Z}_+} kq_k > 2\), a naive physicist's mean-field analysis suggests that, with positive probability, an everywhere explosion occurs in finite time, even when the initial number of particles is finite. This leads to several questions: What is the probability of the explosion? How can the explosion time and the growth of the population before the explosion be characterized? What is the behavior of the system conditioned on non-explosion?
	For (non-spatial) branching processes with pairwise interactions, explosion probability and explosion time are characterized in \cite{MR4929072}*{Theorem 1}; the third question, even in the non-spatial setting, appears to be open.
	
	The critical case, \(\sum_{k\in \mathbb{Z}_+} kq_k = 2\), is perhaps more intriguing. We do not expect the CDI property to hold in this case.
	However, under suitable assumptions and rescaling, we anticipate that the empirical measure of the SBBM with critical catalytic branching is likely to converge to the multiplicative linear SHE, \(\partial_t z = \frac{\Delta}{2}z + z \dot{\xi},\) where \(\dot{\xi}\) denotes the space-time white noise. This conjecture is by analogy with the classical result that critical branching Brownian motion converges to the Dawson-Watanabe superprocess under rescaling \cite{MR1779100}*{Chapter 1, \S1.4}.
	Notably, the multiplicative linear SHE is closely related to the KPZ equation, \(\partial_t h = \frac{1}{2} \partial_x^2 h - \frac{1}{2} (\partial_x h)^2 + \dot{\xi},\) for which we refer our readers to \cite{MR3098078}. This connection invites a further question: Do certain observables of the SBBM with critical catalytic branching belong to the KPZ universality class?

	The non-parity-preserving assumption \eqref{asp:A2} is also crucial for our main results.
	To illustrate this, consider the scenario where \(\beta_{\mathrm{o}} = 0\) and \(q_0 = 1\), in which \eqref{asp:A2} fails to hold.
	The SBBM model with these parameters is referred to as the local-time annihilating Brownian motions.
	Using the methods in this article, it can be shown that for any \(t > 0\), the sequence of \(\mathcal{N}\)-valued random elements \((Z_t^{(n)})_{n \in \mathbb{N}}\) is tight. However, we do not expect the sub-sequential convergence-in-distribution limit to be unique. Without assumption \eqref{asp:A2}, the parity of the total population is preserved: if the initial configuration has an even number of particles, it remains even; if odd, it remains odd. Consequently, subsequences with different initial parities may converge to different limits. We refer interested readers to \cite{MR4212931} for a detailed analysis of such questions in the setting of (instant) annihilating Brownian motions, where particles annihilate immediately upon contact. The local-time case considered here presents additional technical challenges and may exhibit qualitatively different behavior.

	The technical assumption \eqref{asp:A1} is included primarily for convenience, facilitating the application of the duality result in \cite{MR1813840}*{Theorem 1} (see Proposition \ref{prop:D} below).
	We believe this requirement can be significantly weakened, particularly for ordinary branching mechanisms.
	Indeed, when the pairwise interaction is given by the pure coalescing (\(q_1 = 1\)), \cite{MRBMS24} extends the duality result of \cite{MR1813840}*{Theorem 1} without imposing any moment conditions on the ordinary branching.
	Remarkably, their approach even accommodates cases where ordinary branching can produce infinitely many offspring with positive probability.
	\begin{old}
		Exploring the optimal moment conditions on the branching mechanisms necessary for our results to hold is an interesting question, but lies beyond the scope of the current paper.
	\end{old}
	However, extending the novel framework of \cite{MRBMS24} to the more general setting is highly non-trivial. The arguments in \cite{MRBMS24} rely heavily on the monotonicity inherent to the pure coalescing mechanism ($q_1=1$). For the more general self-catalytic branching ($q_1 \neq 1$), this monotonicity is lost. Overcoming this lack of monotonicity to establish the optimal moment conditions for the duality result (Proposition \ref{prop:D}) is an interesting problem, but lies beyond the scope of the current paper.

	\subsubsection{Coming down from infinity}
	The CDI property describes how certain observables in a time-homogeneous dynamical system, starting from an infinite value, become finite immediately after the system starts to evolve.
	For example, the solution to the ordinary differential equation
	\begin{equation}
		\begin{cases}
			\frac{\mathrm{d}}{\mathrm{d} t} x(t) = -x(t)^2, \quad t > 0, \\
			x(0) = \infty,
		\end{cases}
	\end{equation}
	which is explicitly given by \(x(t) = 1/t\), exhibits the CDI property.
	
	Recently, the CDI phenomenon has also been observed across various stochastic dynamical systems. A well-known example is Kingman's coalescent \cites{MR0671034, MR0633178, MR1673235}, where every pair of particles coalesces according to independent exponential clocks.
	Other examples include the $\Lambda$-coalescent, which generalizes Kingman's coalescent by allowing simultaneous mergers of multiple particles \cite{MR1736720}, \cite{MR2599198}; the spatial $\Lambda$-coalescent, where particles perform continuous-time independent random walks on a lattice, and particles occupying the same site coalesce according to the usual $\Lambda$-coalescent \cite{MR2223040}, \cite{MR2892958}; the branching process with pairwise interactions (BPI), which generalizes the Galton-Watson process by incorporating pairwise branchings \cite{MR4929072}; the logistic continuous-state branching processes (logistic CSBP), which is the continuous-state analog of the BPI where the interaction is competitive \cite{MR2134113}, \cite{MR3940763}; the non-linear CSBP, which generalizes the logistic CSBP by allowing more complex density-dependent interactions \cite{MR3983343}, \cite{MR4807290}; the time-changed L\'evy processes, which is closely related to the non-linear CSBP \cite{MR4255235}, \cite{MRbaguley2024structure}; branching random walk with non-local competition \cite{MR4751871};  and, last but not least, the dynamical $\Phi^4_3$ model, which formally solves the 3-dimensional singular SPDE
	\(
	\partial_t \rho = \frac{\Delta}{2} \rho - \rho^3 + m \rho + \dot \xi,
	\)
	where \(\dot \xi\) denotes the space-time white noise \cite{MR3719541}.

	\subsubsection{The CDI profile equation v.s.~the Mean-field equation}
	\label{sec:MFE}
	In Theorem 1.4, we characterized the CDI rate of SBBM using its CDI profile equation \eqref{PDE}.
	To better understand the origin of this equation from system's dynamics, it is helpful to consider the mean-field approximation of the SBBM.
	Broadly speaking, the Mean-Field Equation (MFE) approximates the behavior of a high-dimensional random particle system by averaging over its degrees of freedom. For the SBBM, the degrees of freedom include:
	\begin{itemize}
		\item the movement of the particles;
		\item the occurrence time and the number of children of the ordinary branchings;
		\item the occurrence time and the number of children of the catalytic branchings.
	\end{itemize}
\begin{old}
		We call the following equation
		\begin{equation} \label{eq:TMFE}
			\partial_t \tilde v_t(x) = \frac{\Delta}{2} \tilde v_t(x) + \Phi'(0+) \tilde v_t(x) -  \frac{\Psi'(0+)}{2} \tilde v_t(x)^2,
		\end{equation}
		subjected to the initial condition similar to that of the equation \eqref{PDE},
		the MFE of the SBBM model, with the idea that  $\tilde v_t(x) \mathrm{d} x$  is an approximation of the empirical measure, and that the three terms on the right-hand side of \eqref{eq:TMFE} are the mean-field averages of the three groups of randomness \eqref{MP}--\eqref{CB}, respectively.

		Note that, in Theorem \ref{Comming-down-finite-first-moment}, the rate of CDI is given by the solution to the equation \eqref{PDE} instead of \eqref{eq:TMFE}, where the linear term corresponding to the ordinary branching is absent.
		This is fine, since one can verify that, uniformly in \(x \in \mathbb{R}\), \(v_t(x)\) and \(\tilde v_t(x)\) are asymptotically equivalent as \(t \downarrow 0\).
		By that, we mean
		\(
		\lim_{t \to 0} \sup_{x \in \mathbb{R}} | v_t(x)/\tilde v_t(x) - 1 | = 0.
		\)
	\end{old}
	The MFE for the SBBM model is given by the following reaction-diffusion equation:
	\begin{equation} \label{eq:TMFE}
		\partial_t \tilde v_t(x) = \frac{\Delta}{2} \tilde v_t(x) - \Phi'(0+) \tilde v_t(x) -  \frac{\Psi'(0+)}{2} \tilde v_t(x)^2,
	\end{equation}
	subject to the same initial condition as \eqref{PDE} where
	\begin{align}\label{eq:Phi'}
		\Phi'(0+) & := \beta_{\mathrm o}\pr{1-\sum_{k=0}^\infty kp_k }.
	\end{align}
	Here, the three terms on the right-hand side of \eqref{eq:TMFE} correspond to the mean-field averages of the diffusion, ordinary branching, and catalytic branching, respectively.

	One might wonder why, in Theorem \ref{thm:CDIrate}, the exact rate of CDI is given by the solution of the CDI Profile Equation \eqref{PDE} instead of the MFE \eqref{eq:TMFE}, as the linear term corresponding to the ordinary branching is absent in the former. However, this omission is ultimately inconsequential: one can verify that, uniformly in $x \in \mathbb{R}$, $v_t(x)$ and $\tilde v_t(x)$ are asymptotically equivalent as $t \downarrow 0$.
	To be precise, we have the following lemma whose proof is given in the Appendix \ref{sec:AA}.
	\begin{lemma}\label{lem:AE}
		Let $\tilde \Lambda$ be a closed subset of $\mathbb R$ and $\mu$ a Radon measure on $\tilde \Lambda^{\mathrm c}$.
		Suppose that $v^{(\tilde \Lambda, \tilde \mu)}\in \mathcal C^{1,2}((0,\infty)\times \mathbb R)$ is the unique solution to \eqref{PDE} and $\tilde v_t^{(\tilde \Lambda, \tilde \mu)}\in \mathcal C^{1,2}((0,\infty)\times \mathbb R)$ is the unique solution to \eqref{eq:TMFE} with the same initial condition as in \eqref{PDE}. Then,
		\begin{equation}
			\lim_{t\downarrow 0} \sup_{x\in \mathbb R} \abs{\frac{\tilde v_t^{(\tilde \Lambda, \tilde \mu)}(x)}{v^{(\tilde \Lambda, \tilde \mu)}_t(x)} - 1} = 0.
		\end{equation}
	\end{lemma}
	As a corollary of Lemma \ref{lem:AE}, we see that Theorem \ref{thm:CDIrate} still holds with the solution to the CDI profile equation $v^{(\Lambda, \mu)}$ being replaced by the solution to the mean-field equation $\tilde v^{(\Lambda, \mu)}$.
	In other word, the ordinary branching is dominated by the catalytic branching in the CDI asymptotic regime ($t \downarrow 0$), and only the latter contributes to the leading-order profile.

	\subsubsection{Proof Strategy: Duality} \label{sec:PSD}
	
	In general, we say a Markov process $(X_t)_{t \geq 0}$, with state space $\mathbb{X}$ and transition kernel $(P_t)_{t \geq 0}$, and a Markov process $(Y_t)_{t \geq 0}$, with state space $\mathbb{Y}$ and transition kernel $(Q_t)_{t \geq 0}$, satisfy a dual relationship w.r.t.~a given dual function $H: \mathbb{X} \times \mathbb{Y} \to \mathbb{R}$, if
	\[
	\int_{\mathbb{X}} H(x, y_0) P_t(x_0, \mathrm{d}x)
	=\int_{\mathbb{Y}} H(x_0, y) Q_t(y_0, \mathrm{d}y), \quad x_0 \in \mathbb{X}, y_0 \in \mathbb{Y}.
	\]
	For example, the one-dimensional standard Brownian motion $(B_t)_{t \geq 0}$ and the solution $(h_t)_{t \geq 0}$ to the heat equation
	\[
	\partial_t h_t(x) = \frac{\Delta}{2}h_t(x), \quad t \geq 0, x \in \mathbb{R},
	\]
	which can be regarded as a (deterministic) Markov process with state space $\mathcal{C}_{\mathrm{b}}(\mathbb{R})$,
	satisfy a dual relationship w.r.t.~the dual function $(B, h) \in \mathbb{R} \times \mathcal{C}_{\mathrm{b}}(\mathbb{R}) \mapsto h(B)$.
	Here, $\mathcal{C}_{\mathrm{b}}(\mathbb{R})$ represents the collection of bounded continuous functions on $\mathbb{R}$.
	
	Our proofs of the main theorems rely on a moment duality between the SBBM $(Z_t)_{t \geq 0}$ and the following stochastic reaction-diffusion equation:
	\begin{equation}\label{DSPDE}
		\partial_t u_t(x) = \frac{\Delta}{2}u_t(x) - \Phi(u_t(x)) + \sqrt{\Psi(u_t(x))} \dot{W}_t(x), \quad t \geq 0, x \in \mathbb{R},
	\end{equation}
	where $\dot{W}$ is a space-time white noise.
	The corresponding dual function $H$ is given by
	\begin{equation}\label{eq:MDP}
		H(u, Z) := \prod_{x \in \mathbb{R}} (1 - u(x))^{Z(\{x\})},
		\quad u \in \mathcal{C}(\mathbb{R}, [0, z^*]), Z \in \mathcal{N},
	\end{equation}
	where $z^* \in [1, 2)$ is a constant determined by $\Psi$.
	The well-definedness of the infinite product in~\eqref{eq:MDP} will be addressed in Lemma~\ref{eq:infinite-product}.
	This duality first appeared in \cite{MR0948717} for the LCBM and was later generalized for a large family of SBBM in \cite{MR1813840}.
	As we will see in this paper, this duality can be further generalized to incorporate infinitely many initial particles.
	We will be more precise about the solution theory of the SPDE \eqref{DSPDE}, the possibly infinite product in \eqref{eq:MDP}, and the generalized version of this moment dual, in Sections \ref{sec:Dual}, \ref{sec:DSPDE}, and \ref{sec:SS5}, respectively.
	We note here that, if there is no catalytic branching, i.e.,~$\beta_{\mathrm{c}} = 0$, then this duality degenerates to McKean's duality \cite{MR0400428} between the BBM and the FKPP equations.
	Furthermore, if there is no ordinary branching, i.e.,~$\beta_{\mathrm{o}} = \beta_{\mathrm{c}} = 0$, this duality degenerates to the trivial duality between Brownian motions and the heat equation.
	
	Our proof strategy extends the approach of \cite{MR4698025} to the general SBBM setting, where the lack of monotonicity necessitates substantial technical modifications.
	The same idea is to compare the above moment dual with a Laplacian dual connecting the MFE \eqref{eq:TMFE} to a super-Brownian motion, which is a measure-valued Markov process whose density evolves according to the SPDE:
	\begin{equation} \label{eq:SPDESBM}
		\partial_t \tilde{u}_t(x) = \frac{\Delta}{2} \tilde{u}_t(x) - \Phi'(0+) \tilde{u}_t(x) + \sqrt{\Psi'(0+) \tilde{u}_t(x)} \dot{W}_t(x), \quad t \geq 0, x \in \mathbb{R},
	\end{equation}
	w.r.t.~the dual function $(\tilde{v}, \tilde{u}) \in \mathcal{C}(\mathbb{R})^2 \mapsto \exp\{-\int \tilde{v}(x)\tilde{u}(x)\mathrm{d}x\}$.
	The super-Brownian motions arise originally as the rescaling limit of the empirical measure of the (near) critical BBM, and its study has expanded into a major area of research over the last few decades.
	For the precise dual connection between the PDE \eqref{eq:TMFE} and the SPDE \eqref{eq:SPDESBM}, see \cite{MR0958288} and \cite{MR1429263}.
	For a modern overview of the super-Brownian motions, we refer our readers to \cite{MR2760602}.
	Note that the coefficients $-\Phi'(0+)\tilde{u}_t(x)$ and $\Psi'(0+)\tilde{u}_t(x)$ in the SPDE \eqref{eq:SPDESBM} can be considered as the linearizations of $-\Phi(\tilde{u}_t(x))$ and $\Psi(\tilde{u}_t(x))$, respectively.
	Thus, the behavior of the solutions to the SPDEs \eqref{DSPDE} and \eqref{eq:SPDESBM} are similar at small times provided that they share the same initial value.
	Because of this, many properties of the dual SPDE \eqref{DSPDE}, including its compact support property among others, can be obtained via analytical tools such as the weak comparison principle, the Feynman-Kac formula, It\^{o}'s formula, and the BDG inequality, etc.
	It is exactly those analytical results on the dual SPDE \eqref{DSPDE} that can be translated via duality into the CDI property of the SBBM.
	
	However, as mentioned above, the lack of monotonicity necessitates several technical innovations.
	For example, the result of Theorem \ref{Comming-down-finite-first-moment} (i) for the LCBM is proved in \cite{MR4698025} via a straightforward coupling argument.
	But here, the proof for the SBBM is much more involved relying on the construction of certain super-martingales.
	Another technical thing is that one cannot take the logarithm of the dual function $H$ in \eqref{eq:MDP} as it is done in \cite{MR4698025} for the sake of a better comparison between the moment dual and the Laplacian dual.
	This is because the function $H$ might take negative values in our general settings.

	\subsubsection{Restriction on the monotone approximation} \label{sec:initial_data_convergence}
	As highlighted in Theorem \ref{thm:existence-of-Z-t}, our results regarding the approximation of the SBBMs with infinitely many initial particles are formulated under the strict assumption that the initial configurations $(\mu_n)_{n\in \mathbb N}$ monotonically increase to the initial trace.
	This methodological constraint is intrinsically necessitated by the topological discontinuity inherent in the moment duality framework.
	If one solely assumes that the sequence $(\mu_n)_{n\in \mathbb N}$ converges m-weakly to the initial trace without imposing the monotonicity condition, it is possible that the limit of the dual functional $\prod_{x\in \mathbb{R}}(1-u(x))^{\mu_n(\{x\})}$ as $n \to \infty$ fail to be unique, fundamentally depending on the specific choice of the approximating sequence $\mu_n$.
	
	To illustrate this obstruction, consider a simplified scenario where the dual field satisfies the local asymptotic behavior $1-u(x) \sim x$ as $x \downarrow 0$, and the target initial trace is $(\{0\}, \mathbf 0)$ which contains a singularity at the origin.
	If one constructs an approximating sequence of measures $\mu_n^{(1)} := n \delta_{1/n}$, the corresponding dual functional evaluates to $(1-1/n)^n$, which converges to $e^{-1}$ as $n \to \infty$.
	Conversely, selecting an alternative approximating sequence $\mu_n^{(2)} := 2n \delta_{1/n}$, which is m-weakly equivalent to $\mu_n^{(1)}$ and converges to the same singular trace, yields a dual functional converging to $e^{-2}$.
	This structural discrepancy reveals that the limit of the dual functional cannot be universally defined under the general m-weak approximation.
	
	By restricting our framework to monotone approximations, the limit of the dual functional becomes well-defined (see Lemma~\ref{Lemma: Duality3}), bypassing the aforementioned discontinuity without necessitating delicate \textit{a priori} regularity estimates on the solutions of the dual SPDEs.

	\subsection*{Acknowledgement}
	
	We want to thank Omer Angle, Clayton Barnes, Zhen-Qing Chen, Louis Fan, Pei-Sen Li, Leonid Mytnik, Fan Yang, Xiaolong Zhang, and Xicheng Zhang for helpful conversations.
	
	Haojie Hou is partially supported by the China Postdoctoral Science Foundation
	\\no.~2024M764112.
	Zhenyao Sun is partially supported by National Key R\&D Program of China no.~2023YFA1010100 and National Natural Science Foundation of China no.~12301173.

	Zhenyao Sun is the corresponding author.

	\section{Non-explosion}\label{ss:3.1}
	
	Let $(x_i)_{i=1}^n$, $\beta_{\mathrm o}$, $(p_k)_{k=0}^\infty$, $\beta_{\mathrm c}$ and $(q_k)_{k=0}^\infty$ be given as in \eqref{eq:IC}--\eqref{eq:CB}.
	Assume that \eqref{asp:FFM} and \eqref{asp:A3} hold.
	As it has been mentioned in Subsection \ref{sec:MR}, an SBBM, following the rules \eqref{eq:R1}--\eqref{eq:R4}, is  a priori  defined only up to its explosion time $\tau_\infty$.

	In this section, we will prove Proposition \ref{prop:WD} which says that the explosion won't really happen.
	{We denote by $I_t$  the collection of unique labels of the particles that is alive at time $t$.}
	For every $t\in [0,\tau_\infty)$ and $\alpha \in I_t$, denote by $X^{\alpha}_t$
	the spatial location of the particle labeled by $\alpha$ at time $t$.
	For every $t\in [0,\tau_\infty)$ and $U\in \mathcal B(\mathbb R)$, denote by $Z_t(U) := |\{\alpha \in I_t: X_t^{\alpha} \in U\}|$ the number of alive particles at time $t$ whose locations belong to $U$.
	Also, define
	\begin{align}\label{Def-of-lambda}
		\lambda_{\mathrm o}:= \beta_{\mathrm o} \sum_{k=0}^\infty kp_k \in  [0,\infty),
		\qquad \lambda_{\mathrm c}:= \beta_{\mathrm c} \sum_{k=0}^\infty kq_k \in  [0,\infty),
	\end{align}
	and note that $\Phi'(0+) = \beta_{\mathrm o} - \lambda_{\mathrm o}$.

	Let us warm up with a simpler version of Proposition \ref{prop:WD} by excluding the ordinary branching.
	That is, we first consider the case that $\beta_{\mathrm o} = 0$.
	For every $t\in [0,\infty]$, define $N_t := |\{ k \in \mathbb N: \tau_k < t\}|$ to be the total number of branchings in the time interval $[0,t)$.
	In particular, $N_\infty$ is the total number of branchings that will ever occur in finite time.
	\begin{lemma} \label{lem:SL}
		Suppose that there is no ordinary branching, i.e.~$\beta_{\mathrm o} = 0$, then $N_\infty < \infty$ almost surely. In particular, $\tau_\infty \geq \tau_{N_\infty +1} = \infty$ almost surely.
	\end{lemma}
	\begin{proof}
		\begin{old}
			In this case, we know that all the  branchings are catalytic branchings.
			It is not hard to see that the process $(Z_{\tau_k}(\mathbb R))_{k=0}^{N_\infty}$ is a Markov chain taking values in $\mathbb Z_+$ stopped at the (possibly infinite)  random step $N_\infty$.
			Observe from \eqref{eq:R1} and \eqref{eq:R4}, the initial value of this Markov chain is $n$, and the corresponding transition matrix is
			\begin{equation}
				P_{i,j} := \mathbf 1_{\{i\geq 2, j \geq i-2\}} q_{j+2- i} + \mathbf 1_{\{i = j = 1\}} + \mathbf 1_{\{i = j = 0\}}, \quad i,j \in \mathbb Z_+.
			\end{equation}
			Since we assumed that $q_2  = 0$,  the only two absorbing states of this Markov chain are $0$ and $1$.
			From \eqref{asp:A1} and the standard theory of the Markov chains, we can verify that a Markov chain with the above transition matrix $P$ will be absorbed by its absorbing states $\{0,1\}$ in finite steps almost surely.
	\end{old}
		Let us consider a filtration $(\mathcal G_k)_{k\in \mathbb N}$ and a $(\mathcal G_k)_{k\in \mathbb N}$-Markov chain $(X_k)_{k\in \mathbb N}$ with transition matrix
		\begin{equation}
			P_{i,j} := \mathbf 1_{\{i\geq 2, j \geq i-2\}} q_{j+2- i} + \mathbf 1_{\{i = j = 1\}} + \mathbf 1_{\{i = j = 0\}}, \quad i,j \in \mathbb Z_+.
		\end{equation}
		Since we assumed that $q_2 =0$, we can verify that $P_{i,i} = 0$ for every integer $i>1$.
		Therefore, $0$ and $1$ are the only two absorbing states of this Markov chain.
		Note that for every $k\in \mathbb N$, almost surely,
		\begin{align}
			& \mathbb E[X_{k+1}| \mathcal G_k] = \begin{cases}
				0, & \quad X_k = 0,
				\\1, &\quad X_k = 1,
				\\ X_k-2+\sum_{j}jq_j ,&\quad X_k >1,
			\end{cases}
			\\&\overset{\eqref{asp:A3}}{\leq} X_k.
		\end{align}
		Therefore, $(X_k)_{k\in \mathbb N}$ is a non-negative $(\mathcal G_k)_{k\in \mathbb N}$-supermartingale.
		By \cite{MR3930614}*{Theorem 4.2.12}, almost surely the process $(X_k)_{k\in \mathbb N}$ converges.
		Combine this with the fact that the process $(X_k)_{k\in \mathbb N}$ takes integer values, we know that the process must eventually stabilize.
		In other word, there are finite random integers $N$ and $m$ such that almost surely $X_k = m$ for every $k\geq N$.
		Recalling that $(X_k)_{k\in \mathbb N}$ is a Makov chain, the random integer $m$ here must take values from the set of the absorbing states $\{0,1\}$.

		Now, since 1-d Brownian motion is recurrent, any pair of particles will eventually meet and undergo catalytic branching almost surely. Therefore, $(Z_{\tau_k}(\mathbb R))_{k=0}^{N_\infty}$ is exactly a Markov chain with the above transition matrix $P$ stopped upon hitting its absorbing states $\{0,1\}$: when there are $i \geq 2$ particles, catalytic branching will eventually occur, removing 2 particles and adding $j$ new particles with probability $q_j$, giving the transition $i \to i-2+j$. The process cannot stop at any state $i \geq 2$ since branching is guaranteed by recurrence; it stops only upon hitting $\{0,1\}$ where no further catalytic branching is possible.
		Since the Markov chain is absorbed by $\{0,1\}$ in finitely many steps almost surely, we have $N_\infty < \infty$ almost surely, as desired.
	\end{proof}
	
	In the rest of this section, let us consider the more general case $\beta_{\mathrm o} \geq 0$.
	Lemma \ref{lem:SL} says that the SBBM can be defined for all finite time if $ \beta_{\mathrm o} = 0$.
	This fact is crucial for a coupling argument below, where we couple two SBBMs together with one of them not allowed to have ordinary branching.
	
	Define $\tau^{\mathrm o}_0 := 0$, and inductively for every $k\in \mathbb Z_+$, let $\tau^{\mathrm o}_{k+1}$ be the earliest occurring time of an ordinary branching after the time $\tau^{\mathrm o}_{k}$ (if $\tau^{\mathrm o}_k = \tau_\infty$, or if $\tau^{\mathrm o}_k<\tau_\infty$ and there is no ordinary branching occurring in the time interval $(\tau^{\mathrm o}_k, \tau_\infty)$, we define $\tau^{\mathrm o}_{k+1} :=\tau_\infty$ by convention.)
	Similarly, denote by $(\tau^{\mathrm c}_k)_{k\in \mathbb Z_+}$ the occurring times of the catalytic branchings.
	
	\begin{lemma}\label{lem:FCB}
		Almost surely, for every $k \in \mathbb Z_+$, there are only finitely many catalytic branchings occurring in the time interval $(\tau^{\mathrm o}_k, \tau^{\mathrm o}_{k+1})$.
	\end{lemma}
	\begin{proof}
		Let us fix an arbitrary $k\in \mathbb Z_+$.  On one hand, almost surely on the event $\{\tau^{\mathrm o}_k = \tau_\infty\}$, the time interval $(\tau^{\mathrm o}_k, \tau^{\mathrm o}_{k+1}) = \emptyset$.
		So obviously, there is no catalytic branching occurring in this empty time interval.
		On the other hand, by the strong Markov property of the Brownian motions, the process $(Z_{t})_{t\in [0,\tau_\infty)}$ can be coupled with a process $(\tilde Z_t)_{t\geq 0}$, which is an SBBM without the ordinary branching, such that almost surely on the event $\{\tau^{\mathrm o}_k < \tau_\infty\}$, $Z_{t+ \tau^{\mathrm o}_k} = \tilde Z_t$ for every $t\in [0, \tau^{\mathrm o}_{k+1}-\tau^{\mathrm o}_k)$.
		From Lemma \ref{lem:SL}, we know that $(\tilde Z_t(\mathbb R))_{t\geq 0}$ jumps only finitely many times. So almost surely on the event $\{\tau^{\mathrm o}_k < \tau_\infty\}$, $(Z_{t}(\mathbb R))_{t\in (\tau^{\mathrm o}_k, \tau^{\mathrm o}_{k+1})}$ jumps only finitely many times.
		The desired result now follows.
	\end{proof}
	
	The following lemma intuitively establishes that a finite-time explosion cannot happen if there are only finitely many ordinary branchings.
	\begin{lemma}
		For every $m\in \mathbb N$, almost surely on the event $\{\tau_m^{\mathrm o} = \tau_\infty\}$, $\tau_m^{\mathrm o}=\tau_\infty = \infty$.
	\end{lemma}
	\begin{proof}
		Almost surely on the event $\{\tau_m^{\mathrm o} = \tau_\infty\}$, there exists an $N \in \mathbb N$ such that $\tau^{\mathrm o}_N < \tau^{\mathrm o}_{N+1} = \tau_\infty$.
		Therefore, almost surely on this event, the time interval $[0,\tau_\infty)$ can be decomposed into finitely many disjoint sub-intervals in the following way:
		\begin{equation} \label{eq:ID}
			[0,\tau_\infty) = \bigcup_{k=0}^{N} [\tau^{\mathrm o}_k,\tau^{\mathrm o}_{k+1}).
		\end{equation}
		Note from Lemma \ref{lem:FCB} that, almost surely on the event $\{\tau_m^\mathrm o = \tau_{\infty}\}$, there are only finitely many catalytic branchings occurring in each of the sub-intervals on the right hand side of \eqref{eq:ID}.
		Therefore, almost surely on the event $\{\tau_m^{\mathrm o} = \tau_\infty\}$, the total number of branchings is the sum of at most $m$ ordinary branchings and finitely many catalytic branchings in each of the finitely many sub-intervals, hence $N_{\tau_\infty} < \infty$, i.e.~ there are only finitely many branchings occurring before the explosion; from how the explosion time $\tau_\infty$ is defined, we must have $\tau_\infty = \infty$, as desired.
	\end{proof}
	
	In the rest of this section, let us fix an arbitrary $m \in \mathbb N$ and define a new process $(Z_t^{\mathrm o,m})_{t\geq 0}$ so that almost surely on the event $\{\tau^{\mathrm o}_m = \tau_\infty\}= \{\tau^{\mathrm o}_m = \tau_\infty=\infty\}$,  $Z_t^{\mathrm o,m} = Z_t$ for every $t\in [0,\infty)$; and almost surely on the event $\{\tau^{\mathrm o}_m < \tau_\infty\}$,
	\begin{equation}
		Z_t^{\mathrm o,m} := \begin{cases}
			Z_t, \quad & t \in [0,\tau^{\mathrm o}_m),
			\\ \tilde Z^{\mathrm o,m}_{t-\tau^{\mathrm o}_m}, \quad & t \in [\tau^{\mathrm o}_m, \infty),
		\end{cases}
	\end{equation}
	where $(\tilde Z^{\mathrm o,m}_{t})_{t\geq 0}$ is  an  SBBM without ordinary branching whose initial value is given by $Z_{\tau^{\mathrm o}_m}$. Loosely speaking, $(Z^{\mathrm o,m}_{t})_{t\geq 0}$
	is an SBBM with at most $m$ many ordinary branchings allowed.
	
	\begin{lemma} \label{lem:LSM}
		The process $(e^{\Phi'(0+)(t\wedge \tau^{\mathrm o}_m)} Z^{\mathrm o,m}_{t}(\mathbb R))_{t\geq 0}$ is a local super-martingale where $\Phi'(0+)$ is given as in \eqref{eq:Phi'}.
	\end{lemma}
	\begin{proof}
		Consider the following disjoint decomposition of the time interval $[0,\infty)$:
		\begin{equation} \label{eq:TI}
			[0,\infty) = \pr{ \bigcup_{k=0}^{m-1} [\tau^{\mathrm o}_k, \tau^{\mathrm o}_{k+1} ) } \bigcup [\tau^{\mathrm o}_m, \infty).
		\end{equation}
		Almost surely, in each of the time intervals on the right hand side of \eqref{eq:TI}, we can verify that the integer-valued process $Z^{\mathrm o,m}_\cdot(\mathbb R)$ jumps only finitely many times.
		(To see this, we apply Lemma \ref{lem:FCB} to the intervals $[\tau^{\mathrm o}_k, \tau^{\mathrm o}_{k+1} )$ with $k \in \{0, \dots, m-1\}$, and apply Lemma \ref{lem:SL} for the last interval $[\tau^{\mathrm o}_m, \infty)$.)
		Therefore, $Z^{\mathrm o,m}_\cdot(\mathbb R)$ jumps only finitely many times in the full time interval $[0,\infty)$.
		
		Now, we have the finite sum decomposition
		\begin{equation}\label{eq:TZD}
			Z^{\mathrm o,m}_t(\mathbb R)  - n = \sum_{s\in (0,t]\cap \{\tau^{\mathrm o}_1, \dots, \tau^{\mathrm o}_m\}} \mathbf{\Delta} Z^{\mathrm o,m}_{s}(\mathbb R) + \sum_{s \in (0,t]\setminus \{\tau^{\mathrm o}_1, \dots, \tau^{\mathrm o}_m\}} \mathbf{\Delta} Z^{\mathrm o,m}_{s}(\mathbb R), \quad t\geq 0, \mathrm{a.s.}
		\end{equation}
		where $\mathbf{\Delta} \gamma_t := \gamma_t - \gamma_{t-}$ for every $t\geq 0$ and real-valued c\`adl\`ag process $(\gamma_t)_{t\geq 0}$.
		The first term on the right hand side of \eqref{eq:TZD} are the jumps induced by the ordinary branchings, and the second term are the jumps induced by the catalytic branchings.
		From how those jumps are induced, we can find their compensators.
		In fact, it is not hard to see that the processes:
		\begin{equation}
			M^{\mathrm o, m}_t
			:=  \sum_{s\in (0,t]\cap \{\tau^{\mathrm o}_1, \dots, \tau^{\mathrm o}_m\}} \mathbf{\Delta} Z^{\mathrm o,m}_{s}(\mathbb R) +\Phi'(0+) \int_0^{t\wedge \tau_m^{\mathrm o}} Z^{\mathrm o,m}_s(\mathbb R)  \mathrm ds, \quad t\geq 0,
		\end{equation}
		and
		\begin{equation}
			M^{\mathrm c, m}_t
			:= \sum_{s\in (0,t]\setminus \{\tau^{\mathrm o}_1, \dots, \tau^{\mathrm o}_m\}} \mathbf{\Delta} Z^{\mathrm o,m}_{s}(\mathbb R) +
			\frac{1}{2} \Psi'(0+) L^{\mathrm o,m}_t,
			\quad t\geq 0,
		\end{equation}
		are local-martingales (c.f.~\cite{MRBMS24}*{Lemma 3.3}.)
		Here, $L^{\mathrm o,m}_t$ is the total amount of intersection local times induced by all the unordered pairs of atoms in the process of atomic measures $Z^{\mathrm o,m}_{\cdot}$ up to the time $t\geq 0$. We emphasize that $L^{\mathrm o,m}_t$ is well-defined as an additive functional since $(Z^{\mathrm{o},m}_t)$ undergoes at most $m$ ordinary branchings, and by Lemma \ref{lem:FCB} only finitely many catalytic branchings occur in between, ensuring the total number of particles born up to any finite time $t$ is almost surely finite.
		Also, $\Psi'(0+)>0$ according to \eqref{asp:A3}.
		Substituting these compensators back into the decomposition \eqref{eq:TZD}, we obtain almost surely,
		\begin{equation}
			Z^{\mathrm o,m}_t(\mathbb R)  - n
			= M^{\mathrm o, m}_t + M^{\mathrm c, m}_t  - \Phi'(0+) \int_0^{t\wedge \tau_m^{\mathrm o}}Z^{\mathrm o,m}_s(\mathbb R)  \mathrm ds
			-\frac{1}{2}\Psi'(0+) L^{\mathrm o,m}_t, \quad t\geq 0.
		\end{equation}
		Now, almost surely for every $t\geq 0$,
		\begin{align}
			& e^{\Phi'(0+) (t\wedge \tau^{\mathrm o}_m)} Z^{\mathrm o,m}_{t}(\mathbb R)  - n
			\\&= \int_0^{t} e^{\Phi'(0+)(s\wedge  \tau^{\mathrm o}_m) } \mathrm d Z^{\mathrm o,m}_s(\mathbb R) + \Phi'(0+)  \int_0^{t\wedge \tau^{\mathrm o}_m} Z^{\mathrm o,m}_s(\mathbb R)  e^{\Phi'(0+)s} \mathrm ds
			\\& \label{eq:MD}=  \int_0^{ t } e^{\Phi'(0+) (s\wedge  \tau^{\mathrm o}_m)} \pr{\mathrm{d}  M^{\mathrm o, m}_s + \mathrm{d}  M^{\mathrm c, m}_s}  - \frac{1}{2}\Psi'(0+)\int_0^{t} e^{\Phi'(0+) (s\wedge  \tau^{\mathrm o}_m)}\mathrm{d}  L^{\mathrm o,m}_s.
		\end{align}
		Note that the first term on the right hand side of \eqref{eq:MD} is a local-martingale, while the second term is a non-increasing process with locally finite variation. The desired result follows.
	\end{proof}
	
	For every $t\geq 0$, define a random variable $N_{t}^{\mathrm o, m} := \abs{\brc{ k\in \mathbb N: \tau^{\mathrm o}_k < t, k\leq m }}$ which is the total number of ordinary branchings for the process $(Z^{\mathrm o,m}_t)_{t\geq 0}$ in the time interval $[0,t)$.
	
	\begin{lemma} \label{lem:Not}
		For every $t\geq 0$, $\mathbb E[N_{t}^{\mathrm o, m}] \leq n \beta_{\mathrm o}\int_0^{t}  e^{-\Phi'(0+) s}  \mathrm ds$.
	\end{lemma}
	\begin{proof}
		From Lemma \ref{lem:LSM}, the process $(e^{\Phi'(0+)(t\wedge \tau^{\mathrm o}_m)} Z^{\mathrm o,m}_{t}(\mathbb R))_{t\geq 0}$ is a local super-martingale.
		Let $(\rho_k)_{k\in \mathbb N}$ be a sequence of stopping times increasingly converging to $\infty$ such that $(e^{\Phi'(0+)(t\wedge \rho_k \wedge \tau^{\mathrm o}_m)} Z^{\mathrm o,m}_{t\wedge \rho_k}(\mathbb R))_{t\geq 0}$ is a (true) super-martingale for every $k\in \mathbb N$. By Fatou's lemma, for every $t\geq 0$,
		\begin{align}
			& \mathbb E\brk{e^{\Phi'(0+)(t\wedge \tau^{\mathrm o}_m)} Z^{\mathrm o,m}_{t}(\mathbb R)}=\mathbb E\brk{ \lim_{k\uparrow \infty} e^{\Phi'(0+) (t\wedge \rho_k \wedge \tau^{\mathrm o}_m)} Z^{\mathrm o,m}_{t\wedge \rho_k}(\mathbb R)}
			\\&\leq \liminf_{k\to \infty} \mathbb E\brk{ e^{\Phi'(0+) (t\wedge \rho_k \wedge \tau^{\mathrm o}_m)} Z^{\mathrm o,m}_{t\wedge \rho_k}(\mathbb R)} \le  n.
		\end{align}
		In particular, for every $t\geq 0,$
		\begin{equation}
			\mathbb E\brk{\int_0^{t\wedge \tau_m^{\mathrm o}} Z^{\mathrm o,m}_s(\mathbb R)  \mathrm ds}
			\leq \int_0^{t} e^{-\Phi'(0+) s}\mathbb E\brk{e^{\Phi'(0+)(s\wedge \tau^{\mathrm o}_m)} Z^{\mathrm o,m}_s(\mathbb R)}  \mathrm ds
			\leq n \int_0^{t}  e^{-\Phi'(0+) s}  \mathrm ds.
		\end{equation}
		Note that the following holds almost surely for every $t\geq 0$:
		\begin{equation}
			N_t^{\mathrm o, m} \leq N^{\mathrm o, m}_{t+} = \sum_{s\in (0,t] \cap \{\tau^{\mathrm o}_1, \dots, \tau^{\mathrm o}_m\}} 1.
		\end{equation}
		From how the ordinary branchings are induced, we can find the compensator for the process $(N^{\mathrm o, m}_{t+})_{t\geq 0}$.
		In fact, it is not hard to see that the process
		\begin{equation}
			\hat N^{\mathrm o, m}_{t+} : = N^{\mathrm o, m}_{t+} - \int_0^{t\wedge \tau_m^{\mathrm o}}  \beta_{\mathrm o} Z^{\mathrm o,m}_s(\mathbb R)  \mathrm ds, \quad t\geq 0,
		\end{equation}
		is a (true) martingale (c.f.~\cite{MRBMS24}*{Lemma 3.3}.)
		Now we have
		\begin{align}
			& \mathbb E\brk{N_{t}^{\mathrm o, m}} \leq \mathbb E\brk{N^{\mathrm o, m}_{t+}} = \mathbb E\brk{\int_0^{t\wedge \tau_m^{\mathrm o}}  \beta_{\mathrm o} Z^{\mathrm o,m}_s(\mathbb R)  \mathrm ds}
			\leq n \int_0^{t}  \beta_{\mathrm o}  e^{-\Phi'(0+) s}  \mathrm ds,
		\end{align}
		as desired.
	\end{proof}
	
	We are now ready to present the proof of Proposition \ref{prop:WD}.
	\begin{proof}[Proof of Proposition \ref{prop:WD}]
		Note that $m$ is arbitrarily chosen, and by Lemma \ref{lem:Not},
		\begin{equation}
			\mathbb E\brk{\abs{\brc{ k\in \mathbb N: \tau^{\mathrm o}_k < t, k\leq m }}} \leq n \int_0^{t}  \beta_{\mathrm o}  e^{-\Phi'(0+) s}  \mathrm ds, \quad t\geq 0.
		\end{equation}
		Taking $m\uparrow \infty$, we obtain from the Monotone Convergence Theorem that
		\begin{equation}
			\mathbb E\brk{\abs{\brc{ k\in \mathbb N: \tau^{\mathrm o}_k < t }}}\leq n \int_0^{t}  \beta_{\mathrm o}  e^{-\Phi'(0+)s}  \mathrm ds, \quad t\geq 0.
		\end{equation}
		In particular, we can define the almost surely finite random variable
		\begin{equation}
			N_{t}^{\mathrm o, \infty} := \abs{\brc{ k\in \mathbb N: \tau^{\mathrm o}_k < t }} < \infty, \quad t\geq 0.
		\end{equation}
		Therefore, almost surely $\tau_\infty \geq \tau^{\mathrm o}_{k}|_{k = N_{t}^{\mathrm o, \infty} + 1} \geq t$ for every $t\geq 0$.
		This implies the desired result.
	\end{proof}
	
	In the rest of this section, we establish a result which will be used later in Section \ref{sec:SS5}.
	As explained in Subsection \ref{sec:MR}, by Proposition \ref{prop:WD}, $(Z_t)_{t\geq 0}$ is a c\`adl\`ag process taking values in $\mathcal N$.
	It is also clear that $Z_t(\mathbb R)< \infty$ almost surely for every $t\geq 0$.

	\begin{proposition} \label{prop:TM}
		Suppose that $g$ is a smooth function with bounded derivatives of all orders.
		Then the process
		\begin{align}\label{eq:Mgt}
			M_t^g & :=e^{\Phi'(0+)t} Z_t(g) - \frac{1}{2}\int_0^t e^{\Phi'(0+)s} Z_s(g'') \mathrm{d} s
			\\&\qquad + \frac{1}{2} \Psi'(0+)\int_0^t e^{\Phi'(0+)s} \sum_{\{\alpha, \beta\}\subset I_s: \alpha \neq \beta}g(X_s^\alpha)\mathrm{d} L^{\{\alpha, \beta\}}_s, \quad t\geq 0
		\end{align}
		is a (true) martingale.
		Here, $(L^{\{\alpha, \beta\}}_t)_{t\geq 0}$, the intersection local time between any two particles labelled by $\alpha$ and $\beta$, is defined as the unique continuous process such that
		\[
		L_t^{\{\alpha, \beta\}}
		:= \lim_{\epsilon \to 0}\frac{1}{\epsilon}\int_0^t \mathbf 1_{\brc{\{\alpha, \beta\} \subset I_s}\cap \brc{|X_s^\alpha - X_s^\beta | \leq \epsilon}} \mathrm{d} s, \quad\text{a.s.,~}t\geq 0.
		\]
	\end{proposition}
	\begin{proof}
		Define $\Theta^{\mathrm o}:=\{\tau^{\mathrm o}_k: k \in \mathbb N\}\cap [0,\infty)$, and $\Theta^{\mathrm c}:= \{\tau_k^{\mathrm c}: k\in \mathbb N\}\cap [0,\infty)$, the set of occurrence times of the ordinary branchings, and the set of occurrence times of the catalytic branchings, respectively.
		From Proposition \ref{prop:WD}, we know that almost surely for every $t\geq 0$, $\Theta^\mathrm o\cap (0,t]$ and $\Theta^\mathrm c\cap(0,t]$ are finite sets.
		Notice that, between each two consecutive branching times, $Z_\cdot$ evolves as the empirical measure of a system of independent Brownian motions.
		Therefore, we have the decomposition that almost surely for every $t\geq 0$,
		\begin{equation} \label{eq:MMA}
			Z_t(g) - Z_0(g)
			=  \tilde m_t^{g'} + \frac{1}{2} \int_0^t Z_s(g'') \mathrm{d} s + \sum_{s\in \Theta^{\mathrm o}\cap(0,t]}\mathbf{\Delta} Z_s( g) + \sum_{s\in \Theta^{\mathrm c}\cap(0,t]} \mathbf{\Delta} Z_s(g)
		\end{equation}
		where $\tilde m_\cdot^{g'}$ is a continuous local martingale with  quadratic variation
		\begin{equation} \label{eq:Dg}
			\chv{\tilde m^{g'}_{\cdot}}_t = \int_0^t Z_s\pr{(g')^2}\mathrm{d} s, \quad t\geq 0.
		\end{equation}
		Here, $\mathbf \Delta \gamma_s:= \gamma_s - \gamma_{s-}$ for any $s\geq 0$ and real-valued c\`adl\`ag process $(\gamma_t)_{t\geq 0}$.
		Let $\mathcal U$ be the (deterministic) countable set of all possible labels of the particles, and let $\mathcal U_2 := \{\{\alpha, \beta\}\subset \mathcal U: \alpha \neq \beta\}$ be the set of all the possible  unordered pairs of labels.
		It is standard to see that
		\begin{equation}
			\sum_{s\in \Theta^{\mathrm o}\cap(0,t]}
			\mathbf{\Delta} Z_s( g)
			= \int_{\mathbb Z_+\times\mathcal U \times (0,t] } (k-1) g(X^\alpha_{s-}) N^\mathrm o(\mathrm{d} k, \mathrm{d} \alpha,\mathrm{d} s),
			\quad t\geq 0, \text{a.s.}
		\end{equation}
		and
		\begin{equation}
			\sum_{s\in \Theta^{\mathrm c}\cap(0,t]}
			\mathbf{\Delta} Z_s( g)
			= \int_{\mathbb Z_+ \times \mathcal U_2 \times (0,t]} (k-2) g(X^\alpha_{s-}) N^\mathrm c(\mathrm{d} k, \mathrm{d} \{\alpha,\beta\},\mathrm{d} s), \quad t\geq 0, \text{a.s.},
		\end{equation}
		where $N^{\mathrm o}$ is a point process on $\mathbb Z_+\times \mathcal U\times \mathbb R_+$ with compensator
		\[
		\hat N^\mathrm o(\{k\}\times \{\alpha\}\times \mathrm{d} s)
		:= \beta_\mathrm o p_k \mathbf 1_{\{\alpha \in I_s\}}\mathrm{d} s,
		\quad (k,\alpha, s) \in \mathbb Z_+\times \mathcal U\times \mathbb R_+,
		\]
		and $N^\mathrm c$ is a point process on $\mathbb Z_+\times \mathcal U_2\times \mathbb R$ with compensator
		\[
		\hat N^\mathrm c(\{k\}\times \{\{\alpha,\beta\}\}\times \mathrm{d} s)
		:= \frac{1}{2}\beta_c q_k \mathrm{d} L^{\{\alpha, \beta\}}_s,
		\quad (k,\{\alpha, \beta\}, s) \in \mathbb Z_+\times \mathcal U_2\times \mathbb R_+,
		\]
		in the sense of \cite{MR1011252}*{Definition 3.1}.
		Now, by It\^o's
		formula \cite{MR1011252}*{Theorem 5.1}, we have almost surely for every $t\geq 0$,
		\begin{align}
			& e^{\Phi'(0+)t}Z_t(g)-Z_0(g)
			\\&= \int_0^t e^{\Phi'(0+)s}\mathrm{d} \tilde m^{g'}_s + \frac{1}{2} \int_0^t e^{\Phi'(0+)s} Z_s(g'')\mathrm{d} s +  \Phi'(0+)\int_0^t Z_s(g) e^{\Phi'(0+)s}\mathrm{d} s
			\\& \qquad +\int_{\mathbb Z_+\times \mathcal U \times (0,t]} e^{\Phi'(0+)s}(k-1)g(X_{s-}^\alpha) N^\mathrm o(\mathrm{d} k, \mathrm{d} \alpha, \mathrm{d} s)
			\\&\qquad + \int_{\mathbb Z_+ \times \mathcal U_2 \times (0,t]} e^{\Phi'(0+)s}(k-2) g(X_{s-}^\alpha) N^{\mathrm c}(\mathrm{d} k, \mathrm{d} \{\alpha, \beta\}, \mathrm{d} s).
		\end{align}
		Observe that $(M_t^g)_{t\geq 0}$ is a local martingale, since almost surely for every $t\geq 0$,
		\begin{equation} \label{eq:LMg}
			M_t^g = Z_0(g) + m^{g'}_t + m^{\mathrm o, g}_t + m_t^{\mathrm c, g}
		\end{equation}
		where
		\begin{equation}
			m^{g'}_t:=\int_0^t e^{\Phi'(0+)s} \mathrm{d} \tilde m^{g'}_s,
		\end{equation}
		\begin{equation}
			m_t^{\mathrm o,g}
			:= \int_{\mathbb Z_+\times \mathcal U \times (0,t]} e^{\Phi'(0+)s}(k-1)g(X_{s-}^\alpha) \pr{N^\mathrm o(\mathrm{d} k, \mathrm{d} \alpha, \mathrm{d} s)- \hat N^\mathrm o(\mathrm{d} k, \mathrm{d} \alpha, \mathrm{d} s)},
		\end{equation}
		and
		\begin{equation}
			m_t^{\mathrm c,g}
			:=\int_{\mathbb Z_+ \times \mathcal U_2 \times (0,t]} e^{\Phi'(0+)s}(k-2) g(X_{s-}^\alpha) \pr{ N^{\mathrm c}(\mathrm{d} k, \mathrm{d} \{\alpha, \beta\}, \mathrm{d} s)-\hat N^\mathrm c(\mathrm{d} k, \mathrm{d} \{\alpha, \beta\}, \mathrm{d} s)}.
		\end{equation}
		Now, replacing $g$ by $\mathbf 1_{\mathbb R}$ in \eqref{eq:Mgt}, from the fact that $\Psi'(0+)> 0$, it is clear from \cite{MR4226142}*{Theorem 10.5} that  $(e^{\Phi'(0+)t}Z_t(\mathbb R))_{t\geq 0}$ is a local super-martingale.
		Therefore, by the definition of local super-martingales \cite{MR4226142}*{p.~211}, there exists a sequence of optional times $(\rho_k)_{k\in \mathbb N}$ converging increasingly to $\infty$ such that $(e^{\Phi'(0+)t\wedge \rho_k}Z_{t\wedge \rho_k}(\mathbb R))_{t\geq 0}$ is a super-martingale for every $k\in \mathbb N$. By Fatou's lemma, for every $t\geq 0,$
		\begin{align}\label{eq:MU1}
			\mathbb E\brk{e^{\Phi'(0+)t}Z_t(\mathbb R)}
			\leq \liminf_{k\to \infty} \mathbb E\brk{e^{\Phi'(0+){t\wedge \rho_k}}Z_{t\wedge \rho_k}(\mathbb R)}
			\leq Z_0(\mathbb R) = n .
		\end{align}
		There also exists a sequence of optional times $(\rho'_k)_{k\in \mathbb N}$ converging increasingly to $\infty$ such that $(M^{\mathbf 1_{\mathbb R}}_{t\wedge \rho'_k})_{t\geq 0}$ is a martingale for every $k\in \mathbb N$.
		Taking expectation on the both sides of \eqref{eq:Mgt} while replacing $g$ by $\mathbf 1_{\mathbb R}$ and $t$ by $t\wedge \rho'_k$, we obtain that for every $k\in \mathbb N$ and $t\geq 0$,
		\begin{align}
			&n
			= \mathbb E\brk{M^{\mathbf 1_{\mathbb R}}_{t\wedge \rho'_k}}
			\\&=\mathbb E\brk{e^{\Phi'(0+)t\wedge \rho'_k}Z_{t\wedge \rho'_k}(\mathbb R)} + \frac{1}{2} \Psi'(0+) \mathbb E\brk{\int_0^{t\wedge \rho'_k} e^{\Phi'(0+)s}\sum_{\{\alpha,\beta\}\subset I_s: \alpha \neq \beta} \mathrm{d} L^{\{\alpha, \beta\}}_s}.
		\end{align}
		Taking $k\uparrow \infty$, from the monotone convergence theorem, we have for any $t\geq 0,$
		\begin{equation} \label{eq:MU2}
			\mathbb E\brk{\int_0^{t} e^{\Phi'(0+)s}\sum_{\{\alpha,\beta\}\subset I_s: \alpha \neq \beta} \mathrm{d} L^{\{\alpha, \beta\}}_s} \leq 2 n/\Psi'(0+).
		\end{equation}
		From \eqref{eq:MU1}, we can verify that
		\begin{equation} \label{eq:TM1}
			\mathbb E\brk{\int_0^t e^{2\Phi'(0+)s} Z_s\pr{(g')^2} \mathrm{d} s}
			< \infty,
		\end{equation}
		\begin{equation} \label{eq:TM2}
			\mathbb E\brk{\int_{\mathbb Z_+\times \mathcal U \times (0,t]} \abs{e^{\Phi'(0+)s}(k-1) g(X_{s-}^\alpha) } \hat N^{\mathrm o}(\mathrm{d} k, \mathrm{d} \alpha, \mathrm{d} s)} < \infty,
		\end{equation}
		and from \eqref{eq:MU2} that
		\begin{equation} \label{eq:TM3}
			\mathbb E\brk{\int_{\mathbb Z_+\times \mathcal U_2 \times (0,t]} \abs{e^{\Phi'(0+)s}(k-2) g(X_{s-}^\alpha) } \hat N^{\mathrm c}(\mathrm{d} k, \mathrm{d} \{\alpha,\beta\}, \mathrm{d} s)} < \infty.
		\end{equation}
		From \eqref{eq:TM1}, \eqref{eq:TM2} and \eqref{eq:TM3} we can verify that $(m_t^{g'})_{t\geq 0}$, $(m^{\mathrm o,g}_t)_{t\geq 0}$ and $(m^{\mathrm c, g}_t)_{t\geq 0}$ are (true) martingales, respectively.
		The desired result of this proposition follows.
	\end{proof}

	\section{The dual SPDEs}
	\label{sec:Dual}
	Let the parameters $(x_i)_{i=1}^n$, $\beta_{\mathrm o}$, $(p_k)_{k=0}^\infty$, $\beta_{\mathrm c}$ and $(q_k)_{k=0}^\infty$ be given as in \eqref{eq:IC}--\eqref{eq:CB}. Assume that \eqref{asp:A3} and \eqref{asp:A1} hold.
	Due to Proposition \ref{prop:WD} (which is proved in the previous section), an SBBM w.r.t.~ above  parameters can be constructed up to all time.
	{Let $(I_t)_{t\geq 0}$, $(X^{\alpha}_{t})_{\alpha \in I_t,t\geq 0}$ and $(Z_t)_{t\geq 0}$ be the corresponding notations  for this SBBM, given as in Subsection \ref{sec:MR} (right after Proposition \ref{prop:WD}.) }
	
	In this section, we discuss the duality relation between this SBBM and the following 1-d
	stochastic partial differential equation (SPDE)
	\begin{equation} \label{eq:GSPDE}
		\begin{cases}\displaystyle
			\partial_t u_t(x) = \frac{\Delta}{2} u_t(x) - \Phi(u_t(x)) + \sqrt{\Psi(u_t(x))} \dot W_{t,x}, \quad t> 0, x\in \mathbb R,
			\\ \displaystyle
			u_0(x) = f(x), \quad x\in \mathbb R,
		\end{cases}
	\end{equation}
	where  $\Phi$ and $\Psi$  are defined as in \eqref{Def-Phi} and \eqref{Def-Psi} respectively, and $\dot W$ is a space-time white noise  on $[0,\infty)\times \mathbb R$.
	We need to be careful about the solution concept of the SPDE \eqref{eq:GSPDE}.
	In particular, we want the random variable $u_t(x)$, for every $t\geq 0$ and $x\in \mathbb R$, to take its  values in a subinterval of $\mathbb R$ such that $\Psi (u_t(x))$ is non-negative.
	To this end, let us first analyze the function $\Psi (\cdot)$.
	Define
	\[z^* := \inf\{z\in [1,2]: \Psi(z) = 0\}\]
	with the convention that $\inf \emptyset =  \infty$.
	The following analytic lemma suggests that the random field $(u_t(x))_{t\geq 0, x\in \mathbb R}$ should take its value in $[0,z^*]$.
	(We include the proof of this analytic lemma  in Appendix \ref{append-A}.)
	
	\begin{lemma} \label{lem:CC}
		It can be verified that $z^* \in [1,2]$, $\Psi(z^*) = 0$, $\Phi(z^*) \geq 0$, $\Phi(0)=\Psi(0)= 0$, and $\Psi(z)\geq 0$ for every $z\in [0,z^*]$.
		Furthermore, if \eqref{asp:A2} holds then $z^* < 2$.
	\end{lemma}
	
	We now give the solution concept of SPDE \eqref{eq:GSPDE}.
	Denote by $\mathcal C(\mathbb R, [0, z^*])$ the collection of $[0,z^*]$-valued continuous functions on $\mathbb R$, equipped with the topology of uniform convergence on compact sets.
	Let $f$, the initial value of the SPDE \eqref{eq:GSPDE}, be an arbitrary  element of $\mathcal C(\mathbb R, [0,z^*])$.
	We say $(\tilde \Omega, \tilde{\mathcal F}, (\tilde{\mathcal F}_t)_{t\geq 0}, \tilde{\mathbb P}_f, W)$ is a stochastic basis, if $(\tilde \Omega, \tilde{\mathcal F}, \tilde{\mathbb P}_f)$ is a complete probability space equipped with an augmented filtration $(\tilde {\mathcal F}_t)_{t\geq 0}$ in the sense of \cite{MR4226142}*{Lemma 9.8}, and $W:=(W_t(\phi): t\geq 0, \phi\in L^2(\mathbb R))$ is an $(\tilde {\mathcal F}_t)_{t\geq 0}$-adapted cylindrical Wiener process on $L^2(\mathbb R)$ with covariance structure \[\tilde{\mathbb E}_f[W_t(\phi)W_s(\psi)]=(t\wedge s) \int \phi(x)\psi(x)\mathrm{d} x, \quad t,s\geq 0,\phi,\psi \in L^2(\mathbb R)\] in the sense of \cite{MR3236753}*{Section 4.1.2.}.
	Given a stochastic basis $(\tilde \Omega, \tilde{\mathcal F}, (\tilde{\mathcal F}_t)_{t\geq 0}, \tilde{\mathbb P}_f, W)$, we say an $(\tilde{\mathcal F}_t)_{t\geq 0}$-adapted  $\mathcal C(\mathbb R, [0,z^*])$-valued continuous process $(u_{t})_{t\geq 0}$ solves the SPDE \eqref{eq:GSPDE}, if $u_0 = f$ and for every $(t,x)\in (0,\infty)\times \mathbb R$,  almost surely
	\begin{align}\label{eq:mild}
		u_{t}(x)
		& = \int p_{t}(x-y)f(y)\mathrm dy - \iint_0^t p_{t-s}(x-y) \Phi (u_s(y))\mathrm ds \mathrm dy + {}
		\\&\qquad \iint_0^t p_{t-s}(x-y) \sqrt{\Psi(u_s(y))} W(\mathrm ds\mathrm dy).
	\end{align}
	Here, $p_{t}(x):= (2\pi t)^{-1/2}e^{-x^2/(2t)}$ for $ t>0, x\in \mathbb{R}$
	is the heat kernel, and the third term on the right hand side of \eqref{eq:mild}
	is the stochastic integral driven by the space-time white noise (see \cite{MR0876085} or equivalently \cite{MR3236753}*{Section 4.2.1.}).
	In particular, for any $(\tilde{\mathcal F}_t)_{t\geq 0}$-predictable $L^2(\mathbb R)$-valued process $(H_t)_{t\geq 0}$ satisfying that almost surely $\int_0^t \|H_s\|^2_{L^2(\mathbb R)} \mathrm{d} s<\infty$ for every $t\ge 0$, the stochastic integral
	$(\iint_0^t H_s(y) W(\mathrm{d} s\mathrm{d} y))_{t\geq 0}$ is an $(\tilde{\mathcal F}_t)_{t\geq 0}$-adapted continuous local martingale with quadratic variation $(\int_0^t \|H_s\|^2_{L^2(\mathbb R)}\mathrm{d} s)_{t\geq 0}$.
	Equation \eqref{eq:mild} is also known as the mild form of the SPDE \eqref{eq:GSPDE}.
	
	Let us be more precise about the existence of the solutions.
	In this paper, we will be  only  considering the weak existence.
	By that, we mean the existence of a stochastic basis $(\tilde \Omega, \tilde{\mathcal F}, (\tilde{\mathcal F}_t)_{t\geq 0}, \tilde{\mathbb P}_f, W)$, and an $(\tilde{\mathcal F}_t)_{t\geq 0}$-adapted  $\mathcal C(\mathbb R, [0,z^*])$-valued continuous process $(u_{t})_{t\geq 0}$ solving the SPDE \eqref{eq:GSPDE}.
	
	Let us also be more precise about the uniqueness of the solutions.
	We will be only considering the  uniqueness in law.
	We say the uniqueness in law  holds for the SPDE \eqref{eq:GSPDE} if any two solutions sharing the same initial value,  but not necessarily the same stochastic basis,  induce the same law  on  the path space $\mathcal C([0,\infty), \mathcal C(\mathbb R, [0, z^*]))$.
	
	\begin{lemma}
		The weak existence holds for the SPDE \eqref{eq:GSPDE}.
	\end{lemma}
	
	\begin{proof}
		Thanks to Lemma \ref{lem:CC}, the existence of the SPDE \eqref{eq:GSPDE} is standard.
		See  \cite{MR1271224}*{Theorem 2.6} and \cite{MR4259374}*{Section 2.1} for example.
	\end{proof}

	The uniqueness in law  for  the SPDE \eqref{eq:GSPDE} also holds. To show this, let us first give the moment duality relation between the SPDE \eqref{eq:GSPDE} and the SBBM.

	\begin{proposition} \label{prop:D}
		Suppose that  $\mathcal C(\mathbb R,[0, z^*])$-valued continuous process $(u_t)_{t\geq 0}$ is a  solution to the SPDE \eqref{eq:GSPDE}  w.r.t.~ a
		stochastic basis $(\tilde{\Omega}, \tilde{\mathcal F}, (\tilde{\mathcal F}_t)_{t\geq 0}, \tilde{\mathbb P}_f, W)$.
		Then it holds for every $t \geq 0$ that
		\begin{equation} \label{eq:Duality}
			\tilde{\mathbb E}_f\brk{\prod_{i=1}^n \pr{1-u_t(x_i)}}
			= \mathbb E \brk{ \prod_{\alpha \in {I_t}} \pr{1-f({X_t^{\alpha}})} }.
		\end{equation}
	\end{proposition}
	
	\begin{proof}
		In \cite{MR1813840}, Athreya and Tribe considered the 
		the SPDEs
		\begin{equation} \label{eq:ATSPDE}
			\partial_t w_{t}(x) = \frac{1}{2}\Delta w_t(x) + b(w_t(x)) + \sqrt{\sigma(w_t(x))} \dot \xi_{t,x}, \quad t\geq 0, x\in \mathbb R
		\end{equation}
		where $\dot \xi$ is a space-time white noise on $[0,\infty)\times\mathbb R$,
		\begin{equation} \label{eq:b}
			b(z) := \sum_{k=0}^\infty b_k z^k,\quad z\in (-R_b,R_b),
		\end{equation}
		and
		\begin{equation}  \label{eq:sigma}
			\sigma(z) := \sum_{k=0}^\infty \sigma_k z^k, \quad z\in (-R_\sigma, R_\sigma).
		\end{equation}
		Here, $R_b>0$ and $R_\sigma>0$ are the convergence radius for the infinite series
		on the right hand sides of  \eqref{eq:b} and \eqref{eq:sigma}
		respectively.
		To  utilize the result in \cite{MR1813840}, let us take
		\begin{equation} \label{eq:bk}
			b_k :=
			\begin{cases}
				\beta_{\mathrm o}
				p_k, & \quad k\in \mathbb Z_+\setminus \{1\},
				\\ -
				\beta_{\mathrm o},
				& \quad k = 1,
			\end{cases}
			\quad \mbox{and}\quad
			\sigma_k:= \begin{cases}
				\beta_{\mathrm c} q_k, & \quad k \in \mathbb Z_+ \setminus \{2\},
				\\ -\beta_{\mathrm c}, & \quad k = 2.
			\end{cases}
		\end{equation}
		In this way, it is clear from \eqref{asp:A1} that $R_b > 1$ and $R_\sigma > 1$.
		One can also verify that the $[1-z^*,1]$-valued continuous random field $(1-u_t(x))_{t\geq 0,x\in \mathbb R}$  solves
		the SPDE \eqref{eq:ATSPDE}.
		In the following, we analyze this particular solution $w := 1 - u$ whose initial value is $w_0 = 1-f$.
		
		Let us assume for the moment that the ordinary offspring law is subcritical:
		\begin{equation} \label{eq:NA}
			\sum_{k=0}^\infty kp_k < 1.
		\end{equation}
		Define
		\[
		\tilde b(z) := \sum_{k\in \mathbb Z_+: z \neq 1} |b_k|z^{k-1}= \beta_{\mathrm o}\sum_{k=0}^\infty p_k z^{k-1}, \quad z\in (-R_\sigma, R_\sigma)\setminus\{0\},
		\]
		and
		\[
		\tilde \sigma(z)
		:=
		\sum_{k\in \mathbb Z_+: z \neq 2} |\sigma_k|z^{k-2} = \beta_{\mathrm c}\sum_{k=0}^\infty q_k z^{k-2}
		, \quad z\in (-R_\sigma, R_\sigma)\setminus\{0\}.
		\]
		Note that in this case,
		\begin{equation} \label{eq:Tbpo}
			\tilde b'(1) = \beta_{\mathrm o}\sum_{k=0}^\infty p_k(k-1) < 0,
			\quad \mbox{and}\quad
			\tilde \sigma'(1) = \beta_{\mathrm c}\sum_{k=0}^\infty q_k(k-2) < 0.
		\end{equation}
		Therefore, the condition (H1) in \cite{MR1813840}*{Theorem 1} holds.
		Moreover, from
		\begin{equation}
			\tilde b(1) = \sum_{k\in \mathbb Z_+: z\neq 1} |b_k| =
			\beta_{\mathrm o}
			\sum_{k=0}^\infty p_k =
			\beta_{\mathrm o},
		\end{equation}
		and \eqref{eq:Tbpo}, there exists  a  $\gamma > 0$ such that
		\begin{equation} \label{eq:H2}
			\tilde b(e^\gamma) <
			\beta_{\mathrm o}
			= -b_1.
		\end{equation}
		Similarly, there exists  a $\gamma '>0$ such that $\tilde \sigma(e^{\gamma'}) < -\sigma_2$.
		These give the condition (H2) in \cite{MR1813840}*{Theorem 1}.
		So, under this additional assumption \eqref{eq:NA}, one obtains from \cite{MR1813840}*{Theorem 1} for every $t\geq 0$ the following identity
		\begin{align}\label{eq:FCDuality}
			& \qquad \tilde{\mathbb E}_f\brk{\prod_{i=1}^n \pr{1-u_t(x_i)}} =\tilde{\mathbb E}\brk{\prod_{i=1}^n w_t(x_i) \middle| w_0 = 1-f }
			\\&= \mathbb E \brk{ (-1)^{K_t} \exp\pr{(\nu+\frac{\sigma_2}{2}){L_t}+ (\mu+b_1)\int_0^t {|I_s|}\mathrm ds}\prod_{\alpha \in {I_t}} \pr{1-f({X_t^{\alpha}})} }
		\end{align}
		where $\mu := \sum_{k\neq 1} |b_k|$, $\nu := \frac{1}{2} \sum_{k\neq 2} |\sigma_k|$, $K_t := K_t^1+K_t^2$, $K^1_t:=$ the number of ordinary branching events whose offspring size belongs to $S_1$ up to time $t$, $K^2_t:=$ the number of catalytic branching events whose offspring size belongs to $S_2$ up to time $t$, $S_1:= \{k \in \mathbb N: b_k < 0, k\neq 1\}$, $S_2:=\{k\in \mathbb N: \sigma_k < 0, k\neq 2\}$ and ${L_t}=$ the total intersection local time accrued by all pairs of particles up to time $t$.
		From how we defined $(b_k)_{k\in \mathbb N}$ and $(\sigma_k)_{k\in \mathbb N}$ in \eqref{eq:bk}, we see that $S_1$ and $S_2$ are empty sets, and therefore $K_t = 0$.
		We also have that $\mu + b_1 =0$ since $\mu = \sum_{k\neq 1}|b_k| = \sum_{k\neq 1}\beta_{\mathrm o} p_k = \beta_{\mathrm o} = -b_1$ where we used the fact that $(p_k)_{k\in \mathbb N}$ is a probability measure on $\mathbb N$ with $p_1 = 0$.
		Similarly, $\nu+\frac{\sigma_2}{2} = 0$.
		Therefore, the alternating term $(-1)^{K_t}$ and the exponential term in the expectation on the right hand side of \eqref{eq:FCDuality} are $1$.
		So, the Feynman-Kac type formula \eqref{eq:FCDuality} degenerates to the desired identity \eqref{eq:Duality} under our particular choice of the coefficients \eqref{eq:bk}.

		Note that  this additional condition
		\eqref{eq:NA} is used
		to deduce \eqref{eq:Tbpo} and \eqref{eq:H2}, which are mainly used in \cite{MR1813840} to prevent the explosion of the dual particle system, and to ensure the finiteness of the expectation of the following  term
		\begin{equation} \label{eq:term}
			\sup_{0\leq t\leq T}\exp\pr{(\nu+\frac{\sigma_2}{2}){L_t} + (\mu+b_1)\int_0^t {|I_s|}\mathrm{d} s},
		\end{equation}
		respectively for every $T>0$.
		See Section 2.2, especially Lemma 3, of \cite{MR1813840}.
		
		For our case when $\sum_{k=0}^\infty k p_k \geq 1$, since the explosion for the dual particle system won't happen by Proposition \ref{prop:WD}, and the term \eqref{eq:term} equals $1$ under our particular choice of the coefficients \eqref{eq:bk},
		one can still verify the desired  identity \eqref{eq:Duality} following the  steps of  \cite{MR1813840}. We omit the details.
	\end{proof}
	
	\begin{lemma}
		The  uniqueness in law holds for the SPDE \eqref{eq:GSPDE}.
	\end{lemma}
	\begin{proof}
		This follows from Proposition \ref{prop:D} and \cite{MR1813840}*{Lemma 1}.
	\end{proof}

	Recall that the initial value $f\in \mathcal C(\mathbb R, [0,z^*])$ of the SPDE \eqref{eq:GSPDE} is chosen arbitrarily.
	Let $\mathscr L_{f}$ represent the law of the unique in law solution $u=(u_t(x))_{t\geq 0, x\in \mathbb R}$ to the SPDE \eqref{eq:GSPDE} on  $\mathcal C([0,\infty), \mathcal C(\mathbb R, [0,z^*]))$.
	The following lemma is  standard, and is  known as the weak comparison principle.
	(We include its proof in Appendix \ref{append-A}.)
	
	\begin{lemma} \label{eq:WC}
		Suppose that $u^{(1)}=(u^{(1)}_t(x))_{t\geq 0, x\in \mathbb R}$ and $u^{(2)}=(u^{(2)}_t(x))_{t\geq 0, x\in \mathbb R}$ are two solutions to the SPDE \eqref{eq:GSPDE} with initial values $f^{(1)}$ and $f^{(2)}$ in $\mathcal C(\mathbb R, [0,z^*])$ respectively.
		Suppose that $f^{(1)}\leq f^{(2)}$ on $\mathbb R$.
		Then $u^{(1)}$ is stochastically dominated by $u^{(2)}$, in the sense that, there exists a probability kernel $\mathscr K_{f^{(1)},f^{(2)}}$ on $\mathcal C([0,\infty), \mathcal C(\mathbb R, [0,z^*]))$ such that, for $\mathscr L_{f^{(1)}}$-a.s.~
		$w^{(1)} \in \mathcal C([0,\infty),\mathcal C(\mathbb R, [0,z^*]))$
		and $\mathscr K_{f^{(1)},f^{(2)}}(w^{(1)}, \cdot)$-a.s.~ $w^{(2)}$, we have \[w^{(1)}_t(x) \leq w^{(2)}_t(x), \quad t\geq 0, x\in \mathbb R;\]
		and that, for any Borel subset $A$ of $\mathcal C([0,\infty), \mathcal C(\mathbb R, [0,z^*]))$,
		\[
		\mathscr L_{f^{(2)}}(A)
		= \int \mathscr K_{f^{(1)},f^{(2)}}(w^{(1)}, A)\mathscr L_{f^{(1)}}(\mathrm{d} w^{(1)}) .
		\]
	\end{lemma}

	For our purpose, we sometimes need the initial value of the SPDE \eqref{eq:GSPDE} to be the non-decreasing limit of a sequence of continuous functions on $\mathbb R$.

	\begin{proposition} \label{prop:GI}
		Let $g$ be a measurable function on $\mathbb R$ which can be approximated by the elements of $\mathcal C(\mathbb R, [0,z^*])$ monotonically from below, i.e.~ there exists a pointwisely non-decreasing sequence $(f^{(m)})_{m=1}^\infty$ in $\mathcal C(\mathbb R, [0,z^*])$ such that $f^{(m)}(x) \uparrow g(x)$ for every $x\in \mathbb R$ as $m\uparrow \infty$.
		Then, there exists a $\mathcal C(\mathbb R, [0,z^*])$-valued continuous process $(u_t)_{t>0}$  with initial value $u_0:=g$, on a probability space whose probability measure will be denoted by $\tilde {\mathbb P}_g$, such that the following two statements hold.
		\begin{itemize}
			\item[\eq \label{eq:S1}] For each $\phi \in \mathcal C_\mathrm c^\infty(\mathbb R)$, almost surely
			\begin{align}
				\int u_t(x) \phi(x) \mathrm{d} x & = \int g(x) \phi(x) \mathrm{d} x +  \iint_0^t u_s(y) \frac{\phi''(y)}{2} \mathrm{d} s \mathrm{d} y  - {}
				\\& \qquad \iint_0^t \Phi(u_s(y)) \phi(y) \mathrm{d} s \mathrm{d} y + M_t^{(\phi)},  \quad t\geq 0,
			\end{align}
			where $(M^{(\phi)}_t)_{t\geq 0}$ is a $(\mathcal G_t)_{t\geq 0}$-adapted  continuous martingale with quadratic variation
			\begin{equation}
				\chv{M_\cdot^{(\phi)}}_t  = \iint_0^t \Psi(u_s(y)) \phi(y)^2 \mathrm{d} s\mathrm{d} y, \quad t\geq 0
			\end{equation}
			and $(\mathcal G_t)_{t\geq 0}$ is the natural filtration of the process $(u_t)_{t\geq 0}$.
			\item[\eq\label{eq:S2}] For every  $t\geq 0$, the duality formula \eqref{eq:Duality} holds with $f$ being replaced by $g$.
		\end{itemize}
	\end{proposition}
	The proof of the above proposition is standard using the weak comparison principle.
	(We include a proof in Appendix \ref{append-A}.)

	\section{Probabilistic estimates for the dual SPDE} \label{sec:DSPDE}
	
	In this section, we give some probabilistic estimates for the dual SPDE $(u_t)_{t\geq 0}$.
	Especially, we give upper/lower bounds for some finite/infinite moments of the random field $1-u_t$ when the initial value $u_0$ is close to $0$.
	Those bounds will be crucial for the proof of Theorems \ref{thm:existence-of-Z-t} and \ref{Comming-down-finite-first-moment} in the later sections.

	Let $(x_i)_{i=1}^\infty$ be a sequence in $\mathbb R$, and let $(\Lambda, \mu) \in \mathcal T_\mathrm a$ denote the m-weak limit of the monotonically increasing sequence of counting measures $(\sum_{i=1}^n \delta_{x_i})_{n\in\mathbb N}$.
	For every $n \in \mathbb N$, let $(I_t^{(n)})_{t\geq 0}$, $(X^{(n),\alpha}_{t})_{\alpha \in I_t^{(n)},t\geq 0}$ and $(Z^{(n)}_t)_{t\geq 0}$ be notations  $(I_t)_{t\geq 0}$, $(X^{\alpha}_{t})_{\alpha \in I_t,t\geq 0}$ and $(Z_t)_{t\geq 0}$ given as in Subsection \ref{sec:MR} (right after Proposition \ref{prop:WD}) for an SBBM with initial configuration $(x_i)_{i=1}^n$, and
	parameters $\beta_{\mathrm o}$, $(p_k)_{k=0}^\infty$, $\beta_{\mathrm c}$ and $(q_k)_{k=0}^\infty$ given as in \eqref{eq:OBR}--\eqref{eq:CB}.
	Assume that \eqref{asp:A3}, \eqref{asp:A1} and \eqref{asp:A2} hold.
	For every $(\tilde \Lambda, \tilde \mu)\in \mathcal T$, let $(v_t^{(\tilde \Lambda, \tilde \mu)}(x))_{t>0,x\in \mathbb R} \in \mathcal C^{1,2}((0,\infty)\times \mathbb R)$ be the unique non-negative solution to the CDI Profile Equation \eqref{PDE}.
	Let $\Phi$ and $\Psi$ be given as in \eqref{Def-Phi} and \eqref{Def-Psi} respectively.
	Let $f$ be a measurable function on $\mathbb R$ which can be approximated by the elements of $\mathcal C(\mathbb R, [0,z^*])$ monotonically from below.
	Let $(u_t)_{t> 0}$ be the continuous $\mathcal C(\mathbb R, [0,z^*])$-valued process given as in Proposition \ref{prop:GI}, with  initial value $u_0 := f$,  on a probability space whose probability measure will be denoted by $\tilde {\mathbb P}_f$.

	When the initial value $u_0$ is close to $0$, the behavior of the random field $(u_t)_{t\geq 0}$ is largely related to the linearization of the functions $\Phi$ and $\Psi$ at $0$.
	The following analytical lemma helps us with the linearization technique.

	\begin{lemma}\label{lem: general-Bernoulli-ineq}
		Suppose that $N\in \mathbb N$ and $(z_i)_{i=1}^N$ is a finite list in $[0,2]$, then
		\begin{align}
			0\leq 1- \prod_{i=1}^N (1-z_i)\leq  \sum_{i=1}^N z_i.
		\end{align}
	\end{lemma}
	\begin{proof}
		If  $z_i\in [0,1]$ for all $i=1, \dots, N$,  then the desired  inequality follows from Bernoulli's inequality.
		If there is $i\neq j$ such that $z_i, z_j >1$, then $0\leq 1- \prod_{i=1}^N (1-z_i)\leq 2< \sum_{i=1}^N z_i.$
		Otherwise, there exists only one $i_0$ such that $z_{j}\in [0,1]$ for all $j\neq i_0$ and that $z_{i_0}\in(1,2].$
		Without loss of generality, we assume in this case $i_0=1$.
		Then by Bernoulli's inequality,
		\[
		\prod_{i=1}^N(1-z_i) = -\frac{z_1}{2}\prod_{i=2}^N(1-z_i) + (1-\frac{z_1}{2})\prod_{i=2}^N(1-z_i) \geq -\frac{z_1}{2}+1-\frac{z_1}{2}-\sum_{i=2}^N z_i,
		\]
		which again implies the desired result.
	\end{proof}
	
	Combining  Lemmas \ref{lem:CC} and \ref{lem: general-Bernoulli-ineq}, as well as the
	fact
	that $ (1-z)^k \leq 1$  for every $z\in [0,2]$ and $k\in \mathbb N$,
	we see that  for all $z\in [0,z^*]\subset [0,2]$,
	\begin{align}\label{eq: bounds-for-branching-mechanism}
		& -\lambda_{\mathrm o} z\leq \Phi'(0+) z \leq \Phi(z) \leq \beta_{\mathrm o} z \ \mbox{ and }\  0\leq \Psi(z) \leq 2\beta_{\mathrm c} z,
	\end{align}
	where $\lambda_{\mathrm o}$ is defined in \eqref{Def-of-lambda}.
	Moreover, define
	\begin{align}\label{Def-Of-kappa}
		\kappa( \tilde \gamma ):=\inf_{w\in [0,  \tilde \gamma ]} \frac{\Psi'(0+)w}{\Psi(w)}, \quad \tilde \gamma \in [0,1).
	\end{align}
	It is clear that  $\kappa (\tilde \gamma) \in [0,1]$ and  $\lim_{\tilde \gamma\to 0} \kappa(\tilde \gamma) =1$.
	With those linear bounds of the functions $\Phi$ and $\Psi$, we  have  a preliminary upper bound for the first moment of the random field $u_t$.
	
	\begin{lemma}\label{Upper-bound-w-t-x-2}
		For every $t>0$ and $x\in \mathbb R$, it holds that
		\begin{align}
			& \tilde{\mathbb{E}}_{f} \brk{u_{t}(x) }
			\leq  e^{-\Phi'(0+)t} \mathbf{E}_x[f(B_t)]
			\leq e^{\lambda_{\mathrm o} t} \mathbf{E}_x[f(B_t)],
		\end{align}
		where $(B_t)_{t\geq 0}$ is a 1-d standard Brownian motion with initial value $x$ under the expectation operator $\mathbf E_x$.
	\end{lemma}
	\begin{proof}
		Taking expectation in the mild formula \eqref{eq:mild}, we see that for every $t>0$ and $x\in \mathbb R$,
		\begin{equation} \label{eq:R22}
			\tilde{\mathbb{E}}_{f} \brk{u_{t}(x) }
			=\int p_t(x-y) f(y) \mathrm{d} y - \iint_0^t p_{t-s}(x-y)\tilde{ \mathbb{E}}_{f}\brk{ \Phi(u_{s}(y)) }\mathrm{d}  s \mathrm{d}  y.
		\end{equation}
		Fix the arbitrary $t>0$.
		Define $g(r,x):=\tilde{\mathbb{E}}_{f} \brk{u_{t-r}(x) }$ and $\rho(r,x):=\frac{\tilde{\mathbb{E}}_{f} \brk{\Phi(u_{t-r}(x))}}{\tilde{\mathbb{E}}_{f} \brk{u_{t-r}(x)}}$ for every $r\in [0,t]$ and $x\in \mathbb R$.
		It is clear that $g$ is bounded on $[0,T] \times \mathbb R$; and so is $\rho$, thanks to \eqref{eq: bounds-for-branching-mechanism}.
		For every $(r,x)\in [0,T]\times \mathbb R$, denote by $(B_t)_{t\geq r}$ a 1-d standard Brownian motion initiated at time $r$ and location $x$ under the expectation operator $\mathbf E_{(r,x)}$.
		Then, for every $(r,x)\in [0,t)\times \mathbb R$,
		\begin{align}
			& g(r,x)+\mathbf E_{(r,x)}\brk{\int_r^t (\rho g)(s,B_s)\mathrm ds}
			\\&=\tilde{\mathbb{E}}_{f} \brk{u_{t-r}(x)} + \int_r^t \mathrm ds\int p_{s-r}(x-y) \tilde{ \mathbb{E}}_{f}\brk{ \Phi(u_{t-s}(y)) } \mathrm dy
			\\&\label{eq:R23} \overset{\eqref{eq:R22}}{=}\int p_{t-r}(x-y) f(y) \mathrm{d} y - \iint_0^{t-r} p_{t-r-s}(x-y)\tilde{ \mathbb{E}}_{f}\brk{ \Phi(u_{s}(y)) }\mathrm{d}  s \mathrm{d}  y
			\\&\qquad + \int_r^t \mathrm ds\int p_{s-r}(x-y) \tilde{ \mathbb{E}}_{f}\brk{ \Phi(u_{t-s}(y)) } \mathrm dy.
		\end{align}
		Noticing that the last two terms on the right hand side of \eqref{eq:R23} cancel each other, so
		\begin{equation}
			g(r,x)+\mathbf E_{(r,x)}\brk{\int_r^t (\rho g)(s,B_s)\mathrm ds}= \mathbf E_{(r,x)}[f(B_t)], \quad (r,x)\in [0,t)\times \mathbb R.
		\end{equation}
		Now, by Feynman-Kac formula \cite{MR1235414}*{Lemma 1.5 of Appendix of Part I},
		\begin{align}
			& \tilde{\mathbb{E}}_{f} \brk{u_{t}(x) }
			= g(0,x)
			= \mathbf{E}_{(0,x)} \brk{\exp\brc{- \int_0^t \rho(s,B_s) \mathrm{d}  s } f(B_t)}
			\\
			& \overset{\eqref{eq: bounds-for-branching-mechanism}}{\leq}  e^{-\Phi'(0+)t} \mathbf{E}_x[f(B_t)] \leq
			e^{\lambda_{\mathrm o} t} \mathbf{E}_x[f(B_t)], \quad x\in \mathbb R.
		\end{align}
	\end{proof}

	Rather than the first moment, we are more interested in the infinite moments of the random field $1-u_t$, since they arise naturally when one takes $n$ to $\infty$ in the duality relation \eqref{eq:Duality}.
	To this end, we need to work with the infinite product.
	For a sequence of real numbers $(z_i)_{i\in \mathbb N}$, define \[\prod_{i=1}^\infty z_i = \lim_{m\to \infty} \prod_{i=1}^m z_i\] whenever the limit exists.
	(This definition is standard, see \cite{MR0924157}*{Definition 15.2}.)
	Recall that $z^* \in [0,2)$.
	For every $[0,z^*]$-valued measurable function $u$ on $\mathbb R$, closed subset $\tilde \Lambda$ of $\mathbb R$, and integer-valued locally finite measure $\tilde \mu$ on $\tilde \Lambda^{\mathrm c}$, define
	\begin{equation} \label{eq:IP}
		\prod_{x\in \tilde {\Lambda}^{\mathrm c}} \pr{1-u(x)}^{\tilde \mu(\{x\})}
		:= \prod_{i=1}^{\tilde \mu(\tilde \Lambda^\mathrm c)} (1-u(z_i))
	\end{equation}
	where $(z_i)_{i=1}^{\tilde \mu(\tilde \Lambda^\mathrm c)}$ is a (possibly finite or infinite) sequence in $\tilde \Lambda^\mathrm c$ such that $\tilde \mu = \sum_{i=1}^{\tilde \mu(\tilde \Lambda^\mathrm c)} \delta_{z_i}$.
	This definition does not depend on the particular choice of the sequence $(z_i)_{i=1}^{\tilde \mu(\tilde \Lambda^\mathrm c)}$ thanks to the following lemma.
	\begin{lemma}\label{eq:infinite-product}
		Let $u$ be a $[0,z^*]$-valued measurable function on $\mathbb R$, $\tilde \Lambda$ a closed subset of $\mathbb R$, and $\tilde \mu$ an integer-valued locally finite measure  on $\tilde \Lambda^{\mathrm c}$.
		Then for any (possibly finite or infinite) list $(\tilde x_i)_{i=1}^{\tilde \mu(\tilde \Lambda^\mathrm c)}$ in $\tilde \Lambda^\mathrm c$ such that $\tilde \mu = \sum_{i=1}^{\tilde \mu(\tilde \Lambda^\mathrm c)} \delta_{\tilde x_i}$, it holds that
		\begin{align}\label{eq:IPP}
			& \qquad \prod_{i=1}^{\tilde \mu(\tilde \Lambda^\mathrm c)} (1-u(\tilde x_i))
			\\&=
			\mathbf 1_{\brc{\tilde \mu\pr{ \brc{x\in \tilde \Lambda^{\mathrm c}: u(x) \in (1,z^*]}} < \infty}}
			(-1)^{\tilde \mu(\{x\in \tilde \Lambda^{\mathrm c}: u(x) \in (1,z^*]\})}
			\exp \brc{\int_{\tilde \Lambda^\mathrm c} \log \pr{\abs{1-u(x)}} \tilde\mu(\mathrm dx)}.
		\end{align}
	\end{lemma}
	
	\begin{proof}
		If $\tilde \mu$ is a finite measure on $\tilde \Lambda^{\mathrm c}$ then the desired result is trivial. So we assume that $\tilde \mu(\tilde \Lambda^\mathrm c) = \infty$.
		
		\firststep Suppose that \[\tilde \mu\pr{ \brc{x\in \tilde \Lambda^{\mathrm c}: u(x) \in (1,z^*]}} = \infty.\]
		In this case, the right hand side of \eqref{eq:IPP} is zero.
		Let $I$ be the collection of indices $i\in \mathbb N$ such that $1< u(\tilde x_i) \leq z^*$.
		Then,
		\begin{equation}\label{eq:IFI}
			\abs{I}=\sum_{i=1}^{\infty} \mathbf 1_{\brc{x\in \tilde \Lambda^{\mathrm c}: u(x) \in (1,z^*]}}(\tilde x_i) = \tilde \mu\pr{ \brc{x\in \tilde \Lambda^{\mathrm c}: u(x) \in (1,z^*]}} = \infty.
		\end{equation}
		Note by Lemma \ref{lem:CC} that $1\leq z^*<2$, and therefore
		\begin{itemize}
			\item[\eq\label{eq:zstar}] for any $a\in [0,z^*]$, $|1-a| \leq 1$.
		\end{itemize}
		For any $n\in \mathbb N$, it holds that
		\begin{align}
			& \abs{\prod_{i=1}^{n} (1-u(\tilde x_i))}
			= \pr{\prod_{i\in I \cap \{1,\dots, n\}} \abs{1-u(\tilde x_i)}} \pr{\prod_{i\in I^{\mathrm{c}} \cap \{1,\dots, n\}} \abs{1-u(\tilde x_i)} }
			\\&\overset{\eqref{eq:zstar}}{\leq} \prod_{i\in I \cap \{1,\dots, n\}} \abs{1-u(\tilde x_i)}
			\leq \pr{z^* - 1}^{|I \cap \{1,\dots, n\}|}.
		\end{align}
		Now
		\begin{equation}
			\abs{\prod_{i=1}^{\infty} (1-u(\tilde x_i))} \leq \limsup_{n\to \infty} ~ \pr{z^* - 1}^{|I \cap \{1,\dots, n\}|} \overset{\eqref{eq:IFI}}{=} 0
		\end{equation}
		which says that the desired equality \eqref{eq:IPP} holds in this case.

		\nextstep Suppose that \[\tilde \mu\pr{ \brc{x\in \tilde \Lambda^{\mathrm c}: u(x) \in (1,z^*]}} < \infty.\]
		Let $I$ be the collection of indices $i\in \mathbb N$ such that $1< u(\tilde x_i) \leq z^*$.
		Then,
		\begin{equation}
			\abs{I}=\sum_{i=1}^{\infty} \mathbf 1_{\brc{x\in \tilde \Lambda^{\mathrm c}: u(x) \in (1,z^*]}}(\tilde x_i) = \tilde \mu\pr{ \brc{x\in \tilde \Lambda^{\mathrm c}: u(x) \in (1,z^*]}} < \infty.
		\end{equation}
		Now, for any $n\in \mathbb N$,
		\begin{align}
			& \prod_{i=1}^{n} (1-u(\tilde x_i))
			= \pr{\prod_{i\in I \cap \{1,\dots, n\}} \pr{1-u(\tilde x_i)}} \pr{\prod_{i\in I^{\mathrm{c}} \cap \{1,\dots, n\}} \pr{1-u(\tilde x_i)} }
			\\&=(-1)^{|I\cap \{1,\dots, n\}|}\pr{\prod_{i\in I \cap \{1,\dots, n\}} \abs{1-u(\tilde x_i)}} \pr{\prod_{i\in I^{\mathrm{c}} \cap \{1,\dots, n\}} \abs{1-u(\tilde x_i)} }
			\\&=(-1)^{|I\cap \{1,\dots, n\}|} \exp \brc{\sum_{i=1}^n \log\pr{\abs{1-u(\tilde x_i)}}}.
		\end{align}
		Taking $n\to \infty$, we see that the desired equality \eqref{eq:IPP} also holds in this case.
	\end{proof}

	It is worth noting from  \cite{MR3642325}*{Theorem 4.2} and Lemma \ref{eq:infinite-product}
	that
	\begin{itemize}
		\item[\eq\label{eq:MM}]
		for every $[0,z^*]$-valued measurable function $u$ on $\mathbb R$,
		\[
		\tilde \mu \mapsto \prod_{x\in \mathbb R} \pr{1-u(x)}^{\tilde \mu(\{x\})}
		\]
		is a Borel measurable map from $\mathcal N$ to $(-1,1]$.
	\end{itemize}
	
	The next lemma shows how the infinite moments of the random field $1-u_t$ are related to the SBBM. It motivates the rest of the results of this section.
	\begin{lemma} \label{Lemma: Duality3}
		For any $t>0$,
		\begin{align}\label{eq:Duality3}
			& \lim_{n\to\infty}\mathbb E\brk{  \prod_{\alpha \in I_t^{(n)}} \pr{1-f\pr{X_t^{(n),\alpha}}}} = \tilde{\mathbb E}_{f} \brk{\prod_{i=1}^\infty \pr{1-u_t(x_i)}}
			\\&
			= \tilde{\mathbb E}_{f} \brk{\mathbf 1_{\brc{u_t(x)=0,\forall x\in  \Lambda}}\prod_{x\in \Lambda^{\mathrm c}} \pr{1-u_t(x)}^{\mu(\{x\})}}.
		\end{align}
	\end{lemma}
	\begin{proof}
		Thanks to Proposition \ref{prop:D}, it suffices to consider the limit
		\begin{equation}\label{eq:EIP}
			\lim_{n\to\infty}\mathbb E\brk{  \prod_{\alpha \in I_t^{(n)}} \pr{1-f\pr{X_t^{(n),\alpha}}}}=
			\lim_{n\to\infty} \tilde{\mathbb E}_{f} \brk{\prod_{i=1}^n\pr{1-u_t(x_i)}}.
		\end{equation}
		Let us first explain that the infinite product
		\[
		\prod_{i=1}^\infty \pr{1-u_t(x_i)} := \lim_{n\to \infty} \prod_{i=1}^n \pr{1-u_t(x_i)}
		\]
		is a well-defined random variable.
		Note that $u_t(x_i)$ take values in $[0,z^*].$
		From Lemma \ref{lem:CC} and \eqref{asp:A2}, we have $z^* < 2$.
		On one hand, if there are infinitely many $i$ such that
		$u_t(x_i)\in [1, z^*]\subset [1,2)$, then $\prod_{i=1}^\infty(1-u_t(x_i))=0$.
		On the other hand,
		if there are only finitely many
		$i$ such that $u_t(x_i)\in [1,z^*]$, then
		the infinite product is also well-defined since the sequence
		\begin{equation}
			\prod_{i=1}^n (1-u_t(x_i)),   \quad n\in \mathbb N
		\end{equation}
		is eventually monotone.
		Since $u_t(x)$ is continuous in $x\in \mathbb R$, by  Lemma \ref{eq:infinite-product},
		almost surely,
		\[
		\prod_{i=1}^\infty \pr{1-u_t(x_i)} =\mathbf 1_{\brc{u_t(x)=0,\forall x\in  \Lambda}}\prod_{x\in \Lambda^{\mathrm c}} \pr{1-u_t(x)}^{\mu(\{x\})}.
		\]
		Now, by the bounded convergence theorem, we can exchange the limit and the expectation on the right hand side of \eqref{eq:EIP}, and get the desired result.
	\end{proof}

	The next lemma says that the infinite moments of the random field $1-u_t$ can be estimated by the Laplace transform of $u_t$.

	\begin{lemma}\label{lem:1}
		Let $\varepsilon\in (0,\frac{1}{2})$ and $U\subset \R$ be an open interval.
		Suppose that $F$ is a closed interval  containing $\{x_i: i\in \N\}$.
		Define
		\begin{equation} \label{eq:Theta}
			\theta (\gamma) := \begin{cases}
				-\log (1-\gamma)/\gamma, \quad & \gamma\in (0,1),
				\\ 1, \quad &\gamma = 0.
			\end{cases}
		\end{equation}
		Then
		for any $t>0$ and $\gamma \in (\varepsilon, 1)$,
		\begin{align}\label{Lower:w}
			\tilde{\mathbb E}_{\varepsilon  \mathbf 1_U} \brk{\prod_{i=1}^\infty (1-u_t(x_i))}
			\geq \tilde{\mathbb E}_{\varepsilon  \mathbf 1_U } \brk{\exp\brc{-\theta(\gamma)\sum_{i=1}^\infty u_{t}(x_i) } }  -2\tilde{\mathbb{P}}_{\varepsilon  \mathbf 1_U }\pr{\sup_{s\leq t, y\in F} u_s(y)> \gamma},
		\end{align}
		and
		\begin{align}\label{Upper:w}
			\tilde{\mathbb E}_{\varepsilon  \mathbf 1_U} \brk{\prod_{i=1}^\infty (1-u_t(x_i))}
			\leq \tilde{\mathbb E}_{\varepsilon \mathbf 1_U }\brk{\exp\brc{-\sum_{i=1}^\infty u_t(x_i)}} +  \tilde{\mathbb{P}}_{\varepsilon  \mathbf 1_U }\pr{\sup_{s\leq t, y\in F} u_s(y)> \frac{1}{2}}.
		\end{align}
	\end{lemma}
	
	\begin{proof}
		Let $t>0$ and $\gamma \in (\varepsilon,1)$.
		Define $\tau_\gamma:= \inf\brc{s\geq 0: \sup_{y\in F} u_s(y)>\gamma}$.
		For the lower bound, noting that $1-w\geq e^{-\theta(\gamma)w}$  for $w\in [0, \gamma]$, we have
		\begin{align}
			& \tilde{\mathbb E}_{\varepsilon  \mathbf 1_U } \brk{\prod_{i=1}^\infty (1-u_t(x_i)) }
			= \tilde{\mathbb E}_{\varepsilon \mathbf 1_U } \brk{\prod_{i=1}^\infty (1-u_t(x_i))  \mathbf 1_{\{\tau_\gamma  \leq  t \}}} +  \tilde{\mathbb E}_{\varepsilon  \mathbf 1_U} \brk{\prod_{i=1}^\infty (1-u_t(x_i)) \mathbf 1_{\{\tau_\gamma > t \}} }
			\\
			& \geq - \tilde{\mathbb{P}}_{\varepsilon  \mathbf 1_U }\pr{\tau_\gamma \leq t} +  \tilde{\mathbb E}_{\varepsilon \mathbf 1_U}\brk{\exp\brc{-\theta(\gamma)\sum_{i=1}^\infty u_{t}(x_i) } \mathbf 1_{\{\tau_\gamma > t \}}} \\
			& \geq \tilde{\mathbb E}_{\varepsilon  \mathbf 1_U }\brk{\exp\brc{-\theta(\gamma)\sum_{i=1}^\infty u_{t}(x_i) } }-2\tilde{\mathbb{P}}_{\varepsilon  \mathbf 1_U }\pr{\tau_\gamma \leq t}.
		\end{align}
		For the upper bound, using the inequality $1-w \leq e^{-w}$ and the fact that $1-u_t(x_i)\geq 0$ for every $i\in \mathbb N$ almost surely on the event $\{\tau_{1/2}>t\}$, we have
		\begin{align}
			& \tilde{\mathbb E}_{\varepsilon  \mathbf 1_U } \brk{\prod_{i=1}^\infty (1-u_t(x_i))}
			= \tilde{\mathbb E}_{\varepsilon  \mathbf 1_U } \brk{\prod_{i=1}^\infty (1-u_t(x_i))  \mathbf 1_{\{\tau_{1/2}\leq t \}} } + \tilde{\mathbb E}_{\varepsilon  \mathbf 1_U } \brk{\prod_{i=1}^\infty (1-u_t(x_i))  \mathbf 1_{\{\tau_{1/2}> t\}}}
			\\
			& \leq \tilde{\mathbb{P}}_{\varepsilon  \mathbf 1_U }\pr{\tau_{1/2}\leq t}+  \tilde{\mathbb E}_{\varepsilon  \mathbf 1_U }\brk{\exp\brc{-\sum_{i=1}^\infty u_t(x_i)}}.
		\end{align}
		This completes the proof.
	\end{proof}

	Now,  we want to give upper/lower bounds for the Laplace  transform of the random field $u_t$ in terms of the solution $(v_t)_{t\geq 0}$ to the CDI Profile Equation \eqref{PDE}.
	The idea is to investigate Doob's decomposition of the semimartingale
	\begin{equation}
		s\mapsto \exp\brc{-\int u_s(y) v_{t-s}(y)\mathrm{d} y}
	\end{equation}
	on $[0,t]$.

	We will write $f \lesssim g$ to indicate that $f$ is bounded by $g$ up to a positive multiplicative constant.
	To be precise about this notation, let $f$ and $g$ be arbitrary real-valued functions defined on the Cartesian product $A\times B$ of two arbitrary sets $A$ and $B$.
	For a fixed $b\in B$, we say $f(a,b) \lesssim g(a,b)$ uniformly in $a\in A$ provided that there exists a constant $c>0$ (which may depend on $b$) such that $f(a,b) \leq c g(a,b)$ for all $a\in A$.

	For every $(\tilde \Lambda, \tilde \mu) \in \mathcal T$, we  use $(\hat{v}_t^{(\tilde{\Lambda}, \tilde{\mu})}(x))_{t>0,x\in \mathbb R} \in \mathcal C^{1,2}((0,\infty)\times \mathbb R)$  to denote the
	unique non-negative  solution of \eqref{PDE} with  $\Psi'(0+)$ being replaced by $1$.
	It is easy to  check that for all $(\tilde \Lambda, \tilde \mu)\in \mathcal T$ and $t>0, x\in \R$,
	\begin{align}\label{Identity}
		v_t^{(\tilde{\Lambda}, \tilde{\mu})} (x) = \Psi'(0+)^{-1}\hat{v}_t^{(\tilde{\Lambda}, \Psi'(0+)\tilde{\mu})} (x) .
	\end{align}
	From this and \cite{MR4698025}*{(2.4)--(2.10)}, we can verify that for every $(\tilde \Lambda, \tilde \mu)\in \mathcal T$  and $t>0, x\in \R$,
	\begin{align}
		&
		v_{t}^{(\tilde \Lambda, \tilde \mu)}(x)
		\leq v_{t}^{(\R, \mathbf 0)}(x)
		= \frac{2}{\Psi'(0+)t},\label{eq:Upper-bound-of-v}
		\\& \label{eq:TEV}
		v_{t}^{(\tilde \Lambda, \mathbf 0)}(x)= v_{t}^{(\tilde \Lambda+\{z\}, \mathbf 0)}(x+z), \quad \forall z\in \mathbb R,
		\\&\label{eq:SV}
		v_{t}^{(\tilde \Lambda,\mathbf 0)}(x)= v_{t}^{(-\tilde \Lambda,\mathbf 0)}(-x).
	\end{align}
	Here, we used Minkowski's notation  $a A + b B := \{ax+by:x\in A, y\in B\}$ for $a,b\in \mathbb R$ and $A,B \subset \mathbb R$; and $\mathbf 0$ represents the null measure.
	We also have that
	\begin{align}\label{Property-2}
		v_{t}^{((-\infty, 0], \mathbf 0)}(|x|) \lesssim \frac{1}{t}\pr{1+\frac{|x|}{\sqrt{t}}}e^{-\frac{x^2}{2t}},
	\end{align}
	and
	\begin{align}\label{Property-3}
		v_{t}^{([-k,k], \mathbf 0)}(x)
		\lesssim \frac{1}{t}\pr{1+\frac{ \operatorname{dist}(\{x\}, [-k,k])}{\sqrt{t}}} e^{-\frac{\operatorname{dist}(\{x\}, [-k,k])^2}{2t}}
	\end{align}
	uniformly for  $t>0, x\in \mathbb R$ and $k\geq 0$.
	Here, $\operatorname{dist}(A,B):= \inf\{|a-b|:a\in A,b\in B\}$ represents the distance between two given sets $A,B\subset \mathbb R$.
	For every $t>0$,  $(\tilde \Lambda, \tilde \mu) \in \mathcal T$ and closed interval $F \subset \mathbb R,$ define
	\begin{equation} \label{eq:CV}
		\mathcal{V}_t^{(\tilde \Lambda,\tilde \mu , F)}:= \int_0^t \int_{F^c} v_{r}^{(\tilde \Lambda, \tilde \mu)}(z)^2 \mathrm{d}  z \mathrm{d}  r.
	\end{equation}
	
	\begin{lemma}\label{lemma:integral-of-v}
		The following statements hold.
		\begin{itemize}
			\item[(i)]
			If $U$ is an open interval such that $U\cap \{x_i: i\in \N \}$ is bounded,
			then for every $b\geq a>0$,
			\begin{equation} \label{eq:DU}
				\sup_{t\in [a,b]}\int_U v_{t}^{(\Lambda,\mu)}(y)\mathrm{d}  y<\infty.
			\end{equation}
			\item[(ii)]
			If $F$ is a closed interval containing $\bigcup_{i\in \N} (x_i -1, x_i+1)$, then $\mathcal{V}_t^{(\Lambda,\mu, F)}<\infty$ for every $t>0$. In particular, it holds in this case that
			$
			\lim_{t\downarrow 0}\mathcal{V}_t^{(\Lambda,\mu, F)} = 0.
			$
			\item[(iii)]
			For each $K\in \mathbb N$, let $\Lambda_K:= [a_K,\infty)$ and $ F_K:= [b_K,\infty)$
			be unbounded intervals where $(a_K)_{K\in \mathbb N}, (b_K)_{K\in \mathbb N}$ are sequences in $\mathbb R$  such that
			\[
			\lim_{K\to\infty}  \operatorname{dist}(\Lambda_K, F_K^c)
			= \lim_{K\to\infty} (a_K- b_K)^+=\infty.
			\]
			Then for every $t>0$,
			$
			\lim_{K \to \infty} \mathcal{V}_t^{(\Lambda_K, \mathbf 0  , F_K)} =0.
			$
		\end{itemize}
	\end{lemma}

	\begin{proof}
		Let us prove (i).
		Let us fix an arbitrary open interval $U$ such that $U\cap \{x_i:i\in \mathbb N\}$ is bounded.
		Let $\tilde F$ be the smallest closed interval which contains $\{x_i:i\in \mathbb N\}$.
		There are four different cases to consider.
		\begin{itemize}
			\item
			If $\tilde F=\mathbb R$, then  $U$ is bounded.
			In this case, \eqref{eq:DU} follows from \eqref{eq:Upper-bound-of-v}.
			\item
			If $\tilde F=(-\infty, \beta]$ for some $\beta \in \mathbb R$, then $U$ is the subset of $(\alpha,\infty)$ for some $\alpha \in \mathbb R$.
			In this case, it was argued in the proof of \cite{MR4698025}*{Lemma 2.3} that
			\begin{equation}
				\int_U v_t^{(\Lambda, \mu)}(y) \mathrm{d} y
				\leq \int_\alpha^\beta v^{(\mathbb R, \mathbf 0)}_t(y)\mathrm{d} y + \int_0^\infty v_{t}^{((-\infty,0],\mathbf 0)}(y)\mathrm{d} y.
			\end{equation}
			The desired \eqref{eq:DU} now follows from \eqref{eq:Upper-bound-of-v} and \eqref{Property-2}.
			\item
			If $\tilde F = [\alpha, \infty)$ for some $\alpha \in \mathbb R$ then, similarly to the previous case, \eqref{eq:DU} holds.
			\item
			If $\tilde F = [\alpha, \beta]$ for some $-\infty < \alpha \leq \beta<\infty$, then it was argued in the proof of \cite{MR4698025}*{Lemma 2.3} that
			\begin{equation}
				\int_U v_t^{(\Lambda, \mu)}(y) \mathrm dy
				\leq \int v^{([\frac{\alpha - \beta}{2}, \frac{\beta - \alpha}{2}],\mathbf 0)}_{t}(y)\mathrm{d} y.
			\end{equation}
			The desired \eqref{eq:DU} now follows from \eqref{Property-3}.
		\end{itemize}

		Noticing from \eqref{Identity},  (ii) of this lemma is essentially given by \cite{MR4698025}*{Lemma 2.3 and Lemma 3.1 (2)}.
		
		We now prove (iii). Without loss of generality, we assume that
		$c_K:= a_K-b_K> 1$ for every $K\in \mathbb N$.
		By \eqref{eq:TEV}, \eqref{eq:SV} and \eqref{Property-2}, uniformly for every $K\in \mathbb N$ and $t>0$,
		\begin{align}
			& \mathcal{V}_t^{(\Lambda_K, \mathbf 0  , F_K)}
			=  \int_0^t \int_{-\infty}^{b_K} v_{r}^{([a_K, \infty) , \mathbf 0)}(z)^2 \mathrm{d}  z \mathrm{d}  r
			= \int_0^t \int_{-\infty}^{-c_K} v_{r}^{([0, \infty) , \mathbf 0)}(z)^2 \mathrm{d}  z \mathrm{d}  r
			\\& = \int_0^t \int_{c_K}^{\infty } v_{r}^{((-\infty, 0], \mathbf 0)}(z)^2 \mathrm{d}  z \mathrm{d}  r
			\lesssim\int_0^t \int_{c_K}^{\infty } \frac{ 1}{r^2}\pr{1+\frac{z}{\sqrt{r}}}^2e^{-\frac{z^2}{r}}  \mathrm{d}  z \mathrm{d}  r.
		\end{align}
		Therefore,
		noticing that the function $a \mapsto a^4 (1+a)^2e^{-\frac{a^2}{2}}$ is bounded on $(0, \infty)$ and $z\geq c_K > 1$,
		we have uniformly for every $K\in \mathbb N$ and $t>0,$
		\begin{align}
			& \mathcal{V}_t^{(\Lambda_K, \mathbf 0  , F_K)}
			\lesssim \int_0^t \int_{c_K}^{\infty } e^{-\frac{z^2}{2r}}  \mathrm{d}  z \mathrm{d}  r
			\leq t  \int_{c_K}^{\infty } e^{-\frac{z^2}{2t}}  \mathrm{d}  z.
		\end{align}
		The desired result of (iii) now follows.
	\end{proof}

	We now present our upper/lower bounds for certain Laplace transform of $u_t$.

	\begin{lemma}\label{lem:2}
		Let $F$ be a closed interval containing $\bigcup_{i\in \mathbb N} (x_i -1, x_i+1)$ and $U$ be an open interval.
		Let $0\leq \varepsilon \leq \gamma <1$. Suppose that the initial value of $(u_t)_{t\geq 0}$ is $f = \varepsilon \mathbf 1_{U}.$
		Let $\theta(\gamma)$ be given as in \eqref{eq:Theta}.
		Then for any $t>0$,
		\begin{align}\label{eq:LIL}
			& \tilde{\mathbb{E}}_{\varepsilon \mathbf 1_U }\brk{\exp\brc{-\theta(\gamma)\sum_{i=1}^\infty u_{t}(x_i) } }
			\geq  \exp\brc{-\frac{\varepsilon e^{\lambda_{\mathrm o} t}}{1-\gamma}\int_U v_{t}^{(\Lambda,\mu)}(y)\mathrm{d}  y }
			\\ &\qquad - \tilde{\mathbb{P}}_{\varepsilon \mathbf 1_U }\pr{\sup_{s\leq t, y\in F} u_s(y)
				> \frac{\gamma}{2 \beta_{\mathrm c}}\Psi'(0+) } -  {}
			\frac{\varepsilon \Psi'(0+)e^{2\lambda_{\mathrm o} t}}{2(1-\gamma)} \mathcal{V}_t^{(\Lambda, \mu, F)},
		\end{align}
		and
		\begin{align}\label{eq:LIU}
			& \qquad  \tilde{\mathbb{E}}_{\varepsilon \mathbf 1_U }\brk{\exp\brc{-\sum_{i=1}^\infty u_t(x_i) } } \\&\leq \exp\brc{-\varepsilon \kappa(\gamma) e^{-\beta_{\mathrm o}  t}  \int_U v_{t}^{(\Lambda,\mu)}(y) \mathrm{d}  y  } + \tilde{\mathbb{P}}_{\varepsilon  \mathbf 1_U}\pr{\sup_{s\leq t, y\in F} u_s(y)> \gamma } + \varepsilon \beta_{\mathrm c}
			e^{\lambda_{\mathrm o} t}
			\mathcal{V}_t^{(\Lambda, \mu, F)}.
		\end{align}
		Here,  $\Psi'(0+)$, $\lambda_{\mathrm o}$, $\Phi'(0+)$ and $\kappa$ are given as in \eqref{eq:Psi'}, \eqref{Def-of-lambda}, \eqref{eq:Phi'} and \eqref{Def-Of-kappa} respectively.
	\end{lemma}
	
	The  proof of Lemma \ref{lem:2} is similar to that of \cite{MR4698025}*{Lemma 2.4} so we postpone it to  Appendix \ref{append-B}.
	
	In the upper/lower bounds for the Laplace transform of the random field $u_t$ in Lemma \ref{lem:2} above, except the terms involving the solution $(v_t)_{t\geq 0}$ to the CDI Profile Equation \eqref{PDE},
	there are still terms related to the probability of the random field $u_t$ itself.
	In the next lemma, we explain how these terms can be controlled by the initial data $\varepsilon \mathbf 1_U$.
	
	\begin{lemma}\label{lem: Prob-w}
		Let $U$ be an open interval and $F$ be a closed interval. Let $t>0$ and $\gamma\in (0,1)$.  Suppose that $U\cap F$ is bounded. Then
		\C{c:UFtg}
		\begin{align}\label{eq:UFtg}
			\Cr{c:UFtg}(U,F,t,\gamma)
			:= \sup_{\varepsilon\in (0, \gamma/2)} \frac{1}{\varepsilon}\tilde{\mathbb{P}}_{\varepsilon \mathbf 1_U }\pr{\sup_{s\leq t, y\in F} u_{s}(y)> \gamma}<\infty.
		\end{align}
		In particular,
		$
		\limsup_{t \downarrow 0} \Cr{c:UFtg}(U,F, t,\gamma) <\infty.
		$
		Moreover, if $\tilde U$ is an open interval such that its intersection with $F_K:=[K,\infty)$ is bounded for every $K\in \mathbb R$,
		then we also have that
		$
		\lim_{K\to\infty}  \Cr{c:UFtg}(\tilde U,F_K,t,\gamma) =0.
		$
	\end{lemma}
	
	The proof of the above Lemma  is standard and similar to that of \cite{MR1339735}*{Lemma 3.1}, so we postpone it to Appendix \ref{append-C}.
	
	As an application of the results in this section, we establish a compact support property for the random field $u_t$.
	
	\begin{lemma}\label{Local-Compactness}
		Suppose that $0< \varepsilon < (2-\sum_{k=1}^\infty kq_k)/4 = \Psi'(0+)/(4\beta_{\mathrm c})$.
		Let $U$ be a bounded open subset of $\mathbb R$.
		Let $f$, the initial value of the SPDE \eqref{eq:GSPDE}, be bounded by $\varepsilon\mathbf 1_U$.
		Let $t > 0$ be arbitrary.
		Then,
		it holds that
		\begin{align}
			\lim_{K\to +\infty} \tilde{\mathbb{P}}_{f}\pr{u_t(x)=0,\forall x>K}= 1
			\quad \mbox{and}\quad
			\lim_{K\to +\infty} \tilde{\mathbb{P}}_{f }\pr{u_t(x)=0,\forall x<-K}= 1.
		\end{align}
	\end{lemma}
	\begin{proof}
		We only prove the first limit,  since  the second one follows  by symmetry. By the weak comparison principle, i.e.~ Lemma \ref{eq:WC},
		we can assume  $f= \varepsilon\mathbf 1_{U}$ without loss of generality.
		Let $K \geq 2$ be arbitrary.
		We can construct a sequence  $(\hat{x}_i)_{i\in \mathbb N}$ in $\mathbb R$ such that $\{\hat{x}_i: i\in \N\} = (K,+\infty)\cap \Q$.
		\begin{old} Note that  \eqref{Def-Lambda} holds with  $(x_i)_{i=1}^\infty$  and $(\Lambda, \mu)$ being replaced by $(\hat x_i)_{i=1}^\infty$ and
			$(\Lambda_K, \mathbf 0)$ where $\Lambda_K:=[K, +\infty)$. \end{old}
		Let $\gamma = \gamma(\varepsilon) <1$ be close enough to $1$ such that $4\varepsilon < \frac{\Psi'(0+)}{\beta_{\mathrm c}}\gamma$.
		In particular, it holds that $0\leq \varepsilon \leq \gamma < 1$.
		Fix $\theta =\theta(\gamma)$ according to \eqref{eq:Theta}.
		By Lemma \ref{lem:2},
		with $F_K:= [K/2, +\infty)$ and $(\Lambda, \mu)= (\Lambda_K, \mathbf 0)$ for $\Lambda_K:=[K, +\infty)$,
		we have
		\begin{align}
			& \tilde{\mathbb{P}}_{\varepsilon  \mathbf 1_U}\pr{u_t(x)=0,\forall x>K}= \tilde{\mathbb{E}}_{\varepsilon \mathbf 1_U}\brk{\exp\brc{-\theta(\gamma)\sum_{i=1}^\infty u_{t}(\hat{x}_i) }}
			\\&
			\label{Lower:w-2}
			\geq \exp\brc{-\frac{\varepsilon e^{\lambda_{\mathrm o} t}}{1-\gamma}\int_U v_{t}^{(\Lambda_K,\mathbf{0})}(y)\mathrm{d}  y } - \tilde{\mathbb{P}}_{\varepsilon \mathbf 1_U }\pr{\sup_{s\leq t, y\in F_K} u_s(y)> \frac{\gamma}{2\beta_{\mathrm c}
				}\Psi'(0+)} - {}
			\\& \qquad  \frac{\varepsilon \Psi'(0+)e^{2\lambda_{\mathrm o} t}}{2(1-\gamma)} \mathcal{V}_t^{(\Lambda_K, \mathbf{0}, F_K)}.
		\end{align}
		By  \eqref{eq:TEV} and \eqref{Property-2},
		since $U$ is bounded,
		the first term on the right hand side of \eqref{Lower:w-2}  converges
		to $1$ as $K\to+\infty$.
		Also, by Lemma \ref{lemma:integral-of-v} (iii), we see that
		the third term on the right hand side of \eqref{Lower:w-2}  converges to $0$
		as $K\to+\infty$.
		
		Therefore, it remains to prove that the second term on the right-hand side of \eqref{Lower:w-2}  converges to  $0$ as $K\to \infty$.
		From $\varepsilon \in (0, \frac{\gamma}{4\beta_{\mathrm c}}\Psi'(0+))$ and Lemma \ref{lem: Prob-w},
		the absolute value of  this term is bounded by $\varepsilon  \Cr{c:UFtg} (U, F_K, t, \frac{\gamma}{2 \beta_{\mathrm c}}\Psi'(0+))\stackrel{K\to\infty}{\longrightarrow}0$.
		We arrive at the desired result.
	\end{proof}
	\section{Existence: Proof of Theorem \ref{thm:existence-of-Z-t}} \label{sec:SS5}
	
	In this section, we will prove Theorem \ref{thm:existence-of-Z-t}.
	
	Assume that \eqref{asp:A3}, \eqref{asp:A1}, and \eqref{asp:A2} hold.
	Let the branching mechanisms $\Phi$ and $\Psi$ be given as in \eqref{Def-Phi} and \eqref{Def-Psi} respectively.
	Let $(\Lambda, \mu) \in \mathcal T_\mathrm{a}$ be an arbitrary initial trace.
	We first observe that it suffices to prove the \textit{existence} of an $\mathcal N$-valued c\`{a}dl\`{a}g Markov process $(Z_t)_{t>0}$ satisfying the convergence statement of Theorem \ref{thm:existence-of-Z-t}. Indeed, if $(Z_t)_{t>0}$ and $(Z'_t)_{t>0}$ are two such processes, then for any monotonically increasing sequence $(\mu_n)_{n\in\mathbb{N}}$ converging m-weakly to $(\Lambda,\mu)$, both $(Z_t)_{t>0}$ and $(Z'_t)_{t>0}$ arise as the distributional limit of the same sequence of SBBMs $(Z_t^{(n)})_{t>0}$; hence they must have the same law on $\mathbb{D}([t_0,\infty),\mathcal{N})$ for every $t_0>0$, which confirms they also share the same law in $\mathbb{D}((0,\infty),\mathcal{N})$.
	
	Thanks to this, fixing arbitrary $(\mu_n)_{n\in \mathbb{N}}$ and $(Z^{(n)}_{\cdot})_{n\in \mathbb N}$ satisfying \eqref{eq:Zt-monotone} and \eqref{eq:Zt-SBBM}, to prove Theorem \ref{thm:existence-of-Z-t}, we only have to show that
	\begin{itemize}
		\item[\eq\label{eq:MM1}] there exists an $\mathcal N$-valued c\`{a}dl\`{a}g Markov process $(Z_t)_{t>0}$ such that the $\mathbb D([t_0,\infty),\mathcal N)$-valued random elements $(Z^{(n)}_t)_{t\geq t_0}$ converge in distribution to $(Z_t)_{t\geq t_0}$ as $n \to \infty$ for every $t_0 > 0$;
		\item[\eq\label{eq:MM2}] the law of the limit process $(Z_t)_{t>0}$ depends only on $(\Lambda, \mu)$ and $(\Phi, \Psi)$, but not on the particular choice of the sequences $(\mu_n)_{n\in \mathbb N}$ and $(Z^{(n)}_\cdot)_{n\in \mathbb N}$;
		\item[\eq\label{eq:MM3}] if we assume in addition that $\Lambda =\emptyset$, then $(Z_t^{(n)})_{t\geq 0}$ converge in distribution to $(Z_t)_{t\geq 0}$ in $\mathbb D([0,\infty), \mathcal N)$ with $Z_0 := \mu$.
	\end{itemize}
	
	Furthermore, we can assume without loss of generality that there exists a sequence of points $(x_i)_{i=1}^\infty$ in $\mathbb R$ such that $\mu_n = \sum_{i=1}^n \delta_{x_i}$ for each $n \geq 1$.
	In fact, if the statements \eqref{eq:MM1}--\eqref{eq:MM3} are established under this additional condition, then any general monotonically increasing sequence can simply be viewed as a subsequence of the one-by-one accumulation $\big(\sum_{i=1}^k \delta_{x_i}\big)_{k=1}^\infty$, which immediately guarantees the validity of those statements in their full generality.
	
	For each $n \geq 1$, let $(I_t^{(n)})_{t\geq 0}$, $(X^{(n),\alpha}_{t})_{\alpha \in I_t^{(n)},t\geq 0}$, and $(Z^{(n)}_t)_{t\geq 0}$ play the respective roles of $(I_t)_{t\geq 0}$, $(X^{\alpha}_{t})_{\alpha \in I_t,t\geq 0}$, and $(Z_t)_{t\geq 0}$ defined in Subsection \ref{sec:MR} (right after Proposition \ref{prop:WD}) associated with an SBBM having initial configuration $\mu_n = \sum_{i=1}^n \delta_{x_i}$ and the branching mechanisms \eqref{eq:OBR}--\eqref{eq:CB}.
	
	For every $(\tilde \Lambda, \tilde \mu)\in \mathcal T$, let $(v_t^{(\tilde \Lambda, \tilde \mu)}(x))_{t>0,x\in \mathbb R} \in \mathcal C^{1,2}((0,\infty)\times \mathbb R)$ be the unique non-negative solution to the equation \eqref{PDE}.
	Let $f$ be a measurable function on $\mathbb R$  which can be approximated by the elements of $\mathcal C(\mathbb R, [0,z^*])$ monotonically from below.
	Let $(u_t)_{t> 0}$
	be the continuous $\mathcal C(\mathbb R, [0,z^*])$-valued process given as in Proposition \ref{prop:GI}, with  initial value $u_0=f$, on a probability space whose probability measure will be denoted by $\tilde {\mathbb P}_f$.

	The proof of Theorem \ref{thm:existence-of-Z-t} will be  divided into four parts.
	In Subsection \ref{sec:CFDD}, we will show the convergence in  finite-dimensional  distributions of $(Z_t^{(n)})_{t>0}$ as $n\to\infty$ to a Markov process $(\tilde{Z}_t)_{t>0}$.
	In Subsection \ref{sec:PZ}, we will give several preliminary results on the process $(\tilde{Z}_t)_{t>0}$, in particular, showing that it is stochastically right-continuous.
	In Subsection \ref{sec:CR}, we will show that $(\tilde{Z}_t)_{t>0}$ has a c\`{a}dl\`{a}g modification $(Z_t)_{t>0}$.
	Finally, Theorem \ref{thm:existence-of-Z-t} is proved in Subsection \ref{SS5.4}.
	
	\subsection{Convergence in finite dimensional distributions} \label{sec:CFDD}
	In this subsection, we will show the convergence in finite-dimensional distributions of $(Z_t^{(n)})$ as $n\to \infty$.
	Let us start with an analytic lemma which gives the continuity of certain functions on $\mathcal N$.
	(We include its proof  in Appendix \ref{append-D}.)

	\begin{lemma}\label{lemma:general-vague-convergence1}
		Suppose that $\nu_m$ converges  to $\nu$ in $\mathcal N$ when $m\to \infty$. Let $g\in \mathcal{C}(\R, [-1, 1])$ satisfy that $1-g$  has compact support.
		Then
		\begin{align}\label{step_6}
			\lim_{m\to\infty} \prod_{z\in \mathbb R}g(z)^{\nu_m(\{z\})}
			=  \prod_{z\in \mathbb R} g(z)^{\nu(\{z\})}.
		\end{align}
	\end{lemma}
	
	\begin{corollary}\label{lemma:general-vague-convergence2}
		If $f$
		has compact support and is bounded by
		$\Psi'(0+)/(4\beta_{\mathrm c})$,
		then for each $t>0$,
		\begin{align}\label{step_7}
			\lim_{m\to\infty} \tilde{ \mathbb{E}}_f\brk{\prod_{z\in \mathbb R}\pr{1-u_t(z)} ^{\nu_m(\brc{z})} }
			= \tilde{ \mathbb{E}}_f\brk{\prod_{z\in \mathbb R}\pr{1-u_t(z)}^{\nu(\brc{z})} }.
		\end{align}
	\end{corollary}
	\begin{proof}
		For each $K>0$ and $t>0$, it holds that
		\begin{align}\label{step_8}
			& \abs{\tilde{ \mathbb{E}}_f\brk{\prod_{z\in \mathbb R}\pr{1-u_t(z)}^{\nu_m(\{z\})} } - \tilde{ \mathbb{E}}_f\brk{\prod_{z\in \mathbb R}\pr{1-u_t(z)}^{\nu(\{z\})} }}
			\\
			& \leq
			\tilde{ \mathbb{E}}_f \brk{  \abs{ \prod_{z\in \mathbb R}\pr{1-u_t(z)}^{\nu_m(\{z\})} - \prod_{z\in \mathbb R}\pr{1-u_t(z)}^{\nu(\{z\})}   } \mathbf 1_{\brc{u_t(x) =0,\forall |x|>K}} }
			\\
			& \quad
			+ 2\tilde{\mathbb{P}}_{f} \pr{ \exists x\text{~s.t.~}|x|>K, u_{t}(x)>0 }.
		\end{align}
		From Lemma \ref{lemma:general-vague-convergence1} and the bounded convergence theorem,
		we know that the first term on the right hand side converges to $0$ when $m\to \infty$.
		Now from Lemma \ref{Local-Compactness}, taking $m\to\infty$ first and then $K\to\infty$, we get \eqref{step_7} as desired.
	\end{proof}

	In the next proposition, we show that $(Z_t^{(n)})_{t>0}$ converges to some Markov process in finite dimensional distributions.

	\begin{proposition}\label{Part-1-proof-of-thrm-1.2}
		There exists an $\mathcal N$-valued time-homogeneous Markov process $(\tilde{Z}_t)_{t>0}$ such that $(Z_t^{(n)} )_{t>0}$ converges to $(\tilde{Z}_t)_{t>0}$ as $n\to \infty$ in finite dimensional distributions.
		The entrance law $(\mathscr P^{(\Lambda, \mu)}_t)_{t> 0}$ of $(\tilde Z_t)_{t>0}$ is characterized so that, for any non-negative continuous function $g$ on $\mathbb R$ with compact support and $t>0$,
		\begin{align}
			& \int e^{-\tilde \nu (g)} \mathscr P^{(\Lambda, \mu)}_t(\mathrm{d} \tilde \nu)
			=
			\tilde {\mathbb E}_{1-e^{- g}}\brk{\prod_{i=1}^\infty \pr{1-u_t(x_i)}}
			\\&\label{eq:EC}=
			\tilde {\mathbb E}_{1-e^{- g}}\brk{\mathbf 1_{\brc{u_t(x) = 0, \forall x\in \Lambda }} \prod_{ x\in \Lambda^{\mathrm c}} \pr{1-u_t(x)}^{\mu(\{x\})}}.
		\end{align}
		The transition kernels $(\mathscr Q_t)_{t\geq 0}$ of $(\tilde Z_t)_{t>0}$ are characterized so that, for any non-negative continuous function $g$ on $\mathbb R$ with compact support, $t\geq 0$, and $\nu\in \mathcal N$,
		\begin{equation} \label{eq:KC}
			\int e^{-\tilde \nu(g)} \mathscr Q_t(\nu,\mathrm{d} \tilde \nu ) =
			\tilde {\mathbb E}_{1-e^{-g}} \brk{\prod_{z\in \mathbb R} \pr{1-u_t(z)}^{\nu(\{z\})}}.
		\end{equation}
		In particular, the finite dimensional distributions of $(\tilde Z_t)_{t>0}$ are determined by $(\Lambda, \mu)$, $\Phi$, and $\Psi$.
	\end{proposition}
	\begin{remark} \label{rem:LQ}
		Comparing \eqref{eq:EC} and \eqref{eq:KC},
		we have $\mathscr P^{(\emptyset, \nu)}_t = \mathscr Q_t(\nu, \cdot)$ for any $\nu\in \mathcal N$ and $t>0$.
	\end{remark}
	\begin{proof}
		\firststep \label{step:K}
		Let us fix an arbitrary $t>0$ and show that the $\mathcal N$-valued random element $Z_t^{(n)}$ converges in distribution to some $\mathcal N$-valued random element  $\hat Z_t$ as $n\to \infty$.
		Fix an arbitrary non-negative continuous function $g$ on $\mathbb R$ with compact support.
		From \cite{MR3642325}*{Corollary 4.14},
		it suffices to show the convergence in distribution of the random variable  $Z_t^{(n)}(g)$ as $n\to \infty$.
		\begin{draft}
			If a sequence of integer-valued measures converges vaguely, then the limit is also a an integer-valued measure. This fact can be verified, and is actually used here. However, we omit the details.
		\end{draft}
		By Lemma \ref{Lemma: Duality3} with  $f:=1-e^{-\theta g}$, we see that the following limit exists  for each $\theta \geq 0$:
		\begin{align}\label{eq:ECR}
			& \lim_{n\to\infty} \mathbb{E}\brk{ e^{-\theta Z_t^{(n)}(g)} }
			=\lim_{n\to\infty} \mathbb{E}\brk{ \prod_{\alpha \in I_t^{(n)}} e^{-\theta g(X_t^{(n), \alpha})} }
			\\&= \tilde {\mathbb E}_{1-e^{-\theta g}}\brk{\prod_{i=1}^\infty \pr{1-u_t(x_i)}}
			=\tilde {\mathbb E}_{1-e^{-\theta g}}\brk{\mathbf 1_{\brc{u_t(x) = 0, \forall x\in \Lambda }} \prod_{x\in \Lambda^{\mathrm c}} \pr{1-u_t(x)}^{\mu(\{x\})}}.
		\end{align}
		To show the weak convergence of $Z_t^{(n)}(g)$,  by L\'evy's  continuity theorem, it suffices to prove that $\lim_{\theta \to0}\lim_{n\to\infty}  \mathbb{E} [e^{-\theta Z_t^{(n)}(g)} ]=1$.
		Recall from \eqref{asp:A3} and \eqref{eq:Psi'} that $\Psi'(0+) /\beta_{\mathrm c} = 2 - \sum_{k=0}^\infty k q_k \in (0,2)$.
		Let $\theta>0$ be small enough so that
		\[
		0<\theta' := 1-e^{-\theta \Vert g\Vert_\infty} <  \frac{\Psi'(0+)}{8\beta_{\mathrm c}} \leq  \frac{1}{4}.
		\]
		Let $U$ be a bounded open interval  containing the support of $g$, and $F$ be a closed interval containing $\cup_{i\in \mathbb N}(x_i-1,x_i+1)$.
		By Lemma \ref{Lemma: Duality3}, we have that
		\begin{align}
			& \lim_{n\to\infty} \mathbb{E}\brk{ e^{-\theta Z_t^{(n)}(g)}}
			\geq \lim_{n\to\infty} \mathbb{E}\brk{ (1-\theta')^{Z_t^{(n)}(U)}}
			= \tilde{\mathbb E}_{\theta' \mathbf 1_{U}} \brk{\prod_{i=1}^\infty (1-u_t(x_i))}.
		\end{align}
		From Lemmas \ref{lem:1} and  \ref{lem:2}, for $\gamma =1/2$,
		\begin{align}\label{eq:CR}
			& \tilde{\mathbb E}_{\theta' \mathbf 1_{U}} \brk{\prod_{i=1}^\infty (1-u_t(x_i))}
			\\&\geq \exp\brc{-\frac{\theta' e^{\lambda_{\mathrm o} t}}{1-\gamma}\int_U v_{t}^{(\Lambda,\mu)}(y)\mathrm{d}  y } - \tilde{\mathbb{P}}_{\theta' \mathbf 1_U}\pr{\sup_{s\leq t, y\in F} u_s(y)> \frac{\gamma}{2\beta_{\mathrm c}}\Psi'(0+)} - {}
			\\& \qquad \frac{\theta' \Psi'(0+)e^{2\lambda_{\mathrm o}  t}}{2(1-\gamma)} \mathcal{V}_t^{(\Lambda, \mu, F)} -2\tilde{\mathbb{P}}_{\theta' \mathbf 1_{U}}\pr{\sup_{s\leq t, y\in F} u_s(y)> \gamma}.
		\end{align}
		Note that $\theta'\to 0$ as $\theta \to 0$.
		By Lemma \ref{lem: Prob-w}, we have
		\[
		\tilde {\mathbb P}_{\theta' \mathbf 1_U} \pr{\sup_{s\leq t, y \in F} u_s(y) > \frac{\gamma}{2\beta_{\mathrm c}}\Psi'(0+)} \leq \theta' \Cr{c:UFtg}\pr{U,F,t,\frac{\gamma}{2\beta_{\mathrm c}}\Psi'(0+)}
		\xrightarrow[]{\theta \to 0} 0,
		\]
		and
		\[
		\tilde {\mathbb P}_{\theta' \mathbf 1_U} \pr{\sup_{s\leq t, y \in F} u_s(y) > \gamma}
		\leq \theta'  \Cr{c:UFtg}\pr{U,F,t,\gamma} \xrightarrow[]{\theta \to 0} 0.
		\]
		Here, $\Cr{c:UFtg}(\cdot, \cdot, \cdot, \cdot)$ is the constant given as in Lemma \ref{lem: Prob-w}.
		From Lemma \ref{lemma:integral-of-v}, we have $\mathcal V_t^{(\Lambda, \mu,F)} < \infty$. Therefore, the third term on the right hand side of \eqref{eq:CR} converges to $0$ as $\theta \to 0$.
		Now, we have
		\begin{align}
			& \liminf_{\theta\to 0} \lim_{n\to\infty} \mathbb{E}\brk{ e^{-\theta Z_t^{(n)}(g)}   }
			\geq
			\liminf_{\theta\to 0} \tilde{\mathbb E}_{\theta' \mathbf 1_{U}} \brk{\prod_{i=1}^\infty (1-u_t(x_i))}
			\\&\geq
			\lim_{\theta\to 0} \exp\brc{-\frac{\theta' e^{\lambda_{\mathrm o}  t}}{1-\gamma}\int_U v_{t}^{(\Lambda,\mu)}(y)\mathrm{d}  y } =1,
		\end{align}
		as desired for this step.
		Moreover, we can verify from \eqref{eq:ECR} that the distribution of $\tilde Z_t$, denoted by $\mathscr P^{(\Lambda, \mu)}_t$, satisfies \eqref{eq:EC}.
		In fact, we know that $\mathscr P^{(\Lambda, \mu)}_t$ is the unique probability measure on $\mathcal N$ satisfying \eqref{eq:EC}, thanks to \cite{MR3642325}*{Theorem 4.11 (iii)}.

		\nextstep \label{step:FC}
		Fixing integer $m>1$ and real numbers $0< t_1 < \dots < t_m$, we will show in this step the convergence in distribution of the $\mathcal N^m$-valued random element $(Z^{(n)}_{t_1}, \dots, Z^{(n)}_{t_m})$ as $n\to \infty$.
		From Step \ref{step:K}, for each $k\in \{1,\dots, m\}$, the family of $\mathcal N$-valued random elements $\{Z_{t_k}^{(n)}: n\in \mathbb N\}$ is tight.
		From this, it is easy to see that the family of $\mathcal N^m$-valued random elements $\{(Z^{(n)}_{t_1}, \dots, Z^{(n)}_{t_m}):n \in \mathbb N\}$ is also tight, and therefore, by \cite{MR4226142}*{Theorem 23.2}, is relatively compact in distribution.
		This implies the existence of a strictly increasing sequence $(n_k)_{k=1}^\infty$ in $\mathbb N$ satisfying that the $\mathcal N^m$-valued random element $(Z^{(n_k)}_{t_1}, \dots, Z^{(n_k)}_{t_m})$ converges in distribution as $k\to \infty$.
		Let the $\mathcal N^m$-valued random element $(\hat Z_{t_1}, \dots, \hat Z_{t_m})$ be the corresponding subsequential convergence in distribution limit.
		We will show that $(\hat Z_{t_1}, \dots, \hat Z_{t_m})$ is also the convergence in distribution limit of $(Z^{(n)}_{t_1}, \dots, Z^{(n)}_{t_m})$ as $n\to \infty$.
		To do this, by \cite{MR3642325}*{Theorem 4.11 (iii)}, it suffices to show that
		\begin{align}\label{eq:fdd-limit}
			\lim_{n\to\infty} \mathbb{E}  \brk{ e^{- \sum_{i=1}^m Z_{t_i}^{(n)}(g_i)} } \quad \text{exists in}~\mathbb R,
		\end{align}
		where, for each $i\in \{1,\dots, m\}$, $g_i$ is an arbitrarily chosen non-negative continuous function on $\mathbb R$ with compact support satisfying that  $\norm{1-e^{-g_i}}_{\infty} < \Psi'(0+)/(4\beta_\mathrm c)$.
		In fact, if \eqref{eq:fdd-limit} holds, it must be the case that
		\begin{align}\label{eq:fdd-limit2}
			\mathbb{E}  \brk{e^{- \sum_{i=1}^m \hat Z_{t_i}(g_i)}  } =\lim_{n\to\infty} \mathbb{E}  \brk{ e^{- \sum_{i=1}^m Z_{t_i}^{(n)}(g_i)} }.
		\end{align}
		
		By the principle of induction, we can assume without loss of generality that \eqref{eq:fdd-limit} holds with $m$ being replaced by $m-1$.
		In particular, we can assume that $(\hat Z_{t_1}, \dots, \hat Z_{t_{m-1}})$ is the convergence in distribution limit of $(Z_{t_1}^{(n)}, \dots, Z^{(n)}_{t_{m-1}})$ as $n\to \infty$.
		By Skorokhod's representation theorem, we can further assume (in this step) without loss of generality that $\{(Z_{t_1}^{(n)}, \dots, Z^{(n)}_{t_{m-1}}): n\in \mathbb N\}$ and $(\hat Z_{t_1}, \dots, \hat Z_{t_{m-1}})$ are coupled in one probability space so that $(Z_{t_1}^{(n)}, \dots, Z^{(n)}_{t_{m-1}})$ converges almost surely to $(\hat Z_{t_1}, \dots, \hat Z_{t_{m-1}})$ when $n\to \infty$.
		
		For every $t>0$ and $h\in \mathcal C(\mathbb R, [0,\Psi'(0+)/(4\beta_{\mathrm c})])$ with compact support, define
		\begin{align}\label{Def-H-f}
			& H_t^{h}(\nu):=\tilde{\E}_{h}\brk{\prod_{z\in \mathbb R } \pr{1-u_t(z)}^{\nu(\{z\})} }, \quad \nu \in \mathcal N,
		\end{align}
		which is a bounded continuous function on $\mathcal N$, according to  Corollary \ref{lemma:general-vague-convergence2}.
		From Proposition \ref{prop:D} and the Markov property of the process $(Z^{(n)}_t)_{t\geq 0}$,
		almost surely
		\begin{align}
			& \E \brk{\cond{e^{- \sum_{i=1}^m Z_{t_i}^{(n)}(g_i) }  }(Z_{s}^{(n)})_{s\leq t_{m-1}}}
			=e^{ -  \sum_{i=1}^{m-1} Z_{t_i}^{(n)}(g_i)}H_{t_m-t_{m-1}}^{1-e^{-g_m}}( Z_{t_{m-1}}^{(n)}).
		\end{align}
		Therefore, by the dominated convergence theorem, we have
		\begin{align}\label{eq:fdd-cal}
			& \lim_{n\to\infty}  \mathbb{E}\brk{e^{- \sum_{i=1}^m Z_{t_i}^{(n)}(g_i) }}=\lim_{n\to\infty}  \mathbb{E}\brk{ e^{ - \sum_{i=1}^{m-1}  Z_{t_i}^{(n)}(g_i)} H_{t_m-t_{m-1}}^{1-e^{-g_m}}( Z_{t_{m-1}}^{(n)})}
			\\&=  \mathbb{E} \brk{ e^{ - \sum_{i=1}^{m-1} \hat Z_{t_i}(g_i)} H_{t_m-t_{m-1}}^{1-e^{-g_m}}( \hat Z_{t_{m-1}})}.
		\end{align}
		This implies \eqref{eq:fdd-limit}, and therefore, the desired result of this step.
		
		\nextstep
		For every integer $m\in \mathbb N$ and real numbers $0<t_1<\dots < t_m$, denote by $\mathscr P^{(\Lambda, \mu)}_{t_1, \dots, t_m}$ the distribution of the $\mathcal N^m$-valued random elements $(\hat Z_{t_1}, \dots, \hat Z_{t_m})$ given as in the previous steps.
		It is straightforward to verify that the family of distributions $(\mathscr P^{(\Lambda, \mu)}_{t_1,\dots, t_m})_{m\in \mathbb N, 0<t_1 < \dots < t_m}$ is projective in the sense of \cite{MR4226142}*{p.~179}.
		Recall that $\mathcal N$ is Polish.
		Therefore, by Kolmogorov's extension theorem (see \cite{MR4226142}*{Theorem 8.23} for example), there exists an $\mathcal N$-valued process $(\tilde Z_t)_{t>0}$ such that, for every integer $m\in \mathbb N$ and real numbers $0<t_1<\dots < t_m$, the distribution of $(\tilde Z_{t_1}, \dots, \tilde Z_{t_m})$ is given by $\mathscr P^{(\Lambda, \mu)}_{t_1, \dots, t_m}$.
		Moreover, from \eqref{eq:fdd-limit2} and \eqref{eq:fdd-cal}, we have
		\begin{equation}\label{Transition-probability}
			\mathbb{E}\brk{ e^{- \sum_{i=1}^m \tilde Z_{t_i}(g_i)}  } =  \mathbb{E}\brk{e^{ - \sum_{i=1}^{m-1} \tilde Z_{t_i}(g_i)} H_{t_m-t_{m-1}}^{1-e^{-g_m}}(\tilde  Z_{t_{m-1}})}
		\end{equation}
		for every $m\in \{2,3,\dots\}$, $0< t_1 < \dots < t_m$ and testing functions $(g_i)_{i=1}^m$ given as in Step \ref{step:FC}.

		\nextstep \label{step:Four}
		Note that, the result in Step \ref{step:K} essentially implies that, for any $t>0$, closed subset $\tilde \Lambda$ of $\mathbb R$, and locally finite integer-valued measure $\tilde \mu$ on $\tilde \Lambda^{\mathrm c}$, there exists a unique probability measure $\tilde {\mathscr P}$ on $\mathcal N$, such that, for any non-negative continuous function $g$ on $\mathbb R$ with compact support,
		\[
		\int e^{-\tilde \nu(g)} \tilde {\mathscr P}(\mathrm{d} \tilde \nu) =
		\tilde {\mathbb E}_{1-e^{-g}}\brk{\mathbf 1_{\brc{u_t(x) = 0, \forall x\in \tilde \Lambda}} \prod_{  x\in \tilde{\Lambda}^c  } \pr{1-u_t(x)}^{\tilde \mu(\{x\})} }.
		\]
		For every $t>0$ and $\nu \in \mathcal N$, by taking $\tilde \Lambda = \emptyset$ and $\tilde \mu = \nu$ in the above statement, we know that there exists a unique probability measure $\mathscr Q_t(\nu, \cdot)$ on $\mathcal N$ such that \eqref{eq:KC} holds.
		(When $t=0$, we set $\mathscr Q_t(\nu, \cdot) = \delta_{\nu}$.)
		
		It can be verified that $(\mathscr Q_t)_{t\geq 0}$ is a family of kernels on $\mathcal N$.
		In fact, fixing $t> 0$, denote by $\mathcal H$ the collection of bounded measurable function $F$ on $\mathcal N$ such that $\nu \mapsto \int F(\tilde\nu) \mathscr Q_t(\nu,\mathrm d\tilde \nu)$ is a measurable map from $\mathcal N$ to $\mathbb R$.
		It is clear that $\mathcal H$ is a monotone vector space (MVS) in the sense of \cite{MR0958914}*{Appendix A0}.
		Denote by $\mathcal K$ the collection of bounded measurable map $\tilde \nu\mapsto e^{-\tilde \nu(g)}$ from $\mathcal N$ to $\mathbb R$ where $g$ is a non-negative continuous function on $\mathbb R$ with compact support.
		Now, from \eqref{eq:MM}, \eqref{eq:KC} and \cite{MR4226142}*{Lemma 1.28},
		it can be argued
		that, for every $F \in \mathcal K$, $\nu \mapsto \int F(\tilde \nu) \mathscr Q_t(\nu, d\tilde \nu)$ is a measurable map from $\mathcal N$ to $\mathbb R$.
		In other words, $\mathcal K \subset \mathcal H$.
		Also, note that $\mathcal K$ is a multiplicative class of bounded real functions on $\mathcal N$ in the sense of  \cite{MR0958914}*{Appendix A0}.
		It is also clear that $\sigma(\mathcal K)$ is the Borel $\sigma$-field $\mathcal B_{\mathcal N}$ of $\mathcal N$ generated by the vague topology.
		So from \cite{MR0958914}*{Theorem A0.6}, we have  $\mathcal B_{\mathrm b}(\mathcal N) \subset \mathcal H$.
		Here, $\mathcal B_{\mathrm b}(\mathcal N)$ represents the collection of bounded Borel measurable functions on $\mathcal N$.
		This proves that $\mathscr Q_t$ is a kernel from $\mathcal N$ to itself.
		
		\nextstep
		Let $(\mathscr Q_t)_{t\geq 0}$ be the family of probability kernels given as in Step \ref{step:Four}.
		For every $m\in \{2,3,\dots\}$, $0< t_1 < \dots < t_m$ and testing functions $(g_i)_{i=1}^m$ chosen as in Step \ref{step:FC}, from \eqref{Transition-probability} we have
		\begin{align}
			& \mathbb{E}\brk{ e^{- \sum_{i=1}^m \tilde Z_{t_i}(g_i)}  }
			=  \mathbb{E}\brk{e^{ - \sum_{i=1}^{m-1} \tilde Z_{t_i}(g_i)} \int e^{-\tilde \nu(g_m)}\mathscr Q_{t_m - t_{m-1}}(\tilde Z_{t_{m-1}},\mathrm{d} \tilde \nu)}.
		\end{align}
		From this, it is clear that $(\tilde Z_t)_{t > 0}$ is a time-homogeneous Markov process with transition kernels $(\mathscr Q_t)_{t\geq 0}$.
		The entrance laws of this Markov process, i.e.~its one-dimensional distributions $(\mathscr P^{(\Lambda, \mu)}_t)_{t>0}$,  were  characterized already in Step \ref{step:K} through \eqref{eq:EC}. We are done.
	\end{proof}

	\subsection{Stochastic right continuity}\label{sec:PZ}
	In this subsection, we are going to give several preliminary results on the Markov process $(\tilde Z_t)_{t>0}$ given as in Proposition \ref{Part-1-proof-of-thrm-1.2}.
	Without loss of generality, we assume that $(\tilde Z_t)_{t\geq 0}$ is the canonical process defined on the path space $\Omega :=\mathcal N^{(0,\infty)}$, which is the collection of maps from $(0,\infty)$ to $\mathcal N$.
	More precisely, $\tilde Z_t(\omega) = \omega(t)$ for every $\omega \in \Omega$ and $t>0$.
	Let $\mathcal F^{\tilde Z}$ and $(\mathcal F^{\tilde Z}_t)_{t>0}$ be the natural $\sigma$-field, and the natural filtration, generated by the process $(\tilde Z_t)_{t>0}$.
	The corresponding probability measure, and the expectation operator, will be denoted by
	$\mathbb P_{(\Lambda, \mu)}$, and $\mathbb E_{(\Lambda, \mu)}$, respectively.

	To show the existence of a c\`adl\`ag modification of a process, one typically needs information about its finite dimensional distributions.
	Equations \eqref{eq:EC} and \eqref{eq:KC} can be regarded as the duality formulas between the process $(\tilde Z_t)_{t> 0}$ and the SPDE \eqref{eq:GSPDE}, which essentially characterizes the finite dimensional distributions of $(\tilde Z_t)_{t> 0}$.
	Note that they hold
	under the condition that the initial value of the dual SPDE $(u_t)_{t\geq 0}$ is a non-negative, compactly supported, continuous function bounded by $1$.
	This condition will be relaxed in Proposition \ref{prop:ID} and Corollary \ref{cor:ID}
	below where the following analytical lemma will play a role.
	(The proof of this analytic lemma is included in Appendix \ref{append-D}.)
	\begin{lemma} \label{lem:IM}
		For each $i\in \mathbb N$, let $(z^{(m)}_i)_{m\in \mathbb N}$ be an increasing sequence in $[0,z^*]$ whose limit is denoted by $z_i\in [0,z^*]$.
		Then
		\begin{equation} \label{eq:UM}
			\lim_{m\to \infty} \prod_{i=1}^\infty \pr{1-z_i^{(m)}} = \prod_{i=1}^\infty (1-z_i).
		\end{equation}
	\end{lemma}
	
	Recall that, throughout Section \ref{sec:SS5}, $f$ is a measurable function on $\mathbb R$ which can be approximated by the elements of $\mathcal C(\mathbb R, [0,z^*])$ monotonically from below.
	
	\begin{proposition} \label{prop:ID}
		Let $t>0$. It holds that
		\begin{equation}
			\int \pr{\prod_{x\in \mathbb R}\pr{1-f(x)}^{\nu(\{x\})} } \mathscr P^{(\Lambda,\mu)}_t(\mathrm{d} \nu)  = \tilde {\mathbb E}_f\brk{\prod_{i=1}^\infty (1-u_t(x_i))}.
		\end{equation}
	\end{proposition}
	\begin{proof}
		\firststep \label{step:CS}
		Let us first assume that $f$ is a continuous $[0,z^*]$-valued function on $\mathbb R$ with compact support.
		From Lemma \ref{lemma:general-vague-convergence1}, we know that the map $\nu \mapsto \prod_{x\in \mathbb R}(1-f(x))^{\nu(\{x\})}$ from $\mathcal N$ to $[-1,1]$ is continuous.
		Now, from  Proposition \ref{Part-1-proof-of-thrm-1.2} and Lemma \ref{Lemma: Duality3},  we have
		\begin{equation}
			\mathbb E_{(\Lambda, \mu)} \brk{\prod_{x\in \mathbb R}\pr{1-f(x)}^{\tilde Z_t(\{x\})}}
			= \lim_{n\to \infty} \mathbb E\brk{\prod_{x\in \mathbb R}\pr{1-f(x)}^{Z^{(n)}_t(\{x\})}}
			= \tilde {\mathbb E}_f \brk{\prod_{i=1}^\infty \pr{1-u_t(x_i)}}
		\end{equation}
		as desired.
		
		\nextstep Let us now assume that $f$ is a $[0,z^*]$-valued measurable function on $\mathbb R$ which can be approximated by a sequence $(\tilde f_m)_{m\in \mathbb N}$ in $\mathcal C(\mathbb R, [0,z^*])$ monotonically from below.
		Let $(\tilde g_m)_{m\in \mathbb N}$ be a sequence of continuous functions on $\mathbb R$ with compact support which approximates $\mathbf 1_\mathbb R$ from below.
		Then it is clear that $f$ can be approximated from below by $(f_m)_{m\in \mathbb N} := (\tilde f_m\tilde g_m)_{m\in \mathbb N}$, 	a sequence of $[0,z^*]$-valued continuous functions on $\mathbb R$ with compact support.
		Without loss of generality, it is standard (see the proof of Proposition \ref{prop:GI} in Appendix \ref{append-A}) to assume
		the existence of a sequence of $\mathcal C([0,\infty), \mathcal C(\mathbb R, [0,z^*]))$-valued random elements $(u^{(m)})_{m\in \mathbb N}$ such that
		\begin{itemize}
			\item for each $m\in \mathbb N$, $u^{(m)}$ has the law $\mathscr L_{f_m}$;
			\item almost surely, $u^{(m)} \leq u^{(m+1)}$ on $[0,\infty) \times \mathbb R$ for each $m\in \mathbb N$; and that
			\item almost surely, for almost every $(s,y)\in (0,\infty)\times \mathbb R$ w.r.t.~the Lebesgue measure, $u^{(m)}_s(y)\uparrow u_s(y)$.
		\end{itemize}
		From what we have proved in Step \ref{step:CS}, we have
		\begin{equation}
			\mathbb E_{(\Lambda, \mu)} \brk{\prod_{x\in \mathbb R}\pr{1-f_m(x)}^{\tilde Z_t(\{x\})}}
			= \tilde {\mathbb E}_f\brk{\prod_{i=1}^\infty \pr{1-u^{(m)}_t(x_i)}}.
		\end{equation}
		Taking $m\uparrow \infty$, from Lemma \ref{lem:IM} and bounded convergence theorem, we obtain the desired result.
	\end{proof}

	From Proposition \ref{prop:ID} and Remark \ref{rem:LQ}, one can verify the following result which generalizes \eqref{eq:KC}.
	\begin{corollary} \label{cor:ID}
		Let $t>0$ and $\nu \in \mathcal N$. It holds that
		\begin{equation}
			\int \pr{\prod_{x\in \mathbb R} \pr{1-f(x)}^{\tilde \nu(\{x\})} } \mathscr Q_t(\nu, \mathrm{d} \tilde \nu)  = \tilde {\mathbb E}_{f}\brk{\prod_{x\in \mathbb R}\pr{1-u_t(x)}^{\nu(\{x\})}}.
		\end{equation}
	\end{corollary}

	As an application of Proposition \ref{prop:ID}, the following proposition controls the expected number of particles in specific intervals at any fixed time. This first-moment estimate is crucial for later establishing that $(\tilde Z_t)_{t>0}$ possesses a c\`adl\`ag version.

	\begin{proposition}\label{prop1}
		Suppose that $F$ is a closed interval containing $\cup_{i\in \N}(x_i -1, x_i +1)$ and that $U$ is an open interval. Suppose that $U\cap F$ is bounded.
		Let $\gamma\in (0,1)$ and $\gamma_0= \gamma \Psi'(0+)/(2\beta_\mathrm c)$.
		Then for any $t>0$,
		\begin{align}
			\mathbb{E}_{(\Lambda, \mu)}\brk{\tilde Z_t(U)}
			\leq \frac{e^{\lambda_{\mathrm o} t}}{1-\gamma}\pr{\int_U v_{t}^{(\Lambda, \mu)}(y)\mathrm{d}  y + \frac{\Psi'(0+)e^{\lambda_{\mathrm o} t}}{2}\mathcal{V}_t^{(\Lambda, \mu, F)}}+
			3 \Cr{c:UFtg}(U,F,t,  \gamma_0)
			< \infty.
		\end{align}
		Here, $\Cr{c:UFtg}(\cdot, \cdot, \cdot, \cdot)$ is the constant given as in Lemma \ref{lem: Prob-w}.
	\end{proposition}
	
	\begin{proof}
		It is clear from \eqref{asp:A1} and \eqref{eq:Psi'} that $\gamma_0 \leq \gamma$.
		Let $t>0$ and $\varepsilon \in (0, \gamma_0 \wedge \frac{1}{2})$. From Proposition \ref{prop:ID},
		\begin{equation}
			\mathbb{E}_{(\Lambda,\mu)}\brk{(1-\varepsilon)^{\tilde Z_t(U)}}
			=  \tilde {\mathbb E}_{\varepsilon \mathbf 1_{U}} \brk{\prod_{i=1}^\infty \pr{1-u_t(x_i)}}.
		\end{equation}
		Together with Lemmas  \ref{lem:1}  and \ref{lem:2}, we conclude that
		\begin{align}
			& \mathbb{E}_{(\Lambda,\mu)}\brk{(1-\varepsilon)^{\tilde Z_t(U)}}
			\geq  \tilde{\mathbb{E}}_{\varepsilon\mathbf 1_U}\brk{\exp\brc{-\theta(\gamma)\sum_{i=1}^\infty u_{t}(x_i) }} -2\tilde{\mathbb{P}}_{\varepsilon  \mathbf 1_U}\pr{\sup_{s\leq t, y\in F} u_s(y)> \gamma}                                                                         \\
			& \geq \exp\brc{-\frac{\varepsilon e^{\lambda_{\mathrm o}  t}}{1-\gamma}\int_U v_{t}^{(\Lambda,\mu)}(y)\mathrm{d}  y } - \tilde{\mathbb{P}}_{\varepsilon  \mathbf 1_U }\pr{\sup_{s\leq t, y\in F} u_s(y)> \frac{\gamma}{2 \beta_{\mathrm c} }\Psi'(0+)}                        \\
			& \quad - \frac{\varepsilon \Psi'(0+)e^{2\lambda_{\mathrm o}  t}}{2(1-\gamma)} \mathcal{V}_t^{(\Lambda, \mu, F)}-2\tilde{\mathbb{P}}_{\varepsilon \mathbf 1_U }\pr{\sup_{s\leq t, y\in F} u_s(y) > \gamma}                                                                     \\
			& \geq \exp\brc{-\frac{\varepsilon e^{\lambda_{\mathrm o}  t}}{1-\gamma}\int_U v_{t}^{(\Lambda,\mu)}(y)\mathrm{d}  y } - 3\varepsilon  \Cr{c:UFtg}(U,F,t, \gamma_0) - \frac{\varepsilon \Psi'(0+)e^{2  \lambda_{\mathrm o} t}}{2(1-\gamma)} \mathcal{V}_t^{(\Lambda, \mu, F)},
		\end{align}
		where in the last inequality we used Lemma \ref{lem: Prob-w}.
		It is clear from the monotone convergence theorem that  $\mathbb E[Z] = \lim_{\varepsilon \downarrow 0}\frac{1}{\varepsilon} \mathbb E [1- (1-\varepsilon)^Z]$ for any non-negative integer-valued random variable $Z$. Therefore,
		\begin{align}
			& \mathbb{E}_{(\Lambda,\mu)} \brk{\tilde Z_t(U)}
			= \lim_{\varepsilon \downarrow 0} \frac{1}{\varepsilon}\pr{ 1- \mathbb{E}_{(\Lambda, \mu)}\brk{(1-\varepsilon)^{\tilde Z_t(U)}} }
			\\& \leq \begin{old}  \lim_{\varepsilon \downarrow 0} \frac{1}{\varepsilon}  \pr{1-  \exp\brc{-\frac{\varepsilon e^{\lambda_{\mathrm o}  t}}{1-\gamma}\int_U v_{t}^{(\Lambda,\mu)}(y)\mathrm{d}  y } } + 3\Cr{c:UFtg}(U,F,t,  \gamma_0)
				+  \frac{ \Psi'(0+)e^{2\lambda_{\mathrm o}  t}}{2(1-\gamma)} \mathcal{V}_t^{(\Lambda, \mu, F)} \end{old}
			\frac{e^{\lambda_{\mathrm o}  t}}{1-\gamma}\int_U v_{t}^{(\Lambda,\mu)}(y)\mathrm{d}  y + \frac{ \Psi'(0+)e^{2\lambda_{\mathrm o}  t}}{2(1-\gamma)} \mathcal{V}_t^{(\Lambda, \mu, F)} + 3\Cr{c:UFtg}(U,F,t, \gamma_0), \label{eq:BNM}
		\end{align}
		which implies the desired result.
		From Lemma \ref{lemma:integral-of-v} (i) and (ii), we know that the right hand side of \eqref{eq:BNM} is finite.
	\end{proof}

	\begin{corollary} \label{cor:FE}
		Suppose that $g$ is a bounded non-negative continuous function whose support is contained in an open interval $U$.
		Suppose that $U\cap F$ is bounded where $F$ is a closed interval containing $\cup_{i=1}^\infty (x_i-1,x_i+1)$.
		Then for any $b\geq a > 0$,
		\begin{equation}
			\sup_{t\in [a,b]} \mathbb E_{(\Lambda, \mu)} \brk{\tilde Z_t(g)} < \infty.
		\end{equation}
	\end{corollary}
	\begin{proof}
		Let $\gamma\in (0,1)$ and $\gamma_0= \gamma \Psi'(0+)/(2\beta_\mathrm c)$.
		From Proposition \ref{prop1}, for every $t>0$,
		\begin{align}
			\label{upp-of-exp}
			& \mathbb E_{(\Lambda, \mu)} \brk{\tilde Z_t(g)}
			\leq  \|g\|_\infty \mathbb E_{(\Lambda, \mu)} \brk{\tilde Z_t(U)}
			\\&\leq
			\|g\|_\infty \frac{e^{\lambda_{\mathrm o} t}}{1-\gamma}\pr{\int_U v_{t}^{(\Lambda, \mu)}(y)\mathrm{d}  y + \frac{\Psi'(0+)e^{\lambda_{\mathrm o} t}}{2}\mathcal{V}_t^{(\Lambda, \mu, \mathbb R)}}+ 3  \|g\|_\infty \Cr{c:UFtg}(U,\mathbb R,t,  \gamma_0).
		\end{align}
		From Lemma \ref{lemma:integral-of-v} (i),
		\begin{align}
			\sup_{t\in [a,b]} \int_U v_{t}^{(\Lambda, \mu)}(y)\mathrm{d}  y
			<\infty;
		\end{align}
		from Lemma \ref{lemma:integral-of-v} (ii), $t\mapsto \mathcal V_t^{(\Lambda, \mu, \mathbb R)}$ is a finite increasing function on $(0,\infty)$;
		from Lemma \ref{lem: Prob-w}, $t\mapsto \Cr{c:UFtg}(U, \mathbb R, t, \gamma_0)$ is a finite increasing function on $(0,\infty)$.
		Therefore, for every $b\geq t\geq a>0$, we have
		\begin{align}
			& \sup_{t\in[a,b]}\mathbb E_{(\Lambda, \mu)} \brk{\tilde Z_t(g)}
			\\&\leq  \|g\|_\infty \frac{e^{\lambda_{\mathrm o} b}}{1-\gamma}\pr{\sup_{t\in [a,b]} \int_U v_{t}^{(\Lambda, \mu)}(y)\mathrm{d}  y + \frac{\Psi'(0+)e^{\lambda_{\mathrm o} b}}{2}\mathcal{V}_b^{(\Lambda, \mu, \mathbb R)}}+ 3  \|g\|_\infty \Cr{c:UFtg}(U,\mathbb R,b,  \gamma_0)
			\\&< \infty
		\end{align}
		as desired for this corollary.
	\end{proof}
	
	As mentioned earlier, we want to show that $(\tilde Z_t)_{t>0}$ has a c\`adl\`ag modification.
	The idea is to consider, for each $g\in \mathcal C_{\mathrm c}^\infty(\mathbb R)$, the following ``super-martingale'':
	\begin{equation}\label{eq:PM}
		e^{\Phi'(0+)t}\tilde Z_t(g) - \frac{1}{2}\int_0^t e^{\Phi'(0+)s} \tilde Z_s(g) \mathrm ds, \quad t\geq 0.
	\end{equation}
	Two technical problems arise:
	\begin{itemize}
		\item[\eq\label{P1}] To utilize the regularization theory of martingales/super-martingales, e.g.~\cite{MR4226142}*{Theorem 9.28}, one typically needs to work with a filtration satisfying the usual hypothesis rather than the natural filtration $(\mathcal F^{\tilde Z}_t)_{t>0}$.
		\item[\eq\label{P2}] The integral term on the right hand side of \eqref{eq:PM} is not well-defined yet, because it is not clear whether $\tilde Z_s(g)$ is measurable in $s$ or not.
	\end{itemize}
	Let the $\sigma$-field $\mathcal F$ and filtration $(\mathcal F_t)_{t> 0}$ be the usual augmentation of $\mathcal F^{\tilde Z}$ and  $(\mathcal F^{\tilde Z}_t)_{t> 0}$ w.r.t.~the probability $\mathbb P_{(\Lambda, \mu)}$ in the sense of \cite{MR4226142}*{Lemma 9.8}.
	We will fix the first technical problem \eqref{P1} by showing that $(\tilde Z_t)_{t> 0}$ is a Markov process w.r.t.~ the filtration $(\mathcal F_t)_{t>0}$ in Proposition \ref{prop:WQ} below.
	We will fix the second problem \eqref{P2} by showing that $(\tilde Z_s(g))_{s > 0}$ has a measurable version in Proposition \ref{Prop:measurability-of-Z} below.
	Here, we say a real-valued process $(A_t)_{t> 0}$ defined on the probability space $(\Omega, \mathcal F, \mathbb P_{(\Lambda, \mu)})$ is measurable,
	if $(\omega, t) \mapsto A_t(\omega)$ is a measurable map from the product measurable space $(\Omega \times (0,\infty), \mathcal F \otimes \mathcal B_{(0,\infty)})$ to $(\mathbb R, \mathcal B(\mathbb R))$.
	Our proofs of Propositions \ref{prop:WQ} and \ref{Prop:measurability-of-Z} rely on some preliminary results saying that $(\tilde Z_t)_{t>0}$ is stochastically right-continuous.
	We will establish those results in Lemmas \ref{lem:SCg} and \ref{lem:SC} below.

	\begin{lemma} \label{lem:SCg}
		Suppose that $F$ is a closed interval containing $\cup_{i\in \N}(x_i -1, x_i +1)$ and that $U$ is an open interval. Suppose that $U\cap F$ is bounded.
		Let  $g$ be a  bounded  continuous function on $\mathbb{R}$ such that  the support of $g$ is contained in $U$.
		Then $(\tilde Z_t(g))_{t>0}$ is stochastically right-continuous, i.e., for any $\epsilon>0$ and $s>0$,
		\begin{align}\label{Conv-in-prob-2}
			\lim_{t\downarrow s}  \mathbb{P}_{(\Lambda, \mu)}
			\pr{ \abs{ \tilde{Z}_t(g) -\tilde{Z}_s(g)} >\epsilon }=0.
		\end{align}
	\end{lemma}
	\begin{proof}
		Without loss of generality, we assume that $g$ is non-negative.
		Define
		\begin{align}\label{Condition-initial-measure}
			F_U(\nu):= \nu(U)+ \sum_{i\in \mathbb{Z}} \frac{\nu((i-1, i+1))}{(|i|+1)^2}, \quad \nu\in \mathcal N.
		\end{align}
		
		\firststep \label{S1}
		We will show in this step that $\mathbb E_{(\Lambda, \mu)}[F_U(\tilde Z_t)] < \infty$ for every $t>0$.
		Let $\gamma\in (0,1)$ and denote $\gamma_0 := \gamma \Psi'(0+)/(2\beta_{\mathrm c}).$
		From Proposition \ref{prop1}, we have
		\begin{align}\label{Moment-of-F-U-Z}
			& \mathbb{E}_{(\Lambda,\mu)}\brk{\tilde{Z}_t((i-1,i+1))}
			\\&
			\leq \frac{e^{\lambda_{\mathrm o} t}}{1-\gamma}\pr{\int_{i-1}^{i+1} v_{t}^{(\Lambda, \mu)}(y)\mathrm{d}  y +\frac{\Psi'(0+)e^{\lambda_{\mathrm o} t}}{2}\mathcal{V}_t^{(\Lambda, \mu, \mathbb R)}}+ 3 \Cr{c:UFtg}((i-1, i+1),\mathbb R,t,  \gamma_0).
		\end{align}
		From \eqref{eq:Upper-bound-of-v},  the fact that $\mathcal{V}_t^{(\Lambda, \mu, \mathbb R)}=0$,
		and how $\Cr{c:UFtg}(\cdot, \cdot, \cdot, \cdot)$ is defined in Lemma \ref{lem: Prob-w}, we have
		\begin{align}
			\mathbb{E}_{(\Lambda,\mu)}\brk{\tilde{Z}_t((i-1,i+1))}
			\leq
			\frac{e^{\lambda_{\mathrm o} t}}{1-\gamma}\frac{4}{\Psi'(0+)t}  + 3 \Cr{c:UFtg}((-1, 1), \mathbb R, t, \gamma_0).
		\end{align}
		This, and the fact that $\Cr{c:UFtg}((-1, 1), \mathbb R, t, \gamma_0)$ is non-decreasing in $t$, implies that
		\begin{equation}\label{eq:ES}
			\sup_{t\in [t_0, T]} \sup_{i\in \mathbb{Z}} \mathbb{E}_{(\Lambda,\mu)}[\tilde{Z}_t((i-1,i+1))] < \frac{e^{\lambda_{\mathrm o} T}}{1-\gamma}\frac{4}{\Psi'(0+)t_0}  + 3 \Cr{c:UFtg}((-1, 1), \mathbb R, T, \gamma_0)
		\end{equation}
		for any $0<t_0<T<\infty$.
		Also from Proposition \ref{prop1}, it holds that $\mathbb{E}_{(\Lambda,\mu)}[\tilde{Z}_t(U)]<\infty$.
		Therefore, the desired result for this step follows.
		
		\nextstep \label{S2SC}
		Fixing an arbitrary $\nu \in \mathcal N$ satisfying that $F_U(\nu) < \infty$, we will show that
		\begin{align}\label{eq4}
			\lim_{s\downarrow 0} \int e^{-\theta \tilde \nu (g)} \mathscr Q_s(\nu, \mathrm{d}\tilde \nu) = e^{-\theta  \nu(g)}, \quad \theta\geq 0.
		\end{align}
		To do this, let us fix an arbitrary $\theta > 0$. By Corollary \ref{cor:ID},  for any $s>0$,
		\begin{align}\label{eq1}
			&
			\int e^{-\theta \tilde \nu(g)} \mathscr Q_s(\nu, \mathrm{d} \tilde \nu)  =
			\tilde{ \mathbb{E}}_{1-e^{-\theta g}}\brk{\prod_{z\in \R}\pr{1-u_s(z)}^{\nu(\brc{z})}}.
		\end{align}
		It is clear that almost surely under $\tilde {\mathbb P}_{1-e^{-\theta g}}$, for every $z\in \mathbb R$, $u_s(z)$ converges to $1-e^{-\theta g(z)}$ when $s\downarrow 0$.
		Define the closed interval
		\begin{equation} \label{eq:DefK}
			K
			:=\brc{z\in \mathbb R: \operatorname{dist}(\{z\}, U) \leq 1+\frac{ |z|}{2}}.
		\end{equation}
		Notice that $U \subset K$ and $K\setminus U$ is bounded.
		Since $\operatorname{supp}(g)\subset U \subset K$, we see that $\nu(g)= (\mathbf 1_K \cdot \nu)(g)$.
		From the condition $F_U(\nu) < \infty$, we have $\nu (U)<\infty$.
		Therefore $\nu(K) = \nu(U) + \nu(K\setminus U) < \infty$. In particular, $\mathbf 1_K \cdot \nu$ is a finite integer-valued measure on $\mathbb R$.
		Therefore, by the bounded convergence theorem
		\begin{align}\label{eq3}
			\lim_{s\to 0}\tilde{\mathbb{E}}_{1-e^{-\theta g}} \brk{\prod_{z\in \R}\pr{1-u_s(z)}^{(\mathbf 1_K \cdot \nu)(\brc{z}) } }
			= e^{-\theta(\mathbf 1_K\cdot\nu)(g)}
			=e^{-\theta \nu(g)}.
		\end{align}
		By Lemmas \ref{lem: general-Bernoulli-ineq} and \ref{Upper-bound-w-t-x-2},  for every $s>0$,
		\begin{align}
			& \tilde{\mathbb{E}}_{1-e^{-\theta g}} \brk{  \abs{ \prod_{z\in \R}\pr{1-u_s(z)}^{\nu(\brc{z}) } -  \prod_{z\in \R}\pr{1-u_s(z)}^{(\mathbf 1_K \cdot \nu)(\brc{z}) } }}
			\\&
			\leq
			\tilde{\mathbb{E}}_{1-e^{-\theta g}} \brk{  \abs{ \prod_{z\in \R}\pr{1-u_s(z)}^{(\mathbf 1_{K^\mathrm c}\cdot \nu)(\brc{z}) } -1}}
			\\&\leq \sum_{z\in \R} \tilde{\mathbb{E}}_{1-e^{-\theta g}} \brk{ u_s(z) } (\mathbf 1_{K^c}\cdot \nu)(\brc{z})
			\leq e^{\lambda_{\mathrm o}s}  \sum_{z\in \R} {\mathbf{E}}_{z} \brk{ 1-e^{-\theta g(B_s)} } \mathbf 1_{K^\mathrm c}(z)\nu(\brc{z}).
		\end{align}
		From $1-e^{-\theta g}\leq \mathbf{1}_U$, \eqref{eq:DefK}, and Markov's inequality, we have for every $s>0$,
		\begin{align}\label{eq2}
			& \tilde{\mathbb{E}}_{1-e^{-\theta g}} \brk{  \abs{ \prod_{z\in \R}\pr{1-u_s(z)}^{\nu(\brc{z}) } -  \prod_{z\in \R}\pr{1-u_s(z)}^{(\mathbf 1_K \cdot \nu)(\brc{z}) } }}
			\\&\leq e^{\lambda_{\mathrm o}s}  \sum_{z\in \R} {\mathbf{P}}_{z} \pr{ B_s\in U } \mathbf 1_{\brc{\operatorname{dist}(\{z\},U) > 1+\frac{|z|}{2}}}\nu(\brc{z})
			\\&\leq e^{\lambda_{\mathrm o}s}  \sum_{z\in \mathbb R} {\mathbf{P}}_{0} \pr{ |B_s|\geq 1+\frac{|z|}{2}} \nu(\brc{z})
			\leq 4s e^{\lambda_{\mathrm o}s}  \sum_{z\in \R} \frac{\nu(\brc{z}) }{(2+|z|)^2}
			\\&\leq 4s e^{\lambda_{\mathrm o}s}  \sum_{i\in \mathbb Z} \sum_{z\in \R} \frac{\nu(\brc{z}) }{(2+|z|)^2}  \mathbf 1_{\{z\in (i-1,i+1)\}}
			\leq 4s e^{\lambda_{\mathrm o}s}  \sum_{i\in \mathbb Z}\frac{\nu((i-1,i+1)) }{(1+|i|)^2}.
		\end{align}
		Therefore, from the condition that $F_U(\nu)< \infty$, the left hand side of \eqref{eq2} converges to $0$ when $s\to 0$.
		Combining this with \eqref{eq1} and \eqref{eq3}, we obtain \eqref{eq4} as desired.
		
		\nextstep \label{S3}
		We will finish the proof.
		Let $\theta >0$ be fixed. According to  Step \ref{S2SC},   for any $\nu\in \mathcal{N}$ with $F_U(\nu)<\infty$,
		\begin{align}
			& \lim_{s\downarrow 0} \int  \pr{ e^{-\theta \tilde{\nu}(g)} - e^{-\theta \nu(g)}}^2  \mathscr Q_s(\nu, \mathrm{d} \tilde \nu)                                                                                                                        \\
			& = \lim_{s\downarrow 0} \int e^{-2\theta \tilde \nu(g)} \mathscr Q_s(\nu, \mathrm{d} \tilde \nu) - 2e^{-\theta \nu(g)} \lim_{s\downarrow 0} \int e^{-\theta \tilde \nu(g)} \mathscr Q_s(\nu, \mathrm{d} \tilde \nu)   + e^{-2\theta \nu(g)}=0. \quad
		\end{align}
		Therefore, from Step \ref{S1},  for every $s>0$ almost surely,
		\begin{equation}\label{eq5}
			\lim_{r\downarrow 0} \int  \pr{ e^{-\theta \tilde{\nu}(g)} - e^{-\theta \tilde Z_s(g)}}^2  \mathscr Q_r(\tilde Z_s, \mathrm{d} \tilde \nu) = 0.
		\end{equation}
		By the Markov property of the process $(\tilde Z_\cdot)$, we conclude from the bounded convergence theorem that, for any $s>0$,
		\begin{align}\label{Conv-in-prob-1}
			& \lim_{t\downarrow s} \mathbb{E}_{(\Lambda,\mu)} \brk{\pr{ e^{-\theta \tilde{Z}_t(g)} - e^{-\theta \tilde{Z}_s(g)} }^2}
			\\&= \lim_{t\downarrow s} \mathbb{E}_{(\Lambda,\mu)} \brk{\int  \pr{ e^{-\theta \tilde{\nu}(g)} - e^{-\theta \tilde{Z}_s(g)}}^2  \mathscr Q_{t-s}(\tilde{Z}_s, \mathrm{d} \tilde \nu)} =0.
		\end{align}
		This says that, for any $s>0$,  the random variable $e^{-\theta \tilde Z_t(g)}$ converges to $e^{-\theta \tilde Z_s(g)}$ when $t\downarrow s$ in $L^2$, and therefore, in probability.
		From the continuous mapping theorem (e.g.~\cite{MR4226142}*{Lemma 5.3}), we get the desired result for this lemma.
	\end{proof}

	\begin{lemma} \label{lem:SC}
		The process $(\tilde Z_t)_{t>0}$ is stochastically right-continuous, i.e., for any $\epsilon > 0$ and $s>0$,
		\begin{equation}
			\lim_{t\downarrow s} \mathbb P_{(\Lambda, \mu)} \pr{d_{\mathcal N}(\tilde Z_t,\tilde Z_s) > \epsilon } = 0.
		\end{equation}
	\end{lemma}
	\begin{proof}
		From \cite{MR4226142}*{Lemma 5.2} and the fact that $d_{\mathcal N}$ is bounded by $1$,  we only have to show that, for any $s>0$,
		\[
		\lim_{t\downarrow s}\mathbb E_{(\Lambda, \mu)} \brk{d_{\mathcal N}(\tilde Z_t, \tilde Z_s)}
		= \lim_{t\downarrow s}\sum_{i=1}^\infty 2^{-i}\mathbb E_{(\Lambda, \mu)}\brk{1\wedge \abs{\tilde Z_t(h_i) - \tilde Z_s (h_i)}}
		= 0.
		\]
		Note that the above holds since, from \cite{MR4226142}*{Lemma 5.2} again and Lemma \ref{lem:SCg}, we have, for any $s>0$ and $i\in \mathbb N$,
		\begin{equation}
			\lim_{t\downarrow s} \mathbb E_{(\Lambda, \mu)} \brk{1\wedge \abs{\tilde Z_t(h_i)-\tilde Z_s(h_i)}} = 0.
		\end{equation}
	\end{proof}

	As mentioned earlier, the first technical problem \eqref{P1} will be handled with the help of the next proposition.
	\begin{proposition} \label{prop:WQ}
		$(\tilde Z_t)_{t>0}$ is a Markov process with transition kernels $(\mathscr Q_s)_{s\geq 0}$ w.r.t.~the filtration $(\mathcal F_t)_{t>0}$, i.e.~$(\tilde Z_t)_{t>0}$ is $(\mathcal F_t)_{t>0}$-adapted and
		\begin{equation} \label{eq:DR}
			\mathbb P_{(\Lambda, \mu)}\pr{\tilde Z_{t+s} \in A \middle| \mathcal F_t}
			= \mathscr Q_s(\tilde Z_t, A),
			\quad a.s.,A \in \mathcal B_{\mathcal N}, s\geq 0, t>0
		\end{equation}
		where $\mathcal B_{\mathcal N}$ is the Borel $\sigma$-field of $\mathcal N$ generated by the vague topology.
	\end{proposition}
	
	\begin{proof}
		It is clear that $(\tilde Z_t)_{t>0}$ is $(\mathcal F_t)_{t>0}$-adapted.
		So we only have to verify \eqref{eq:DR}.
		Note that \eqref{eq:DR} is trivial when $s=0$.
		So, let us fix $s> 0$ and $t>0$.
		Take a non-negative continuous function $g$ on $\mathbb R$ with compact support satisfying that $\|1-e^{-g}\|_\infty \leq \Psi'(0+)/(4\beta_{\mathrm c})$.
		Let $\mathscr N:=\{B\in \mathcal F: \mathbb P_{(\Lambda, \mu)}(B) = 0\}$ be the collection of the null  subsets of $\Omega$.
		Define $\bar {\mathcal F}^{\tilde Z}_r:= \sigma(\mathcal F^{\tilde Z}_r, \mathscr N)$ for every $r>0$.
		From \cite{MR4226142}*{Lemma 9.8}, we have $\mathcal F_t= \cap_{k=1}^\infty \bar {\mathcal F}^{\tilde Z}_{t+1/k}$.
		From Proposition \ref{Part-1-proof-of-thrm-1.2}, we have almost surely for each $k\in \mathbb N$,
		\begin{align}
			& \mathbb E_{(\Lambda, \mu)}\brk{e^{-\tilde Z_{t+1/k+s}(g)}\middle| \bar{\mathcal F}^{\tilde Z}_{t+1/k}}
			=\mathbb E_{(\Lambda, \mu)}\brk{e^{-\tilde Z_{t+1/k+s}(g)}\middle| \mathcal F^{\tilde Z}_{t+1/k}}
			\\&\label{eq:PMP}=\int e^{-\nu(g)} \mathscr Q_s(\tilde Z_{t+1/k},\mathrm{d} \nu)
			= H^{1-e^{-g}}_{s}(\tilde Z_{t+1/k}),
		\end{align}
		where $\nu\mapsto H_s^{1-e^{-g}}(\nu)$ is the bounded continuous function on $\mathcal N$ given as in \eqref{Def-H-f}.
		Taking $k\uparrow \infty$, from the continuous mapping theorem (e.g.~\cite{MR4226142}*{Lemma 5.3}) and Lemma \ref{lem:SC}, we know that the most right hand side of  \eqref{eq:PMP} converges to $H_s^{1-e^{-g}}(\tilde Z_{t})$ in probability when $k\uparrow \infty$.
		From \cite{MR3930614}*{Theorem 4.7.3}, we know that
		\begin{equation} \label{eq:PL1}
			\mathbb E_{(\Lambda, \mu)} \brk{e^{-\tilde Z_{t+s}(g)}\middle| \bar{\mathcal F}^{\tilde Z}_{t+1/k} }
			\xrightarrow[k\uparrow \infty]{L^1} \mathbb E_{(\Lambda, \mu)} \brk{e^{-\tilde Z_{t+s}(g)}\middle| \mathcal F_t}.
		\end{equation}
		Also, from Jensen's inequality, Lemma \ref{lem:SC}, and the bounded convergence theorem, we have that
		\begin{align}
			& \mathbb E_{(\Lambda, \mu)}\brk{\abs{\mathbb E_{(\Lambda, \mu)} \brk{e^{-\tilde Z_{t+1/k+s}(g)}\middle| \bar{\mathcal F}^{\tilde Z}_{t+1/k} } - \mathbb E_{(\Lambda, \mu)} \brk{e^{-\tilde Z_{t+s}(g)}\middle| \bar{\mathcal F}^{\tilde Z}_{t+1/k}} } }
			\\& \leq \mathbb E_{(\Lambda, \mu)}\brk{\mathbb E_{(\Lambda, \mu)} \brk{\abs{e^{-\tilde Z_{t+1/k+s}(g)} - e^{-\tilde Z_{t+s}(g)} }\middle| \bar{\mathcal F}^{\tilde Z}_{t+1/k} }}
			\\&=\mathbb E_{(\Lambda, \mu)}\brk{\abs{e^{-\tilde Z_{t+1/k+s}(g)} - e^{-\tilde Z_{t+s}(g)} }} \xrightarrow[k\uparrow \infty]{} 0.
		\end{align}
		Combining this with \eqref{eq:PL1}, we conclude that the most left hand side of \eqref{eq:PMP} converges to the right hand side of \eqref{eq:PL1} in $L^1$ when $k\uparrow \infty$.
		Now, taking $k\uparrow \infty$ in \eqref{eq:PMP}, we obtain that almost surely
		\begin{equation}
			\mathbb E_{(\Lambda, \mu)} \brk{e^{-\tilde Z_{t+s}(g)}\middle| \mathcal F_t} = H_s^{1-e^{-g}} (\tilde Z_{t}) = \int e^{-\nu(g)} \mathscr Q_s(\tilde Z_t,\mathrm{d} \nu).
		\end{equation}
		\begin{final}
			From \cite{MR4226142}*{Theorem 8.5} and \cite{MR3642325}*{Theorem 2.2}, we can verify the desired result for this proposition.
		\end{final}
		\begin{draft}
			Since $\mathcal N$ is Polish, by \cite{MR4226142}*{Theorem 8.5}, there exists a probability kernel $\mathcal A$ from $(\Omega, \mathcal F_t)$ to $(\mathcal N, \mathcal B_{\mathcal N})$ such that
			\begin{equation} \label{eq:EK}
				\mathbb P_{(\Lambda, \mu)}\pr{\tilde Z_{t+s} \in A \middle| \mathcal F_t} (\omega)
				= \mathcal A(\omega, A),
				\quad a.s.~\omega\in \Omega,A \in \mathcal B_{\mathcal N}.
			\end{equation}
			Moreover, it holds that
			\begin{equation}
				\int e^{-\nu(g)} \mathcal A(\omega, \mathrm{d} \nu)
				= \mathbb E_{(\Lambda, \mu)}\brk{e^{-\tilde Z_{t+s}(g)}\middle| \mathcal F_t} (\omega)
				= \int e^{-\nu(g)} \mathscr Q_s(\tilde Z_t(\omega), \mathrm{d} \nu), \quad a.s.~\omega \in \Omega.
			\end{equation}
			Therefore, fixing an arbitrary non-negative bounded measurable random variable $H$ on $(\Omega,\mathcal F_t)$, we have
			\begin{equation}
				\int_\Omega H(\omega) \mathbb P_{(\Lambda, \mu)}(\mathrm{d} \omega)\int_{\mathcal N} e^{-\nu(g)} \mathcal A(\omega, \mathrm{d} \nu) =
				\int_\Omega H(\omega) \mathbb P_{(\Lambda, \mu)}(\mathrm{d} \omega)\int_{\mathcal N} e^{-\nu(g)} \mathscr Q_s(\tilde Z_t(\omega), \mathrm{d} \nu).
			\end{equation}
			Since $g$ is an arbitrarily chosen non-negative continuous function with compact support satisfying that  $\|1-e^{-g}\|_\infty \leq \Psi'(0+)/(4\beta_{\mathrm c})$, from  \cite{MR3642325}*{Theorem 2.2}, it is not hard to observe that
			\[
			\int_\Omega \mathcal A(\omega, A)H(\omega) \mathbb P_{(\Lambda, \mu)} (\mathrm{d} \omega)  = \int_{\Omega} \mathscr Q_s(\tilde Z_t(\omega),A) H(\omega) \mathbb P_{(\Lambda, \mu)} (\mathrm{d} \omega), \quad A\in \mathcal B_{\mathcal N}.
			\]
			For any $A\in \mathcal B_{\mathcal N}$, since $H$ is an arbitrarilly chosen non-negative bounded measurable function on $(\Omega, \mathcal F_t)$, and since both of $\mathcal A(\cdot, A) $ and $\mathscr Q_s(\tilde Z_t,A)$ are random variables on $(\Omega, \mathcal F_t)$, we conclude from above that $\mathcal A(\cdot, A) = \mathscr Q_s(\tilde Z_t,A)$ almost surely.
			Now the desired result of this proposition follows from \eqref{eq:EK}.
		\end{draft}
	\end{proof}
	
	The second technical problem \eqref{P2} will be handled by the next proposition.

	\begin{proposition}\label{Prop:measurability-of-Z}
		Suppose that $F$ is a closed interval containing $\cup_{i\in \N}(x_i -1, x_i +1)$ and that $U$ is an open interval. Suppose that $U\cap F$ is bounded.
		Let  $g$ be a  bounded  continuous function on $\mathbb{R}$ such that  the support of $g$ is contained in $U$.
		Then under $\mathbb{P}_{(\Lambda, \mu)}$, there exists a measurable version $(\tilde{Y}_t^g)_{t>0}$ of the process $(\tilde{Z}_t(g))_{t>0}$.
	\end{proposition}

	We omit the proof of the above proposition, because it follows the standard argument similar to that of \cite{MR1367959}*{Theorem 6.2.3} noticing, from Lemma \ref{lem:SCg}, that $(\tilde Z_t(g))_{t\geq 0}$ is stochastically right continuous.
	(A detailed proof is included in Appendix \ref{append-D}.)

	\subsection{C\`adl\`ag realization}\label{sec:CR}

	In this subsection, we establish that the $\mathcal N$-valued Markov process $(\tilde Z_t)_{t>0}$ given as in Proposition \ref{Part-1-proof-of-thrm-1.2} has a c\`adl\`ag modification. The probability measure, and the expectation operator, corresponding to $(\tilde Z_t)_{t>0}$ will be denoted by
	$\mathbb P_{(\Lambda, \mu)}$, and $\mathbb E_{(\Lambda, \mu)}$, respectively.
	
	\begin{lemma}\label{lemma: Continuity-for-expectation}
		Fix a  smooth function $g$ with bounded derivatives of all orders.
		Assume that $U$ is an open interval containing the support of $g$.
		Let $\nu\in\mathcal{N}$ satisfy $F_U(\nu)<\infty$ where $F_U(\nu)$ is given as in \eqref{Condition-initial-measure}.
		\begin{enumerate}
			\item[(i)]
			It holds that
			\begin{align}
				\abs{ \E_{(\emptyset, \nu)} \brk{\tilde Z_t(g)} - \E_{(\emptyset, \mathbf 1_U\nu)} \brk{\tilde  Z_t(g)} }
				\leq e^{-\Phi'(0+) t} \int_{U^c}
				\mathbf{E}_x\brk{|g(B_t)|}
				\nu(\mathrm{d} x), \quad t>0.
			\end{align}
			Here, for each $x\in \mathbb R$, $(B_t)_{t\geq 0}$ is a Brownian motion initiated at position $x$ w.r.t.~the expectation operator $\mathbf E_x$.
			\item[(ii)]
			$(\E_{(\emptyset, \nu)}[\tilde  Z_t(g)])_{t> 0}$ is continuous, and $\lim_{t\downarrow 0}\E_{(\emptyset, \nu)}[\tilde  Z_t(g)] = \nu(g)$.
			\item[(iii)]
			Suppose further that $g$ is non-negative.
			Then for every $t>0$,
			\begin{align}\label{Ineq:expectation-non-positive}
				e^{\Phi'(0+) t} \E_{(\emptyset, \nu)}\brk{\tilde Z_t(g)}-\nu(g)- \int_0^t e^{\Phi'(0+)s}\frac{1}{2} \E_{(\emptyset, \nu)}
				\brk{\tilde Z_s(g'')} \mathrm{d} s
				\leq 0.
			\end{align}
		\end{enumerate}
	\end{lemma}
	\begin{proof}
		
		(i). Fix an arbitrary $t>0$.
		From $F_U(\nu)<\infty$ we have $\nu(U)<\infty$.
		In particular, since $\nu$ is an integer-valued measure, $\operatorname{supp}(\nu)\cap U$ is bounded.
		Therefore, from Corollary~\ref{cor:FE}, $\mathbb E_{(\emptyset, \nu)}[\tilde Z_t(g)]$ is finite, and similarly, so is $\mathbb E_{(\emptyset, \mathbf 1_U\nu)}[\tilde Z_t(g)]$.
		We can assume without loss of generality that $g$ is non-negative.
		From Proposition~\ref{prop:ID}, we see that
		\begin{align}
			& \abs{ \E_{(\emptyset, \nu)} \brk{\tilde Z_t(g)} - \E_{(\emptyset, \mathbf 1_U\nu)} \brk{\tilde Z_t(g)} }
			= \abs{\lim_{\varepsilon\downarrow 0} \frac{1}{\varepsilon} \pr{ \E_{(\emptyset, \nu)} \brk{e^{-\varepsilon \tilde Z_t(g)}} - \E_{(\emptyset, \mathbf 1_U\nu)} \brk{ e^{-\varepsilon \tilde Z_t(g)}} }}              \\
			& =\abs{\lim_{\varepsilon\downarrow 0} \frac{1}{\varepsilon}  \tilde{\E}_{1-e^{-\varepsilon g} } \brk{\pr{\prod_{x\in U} \pr{1-u_t(x)}^{\nu(\brc{x})} }\pr{ \prod_{x\in U^c} \pr{1-u_t(x)}^{\nu(\brc{z})} -1 }  } }
			\\&\leq \lim_{\varepsilon\downarrow 0} \frac{1}{\varepsilon}  \tilde{\E}_{1-e^{-\varepsilon g} } \brk{\abs{ \prod_{x\in U^c} \pr{1-u_t(x)}^{\nu(\brc{z})} -1 }  }
		\end{align}
		From Lemmas~\ref{lem: general-Bernoulli-ineq} and~\ref{Upper-bound-w-t-x-2}, Fubini's theorem, and inequality $1-e^{-|x|}\leq |x|$, we get that
		\begin{align}
			& \abs{ \E_{(\emptyset, \nu)} \brk{\tilde Z_t(g)} - \E_{(\emptyset, \mathbf 1_U\nu)} \brk{\tilde Z_t(g)} }
			\leq \lim_{\varepsilon\downarrow 0} \frac{1}{\varepsilon}  \tilde{\E}_{1-e^{-\varepsilon g} } \brk{  \int_{U^c} u_t(x) \nu(\mathrm{d} x)}                     \\
			& \leq  e^{-\Phi'(0+)  t}\lim_{\varepsilon\downarrow 0} \frac{1}{\varepsilon} \int_{U^c}  \mathbf{E}_{x} \brk{  1-e^{-\varepsilon g(B_t)}} \nu(\mathrm{d} x)
			\leq e^{-\Phi'(0+)  t} \int_{U^c} \mathbf{E}_x\brk{g(B_t)} \nu(\mathrm{d} x).
		\end{align}
		
		(ii).
		Note that for each $m\in \N$, $U_m:= \{z: \mbox{dist}(\{z\}, U)<m\}$ is an open interval containing $U$ and $\nu(U_m)<\infty$.
		Fix an arbitrary $m\in \mathbb N$ and keep it fixed for the rest of this paragraph.
		Let $(\hat I_t)_{t\geq 0}$, $(\hat X^{\alpha}_{t})_{\alpha \in \hat I_t,t\geq 0}$ and $(\hat Z_t)_{t\geq 0}$
		be notations  $(I_t)_{t\geq 0}$, $(X^{\alpha}_{t})_{\alpha \in I_t,t\geq 0}$ and $(Z_t)_{t\geq 0}$ given as in Subsection \ref{sec:MR} (right after Proposition \ref{prop:WD}) for an SBBM with ordinary branching rate $\beta_{\mathrm o}$, ordinary offspring law $(p_k)_{k=0}^\infty$, catalytic branching rate $\beta_{\mathrm c}$, and catalytic offspring law $(q_k)_{k=0}^\infty$, given as in \eqref{eq:OBR}--\eqref{eq:CB}, and an initial configuration $(\hat x_i)_{i=1}^{\nu(U_m)}$ 	satisfying $\mathbf 1_{U_m}\nu = \sum_{i=1}^{\nu(U_m)} \delta_{\hat x_i}$.
		It is clear from Propositions~\ref{prop:D} and~\ref{prop:ID} that
		\begin{equation}
			\mathbb E_{(\emptyset, \mathbf 1_{U_m} \nu)}\brk{\prod_{x\in \mathbb R} (1-f(x))^{\tilde Z_t(\{x\})}}
			= \mathbb E\brk{\prod_{x\in \mathbb R} (1-f(x))^{\hat Z_t(\{x\})}}, \quad t>0.
		\end{equation}
		Since $f$ is arbitrary, it follows that
		\begin{itemize}
			\item[\eq\label{eq:ED}]
			for each $t>0$, $\tilde Z_t$ under the probability measure  $\mathbb P_{(\emptyset, \mathbf 1_{U_m} \nu)}$ has the same law as $\hat Z_t$ under $\mathbb P $.
		\end{itemize}
		Define
		\begin{equation}\label{Def-of-R-m}
			R^{(m)}_t(g)
			:= \mathbb E\brk{\hat Z_t(g)}
			= \mathbf 1_{\{0\}}(t) \nu(\mathbf 1_{U_m}g) + \mathbf 1_{(0,\infty)}(t) \mathbb E_{(\emptyset, \mathbf 1_{U_m}\nu)}\brk{\tilde Z_t(g)} , \quad t\geq 0.
		\end{equation}
		From Proposition~\ref{prop:TM}, \eqref{eq:MU1} and \eqref{eq:MU2}, we have for every $t\geq 0$,
		\begin{align}\label{eq: MD2}
			e^{\Phi'(0+) t} \E\brk{\hat Z_t(g)} & = \nu(\mathbf 1_{U_m}g)+ \E\brk{\int_0^t e^{\Phi'(0+) s} \frac{1}{2}\hat Z_s(g'')\mathrm{d} s}                                                                                         \\
			& \quad -\frac{1}{2}\Psi'(0+) \E\brk{\int_0^t e^{\Phi'(0+)s} \sum_{\{\alpha,\beta\}\subset \hat I_{s-}: \alpha \neq \beta} g(\hat X_s^\alpha) \mathrm{d}  \hat L_s^{\{\alpha,\beta\}} }.
		\end{align}
		Here, $(\hat L^{\{\alpha, \beta\}}_t)_{t\geq 0}$ represents the intersection local time between any two particles labelled by $\alpha$ and $\beta$.
		Combining \eqref{eq:MU1}, \eqref{eq:MU2}, \eqref{eq: MD2} and the dominated convergence theorem, we see that $(R^{(m)}_t(g))_{t\geq 0}$ is continuous.
		
		Define
		\begin{equation}\label{Def-of-R}
			R_t(g)
			: = \mathbf 1_{\{0\}}(t) \nu(g) + \mathbf 1_{(0,\infty)}(t) \mathbb E_{(\emptyset,\nu)}\brk{\tilde Z_t(g)},
			\quad t\geq 0.
		\end{equation}
		We want to approximate $(R_t(g))_{t\geq 0}$ by $(R^{(m)}_t(g))_{t\geq 0}$.
		Notice that there exists a $c_0 > 0$ such that for every $x\in U_1^\mathrm c$, $\operatorname{dist}(\{x\},U) \geq c_0( |x|+2)$.
		Also note that for every $m\geq 3$ and $i\in \mathbb{Z}$ with $(i-1, i+1)\nsubseteq U_m$, we have $(i-1,i+1)\subset U_1^\mathrm c$.
		Therefore, for every $m\geq 3$, $i\in \mathbb{Z}$ with $(i-1, i+1)\nsubseteq U_m$, and $x\in (i-1,i+1)$, it holds that $\mbox{dist}(\{x\}, U) \geq c_0( |x|+2)\geq c_0(|i|+1)$.
		It also holds that  $U_m^{\mathrm c} \subset \bigcup_{i\in \mathbb{Z}: (i-1, i+1)\nsubseteq U_m} (i-1, i+1)$ for every $m\in \N$.
		Now, by Lemma~\ref{lemma: Continuity-for-expectation} (i), for $t\geq 0$ and $m\geq 3$,
		\begin{align}
			& e^{\Phi'(0+) t} \abs{ R_t(g) - R^{(m)}_t(g) }
			\leq \int_{U_m^\mathrm c} \mathbf{E}_x\brk{\abs{g(B_t)}} \nu(\mathrm{d} x)
			\\& \leq \sum_{ i\in \mathbb{Z}: (i-1, i+1)\nsubseteq U_m}  \nu((i-1, i+1))  \sup_{x\in (i-1, i+1)} \mathbf{E}_x \brk{ \abs{g(B_t)}}\\
			& \leq  \|g\|_\infty \sum_{ i\in \mathbb{Z}: (i-1, i+1)\nsubseteq U_m}  \nu((i-1, i+1))  \sup_{x\in (i-1, i+1)} \mathbf{P}_x \pr{ B_t \in U} \\
			& \leq  \|g\|_\infty \sum_{ i\in \mathbb{Z}: (i-1, i+1)\nsubseteq U_m}  \nu((i-1, i+1))  \mathbf{P}_0 \pr{ |B_t| \geq c_0 (|i|+1)}.
		\end{align}
		Together with Markov's inequality,  we have
		\begin{align}\label{proof-Remain}
			& e^{\Phi'(0+) t} \abs{ R_t(g) - R^{(m)}_t(g) }
			\leq  \frac{t}{c_0^2}\|g\|_\infty \sum_{ i\in \mathbb{Z}: (i-1, i+1)\nsubseteq U_m}  \frac{\nu((i-1, i+1))  }{(|i|+1)^2},
			\quad t\geq 0.
		\end{align}
		By the condition $F_U(\nu)<\infty$ and the monotone convergence theorem, we have
		\begin{equation} \label{eq:ULT}
			\lim_{m\to \infty}\sup_{t\in [0,T]} \abs{R_t(g) - R_t^{(m)}(g)} = 0.
		\end{equation}
		By the uniform limit theorem, $(R_t(g))_{t\geq 0}$ must be continuous.
		
		(iii).
		Fix an arbitrary $m\in \mathbb N$, and let $(\hat Z_t)_{t\geq 0}$ be the SBBM considered in (ii).
		From \eqref{eq:ED}, \eqref{eq: MD2} and the non-negativity of $g$, we have for every $t>0$
		\begin{align}
			& e^{\Phi'(0+) t} \E_{(\emptyset, \mathbf 1_{U_m}\nu)}\brk{\tilde Z_t(g)}-\nu(\mathbf 1_{U_m} g)- \int_0^t e^{\Phi'(0+)s}\frac{1}{2} \E_{(\emptyset,   \mathbf 1_{U_m}\nu)}\brk{\tilde Z_s(g'')} \mathrm{d} s
			\\&=e^{\Phi'(0+) t} \E\brk{\hat Z_t(g)}-\nu(\mathbf 1_{U_m} g)- \int_0^t e^{\Phi'(0+)s}\frac{1}{2} \E\brk{\hat Z_s(g'')} \mathrm{d} s \leq 0.
		\end{align}
		Now, for any $t>0$,  we have
		\begin{align}
			& e^{\Phi'(0+) t} \E_{(\emptyset, \nu)}\brk{\tilde Z_t(g)}-\nu(g)- \int_0^t e^{\Phi'(0+)s}\frac{1}{2} \E_{(\emptyset, \nu)}\brk{\tilde Z_s(g'')} \mathrm{d} s
			\\& \leq  e^{\Phi'(0+) t} \E_{(\emptyset, \mathbf 1_{U_m}\nu)}\brk{\tilde Z_t(g)}-\nu(\mathbf 1_{U_m}g)- \int_0^t e^{\Phi'(0+)s}\frac{1}{2} \E_{(\emptyset,   \mathbf 1_{U_m}\nu)}\brk{\tilde Z_s(g'')} \mathrm{d} s
			\\& \qquad +
			e^{\Phi'(0+) t} \abs{ R_t(g) - R^{(m)}_t(g) } +  \nu(\mathbf 1_{U_m^{\mathrm c}}g)+\int_0^t e^{\Phi'(0+) s} \abs{ R_s(g'') - R^{(m)}_s(g'') }  \mathrm{d} s
			\\&
			\quad \leq  e^{\Phi'(0+) t} \abs{ R_t(g) - R^{(m)}_t(g) } +  \nu(\mathbf 1_{U_m^{\mathrm c}}g)+\int_0^t e^{\Phi'(0+) s} \abs{ R_s(g'') - R^{(m)}_s(g'') }  \mathrm{d} s. \label{eq:FE}
		\end{align}
		Since $\nu(g)<\infty$, it follows from \eqref{eq:ULT} and the fact that \eqref{eq:ULT} also holds with $g$ being replaced by $g''$ that the right hand side of \eqref{eq:FE} converges to $0$ as $m\to \infty$.
		Therefore, the desired result in (iii) of this lemma holds.
	\end{proof}

	\begin{lemma}\label{lemma: super-martingale}
		Let $g$ be a non-negative smooth function on $\mathbb R$
		with bounded derivatives of all orders whose support is contained in an open interval $U$.
		Suppose that $U\cap F$ is bounded, where $F$ is a closed interval containing $\cup_{i=1}^\infty (x_i-1, x_i+1)$.
		Let $a>0$.
		Let $(\tilde Y^{g''}_t)_{t>0}$ be a measurable version of the process $(\tilde Z_t(g''))_{t>0}$ given in Proposition~\ref{Prop:measurability-of-Z}.
		Then
		\begin{equation} \label{Def-of-martingale}
			M_g(t;a)
			:= e^{\Phi'(0+) t} \tilde Z_t(g) - \int_a^t e^{\Phi'(0+) s}\frac{1}{2} \tilde{Y}_s^{g''}\mathrm{d} s, \quad t\geq a,
		\end{equation}
		is a super-martingale on the filtered probability space $(\Omega, \mathcal F, (\mathcal F_t)_{t\geq a}, \mathbb P_{(\Lambda,\mu)})$. In particular, both $(M_g(t;a))_{t\geq a}$ and $(\tilde Z_t(g))_{t\geq a}$ have c\`adl\`ag modification.
	\end{lemma}
	
	\begin{proof}
		
		From Corollary~\ref{cor:FE} and Proposition~\ref{Prop:measurability-of-Z},
		we know that the second term on the right-hand side of \eqref{Def-of-martingale} is a well-defined random variable with finite mean.
		Clearly the process $(M_g(t;a))_{t\geq a}$ is adapted to the filtration $(\mathcal F_t)_{t\geq a}$.
		Notice that almost surely, for every $t\geq a$ and $r\geq 0$,
		\begin{equation}
			M_g(t+r;a) - M_g(t ; a)
			= e^{\Phi'(0+) t}
			\pr{
				e^{\Phi'(0+)r} \tilde Z_{t+r}(g)
				- \tilde Z_t(g)
				- \int_{0}^{r}
				e^{\Phi'(0+)s}
				\frac{1}{2}
				\tilde Y_{t+s}^{g''}
				\mathrm{d} s
			}.
		\end{equation}
		Let us fix an arbitrary $t\geq a$ and an arbitrary bounded continuous function $\phi$ on $\mathbb R$ whose support is contained in $U$.
		
		From Remark~\ref{rem:LQ}, Proposition~\ref{prop:WQ} and \cite{MR4226142}*{Theorem 8.5},
		for each $r\geq 0$ and $m\in \mathbb N$, we have almost surely,
		\[
		\mathbb E_{(\Lambda, \mu)}\brk{\mathcal X_m\pr{\tilde Z_{t+r}(\phi)} \middle | \mathcal F_t}
		= \int \mathcal X_m(\nu(\phi)) \mathscr Q_{r}(\tilde Z_t, \mathrm{d} \nu)
		=\mathbb E_{(\emptyset, \tilde Z_t)}\brk{\mathcal X_m\pr{\tilde Z_{r}(\phi)}},
		\]
		where the truncation function $\mathcal X_m (z) := z \mathbf 1_{[-m,m]}(z)$ for every $z\in \mathbb R$.
		Taking $m\uparrow \infty$, it follows from the dominated convergence theorem and Corollary~\ref{cor:FE} that, for each $r\geq 0$, almost surely, $\mathbb E_{(\Lambda, \mu)}[\tilde Z_{t+r}(\phi) |\mathcal F_t]= \mathbb E_{(\emptyset, \tilde Z_t)}[\tilde Z_{r}(\phi)].$
		Combining this with Proposition~\ref{Prop:measurability-of-Z}, we have, for each $r\geq 0$, almost surely, $\mathbb E_{(\Lambda, \mu)}[\tilde Y^\phi_{t+r}|\mathcal F_t] =  \mathbb E_{(\emptyset, \tilde Z_t)}[\tilde Z_{r}(\phi)]. $
		By Lemma~\ref{lemma: Continuity-for-expectation} (ii), almost surely, $\mathbb E_{(\emptyset, \tilde Z_t)}[\tilde Z_r(\phi)]$ is continuous, and therefore measurable, in $r>0$.
		From Corollary~\ref{cor:FE} and Fubini's theorem, we verify that, for each $r\geq 0,$ almost surely,
		\begin{equation}
			\mathbb E_{(\Lambda, \mu)} \brk{\int_0^r e^{\Phi'(0+)s} \frac{1}{2} \tilde Y^{\phi}_{t+s}\mathrm{d} s \middle| \mathcal F_t}
			= \int_0^r e^{\Phi'(0+)s}\frac{1}{2}\mathbb E_{(\emptyset, \tilde Z_t)}\brk{ \tilde Z_s(\phi)} \mathrm{d} s.
		\end{equation}
		Now, we verify that for each $r\geq 0$, almost surely
		\begin{align}
			& \mathbb E_{(\Lambda, \mu)} \brk{M_g(t+r;a) - M_g(t; a) \middle| \mathcal F_t}
			\\&= e^{\Phi'(0+)t} \pr{e^{\Phi'(0+)r} \mathbb E_{(\emptyset, \tilde Z_t)}\brk{\tilde Z_{r}(g)} - \tilde Z_t(g) -\int_0^r e^{\Phi'(0+)s}\frac{1}{2}  \mathbb E_{(\emptyset, \tilde Z_t)}\brk{ \tilde Z_s(g'')}\mathrm{d} s}.
		\end{align}
		From the Step \ref{S1} of the proof  of Lemma \ref{lem:SCg}, we know that $F_U(\tilde Z_t) < \infty$ almost surely w.r.t.~$\mathbb P_{(\Lambda, \mu)}$.
		Therefore, by Lemma~\ref{lemma: Continuity-for-expectation} (iii),
		for each $r\geq 0$, almost surely,
		$\mathbb E_{(\Lambda, \mu)} [M_g(t+r;a) - M_g(t;a) | \mathcal F_t] \leq 0.$
		Since $t\geq a$ is chosen arbitrarily, we conclude that $(M_g(t;a))_{t\geq a}$ is a super-martingale with respect to the filtration $(\mathcal F_t)_{t\geq a}$.
		It also follows from Lemma~\ref{lemma: Continuity-for-expectation} (ii) and the dominated convergence theorem that $(\mathbb E_{(\Lambda, \mu)}[M_g(t;a)])_{t\geq a}$ is a continuous process.
		Therefore, by \cite{MR4226142}*{Theorem 9.28}, $(M_g(t;a))_{t\geq a}$ has a c\`adl\`ag modification, as does $(\tilde Z_t(g))_{t\geq a}$.
	\end{proof}
	
	The following lemma is standard and the proof is included
	in Appendix~\ref{append-D}.

	\begin{lemma} \label{lem:SRM}
		Suppose that $t_0 \geq 0$ and $(\tilde X_t)_{t\geq t_0}$ is an $\mathcal N$-valued process such that the real-valued process $(\tilde X_t(g))_{t\geq t_0}$ admits a c\`adl\`ag modification for every
		$g\in \mathcal C_{\mathrm c}^\infty (\mathbb R)$.
		Then $(\tilde X_t)_{t\geq t_0}$ itself admits a c\`adl\`ag modification.
	\end{lemma}

	\begin{proposition}\label{prop:cadlag}
		The $\mathcal N$-valued process $(\tilde Z_t)_{t>0}$ admits a c\`adl\`ag modification.
	\end{proposition}
	\begin{proof}
		Let $t_0 >0$ be arbitrary.
		It suffices to show that the $\mathcal N$-valued process  $(\tilde Z_t)_{t\geq t_0}$ admits a c\`adl\`ag modification.
		By Lemma~\ref{lem:SRM}, it further suffices to show that the real-valued process $(\tilde Z_t(g))_{t\geq t_0}$ has a c\`adl\`ag modification for each $g\in \mathcal C_\mathrm{c}^\infty(\mathbb R)$.
		Fix such a $g\in \mathcal C_{\mathrm{c}}^\infty(\mathbb{R})$.
		We decompose $g$ into the difference of two non-negative smooth functions with compact support: to this end, construct a function $\psi \in \mathcal{C}_{\mathrm{c}}^\infty(\mathbb{R})$ dominating $|g|$ on its support.
		Define
		\[
		g_1 := \frac{\psi + g}{2}, \quad g_2 := \frac{\psi - g}{2}.
		\]
		Then $g_1, g_2 \in \mathcal{C}_{\mathrm{c}}^\infty(\mathbb{R})$ and $g_1, g_2 \ge 0$. Moreover, $g = g_1 - g_2$, and by the linearity of $\tilde Z$, $\tilde Z_t(g) = \tilde Z_t(g_1) - \tilde Z_t(g_2)$ for every $t>0$ almost surely.
		From Lemma~\ref{lemma: super-martingale}, both of the real-valued processes $(\tilde Z_t(g_1))_{t\geq t_0}$ and $(\tilde Z_t(g_2))_{t\geq t_0}$ admit c\`adl\`ag modification. Therefore, so does $(\tilde Z_t(g))_{t\geq t_0}$.
	\end{proof}

	\subsection{Skorohod convergence}\label{SS5.4}
	In this subsection, we complete the proof of Theorem~\ref{thm:existence-of-Z-t}.
	
	Recall from \eqref{eq:dN} that we have regarded the space of Radon measures on $\R$ as a separable complete metric space $(\mathcal N, d_{\mathcal N})$.
	For any separable complete metric space $\mathcal S = (\mathcal S, d_{\mathcal S})$ and $t_0\geq 0$, denote by $\mathbb{D}([t_0,\infty), \mathcal{S})$ the space of $\mathcal{S}$-valued c\`adl\`ag functions on $[t_0,\infty)$ equipped with the Skorokhod $J_1$-topology in the sense of \cite{MR4226142}*{Lemma A5.3}. It is known that $\mathbb{D}([t_0,\infty), \mathcal{S})$ is a Polish space.
	
	\begin{definition}
		Let $\mathcal S = (\mathcal S, d_{\mathcal S})$ be a separable complete metric space.
		For each $n\in \mathbb N$, let $(X^{(n)}_t)_{t\geq 0}$ be an $\mathcal{S}$-valued c\`adl\`ag process on a filtered probability space $(\Omega^{(n)}, \mathcal F^{(n)}, (\mathcal F^{(n)}_t)_{t\geq 0}, \mathbb P^{(n)})$.
		Let $t_0\geq 0$.
		We say the sequence $((X^{(n)}_t)_{t\geq 0})_{n\in \mathbb N}$ satisfies Aldous' criteria in $\mathbb{D}([t_0,\infty), \mathcal{S})$ provided that:
		for any
		\begin{itemize}
			\item[\eq\label{eq:T0}] $T>t_0$,
			\item[\eq\label{eq:tau0}] sequence of $[t_0,T]$-valued random variables $(\tau_n)_{n\in \mathbb N}$ satisfying that, for each $n\in \mathbb N$, $\tau_n$ is an $(\mathcal F^{(n)}_t)_{t\geq 0}$-optional time,
			\item[\eq\label{eq:delta0}] and sequence of positive numbers $(\delta_n)_{n\in \mathbb N}$ converging to $0$,
		\end{itemize}
		it holds that
		\[
		d_{\mathcal S}\pr{X^{(n)}_{\tau_n+\delta_n} , X^{(n)}_{\tau_n}} \xrightarrow[n\to \infty]{} 0, \quad \text{in probability}.
		\]
	\end{definition}
	
	\begin{lemma}\label{lem:convergence_aldous}
		Let $t_0\geq 0$.
		For each $n\in \mathbb N$, let $(X_t^{(n)})_{t\geq 0}$ be a c\`adl\`ag process in $\mathcal N$ on a filtered probability space $(\Omega^{(n)}, \mathcal F^{(n)}, (\mathcal F^{(n)}_t)_{t\geq 0}, \mathbb P^{(n)})$.
		Assume that for any
		$h\in \mathcal C_{\mathrm c}^\infty(\mathbb{R})$,
		the sequence of real-valued processes $((X^{(n)}_t(h))_{t\geq 0})_{n\in \mathbb N}$ satisfies Aldous' criteria in $\mathbb{D}([t_0,\infty), \mathbb R)$.
		Then $((X^{(n)}_t)_{t\geq 0})_{n\in \mathbb N}$ itself satisfies Aldous' criteria in $\mathbb{D}([t_0,\infty), \mathcal{N})$.
	\end{lemma}
	
	\begin{proof}
		Fix arbitrary $T$, $(\tau_n)_{n\in \mathbb N}$, and $(\delta_n)_{n\in \mathbb N}$ as in \eqref{eq:T0}--\eqref{eq:delta0}. Let $\eta > 0$ be arbitrary and let $\mathbb P^{(n)}$ be the probability measure corresponding to the process $(X_t^{(n)})_{t\geq 0}$ 	for each $n\in \mathbb N$.
		It suffices to show that
		\begin{equation}\label{eq:DR2}
			\mathbb P^{(n)}\pr{ d_{\mathcal{N}}\pr{X^{(n)}_{\tau_n + \delta_n}, X^{(n)}_{\tau_n}} > \eta } \xrightarrow[n\to \infty]{} 0.
		\end{equation}
		Choose an integer $K$ large enough such that
		\[
		\sum_{k=K+1}^\infty \frac{1}{2^k} 		< \frac{\eta}{2}.
		\]
		For every $n\in \mathbb N$, $\mathbb P^{(n)}$-almost surely,
		\[
		d_{\mathcal{N}}\pr{X^{(n)}_{\tau_n + \delta_n}, X^{(n)}_{\tau_n}}
		\overset{\eqref{eq:dN}} \leq \sum_{k=1}^K \frac{1}{2^k} \left| X^{(n)}_{\tau_n + \delta_n}(h_k) - X^{(n)}_{\tau_n}(h_k) \right| 		+ \frac{\eta}{2},
		\]
		and therefore, we have the event inclusion
		\[
		\left\{ d_{\mathcal{N}}(X^{(n)}_{\tau_n + \delta_n}, X^{(n)}_{\tau_n}) > \eta \right\}
		\subseteq \bigcup_{k=1}^K \left\{ \left| X^{(n)}_{\tau_n + \delta_n}(h_k) - X^{(n)}_{\tau_n}(h_k) \right| > \frac{2^k\eta}{2K} \right\}
		\]
		up to a $\mathbb P^{(n)}$-null set.
		From this, we obtain
		\[
		\bbP^{(n)}\left( d_{\mathcal{N}}(X^{(n)}_{\tau_n + \delta_n}, X^{(n)}_{\tau_n}) > \eta \right)
		\leq \sum_{k=1}^K \bbP^{(n)}\left( \left| X^{(n)}_{\tau_n + \delta_n}(h_k) - X^{(n)}_{\tau_n}(h_k) \right| > \frac{2^k\eta}{2K} \right) \xrightarrow{n \to \infty} 0
		\]
		where in the last step we used the condition that, for any $h\in \mathcal C_{\mathrm c}^\infty(\mathbb{R})$, the sequence of real-valued processes $((X^{(n)}_t(h))_{t\geq 0})_{n\in \mathbb N}$ satisfies Aldous' criteria in $\mathbb{D}([t_0,\infty), \mathbb R)$.
		Now \eqref{eq:DR2} is verified.
	\end{proof}
	
	\begin{lemma}\label{lem:global_convergence}
		Let $f \in \mathcal{C}_{\mathrm{c}}(\mathbb{R}, [0, z^*])$. Let the continuous $\mathcal{C}(\mathbb{R}, [0, z^*])$-valued process $(u_t)_{t \ge 0}$ be a solution to the SPDE~\eqref{eq:GSPDE} with initial condition $u_0=f$. Then,
		\begin{equation}
			\sup_{x \in \mathbb{R}} \tilde{\mathbb{E}}_{f} \brk{ \abs{u_t(x) - f(x)} } \xrightarrow[t\downarrow 0]{} 0.
		\end{equation}
	\end{lemma}
	
	\begin{proof}
		Using the mild formulation of the SPDE~\eqref{eq:mild}, we write for every $t>0$ and $x\in \mathbb R$, almost surely, the difference $u_t(x) - u_0(x)$ as the sum of three terms:
		\begin{align*}
			u_t(x) - u_0(x) & = \underbrace{\int_{\mathbb{R}} p_t(x-y) u_0(y) \, \mathrm{d} y - u_0(x)}_{=:I_1(t, x)} - \underbrace{\iint_0^t p_{t-s}(x-y) \Phi(u_s(y)) \, \mathrm{d} s \mathrm{d} y}_{=:I_2(t, x)} \\
			& \quad + \underbrace{\iint_0^t p_{t-s}(x-y) \sqrt{\Psi(u_s(y))} \, W(\mathrm{d} s \mathrm{d} y)}_{=:I_3(t, x)}.
		\end{align*}
		By the triangle inequality, $\tilde{\mathbb{E}}_f[\abs{u_t(x) - u_0(x)}] \le \abs{I_1(t, x)} + \tilde{\mathbb{E}}_f[\abs{I_2(t, x)}] + \tilde{\mathbb{E}}_f[\abs{I_3(t, x)}]$ for each $t>0$ and $x\in \mathbb R$.
		
		By \cite{MR1681462}*{Lemma 8.4}, $u_0 = f\in \mathcal C_{\mathrm c}(\mathbb R)$ is bounded and uniformly continuous.
		Therefore, by \cite{MR1681462}*{Theorem 8.14},
		\[ \lim_{t \to 0} \sup_{x \in \mathbb R} \abs{I_1(t, x)} = 0. \]
		Since $\Phi$ is bounded on $[0, z^*]$, uniformly for every $t>0$ and $x\in \mathbb R$,
		\begin{align*}
			\tilde{\mathbb{E}}_f\brk{\abs{I_2(t, x)}} & \lesssim \int_0^t \left( \int_{\mathbb{R}} p_{t-s}(x-y) \, \mathrm{d} y \right) \mathrm{d} s = t.
		\end{align*}
		Since $\sqrt{\Psi(\cdot)}$ is bounded on $[0, z^*]$, uniformly for every $t>0$ and $x\in \mathbb R$,
		\begin{align*}
			& \tilde{\mathbb{E}}_f\brk{\abs{I_3(t, x)}}^2 \overset{\text{Jensen}}{\leq} \tilde{\mathbb{E}}_f\brk{\abs{I_3(t, x)}^2}
			\overset{\text{It\^o's isometry}}{=} \iint_0^t  p_{t-s}(x-y)^2 \tilde{\mathbb{E}}_f[\Psi(u_s(y))] \, \mathrm{d} s \mathrm{d} y \\
			& \lesssim \iint_0^t \frac{1}{t-s} e^{-\frac{(x-y)^2}{t-s}} \, \mathrm{d} s \mathrm{d} y
			\lesssim \int_0^t \frac{1}{\sqrt{t-s}} \left( \int_{\mathbb{R}} p_{\frac{t-s}{2}}(x-y) \mathrm{d} y \right) \mathrm{d} s
			\\& = \int_0^t \frac{1}{\sqrt{s}} \mathrm{d} s \lesssim \sqrt{t}.
		\end{align*}
		
		The desired result now follows from those estimates.
	\end{proof}
	
	\begin{lemma}\label{lem:convergence_2}
		
		Let $f \in \mathcal{C}_{\mathrm{c}}(\mathbb{R}, [0, z^*])$ and  $(u_t)_{t \ge 0}$ be the continuous $\mathcal{C}(\mathbb{R}, [0, z^*])$-valued process that is a solution to the SPDE~\eqref{eq:GSPDE} with initial condition $u_0=f$. Then,
		\begin{equation}
			\lim_{t \downarrow 0} \sup_{x \in \mathbb{R}} \pr{(1+x^2) \tilde{\mathbb{E}}_{f} \brk{ \abs{u_t(x) - f(x)} }} = 0.
		\end{equation}
	\end{lemma}
	
	\begin{proof}
		Since $f$ has compact support, there exists $R > 0$ such that $\text{supp}(f) \subset [-R, R]$. Decompose the real line into two parts: the bounded region $D_1 := [-R-1, R+1]$ and the tail region $D_2 := \mathbb{R} \setminus D_1$.
		
		\textit{The bounded region.}
		On $D_1$, the weight $(1+x^2)$ is bounded by $1 + (R+1)^2$. Using Lemma~\ref{lem:global_convergence}, we have
		\[
		\sup_{x \in D_1} (1+x^2) \tilde{\mathbb{E}}_{f} \brk{ \abs{u_t(x) - f(x)} }
		\le (1 + (R+1)^2)\sup_{x \in \mathbb{R}} \tilde{\mathbb{E}}_{f} \brk{ \abs{u_t(x) - f(x)} } \xrightarrow{t \downarrow 0} 0.
		\]
		
		\textit{The tail region.}
		For $x \in D_2$, we have $x \notin \text{supp}(f)$, so $f(x) = 0$. Using Lemma~\ref{Upper-bound-w-t-x-2}, we have for $x\in D_2$ and $t>0$,
		\[
		\tilde{\mathbb{E}}_{f} \brk{ \abs{u_t(x) - f(x)} } = \tilde{\mathbb{E}}_{f}[u_t(x)] \le e^{\lambda_{\mathrm o} t} \mathbf E_x[f(B_t)]
		\]
		where $B$ is a standard Brownian motion starting at $x$ under the probability $\mathbf P_x$.
		Since $f$ is bounded by $z^*$ and supported in $[-R, R]$, for $x\in D_2$ and $t>0$,
		\begin{align}
			& \mathbf E_x[f(B_t)] \le z^* \mathbf P_x(B_t \in [-R, R]) \le z^* \mathbf P_0(\abs{B_t} \ge \abs{x} - R)             \\
			& \overset{\text{Chebyshev}} \leq z^* \frac{\mathbf E[\abs{B_t}^2]}{(\abs{x} - R)^2} = \frac{z^* t}{(\abs{x} - R)^2}.
		\end{align}
		Substituting this back, bound the term on $D_2$:
		\[
		\sup_{x \in D_2} (1+x^2) \tilde{\mathbb{E}}_{f} \brk{ \abs{u_t(x) - f(x)} }
		\le e^{\lambda_{\mathrm o} t} z^* t \underbrace{\sup_{x \in D_2} \frac{1+x^2}{(\abs{x} - R)^2}}_{<\infty}
		\xrightarrow{t \downarrow 0} 0.
		\]
		
		Combining the analysis for $D_1$ and $D_2$ completes the proof.
	\end{proof}
	
	Recall that the probability transition kernels $(\mathscr Q_t)_{t\geq 0}$ for SBBMs with branching mechanisms $(\Phi, \Psi)$ are given by Proposition~\ref{Part-1-proof-of-thrm-1.2}.
	\begin{lemma}\label{lem:transition_bound}
		Suppose that $h \in \mathcal{C}_{\mathrm{c}}^+(\mathbb{R})$.
		Let $(u_t)_{t \ge 0}$ be the solution to the dual SPDE~\eqref{eq:GSPDE} with initial condition $u_0 = 1 - e^{-h} \in \mathcal C_{\mathrm c}(\mathbb R, [0,1])$.
		Then for any $t\geq 0$ and $\mu\in \mathcal N$,
		\begin{align}\label{eq:KK}
			& \abs{ \int_{\mathcal{N}} e^{-\nu(h)} \mathscr{Q}_t(\mu, \mathrm{d} \nu) - e^{-\mu(h)} }
			\\&\le \pr{\int_{\R} (1+x^2)^{-1} \mu(\mathrm d x )} \sup_{x\in \mathbb R} \pr{ (1+x^2) \tilde{\mathbb E}_{1-e^{-h}}\brk{\abs{u_t(x)-u_0(x)}} }.
		\end{align}
	\end{lemma}
	
	\begin{proof}
		By \eqref{eq:KC}, we have for every $t\geq 0$,
		\begin{equation}\label{eq:dual_ident}
			\int_{\mathcal{N}} e^{-\nu(h)} \mathscr{Q}_t(\mu, \mathrm{d} \nu)
			= \tilde{\E}_{f} \brk{ \prod_{x \in \mathbb R} (1 - u_t(x))^{\mu(\{x\})} },
		\end{equation}
		where $f := 1 - e^{-h}$. 	Note that by definition, $u_0 = f = 1 - e^{-h}$, and thus $e^{-\mu(h)} = \prod_{x \in \mathbb R} (1 - u_0(x))^{\mu(\{x\})}$ for any $\mu \in \mathcal N$.
		
		The random field $u$ takes values in $[0, z^*]$ with $z^* \in [1, 2)$. Consequently, the random field $1 - u$ takes values in the interval
		
		$[1-z^*, 1] \subset (-1, 1]$.
		We use the following elementary inequality from \cite{MR3930614}*{Lemma 3.4.3}: for sequences $(a_i)_{i\in \mathbb N}$ and $(b_i)_{i\in \mathbb N}$ in $[-1,1]$, it holds that $\abs{\prod_{i\in \mathbb N} a_i - \prod_{i\in \mathbb N} b_i} \le \sum_{i\in \mathbb N} \abs{a_i - b_i}$.
		By this inequality, we obtain $\tilde{\mathbb P}_f$-almost surely for every $t>0$ and $\mu\in \mathcal N$:
		\begin{align*}
			& \abs{ \prod_{x\in \mathbb R} (1 - u_t(x))^{\mu(\{x\})} - \prod_{x\in \mathbb R} (1 - u_0(x))^{\mu(\{x\})} }
			\le \sum_{x \in \mathbb R} \abs{ (1 - u_t(x)) - (1 - u_0(x)) } \mu(\{x\})                                      \\
			& = \int_{\mathbb{R}} \abs{ u_t(x) - u_0(x) } \, \mu(\mathrm{d} x).
		\end{align*}
		Take the dual expectation $\tilde{\E}_f$ on both sides and apply Jensen's inequality to get
		\begin{align*}
			& \abs{ \int_{\mathcal{N}} e^{-\nu(h)} \mathscr{Q}_t(\mu, \mathrm{d} \nu) - e^{-\mu(h)} }
			\le \tilde{\E}_{f} \brk{ \int_{\mathbb{R}} \abs{ u_t(x) - u_0(x) } \, \mu(\mathrm{d} x) }                         \\
			& \overset{\text{Fubini}}= \int_{\mathbb{R}} \tilde{\E}_{f} \brk{ \abs{ u_t(x) - u_0(x) } } \, \mu(\mathrm{d} x)
			\\&\leq \pr{\int_{\R} (1+x^2)^{-1} \mu(\mathrm d x )} \sup_{x\in \mathbb R} \pr{ (1+x^2) \tilde{\mathbb E}_{f}\brk{\abs{u_t(x)-u_0(x)}} }
		\end{align*}
		for every $t\geq 0$ and $\mu\in \mathcal N$ as desired.
	\end{proof}
	
	Recall that we have fixed the branching mechanisms $\Phi$ and $\Psi$.
	\begin{lemma}\label{lem:stopping_time_bound}
		Let $0 < t_0 < T < \infty$.
		Define the weight function $g(x) := \frac{1}{1+x^2}$ for each $x\in \mathbb R$.
		Then uniformly for any SBBM $(Z_t)_{t\geq 0}$ adapted to a filtration $(\mathcal F_t)_{t\geq 0}$ with branching mechanisms $(\Phi,\Psi)$ whose initial state $Z_0$ is deterministic and finite, and uniformly for any $[t_0, T]$-valued $(\mathcal F_t)_{t\geq 0}$-optional time $\tau$, we have
		\begin{equation}\label{eq:Ztg}
			\mathbb{E}\brk{ Z_{\tau}(g) }
			\lesssim 1,
		\end{equation}
		and, uniformly for any $M>0$,
		\begin{equation} \label{eq:PtM}
			\mathbb P\pr{\sup_{t\in [t_0, T]} Z_t(g) > M} \lesssim \frac{1}{M}.
		\end{equation}
	\end{lemma}
	
	\begin{proof}
		The weight function $g$ belongs to $\mathcal{C}_{\mathrm{b}}^\infty(\mathbb{R})$. Moreover, for any $x \in \mathbb{R}$,
		\[
		g''(x) = \frac{6x^2-2}{(1+x^2)^3} \implies \abs{g''(x)} \le \frac{6(1+x^2)}{(1+x^2)^3} \le \frac{6}{(1+x^2)^2} \le 6 g(x).
		\]
		
		Let $(Z_t)_{t\geq 0}$ be an arbitrary SBBM adapted to a filtration $(\mathcal F_t)_{t\geq 0}$ with branching mechanisms $(\Phi,\Psi)$ whose initial mass $Z_0(\mathbb R)$ is finite, and let $\tau$ be an arbitrary $[t_0, T]$-valued $(\mathcal F_t)_{t\geq 0}$-optional time.
		By reasoning similar to \eqref{eq:ES}, we have
		\begin{equation}
			\sup_{t\in [t_0, T]} \sup_{i\in \mathbb Z} \mathbb E[Z_t((i-1,i+1))] \leq K_0:=\frac{e^{\lambda_{\mathrm o} T}}{1-\gamma}\frac{4}{\Psi'(0+)t_0}  + 3 \Cr{c:UFtg}((-1, 1), \mathbb R, t, \gamma_0) < \infty
		\end{equation}
		where we fix an arbitrary $\gamma\in (0,1)$ and denote $\gamma_0 := \gamma \Psi'(0+)/(2\beta_{\mathrm c}).$
		From this, we see
		\begin{align}\label{eq:AI}
			& \sup_{t\in [t_0, T]} \mathbb E[Z_t(g)] \leq \sup_{t\in [t_0, T]}  \sum_{i\in \mathbb Z} \mathbb E[Z_t((i-1,i+1))] \sup_{x\in (i-1,i+1)} \frac{1}{1+x^2}
			\\&\leq K_0 \sum_{i\in \mathbb Z}\sup_{x\in (i-1,i+1)} \frac{1}{1+x^2} =: K< \infty.
		\end{align}
		Note that the constant $K$ is independent of the arbitrary choice of the SBBM $Z$ and the optional time $\tau$.
		
		From Proposition~\ref{prop:TM} and the fact that the initial state $Z_0$ is deterministic and finite, the process
		\[
		M_t(g) := e^{\Phi'(0+) t} Z_t(g) - \int_0^t \frac{1}{2} e^{\Phi'(0+) s} Z_s(g'') \, \mathrm{d} s, \quad t \ge 0
		\]
		is a supermartingale. By the Optional Stopping Theorem \cite{MR4226142}*{Theorem 9.30} applied to the bounded optional times $t_0$ and $\tau$ (recall $t_0 \le \tau \le T$ a.s.), we have
		\[
		\mathbb{E}\brk{ M_{\tau}(g) } \le \mathbb{E}\brk{ M_{t_0}(g) }.
		\]
		Substituting the definition of $M_\cdot(g)$ and rearranging terms:
		\[
		\mathbb{E}\brk{ e^{\Phi'(0+) \tau} Z_\tau(g) } \le \mathbb{E}\brk{ e^{\Phi'(0+) t_0} Z_{t_0}(g) } + \mathbb{E}\brk{ \int_{t_0}^\tau \frac{1}{2} e^{\Phi'(0+) s} Z_s(g'') \mathrm{d} s }.
		\]
		Since $\tau \in [t_0, T]$, we have $e^{\Phi'(0+) \tau} \ge e^{-|\Phi'(0+)|T}$ almost surely. For the integral term, we use $\abs{Z_s(g'')} \le 6 Z_s(g)$:
		\[
		\abs{\mathbb{E}\brk{ \int_{t_0}^\tau \frac{1}{2} e^{\Phi'(0+) s} Z_s(g'') \mathrm{d} s }}
		\le 3 \int_{t_0}^T e^{\Phi'(0+) s} \mathbb{E}[Z_s(g)] \mathrm{d} s
		\le 3 T e^{|\Phi'(0+)|T} K.
		\]
		Combining these estimates:
		\[
		e^{-|\Phi'(0+)|T} \mathbb{E}[Z_\tau(g)] \le e^{\Phi'(0+) t_0} K + 3 T e^{|\Phi'(0+)|T} K.
		\]
		Multiplying by $e^{|\Phi'(0+)|T}$ yields the first desired inequality for this lemma:
		\begin{equation} \label{eq:3TK}
			\mathbb{E}[Z_\tau(g)] \le e^{2|\Phi'(0+)|T}(1 + 3T) K.
		\end{equation}
		
		For the second inequality, define the first passage time
		\[ \tau_M := \inf \{ t \in [t_0, T] : Z_t(g) \ge M \} \wedge T. \]
		Note that $\tau_M$ is an optional time with values in $[t_0, T]$.
		The event $\{ \sup_{t \in [t_0, T]} Z_t(g) > M \}$ implies $\{ Z_{\tau_M}(g) \ge M \}$ almost surely (using the right-continuity of the process). By Markov's inequality:
		\begin{align}
			& \mathbb{P}\pr{ \sup_{t \in [t_0, T]} Z_t(g) > M } \le \mathbb{P}(Z_{\tau_M}(g) \ge M) \le \frac{\mathbb{E}[Z_{\tau_M}(g)]}{M}
			\\&\overset{\eqref{eq:3TK}}\leq \frac{e^{2|\Phi'(0+)| T} (1 + 3T) K}{M}.
		\end{align}
		This yields the second desired inequality for this lemma.
	\end{proof}
	
	\begin{lemma}\label{lem:stopping_time_bound_t0}
		Let $T > 0$ and define the weight function $g(x) := \frac{1}{1+x^2}$ for each $x\in \mathbb R$.
		Suppose that $\Lambda = \emptyset$.
		Then, uniformly for any $n\in \mathbb N$ and any $[0, T]$-valued $(\mathcal F^{(n)}_t)_{t\geq 0}$-optional time $\tau_n$, we have
		\begin{equation}\label{eq:Ztg0}
			\mathbb{E}\brk{ Z^{(n)}_{\tau_n}(g) }
			\lesssim 1,
		\end{equation}
		and, uniformly for any $M>0$,
		\begin{equation} \label{eq:PtM0}
			\mathbb P^{(n)}\pr{\sup_{t\in [0, T]} Z^{(n)}_t(g) > M} \lesssim \frac{1}{M}.
		\end{equation}
	\end{lemma}
	
	\begin{proof}
		Fix an arbitrary $n\in \mathbb N$.
		Since $(Z^{(n)}_t)_{t\geq 0}$ is an SBBM with deterministic and finite initial state $Z^{(n)}_0 = \mu_n$, by Proposition~\ref{prop:TM}, the process
		\[
		M_t(g) := e^{\Phi'(0+) t} Z^{(n)}_t(g) - \int_0^t \frac{1}{2} e^{\Phi'(0+) s} Z^{(n)}_s(g'') \, \mathrm{d} s, \quad t \ge 0
		\]
		is a supermartingale.
		Applying the Optional Stopping Theorem \cite{MR4226142}*{Theorem 9.30} to the bounded optional times $0$ and $\tau_n$ (recall $0 \le \tau_n \le T$ a.s.), we obtain
		\[
		\mathbb{E}\brk{ M_{\tau_n}(g) } \le \mathbb{E}\brk{ M_{0}(g) } = Z^{(n)}_0(g) = \mu_n(g).
		\]
		Substituting the definition of $M_\cdot(g)$ and using $\abs{g''(x)} \le 6g(x)$:
		\begin{equation}\label{eq:Ztg0-OST}
			e^{-|\Phi'(0+)|T} \, \mathbb{E}\brk{ Z^{(n)}_{\tau_n}(g) } \le \mu_n(g) + 3\int_0^T e^{|\Phi'(0+)|s} \mathbb{E}\brk{ Z^{(n)}_s(g) } \, \mathrm{d} s.
		\end{equation}
		Substituting $\tau_n$ and $T$ by any deterministic time $t \in [0,T]$ into the preceding inequality, we obtain for every $t \in [0,T]$:
		\[
		e^{-|\Phi'(0+)|t} \, \mathbb{E}\brk{ Z^{(n)}_t(g) } \le \mu_n(g) + 3\int_0^t e^{2|\Phi'(0+)|s} \, \big(e^{-|\Phi'(0+)|s}\mathbb{E}\brk{ Z^{(n)}_s(g) }\big) \, \mathrm{d} s.
		\]
		By the Gronwall inequality in the measure-theoretic form \cite{MR0838085}*{Lemma 3.1.4}, we conclude that for every $t \in [0, T]$,
		\[
		\mathbb{E}\brk{ Z^{(n)}_t(g) }
		\le 2\mu_n(g) \, e^{|\Phi'(0+)|t} \sum_{j=0}^{\lfloor 6\ell(t)\rfloor} (6\ell(t))^j
		\le 2\mu(g) \, e^{|\Phi'(0+)|T} \sum_{j=0}^{\lfloor 6\ell(T)\rfloor} (6\ell(T))^j,
		\]
		where $\ell(t) := \frac{1}{2|\Phi'(0+)|}(e^{2|\Phi'(0+)|t} - 1)$, and in the last inequality we used the monotonicity $\mu_n \preceq \mu$ (which holds since $(\mu_n)_{n\in\mathbb N}$ is monotonically increasing and converges m-weakly to $(\emptyset, \mu)$).
		\C{c:C1} \C{c:C2}
		It follows that
		\begin{equation}\label{eq:Zt0 Gronwall}
			\sup_{t\in [0,T]} \mathbb{E}\brk{ Z^{(n)}_t(g) } \le \mu(g) \, \Cr{c:C1}(T),
		\end{equation}
		where $\Cr{c:C1}(T) := 2 e^{|\Phi'(0+)|T} \sum_{j=0}^{\lfloor 6\ell(T)\rfloor} (6\ell(T))^j$.
		Substituting \eqref{eq:Zt0 Gronwall} into the OST inequality \eqref{eq:Ztg0-OST}, we obtain
		\begin{align}
			& e^{-|\Phi'(0+)|T} \, \mathbb{E}\brk{ Z^{(n)}_{\tau_n}(g) }
			\le \mu_n(g) + 3\int_0^T e^{|\Phi'(0+)|s} \, \mathbb{E}\brk{ Z^{(n)}_s(g) } \, \mathrm{d} s
			\\&\le \mu_n(g) + 3\mu(g)\, \Cr{c:C1}(T) \int_0^T e^{|\Phi'(0+)|s} \, \mathrm{d} s
			\le \mu(g)\, \Cr{c:C2}(T),
		\end{align}
		where $\Cr{c:C2}(T) := 1 + 3\Cr{c:C1}(T) \frac{e^{|\Phi'(0+)|T}-1}{|\Phi'(0+)|}$. This yields \eqref{eq:Ztg0}.
		
		For the second inequality, define the first passage time
		\[ \sigma^{(n)}_M := \inf \{ t \in [0, T] : Z^{(n)}_t(g) \ge M \} \wedge T, \quad n\in\mathbb{N}, M>0. \]
		Note that $\sigma^{(n)}_M$ is an optional time with values in $[0, T]$.
		Moreover, on the event \[ \left\{\sup_{t \in [0, T]} Z^{(n)}_t(g) > M\right\}, \] right-continuity implies $Z^{(n)}_{\sigma^{(n)}_M}(g) \ge M$ almost surely.
		By Markov's inequality and \eqref{eq:Ztg0} applied to the optional time $\sigma^{(n)}_M$: uniform in $n\in \mathbb N$ and $M>0$,
		\[
		\mathbb P\pr{ \sup_{t \in [0, T]} Z^{(n)}_t(g) > M } \le \frac{\mathbb{E}\brk{ Z^{(n)}_{\sigma^{(n)}_M}(g) }}{M} \lesssim \frac{1}{M}.
		\]
		This yields \eqref{eq:PtM0}.
	\end{proof}
	
	Recall from the beginning of Section \ref{sec:SS5} that
	\begin{equation}
		Z_t^{(n)} = \sum_{\alpha \in I_t^{(n)}} \delta_{X^{(n), \alpha}_t}, \quad t \geq 0, \quad n\in \mathbb N
	\end{equation}
	is a sequence of SBBMs with initial states $(x_i)_{i=1}^n$ sharing the same branching mechanisms $\Phi$ and $\Psi$.
	
	\begin{lemma}\label{lem:aldous_duality}
		Let $t_0 > 0$. For any non-negative function $h \in \mathcal{C}_{\mathrm{c}}^+(\mathbb{R})$, the sequence of real-valued processes
		\[ Y^{(n)}_t := \exp\pr{-Z^{(n)}_t(h)}, \quad t \ge 0, \quad n\in \mathbb N \]
		satisfies Aldous' criteria in $\mathbb{D}([t_0, \infty), \mathbb R)$.
		Moreover, if we assume in addition that $\Lambda = \emptyset$, then this sequence also satisfies Aldous' criteria in $\mathbb{D}([0, \infty), \mathbb{R})$.
	\end{lemma}
	
	\begin{proof}
		Let $T$, $(\tau_n)_{n \in \mathbb{N}}$ and $(\delta_n)_{n\in \mathbb N}$ be given as in \eqref{eq:T0}--\eqref{eq:delta0}.
		Thanks to Chebyshev's inequality, we only have to show that
		\begin{equation}
			\label{eq:Sq}
			\lim_{n \to \infty} \mathbb{E}\brk{ \abs{Y^{(n)}_{\tau_n + \delta_n} - Y^{(n)}_{\tau_n}}^2 } = 0. \end{equation}
		Let $(u_t)_{t \ge 0}$ be the solution to the dual SPDE \eqref{eq:GSPDE} with initial condition $u_0 = 1 - e^{-\theta h} \in \mathcal C_{\mathrm c}(\mathbb R, [0,1])$ where $\theta \in (0,1)$ is arbitrary.
		Then uniformly for any $n\in \mathbb N$,
		\begin{align}
			& \E\brk{\abs{ \int_{\mathcal N} e^{-\nu(\theta h)} \mathscr{Q}_{\delta_n}(Z^{(n)}_{\tau_n}, \mathrm d \nu) - e^{-Z^{(n)}_{\tau_n}(\theta h)} }}
			\\&\overset{\text{Lemma \ref{lem:transition_bound}}} \le \E\brk{ \int_{\mathbb{R}} (1+x^2)^{-1} Z^{(n)}_{\tau_n}(\mathrm d x) } \sup_{x \in \mathbb{R}} \pr{ (1+x^2) \tilde{\mathbb{E}}_{1-e^{-\theta h}}\brk{\abs{u_{\delta_n}(x) - u_0(x)}} }
			\\& \overset{\eqref{eq:Ztg}}\lesssim \sup_{x \in \mathbb{R}} \pr{ (1+x^2) \tilde{\mathbb{E}}_{1-e^{-\theta h}}\brk{\abs{u_{\delta_n}(x) - u_0(x)}} }
			\xrightarrow[n\to \infty]{\text{Lemma \ref{lem:convergence_2}}} 0. \label{eq:Ce}
		\end{align}
		Expanding the square in \eqref{eq:Sq}, we have for each $n\in \mathbb N$,
		\begin{equation}\label{eq:aldous_expansion}
			\E\brk{ \abs{Y^{(n)}_{\tau_n + \delta_n} - Y^{(n)}_{\tau_n}}^2 }
			= \E\brk{ e^{-2Z^{(n)}_{\tau_n+\delta_n}(h)} } - 2\E\brk{ e^{-Z^{(n)}_{\tau_n+\delta_n}(h)} e^{-Z^{(n)}_{\tau_n}(h)} } + \E\brk{ e^{-2Z^{(n)}_{\tau_n}(h)} }.
		\end{equation}
		We analyze the middle term using the strong Markov property: for each $n\in \mathbb N$,
		\begin{align}
			& \E\brk{ e^{-Z^{(n)}_{\tau_n+\delta_n}(h)} e^{-Z^{(n)}_{\tau_n}(h)} }
			= \E\brk{ e^{-Z^{(n)}_{\tau_n}(h)} \E\brk{ e^{-Z^{(n)}_{\tau_n+\delta_n}(h)} \middle| \mathcal{F}_{\tau_n} } }
			\\&= \E\brk{ e^{-Z^{(n)}_{\tau_n}(h)} \int_{\mathcal N} e^{-\nu(h)} \mathscr{Q}_{\delta_n}(Z^{(n)}_{\tau_n}, \mathrm d \nu) }.
		\end{align}
		Therefore
		\begin{align}
			& \abs{ \E\brk{ e^{-Z^{(n)}_{\tau_n+\delta_n}(h)} e^{-Z^{(n)}_{\tau_n}(h)} } - \E\brk{ e^{-2Z^{(n)}_{\tau_n}(h)} } }
			\\
			& = \abs{\E\brk{e^{-Z^{(n)}_{\tau_n}(h)} \pr{ \int_{\mathcal N} e^{-\nu(h)} \mathscr{Q}_{\delta_n}(Z^{(n)}_{\tau_n}, \mathrm d \nu) - e^{-Z^{(n)}_{\tau_n}(h)} }}}
			\\& \overset{\text{Jensen}}\leq  \E\brk{\abs{ \int_{\mathcal N} e^{-\nu(h)} \mathscr{Q}_{\delta_n}(Z^{(n)}_{\tau_n}, \mathrm d \nu) - e^{-Z^{(n)}_{\tau_n}(h)} }}
			\xrightarrow[n\to \infty]{\text{by \eqref{eq:Ce}}} 0. \label{eq:A1}
		\end{align}
		Similarly, by the strong Markov property again, for each $n\in \mathbb N$,
		\begin{align}
			\E\brk{ e^{-2Z^{(n)}_{\tau_n+\delta_n}(h)}  }
			= \E\brk{  \E\brk{ e^{-2Z^{(n)}_{\tau_n+\delta_n}(h)} \middle| \mathcal{F}_{\tau_n} } }
			= \E\brk{ \int_{\mathcal N} e^{-2\nu(h)} \mathcal{Q}_{\delta_n}(Z^{(n)}_{\tau_n}, \mathrm d \nu) }.
		\end{align}
		Therefore
		\begin{align}
			& \abs{ \E\brk{ e^{-2Z^{(n)}_{\tau_n+\delta_n}(h)}  } - \E\brk{ e^{-2Z^{(n)}_{\tau_n}(h)} } }
			\\
			& = \abs{\E\brk{ \int_{\mathcal N} e^{-2\nu(h)} \mathcal{Q}_{\delta_n}(Z^{(n)}_{\tau_n}, \mathrm d \nu) - e^{-2Z^{(n)}_{\tau_n}(h)} }}
			\\& \overset{\text{Jensen}}\leq  \E\brk{\abs{ \int_{\mathcal N} e^{-2\nu(h)} \mathcal{Q}_{\delta_n}(Z^{(n)}_{\tau_n}, \mathrm d \nu) - e^{-2Z^{(n)}_{\tau_n}(h)} }}
			\xrightarrow[n\to \infty]{\text{by \eqref{eq:Ce}}} 0. \label{eq:A2}
		\end{align}
		Now by taking $n\to \infty$ in \eqref{eq:aldous_expansion}, from \eqref{eq:A1} and \eqref{eq:A2}, we get the desired result \eqref{eq:Sq}.
		
		If we assume in addition that $\Lambda = \emptyset$, then the above argument remains valid with $t_0$ being replaced by $0$, and the reference to \eqref{eq:Ztg} being replaced by \eqref{eq:Ztg0}.
	\end{proof}
	
	\begin{lemma}\label{lem:aldous_duality2}
		Let $t_0 > 0$. For any non-negative function $h \in \mathcal{C}_{\mathrm{c}}^+(\mathbb{R})$, the sequence of real-valued processes $((Z^{(n)}_{t}(h))_{t \ge 0})_{n \in \mathbb{N}}$ satisfies Aldous' criteria in $\mathbb{D}([t_0, \infty), \mathbb{R})$.
		Moreover, if we assume in addition that $\Lambda = \emptyset$, then this sequence also satisfies Aldous' criteria in $\mathbb{D}([0, \infty), \mathbb{R})$.
	\end{lemma}
	
	\begin{proof}
		Let $T$, $(\tau_n)_{n \in \mathbb{N}}$ and $(\delta_n)_{n\in \mathbb N}$ be given as in \eqref{eq:T0}--\eqref{eq:delta0}.
		WLOG, we assume that the sequence $(\delta_n)_{n\in \mathbb N}$ is bounded by $1$.
		We only have to show that for any $\varepsilon > 0$,
		\[ \lim_{n \to \infty} \mathbb{P}\pr{ \abs{Z^{(n)}_{\tau_n + \delta_n}(h) - Z^{(n)}_{\tau_n}(h)} > \varepsilon } = 0. \]
		
		Let $g(x) := (1+x^2)^{-1}$ for every $x\in \mathbb R$.
		Since $h \in \mathcal{C}_{\mathrm{c}}^+(\mathbb{R})$, we have $h(x)\leq \|h/g\|_\infty g(x)$ where $\|h/g\|_\infty < \infty$.
		Thus, $Z^{(n)}_t(h) \le \|h/g\|_\infty Z^{(n)}_t(g)$ almost surely for every $n\in \mathbb N$ and $t\geq 0$.
		Now, uniformly for every $n\in \mathbb N$ and $M > 0$,
		\begin{equation}\label{eq:stoch_bound}
			\mathbb{P}\pr{ \sup_{t \in [t_0, T+1]} Z^{(n)}_t(h) > \|h/g\|_\infty M }
			\leq	\mathbb{P}\pr{ \sup_{t \in [t_0, T+1]} Z^{(n)}_t(g) > M }
			\overset{\eqref{eq:PtM}}\lesssim \frac{1}{M}.
		\end{equation}
		
		Let $Y^{(n)}_t := \exp(-Z^{(n)}_t(h))$ for every $t\geq 0$ and $n\in \mathbb N$.
		Note that, for any $M>0$, on the interval $[e^{-\|h/g\|_\infty M}, 1]$, the derivative $\partial_y (\log y) = 1/y$ is bounded in absolute value by $e^{\|h/g\|_\infty M}$.
		Therefore,
		\begin{itemize}
			\item[\eq\label{eq:MV}] for any $n\in \mathbb N$ and $M>0$, almost surely on the event
			\begin{equation}
				\brc{\sup_{t \in [t_0, T+1]} Z^{(n)}_t(h) \leq \|h/g\|_\infty M}
			\end{equation}
			we have by the Mean Value Theorem that
			\begin{align}
				\abs{Z^{(n)}_{\tau_n+\delta_n}(h) - Z^{(n)}_{\tau_n}(h)}
				= \abs{\log (Y^{(n)}_{\tau_n+\delta_n}) - \log(Y^{(n)}_{\tau_n})}
				\le e^{\|h/g\|_\infty M} \abs{Y^{(n)}_{\tau_n+\delta_n} - Y^{(n)}_{\tau_n}}.
			\end{align}
		\end{itemize}
		
		We now estimate the following probability: uniformly for any $n\in \mathbb N$, $M>0$ and $\varepsilon>0$,
		\begin{align*}
			& \mathbb{P}\pr{ \abs{Z^{(n)}_{\tau_n + \delta_n}(h) - Z^{(n)}_{\tau_n}(h)} > \varepsilon }                                                                                                                    \\
			& \le \mathbb{P}\pr{ \abs{Z^{(n)}_{\tau_n + \delta_n}(h) - Z^{(n)}_{\tau_n}(h)} > \varepsilon,  \sup_{t \in [t_0, T+1]} Z^{(n)}_t(h) \leq \|h/g\|_\infty M}
			\\&\quad + \mathbb{P}\pr{\sup_{t \in [t_0, T+1]} Z^{(n)}_t(h) > \|h/g\|_\infty M} \\
			& \overset{\eqref{eq:MV}}\le \mathbb{P}\pr{ e^{\|h/g\|_\infty M} \abs{Y^{(n)}_{\tau_n + \delta_n} - Y^{(n)}_{\tau_n}} > \varepsilon } + \mathbb{P}\pr{\sup_{t \in [t_0, T+1]} Z^{(n)}_t(h) > \|h/g\|_\infty M} \\
			& \overset{\eqref{eq:stoch_bound}}\lesssim \mathbb{P}\pr{  \abs{Y^{(n)}_{\tau_n + \delta_n} - Y^{(n)}_{\tau_n}} > e^{-\|h/g\|_\infty M}\varepsilon }  + \frac{1}{M}.
		\end{align*}
		Taking $n \to \infty$ and then $M\to \infty$, by Lemma \ref{lem:aldous_duality} we see that
		\[\lim_{n \to \infty}
		\mathbb{P}\pr{ \abs{Z^{(n)}_{\tau_n + \delta_n}(h) - Z^{(n)}_{\tau_n}(h)} > \varepsilon } =0\]
		as desired for this lemma.
		
		If we assume in addition that $\Lambda = \emptyset$, then the above argument remains valid with $t_0$ being replaced by $0$, where \eqref{eq:PtM} in \eqref{eq:stoch_bound} is replaced by \eqref{eq:PtM0}.
	\end{proof}
	
	\begin{lemma}\label{lem:FS}
		Let $t_0 > 0$. The sequence of $\mathcal N$-valued processes $((Z^{(n)}_{t})_{t \geq 0})_{n \in \mathbb{N}}$ satisfies Aldous' criteria in $\mathbb{D}([t_0, \infty), \mathcal N)$.
		Moreover, if we assume in addition that $\Lambda = \emptyset$, then the aforementioned sequence also satisfies Aldous' criteria in $\mathbb{D}([0, \infty), \mathcal N)$.
	\end{lemma}
	
	\begin{proof}
		Let $h \in \mathcal{C}_{\mathrm{c}}^\infty(\mathbb{R})$ be arbitrary.
		Thanks to Lemma \ref{lem:convergence_aldous}, we only have to show that the sequence of real-valued processes $((Z^{(n)}_{t}(h))_{t \geq 0})_{n \in \mathbb{N}}$ satisfies Aldous' criteria in $\mathbb{D}([t_0, \infty), \mathbb R)$.
		We can decompose $h$ into the difference of two non-negative smooth functions with compact support: Construct a function $\psi \in \mathcal{C}_{\mathrm{c}}^\infty(\mathbb{R})$ dominating $|h|$ on its support.
		Define
		\[
		h_1 := \frac{\psi + h}{2}, \quad h_2 := \frac{\psi - h}{2}.
		\]
		Then $h_1, h_2 \in \mathcal{C}_{\mathrm{c}}^\infty(\mathbb{R})$ and $h_1, h_2 \ge 0$. Moreover, $h = h_1 - h_2$.
		
		By linearity, $Z^{(n)}_t(h) = Z^{(n)}_t(h_1) - Z^{(n)}_t(h_2)$.
		Let $T $, $(\tau_n)_{n \in \mathbb{N}}$ and $(\delta_n)_{n\in \mathbb N}$ be given as in \eqref{eq:T0}--\eqref{eq:delta0}.
		By the triangle inequality, for every $n\in \mathbb N$, almost surely,
		\[
		\abs{ Z^{(n)}_{\tau_n+\delta_n}(h) - Z^{(n)}_{\tau_n}(h) }
		\le \abs{ Z^{(n)}_{\tau_n+\delta_n}(h_1) - Z^{(n)}_{\tau_n}(h_1) } + \abs{ Z^{(n)}_{\tau_n+\delta_n}(h_2) - Z^{(n)}_{\tau_n}(h_2) }.
		\]
		Specifically, for any $\varepsilon > 0$ and $n\in \mathbb N$,
		\begin{align*}
			& \mathbb{P}\pr{ \abs{ Z^{(n)}_{\tau_n+\delta_n}(h) - Z^{(n)}_{\tau_n}(h) } > \varepsilon }                   \\
			& \le \mathbb{P}\pr{ \abs{ Z^{(n)}_{\tau_n+\delta_n}(h_1) - Z^{(n)}_{\tau_n}(h_1) } > \frac{\varepsilon}{2} }
			+ \mathbb{P}\pr{ \abs{ Z^{(n)}_{\tau_n+\delta_n}(h_2) - Z^{(n)}_{\tau_n}(h_2) } > \frac{\varepsilon}{2} }.
		\end{align*}
		As $n \to \infty$, by Lemma \ref{lem:aldous_duality2} both probabilities on the right hand side converge to $0$. Thus, the sequence $((Z^{(n)}_t(h))_{t \ge 0})_{n\in \mathbb N}$ satisfies Aldous' criteria in $\mathbb D([t_0, \infty), \mathbb R)$.
		If we assume in addition that $\Lambda = \emptyset$, the extension of Lemma \ref{lem:aldous_duality2} guarantees that the above argument holds also for $t_0 = 0$.
		This yields the desired result.
	\end{proof}
	
	\begin{proof}[Proof of Theorem \ref{thm:existence-of-Z-t}]
		Recall from the beginning of this section that our objective reduces to verifying the statements \eqref{eq:MM1}, \eqref{eq:MM2}, and \eqref{eq:MM3}.
		
		\textit{Verification of \eqref{eq:MM1}.} Let $(\tilde Z_t)_{t> 0}$ be a $\mathcal N$-valued Markov process given as in Proposition \ref{Part-1-proof-of-thrm-1.2}.
		From Proposition \ref{prop:cadlag}, $(\tilde Z_t)_{t> 0}$ admits a c\`adl\`ag modification which will be denoted by $(Z_t)_{t>0}$.
		Note that the $\mathcal N$-valued c\`adl\`ag processes $(Z^{(n)}_t)_{t > 0}$ converge in finite dimensional distributions to $(Z_t)_{t > 0}$ as $n\to\infty$.
		Fix an arbitrary $t_0>0$.
		We only have to argue that the $\mathcal N$-valued c\`adl\`ag processes $(Z^{(n)}_t)_{t\geq t_0}$ converge in distributions to $(Z_t)_{t\geq t_0}$ in the path space $\mathbb D([t_0, \infty), \mathcal N)$ as $n\to \infty$.
		Thanks to \cite{MR4226142}*{Theorems 23.9 (1) \& 23.11}, this reduces to verifying that the sequence of processes $((Z^{(n)}_t)_{t> 0})_{n\in \mathbb N}$ satisfies Aldous' criteria in $\mathbb D([t_0, \infty), \mathcal N)$, which is already established by Lemma \ref{lem:FS}. This verifies \eqref{eq:MM1}.
		
		\textit{Verification of \eqref{eq:MM2}.} It is clear that the limit process $(Z_t)_{t> 0}$ and $(\tilde Z_t)_{t> 0}$ share the same finite dimensional distributions.
		Therefore, both of $(Z_t)_{t> 0}$ and $(\tilde Z_t)_{t> 0}$ are Markov processes sharing the same entrance laws $(\mathscr P^{(\Lambda, \mu)}_t)_{t> 0}$ and transition kernels $(\mathscr Q_t)_{t\geq 0}$ given as in Proposition \ref{Part-1-proof-of-thrm-1.2}.
		As a consequence, the law of the process $(Z_t)_{t>0}$ is uniquely determined by the initial trace $(\Lambda, \mu)$ and the branching mechanisms $(\Psi, \Phi)$. This establishes \eqref{eq:MM2}.
		
		\textit{Verification of \eqref{eq:MM3}.} Assume in addition that $\Lambda = \emptyset$. Since the monotonically increasing sequence $(\mu_n)_{n\in \mathbb N}$ in $\mathcal N$ converges m-weakly to $(\emptyset, \mu)$, the restriction condition in the definition of m-weak convergence reduces to the vague convergence $\mu_n \to \mu$ in $\mathcal N$. In particular, $Z^{(n)}_0 = \mu_n$ converges in $\mathcal N$ to $\mu =: Z_0$. Combining the convergence in finite-dimensional distributions on $(0,\infty)$ (established in Proposition \ref{Part-1-proof-of-thrm-1.2}), the initial consistency $Z_0^{(n)} \to Z_0$ in $\mathcal N$, and the Aldous' criteria in $\mathbb D([0, \infty), \mathcal N)$ provided by Lemma \ref{lem:FS}, we conclude from \cite{MR4226142}*{Theorems 23.9 (1) \& 23.11} that $(Z^{(n)}_\cdot)_{n\in \mathbb N}$ converges in distribution to $(Z_t)_{t\geq 0}$ in $\mathbb D([0, \infty), \mathcal N)$. This establishes \eqref{eq:MM3}.
	\end{proof}

	\section{Coming down from infinity: Proof of \texorpdfstring{\textcolor{blue}{Theorems \ref{Comming-down-finite-first-moment} and \ref{thm:CDIrate}}}{Theorems \ref{Comming-down-finite-first-moment} and \ref{thm:CDIrate}}}\label{Main}

	Let the offspring parameters $\beta_{\mathrm o}$, $(p_k)_{k=0}^\infty$, $\beta_{\mathrm c}$ and $(q_k)_{k=0}^\infty$ be given as in \eqref{eq:OBR}--\eqref{eq:CB}.
	Assume that \eqref{asp:A3}, \eqref{asp:A1} and \eqref{asp:A2} hold.
	
	Let $(\Lambda, \mu) \in \mathcal T_\mathrm{a}$.
	Let $\Phi$ and $\Psi$ be given as in \eqref{Def-Phi} and \eqref{Def-Psi} respectively.
	Let $(Z_t)_{t>0}$  be an SBBM with initial trace $(\Lambda, \mu)$, ordinary branching mechanism $\Phi$, and catalytic branching mechanism $\Psi$.
	That is to say, $(Z_t)_{t>0}$ is the unique in law $\mathcal N$-valued c\`{a}dl\`{a}g Markov process given as in Theorem \ref{thm:existence-of-Z-t}.
	Note that the entrance law $(\mathscr P_t^{(\Lambda, \mu)})_{t>0}$ and the transition kernels $(\mathscr Q_t)_{t\geq 0}$ of the process $(Z_t)_{t>0}$ are given as in Proposition \ref{Part-1-proof-of-thrm-1.2}.
	For every $(\tilde \Lambda, \tilde \mu)\in \mathcal T$, let $(v_t^{(\tilde \Lambda, \tilde \mu)}(x))_{t>0,x\in \mathbb R} \in \mathcal C^{1,2}((0,\infty)\times \mathbb R)$ be the unique non-negative solution to the equation \eqref{PDE}.
	Let $(x_i)_{i=1}^\infty$ be a sequence in $\mathbb R$ such that $(\Lambda, \mu)$ is the m-weak limit of $(\sum_{i=1}^n \delta_{x_i})_{n\in\mathbb N}$.
	
	In this section, we will prove
	Theorems \ref{Comming-down-finite-first-moment} and \ref{thm:CDIrate}. We will first establish Theorem \ref{thm:CDIrate}, as its rate of convergence is instrumental for proving the Coming Down from Infinity (CDI) property in Theorem \ref{Comming-down-finite-first-moment} (iv).
	We assume without loss of generality that $(Z_t)_{t>0}$ is the canonical process of $\mathbb D((0,\infty), \mathcal N)$, the space of $\mathcal N$-valued c\`adl\`ag paths indexed by $(0,\infty)$.
	More precisely, for any $\omega \in \mathbb D((0,\infty), \mathcal N)$ and $t>0$, $Z_t(\omega)= w_t$.
	Note that  this setup is different from Section \ref{sec:PZ} where $(Z_t)_{t>0}$ was the c\`adl\`ag modification of the canonical process of the path space $\mathcal N^{(0,\infty)}$.
	Correspondingly, we redefine our probability space  $\Omega := \mathbb D((0,\infty), \mathcal N)$.
	Let $\mathcal F^Z$ and $(\mathcal F^Z_t)_{t>0}$ be the natural $\sigma$-field and the natural filtration generated by the process $(Z_t)_{t>0}$.
	For any closed subset $\tilde \Lambda$ of $\mathbb R$ and integer-valued locally finite measure $\tilde \mu$ on $\tilde \Lambda^{\mathrm c}$, denote by $\mathbb P_{(\tilde \Lambda, \tilde \mu)}$ the law of an SBBM with initial trace $(\tilde \Lambda, \tilde \mu)$ induced on $(\Omega,\mathcal F^Z)$.
	Let the $\sigma$-field $\mathcal F$ and the filtration $(\mathcal F_t)_{t>0}$ be the usual augmentation of $\mathcal F^Z$ and $(\mathcal F^Z_t)_{t>0}$ with respect to the probability $\mathbb P_{(\Lambda, \mu)}$.
	
	Intuitively speaking, $Z_t$ behaves similarly to its mean field counterpart $v_t$ when $t$ is small.
	The next lemma gives the integrability of $v_t$ on a given interval $U$ in terms of the intersection between $U$ and the initial trace.
	
	\begin{lemma}\label{lemma5}
		Let $U$ be an open interval of $\mathbb R$.
		\begin{itemize}
			\item[(i)] Let $t>0$. Then
			$ \int_U v_{t}^{(\Lambda, \mu)}(x)\mathrm{d}  x<\infty $
			if and only if $U\cap \textup{supp}(\Lambda, \mu)$ is bounded.
			\item[(ii)] Suppose that $U\cap \textup{supp}(\Lambda, \mu)$ is bounded.
			Then \[
			\bar{U} \cap \Lambda = \emptyset \implies \limsup_{t\downarrow 0} \int_U v_{t}^{(\Lambda, \mu)}(x)\mathrm{d}  x<\infty\]
			and
			\[
			\bar{U} \cap \Lambda \neq \emptyset \implies \lim_{t\downarrow 0} \int_U v_{t}^{(\Lambda, \mu)}(x)\mathrm{d}  x=\infty.
			\]
		\end{itemize}
	\end{lemma}
	\begin{proof}
		(i).
		The sufficiency of the boundedness of $U\cap \textup{supp}(\Lambda, \mu)$ follows from Lemma \ref{lemma:integral-of-v} (i).
		To show its necessity, let us assume otherwise that $U\cap \textup{supp}(\Lambda, \mu)$ is unbounded.
		In this case, we can find a sequence $(z_i)_{i\in \mathbb N}$ in $U\cap \textup{supp}(\Lambda, \mu)$ such that $\{(z_i-1,z_i+1): i\in \mathbb N\}$ is a family of disjoint subintervals of $U$.
		Then, by \eqref{Identity} and \cite{MR4698025}*{(2.4)}, we can verify
		\begin{align}
			& \int_U v_{t}^{(\Lambda, \mu)}(x)\mathrm{d}  x
			\geq \sum_{i=1}^\infty \int_{z_i-1}^{z_i+1} v_{t}^{(\Lambda, \mu)}(x)\mathrm{d}  x
			\\&\geq \sum_{i=1}^\infty \int_{z_i-1}^{z_i+1} v_{t}^{(\emptyset, \delta_{z_i})}(x)\mathrm{d}  x
			= \sum_{i=1}^\infty \int_{-1}^{1} v_{t}^{(\emptyset, \delta_{0})}(x)\mathrm{d}  x = \infty.
		\end{align}
		
		(ii). Observing \eqref{Identity}, this is done in \cite{MR4698025}*{Lemma 3.2}.
	\end{proof}

	The next two lemmas demonstrate how the ``density'' of $Z_t$ is comparable to the solution $v^{(\Lambda, \mu)}_t$ of the MFE.

	\begin{lemma}\label{lemma6}
		Let $U$ be an open interval such that $U\cap \textup{supp}(\Lambda, \mu)$ is bounded. Suppose that $\bar{U}\cap \Lambda \neq \emptyset$. Then
		\[
		\pr{\int_U v_{t}^{(\Lambda,\mu)}(x)\mathrm{d}  x}^{-1} \mathbb{E}_{(\Lambda,\mu)} \brk{Z_t(U)}\stackrel{t\downarrow 0}{\longrightarrow}1.
		\]
	\end{lemma}

	\begin{proof}
		Let $\gamma \in (0,1)$ be arbitrary and $\gamma_0=\gamma \Psi'(0+)/(2\beta_\mathrm c)$.
		Let $F$ be the smallest closed interval containing $\cup_{i=1}^\infty (x_i-1,x_i+1)$.
		It is clear that $U\cap F$ is bounded.
		By Lemmas \ref{lemma:integral-of-v} (ii), \ref{lem: Prob-w} and \ref{lemma5} (ii), we have
		\begin{align}\label{limit-zero}
			\lim_{t\downarrow 0} \frac{\mathcal{V}_t^{(\Lambda, \mu, F)}}{\int_U v_{t}^{(\Lambda,\mu)}(x)\mathrm{d}  x} = \lim_{t\downarrow 0} \frac{\Cr{c:UFtg}(U,F,t,\gamma_0)}{\int_U v_{t}^{(\Lambda,\mu)}(x)\mathrm{d}  x} = \lim_{t\downarrow 0} \frac{\Cr{c:UFtg}(U,F,t,\gamma)}{\int_U v_{t}^{(\Lambda,\mu)}(x)\mathrm{d}  x}=0.
		\end{align}
		Therefore, by Proposition \ref{prop1}, we see that
		\begin{align}\label{step_1}
			\limsup_{t\downarrow 0} \frac{\mathbb{E}_{(\Lambda,\mu)}\brk{Z_t(U)}}{\int_U v_{t}^{(\Lambda,\mu)}(x)\mathrm{d}  x}\leq \frac{1}{1-\gamma}\stackrel{\gamma\downarrow 0}{\longrightarrow}1.
		\end{align}
		On the other hand, combining Proposition \ref{prop:ID}, Lemmas \ref{lem:1},  \ref{lem:2} and \ref{lem: Prob-w},  we have for any $\gamma \in (0,1/2)$, $\varepsilon \in (0,\gamma/2)$ and $t>0$, taking $\kappa(\gamma)$ as in \eqref{Def-Of-kappa},
		\begin{align}\label{summary-upper-Z-t-U}
			& \qquad  \mathbb{E}_{(\Lambda,\mu)}\brk{  (1-\varepsilon)^{Z_t(U)} }
			= \tilde {\mathbb E}_{\varepsilon\mathbf 1_U}\brk{\prod_{i=1}^\infty (1-u_t(x_i))}
			\\&\leq \tilde{\mathbb E}_{\varepsilon \mathbf 1_U }\brk{\exp\brc{-\sum_{i=1}^\infty u_t(x_i)}} +  \tilde{\mathbb{P}}_{\varepsilon  \mathbf 1_U }\pr{\sup_{s\leq t, y\in F} u_s(y)> \gamma}.
			\\&\leq \exp\brc{-\varepsilon \kappa(\gamma) e^{-\beta_{\mathrm o}  t}  \int_U v_{t}^{(\Lambda,\mu)}(y) \mathrm{d}  y  } +2\tilde{\mathbb{P}}_{\varepsilon  \mathbf 1_U }\pr{\sup_{s\leq t, y\in F} u_s(y)> \gamma} + \varepsilon \beta_{\mathrm c} e^{\lambda_{\mathrm o} t} \mathcal{V}_t^{(\Lambda, \mu, F)}\\
			& \leq  \exp\brc{-\varepsilon \kappa(\gamma) e^{-\beta_{\mathrm o}  t}  \int_U v_{t}^{(\Lambda,\mu)}(y) \mathrm{d}  y  } +  2\varepsilon \Cr{c:UFtg}(U,F,t,\gamma) + \varepsilon \beta_{\mathrm c} e^{\lambda_{\mathrm o} t} \mathcal{V}_t^{(\Lambda, \mu, F)}.
		\end{align}
		Therefore, for any $\gamma \in (0,1/2)$ and $t>0$,
		\begin{align}
			& \E_{(\Lambda, \mu)} \brk{Z_t(U)}
			= \lim_{\varepsilon \downarrow 0} \frac{1}{\varepsilon}\pr{1- \E_{(\Lambda, \mu)} \brk{(1-\varepsilon)^{Z_t(U)}}}                                                                                                    \\
			& \geq \kappa(\gamma) e^{-\beta_{\mathrm o}  t}  \int_U v_{t}^{(\Lambda,\mu)}(y) \mathrm{d}  y -2\Cr{c:UFtg} \pr{U, F, t, \gamma } - \beta_{\mathrm c} e^{\lambda_{\mathrm o} t} \mathcal{V}_t^{(\Lambda, \mu, F)}.
		\end{align}
		Using \eqref{limit-zero}, we conclude that
		\begin{align}\label{step_2}
			\liminf_{t\downarrow 0} \frac{\mathbb{E}_{(\Lambda, \mu)}\brk{Z_t(U)}}{\int_U v_{t}^{(\Lambda,\mu)}(x)\mathrm{d}  x}\geq \kappa(\gamma) \stackrel{\gamma\downarrow 0}{\longrightarrow} 1.
		\end{align}
		Therefore, we arrive at the desired result by \eqref{step_1} and \eqref{step_2}.
	\end{proof}
	
	\begin{old}
		\begin{lemma}\label{lemma6+}
			Let $U$ be an open interval such that $U\cap \textup{supp}(\Lambda, \mu)$ is bounded. Suppose that $\bar{U}\cap \Lambda \neq \emptyset$. Then
			\[
			\pr{\int_U v_{t}^{(\Lambda,\mu)}(x)\mathrm{d}  x}^{-1}Z_t(U)\stackrel{t\downarrow 0}{\longrightarrow}1, \quad \mbox{in probability}.
			\]
		\end{lemma}
	\end{old}
	
	Now we are ready to prove Theorem \ref{thm:CDIrate}.

	\begin{proof}[Proof of Theorem \ref{thm:CDIrate}]
		For any $\vartheta>0$ and $t>0$, define
		\begin{align}\label{Def-VarEp}
			\varepsilon (U, \vartheta, t):= 1- \exp\brc{- \pr{\int_U v_{t}^{(\Lambda,\mu)}(x)\mathrm{d}  x}^{-1}  \vartheta }.
		\end{align}
		By Lemma \ref{lemma5} (ii), we have
		\[
		e^{-\beta_{\mathrm o}  t}  \varepsilon (U, \vartheta, t)\int_U v_{t}^{(\Lambda,\mu)}(x)\mathrm{d}  x \stackrel{t\downarrow 0}{\longrightarrow} \vartheta, \quad \vartheta > 0.
		\]
		Let $F$ be the smallest closed interval containing $\cup_{i\in \mathbb N}(x_i-1, x_i+1)$.
		Note from \eqref{limit-zero} that for any $\vartheta > 0$,
		\[
		\lim_{t\downarrow 0}  \varepsilon (U, \vartheta, t) \Cr{c:UFtg}(U,F,t,\gamma)
		= \lim_{t\downarrow 0}  \varepsilon (U, \vartheta, t) \pr{\int_U v_{t}^{(\Lambda,\mu)}(x)\mathrm{d} x}  \frac{\Cr{c:UFtg}(U,F,t,\gamma) }{\int_U v_{t}^{(\Lambda,\mu)}(x)\mathrm{d}  x}=0,
		\]
		and similarly,
		$\lim_{t\downarrow 0}  \varepsilon (U, \vartheta, t) \mathcal V_t^{(\Lambda, \mu, F)} = 0. $
		From \eqref{summary-upper-Z-t-U}, for every $\gamma \in (0,1/2)$, $\vartheta>0$, and $t>0$ small enough such that $\varepsilon(U, \vartheta, t) < \gamma/2$,
		\begin{align}\label{step_5}
			& \mathbb{E}_{(\Lambda, \mu)}\brk{\exp\brc{-\vartheta\pr{\int_U v_{t}^{(\Lambda,\mu)}(x)\mathrm{d}  x}^{-1}  Z_t(U)}}
			= \mathbb{E}_{(\Lambda, \mu)}\brk{ \pr{1-\varepsilon (U, \vartheta, t)}^{Z_t(U)} }
			\\& \leq  \exp\brc{-\varepsilon(U, \vartheta, t) \kappa(\gamma) e^{-\beta_{\mathrm o}  t}  \int_U v_{t}^{(\Lambda,\mu)}(y) \mathrm{d}  y  } +  2 \varepsilon(U, \vartheta, t) \Cr{c:UFtg}(U,F,t,\gamma)
			\\& \qquad + \varepsilon(U, \vartheta, t) \beta_{\mathrm c} e^{\lambda_{\mathrm o} t} \mathcal{V}_t^{(\Lambda, \mu, F)}
			\\&\stackrel{ t\downarrow 0}{\longrightarrow}e^{-\vartheta \kappa(\gamma)}\stackrel{\gamma\downarrow 0}{\longrightarrow} e^{-\vartheta}.
		\end{align}
		Therefore, for every $\vartheta > 0$,
		\begin{equation}
			\limsup_{t\downarrow 0} \mathbb{E}_{(\Lambda, \mu)}\brk{\exp\brc{-\vartheta\pr{\int_U v_{t}^{(\Lambda,\mu)}(x)\mathrm{d}  x}^{-1}  Z_t(U)}} \leq e^{-\vartheta}.
		\end{equation}
		On the other hand, by Jensen's inequality $\mathbb{E}\brk{e^{-|Y|}}\geq e^{-\mathbb{E}[|Y|]}$ and Lemma \ref{lemma6}, for every $\vartheta > 0$,
		\begin{align}
			& \mathbb{E}_{(\Lambda, \mu)}\brk{\exp\brc{-\vartheta\pr{\int_U v_{t}^{(\Lambda,\mu)}(x)\mathrm{d}  x}^{-1}  Z_t(U)}}                                                                \\
			& \geq \exp\brc{-\vartheta \pr{\int_U v_{t}^{(\Lambda,\mu)}(x)\mathrm{d}  x}^{-1} \mathbb{E}_{(\Lambda, \mu)} \brk{Z_t(U)} }\stackrel{t\downarrow 0}{\longrightarrow}e^{-\vartheta}.
		\end{align}
		Therefore, we have
		\begin{equation}
			\lim_{t\downarrow 0}\mathbb{E}_{(\Lambda, \mu)}\brk{\exp\brc{-\vartheta\pr{\int_U v_{t}^{(\Lambda,\mu)}(x)\mathrm{d}  x}^{-1}  Z_t(U)}} = e^{-\vartheta}, \quad \vartheta >0.
		\end{equation}
		We are done.
	\end{proof}
	
	The following technical lemma will be used in the proof of Theorem \ref{Comming-down-finite-first-moment}.
	\begin{lemma}\label{lemma1}
		Let $m\in \mathbb N$, $(t_i)_{i=1}^m$ be a list in $(0,\infty)$, and $(f_i)_{i=1}^m$ be a list of non-negative elements in $\mathcal C_\mathrm c(\mathbb R)$.
		Then
		\begin{equation}
			\nu_0\mapsto \int_{\mathcal N^{m}} \prod_{i=1}^m \pr{e^{-\nu_i(f_i)}\mathscr Q_{t_i}(\nu_{i-1}, \mathrm{d} \nu_{i})}
		\end{equation}
		is a continuous function on $\mathcal N$.
	\end{lemma}
	\begin{proof}
		Let $(\nu^{(N)})_{N\in \mathbb N}$ be an arbitrary sequence in $\mathcal N$ converging to some element $\nu\in \mathcal N$.
		From Corollaries  \ref{lemma:general-vague-convergence2} and \ref{cor:ID}, we know that
		\begin{equation}
			\lim_{N\to \infty} \int_{\mathcal N} e^{- \nu_1(f)} \mathscr Q_{t_1}(\nu^{(N)}, \mathrm{d} \nu_1)
			= \int_{\mathcal N} e^{-\nu_1(f)} \mathscr Q_{t_1}(\nu, \mathrm{d} \nu_1)
		\end{equation}
		for every non-negative $f\in \mathcal C_\mathrm c(\mathbb R)$ satisfying $\|1-e^{- f}\|_{\infty}\leq \Psi'(0+)/(4\beta_{\mathrm c})$.
		From this and \cite{MR3642325}*{Theorem 4.11 (iii)}, we know that the probability measures
		$\mathscr Q_{t_1}(\nu^{(N)}, \cdot)$ converge weakly to $\mathscr Q_{t_1}(\nu, \cdot)$ as $N\uparrow \infty$.
		Therefore, for any bounded continuous function $G$ on $\mathcal N$,
		\[
		\lim_{N\to \infty}\int_{\mathcal N} G(\nu_1) \mathscr Q_{t_1}(\nu^{(N)}, \mathrm{d} \nu_1) =  \int_{\mathcal N} G(\nu_1) \mathscr Q_{t_1}(\nu, \mathrm{d} \nu_1).
		\]
		Since the converging sequence  $(\nu^{(N)})_{N\in \mathbb N}$ in $\mathcal N$ is chosen arbitrarily, the above says that
		\begin{itemize}
			\item[\eq\label{eq:CBI}]
			the map $
			\nu_0 \mapsto \int_{\mathcal N} G(\nu_1) \mathscr Q_{t_1}(\nu_0, \mathrm{d} \nu_1)
			$
			from $\mathcal N$ to $\mathbb R$ is continuous for every bounded continuous function $G$ on $\mathcal N$.
		\end{itemize}
		In particular, the desired result of this lemma holds when $m = 1$.
		
		Let us now assume, for the sake of induction, that the desired result of this lemma holds when $m$ is replaced by $m-1$.
		Under this assumption,
		\begin{equation}
			\tilde G(\nu_1) := \int_{\mathcal N^{m-1}} \prod_{i=2}^m (e^{-\nu_i(f_i)}\mathscr Q_{t_i}(\nu_{i-1}, \mathrm{d} \nu_i)), \quad \nu_1\in \mathcal N
		\end{equation}
		is a bounded continuous function on $\mathcal N$.
		Taking $G(\nu_1):=e^{-\nu_1(f_1)}\tilde G(\nu_1)$ in the statement \eqref{eq:CBI}, we obtain the desired result for this lemma.
	\end{proof}

	We are now ready to prove Theorem \ref{Comming-down-finite-first-moment}.
	
	\begin{proof}[Proof of Theorem \ref{Comming-down-finite-first-moment} (i)]

		\firststep \label{SF2}
		Let $T$ be an arbitrary $(0,\infty]$-valued  optional time with respect to the filtration $(\mathcal F_t)_{t\geq 0}$.
		In this step, we show that for any $m\in \mathbb N$, list of distinct times $(s_i)_{i=1}^m$ in $(0,\infty)$, and list of non-negative functions $(f_i)_{i=1}^m$ in $\mathcal C_\mathrm c(\mathbb R)$,
		\begin{equation}
			\mathbb{E}_{(\Lambda,\mu)}\brk{\mathbf 1_{ \{T<\infty\}} \exp\brc{-\sum_{i=1}^m Z_{T+s_i}(f_i)}}
			= \mathbb{E}_{(\Lambda,\mu)} \brk{\mathbf 1_{\{T<\infty\}} \E_{(\emptyset, Z_{T})}\brk{\exp\brc{-\sum_{i=1}^mZ_{s_i}(f_i)}} }.
		\end{equation}
		Without loss of generality, we assume that $0=:s_0 < s_1 < \dots < s_m$.
		Fix an arbitrary $k\in \mathbb N$.
		Define optional time
		\[
		T^{(k)}: = \frac{\lfloor 2^k T \rfloor+1}{2^k} \mathbf 1_{\{T<\infty\}} + \infty \mathbf 1_{\{T = \infty\}}
		\]
		which takes values in the discrete space $2^{-k}\mathbb N\cup \{\infty\}$.
		Notice that $T^{(k)}\downarrow T$ as $k\uparrow \infty$ and that $\{T^{(k)}=d\}\in \mathcal{F}_d$ for every $d \in 2^{-k}\mathbb N$.
		Using the Markov property of the process $(Z_t)_{t>0}$ with respect to the augmented filtration $(\mathcal F_t)_{t>0}$, cf.~Proposition \ref{prop:WQ}, we have
		\begin{equation} \label{eq:Psum}
			\mathbb{E}_{(\Lambda,\mu)}\brk{\mathbf 1_{\{T^{(k)} = d\}} \exp\brc{-\sum_{i=1}^m Z_{d+s_i}(f_i)}}
			= \mathbb{E}_{(\Lambda,\mu)} \brk{\mathbf 1_{\{T^{(k)} = d\}}G(Z_d)}, \quad d\in 2^{-k}\mathbb N,
		\end{equation}
		where
		\begin{equation}
			G(\nu_0):= \int_{\mathcal N^m}\prod_{i=1}^m \pr{e^{-\nu_i(f_i)}\mathscr Q_{s_i - s_{i-1}}(\nu_{i-1},\mathrm{d} \nu_i)}
			= \mathbb E_{(\emptyset, \nu_0)} \brk{\exp\brc{-\sum_{i=1}^m Z_{s_i}(f_i)}}
			, \quad \nu_0 \in \mathcal N
		\end{equation}
		is a bounded continuous function on $\mathcal N$, thanks to Lemma \ref{lemma1}.
		Summing over $d\in 2^{-k}\mathbb N$ in \eqref{eq:Psum}, we get by Fubini's theorem that
		\begin{equation}
			\mathbb{E}_{(\Lambda,\mu)}\brk{\mathbf 1_{\{T<\infty\}} \exp\brc{-\sum_{i=1}^m Z_{T^{(k)}+s_i}(f_i)}}
			= \mathbb{E}_{(\Lambda,\mu)} \brk{\mathbf 1_{\{T<\infty\}} G(Z_{T^{(k)}}) }.
		\end{equation}
		Taking $k\uparrow \infty$, from the fact that the process $(Z_t)_{t>0}$ is right-continuous, we obtain
		\begin{equation}
			\mathbb{E}_{(\Lambda,\mu)}\brk{\mathbf 1_{\{T<\infty\}} \exp\brc{-\sum_{i=1}^m Z_{T+s_i}(f_i)}}
			= \mathbb{E}_{(\Lambda,\mu)} \brk{\mathbf 1_{\{T<\infty\}} G(Z_{T}) }
		\end{equation}
		as desired for this step.

		\nextstep \label{SF3}
		Let $T$ be an arbitrary $(0,\infty]$-valued  optional time with respect to the filtration $(\mathcal F_t)_{t\geq 0}$.
		In this step, we will show that for any $0<a<b<\infty$, $x>0$ and non-negative $f\in \mathcal C_{\mathrm c}(\mathbb R)$,
		\begin{align}
			& \mathbb{P}_{(\Lambda,\mu)}\pr{T<\infty,   \sup_{ s\in [a,b]}  Z_{T+s}(f)\leq x}
			= \mathbb{E}_{(\Lambda,\mu)}\brk{1_{\{T<\infty\}} \mathbb{P}_{(\emptyset, Z_{T})} \pr{ \sup_{s\in [a,b]} Z_{s}(f)\leq x}}.
		\end{align}
		To do this, let $(s_i)_{i=1}^\infty$ be a sequential arrangement of the elements in $[a,b]\cap \mathbb Q$.
		From Step \ref{SF2}, we see that for any $m\in \mathbb N$,
		\begin{equation}\label{eq6}
			\mathbb{P}_{(\Lambda,\mu)}\pr{T<\infty,   \sup_{ 1\leq i\leq m}  Z_{T+s_i}(f)\leq x}
			= \mathbb{E}_{(\Lambda,\mu)}\brk{\mathbf 1_{\{T<\infty\}} \mathbb{P}_{(\emptyset, Z_{T})} \pr{ \sup_{1\leq i\leq m} Z_{s_i}(f)\leq x}}.
		\end{equation}
		Notice that $\sup_{s\in [a,b]} w_s(f)= \sup_{s\in [a,b]\cap \mathbb{Q}} w_s(f)$ for every
		$w\in \mathbb D((0,\infty), \mathcal N)$.
		Therefore, $\mathbb P_{(\Lambda, \mu)}$-almost surely on the event $\{T<\infty\}$,
		\[
		\mathbf 1_{\brc{\sup_{s\in [a,b]}Z_{T+s}(f) \leq x }} = \lim_{m\to \infty} \mathbf 1_{\brc{\sup_{1\leq i \leq m}Z_{T+s_i}(f) \leq x }};
		\]
		and similarly, for any $\nu \in \mathcal N$, $\mathbb P_{(\emptyset, \nu)}$-almost surely,
		\[
		\mathbf 1_{\brc{\sup_{s\in [a,b]}Z_{s}(f) \leq x }} = \lim_{m \to \infty} \mathbf 1_{\brc{\sup_{1\leq i \leq m} Z_{s_i}(f)\leq x}}.
		\]
		Taking $m\uparrow \infty$ in \eqref{eq6}, we get the desired result.

		\nextstep \label{SF4}
		Let $T$ be an arbitrary $(0,\infty]$-valued  optional time with respect to the filtration $(\mathcal F_t)_{t\geq 0}$.
		In this step, we will show that for any $0<a<b<\infty$ and $x>0$,
		\begin{equation} \label{eq:SF4}
			\mathbb{P}_{(\Lambda,\mu)}\pr{T<\infty,   \sup_{ s\in [a,b]}  Z_{T+s}(U)\leq x}
			= \mathbb{E}_{(\Lambda,\mu)}\brk{1_{\{T<\infty\}} \mathbb{P}_{(\emptyset, Z_{T})} \pr{ \sup_{s\in [a,b]} Z_{s}(U)\leq x}}.
		\end{equation}
		Let $(g_N)_{N\in \mathbb N}$ be an increasing sequence of compactly supported non-negative continuous functions on $\mathbb R$ approximating the indicator function $\mathbf 1_U$.
		We claim that for every $\omega \in \Omega=\mathbb D((0,\infty), \mathcal N)$,
		\begin{itemize}
			\item[\eq\label{Claim}]
			$\sup_{s\in [a,b]} w_{s}(g_N)$  increasingly converges to $\sup_{s\in [a,b]} w_{s}(U)$ as $N\uparrow \infty$.
		\end{itemize}
		Indeed, it is obvious from the monotonicity that the large $N$ limit of $\sup_{s\in [a,b]} w_s(g_N)$ exists and is bounded by $\sup_{s\in [a,b]} w_s(U)$.
		On the other hand:
		\begin{itemize}
			\item If $\sup_{s\in [a,b]} w_s(U)<\infty$, then for any $\varepsilon>0$, there exists $s_0\in [a,b]$ such that $w_{s_0}(U)> \sup_{s\in [a,b]} w_s(U)-\varepsilon$, which implies that
			\[
			\sup_{s\in [a,b]} w_s(U)-\varepsilon
			<  w_{s_0}(U)
			= \lim_{N\to\infty} w_{s_0}(g_N)\leq \lim_{N\to\infty} \sup_{s\in [a,b]} w_s(g_N).
			\]
			Letting $\varepsilon \downarrow 0$, we have $\sup_{s\in [a,b]} w_s(U)\leq \lim_{N\to\infty} \sup_{s\in [a,b]} w_s(g_N).$
			\item If $\sup_{s\in [a,b]} w_s(U)=\infty$, then for any $K>0$, there exists $s_0\in [a,b]$ such that $w_{s_0}(U)> K$, which implies that
			\[
			K < w_{s_0}(U) = \lim_{N\to\infty} w_{s_0}(g_N)
			\leq \lim_{N\to\infty} \sup_{s\in [a,b]} w_s(g_N).
			\]
			Letting $K\uparrow \infty$, we have  $\sup_{s\in [a,b]} w_s(U)\leq \lim_{N\to\infty} \sup_{s\in [a,b]} w_s(g_N).$
		\end{itemize}
		Thus \eqref{Claim} is valid.
		From \eqref{Claim}, we see that $\mathbb P_{(\Lambda, \mu)}$-almost surely on the event $\{T<\infty\}$,
		\[
		\mathbf 1_{\brc{\sup_{s\in [a,b]}Z_{T+s}(U) \leq x }} = \lim_{N\to \infty} \mathbf 1_{\brc{\sup_{s\in [a,b]} Z_{T+s}(g_N)\leq x}}, \quad x>0;
		\]
		and similarly, for any $\nu \in \mathcal N$, $\mathbb P_{(\emptyset, \nu)}$-almost surely,
		\[
		\mathbf 1_{\brc{\sup_{s\in [a,b]}Z_{s}(U) \leq x }} = \lim_{N\to \infty} \mathbf 1_{\brc{\sup_{s\in [a,b]} Z_{s}(g_N)\leq x}}, \quad x>0.
		\]
		From Step \ref{SF3}, for every $N\in \mathbb N$ and $x > 0$,
		\begin{equation}
			\mathbb{P}_{(\Lambda,\mu)}\pr{T<\infty,   \sup_{ s\in [a,b]}  Z_{T+s}(g_N)\leq x}
			= \mathbb{E}_{(\Lambda,\mu)}\brk{1_{\{T<\infty\}} \mathbb{P}_{(\emptyset, Z_{T})} \pr{ \sup_{s\in [a,b]} Z_{s}(g_N)\leq x}}.
		\end{equation}
		Taking $N\uparrow\infty$ we get the desired result for this step.

		\nextstep \label{SF5}
		Fix an arbitrary $k\in \mathbb N$ and define $T_k:= \inf\brc{t > 1/k: Z_t(U) \leq k}$ which is an optional time with respect to $(\mathcal F_t)_{t>0}$.
		Let $0<a<b<\infty$.
		We verify in this step that $\mathbb P_{(\Lambda, \mu)}$-almost surely on the event $\{T_k<\infty\}$,
		\begin{equation}
			\mathbb{P}_{(\emptyset, Z_{T_{k}})} \pr{ \sup_{s\in [a,b]} Z_{s}(U)<\infty}=1.
		\end{equation}
		Note that $\mathbb P_{(\Lambda, \mu)}$-almost surely there exists a decreasing sequence $(t_m)_{m\in \mathbb N}$ in $\mathbb R_+$ converging to $0$ such that $Z_{T_k +t_m}(U)\leq k$.
		Let $(g_N)_{N\in \mathbb N}$ be an increasing sequence of compactly supported continuous functions on $\mathbb R$ approximating the indicator function $\mathbf 1_U$.
		Then, by the fact that $(Z_t)_{t>0}$ is an $\mathcal N$-valued c\`adl\`ag process, $\mathbb P_{(\Lambda, \mu)}$-almost surely on the event $\{T_k < \infty\}$,
		\begin{equation}\label{Finite-Stopping-time}
			Z_{T_k}(U)= \lim_{N\to\infty} Z_{T_k}(g_N)= \lim_{N\to\infty} \lim_{m\to\infty} Z_{T_k+t_m}(g_N) \leq k.
		\end{equation}
		
		Fix an arbitrary $\nu \in \mathcal N$ such that $\nu(U)<\infty$.
		Let \[U_1:= \{x\in \mathbb R: |x-x_0|\leq 1~\text{for some}~x_0\in U\}\] be the unit enlargement of the open interval $U$, 	and let  $g$ be a smooth function with bounded derivatives of all orders satisfying that $\mathbf 1_U\leq g\leq \mathbf 1_{U_{1}}$.
		Note that $g''$ is compactly supported.
		Note that $U_1\cap \tilde F$ is bounded if $\tilde F$ is the smallest closed interval containing each $(x-1,x+1)$ such that $\nu(\{x\})>0$.
		Therefore, cf.~ Lemma \ref{lemma: super-martingale},
		\begin{equation}
			M_g(t;a) := e^{\Phi'(0+) t} Z_t(g) - \int_a^t e^{\Phi'(0+) s}\frac{1}{2} Z_s(g'')\mathrm{d} s, \quad t\geq a,
		\end{equation}
		is a super-martingale on the filtered probability space $(\Omega, \mathcal F^Z, (\mathcal F^Z_t)_{t\geq a}, \mathbb P_{(\emptyset,\nu)})$.
		From this and \cite{MR2152573}*{Theorem 1 of Section 1.4.}, we can verify that $\mathbb P_{(\emptyset, \nu)}$-a.s.,
		\[
		\sup_{q\in [a,b+1]\cap \mathbb Q}Z_q(U)\leq \sup_{q\in [a,b+1]\cap \mathbb Q}Z_q(g)
		<\infty.
		\]
		Therefore, $\mathbb P_{(\emptyset, \nu)}$-a.s.,
		\begin{align}
			& \sup_{s\in [a,b]}Z_s(U)
			=  \sup_{s\in [a,b]}\lim_{N\to \infty}   Z_s(g_N)
			= \sup_{s\in [a,b]} \lim_{N\to \infty}  \lim_{q\downarrow s,q\in \mathbb Q} Z_q(g_N)
			\\&\leq \sup_{s\in [a,b]} \lim_{N\to \infty}  \limsup_{q\downarrow s,q\in \mathbb Q} Z_q(U)
			\leq \sup_{q\in [a,b+1]\cap \mathbb Q}Z_q(U)<\infty.
		\end{align}
		
		Now, since we have shown
		\begin{align}
			\mathbb{P}_{(\emptyset, \nu)} \pr{ \sup_{s\in [a,b]} Z_{s}(U)<\infty}=1
		\end{align}
		for the arbitrary $\nu\in \mathcal N$ satisfying $\nu(U)<\infty$, the desired result for this step follows from \eqref{Finite-Stopping-time}.

		\nextstep \label{SF6}
		In this step, we show that
		\[
		\mathbb P_{(\Lambda, \mu)}(Z_t(U) = \infty, \forall t\in \mathbb Q\cap (0,\infty))
		= 1.
		\]
		To do this, for any measurable function $f$ on $\mathbb R$ which can be approximated by the elements of $\mathcal C(\mathbb R, [0,z^*])$ monotonically from below, let $(u_t)_{t>0}$ be a $\mathcal C(\mathbb R, [0,z^*])$-valued continuous process given as in Proposition \ref{prop:GI} with  initial value $f$ on a probability space whose expectation operator will be denoted by $\tilde {\mathbb E}_f$.
		Let $F$ be the smallest closed interval containing $\cup_{i=1}^\infty (x_i-1,x_i+1)$.
		From the condition $U \cap \operatorname{supp}(\Lambda, \mu)$ is unbounded, we have $U\cap F$ is unbounded.
		From Proposition \ref{prop:ID}, Lemmas \ref{lem:1}, \ref{lem:2} and \ref{lemma5} (i), for any $\varepsilon \in (0,1/2)$ and $t>0$,
		\begin{align}
			& \mathbb{E}_{(\Lambda,\mu)}\brk{ (1-\varepsilon)^{Z_t(U)}}
			= \tilde {\mathbb E}_{\varepsilon \mathbf 1_U}\brk{\prod_{i=1}^\infty (1-u_t(x_i))}
			\\& \leq \tilde{\mathbb E}_{\varepsilon \mathbf 1_U }\brk{\exp\brc{-\sum_{i=1}^\infty u_t(x_i)}} +  \tilde{\mathbb{P}}_{\varepsilon  \mathbf 1_U }\pr{\sup_{s\leq t, y\in F} u_s(y)> \frac{1}{2}}
			\\& \leq \exp\brc{-\varepsilon \kappa(1/2) e^{-\beta_{\mathrm o}  t}  \int_U v_{t}^{(\Lambda,\mu)}(y) \mathrm{d}  y  } + 2\tilde{\mathbb{P}}_{\varepsilon  \mathbf 1_U}\pr{\sup_{s\leq t, y\in F} u_s(y)> \frac{1}{2} } + \varepsilon \beta_{\mathrm c}
			e^{\lambda_{\mathrm o} t}
			\mathcal{V}_t^{(\Lambda, \mu, F)}
			\\&= 2\tilde{\mathbb{P}}_{\varepsilon  \mathbf 1_U }\pr{\sup_{s\leq t, y\in F} u_s(y) > \frac{1}{2} } + \varepsilon \beta_{\mathrm c}e^{\lambda_{\mathrm o} t} \mathcal{V}_t^{(\Lambda, \mu, F)}.
		\end{align}
		Here, $\mathcal V_t^{(\Lambda, \mu, F)}$, $\lambda_{\mathrm o}$ and $\kappa(\cdot)$ are given as in \eqref{eq:CV},  \eqref{Def-of-lambda}
		and \eqref{Def-Of-kappa} respectively.
		From Lemma \ref{lemma:integral-of-v} (ii) and Lemma \ref{lem: Prob-w},  taking $\varepsilon\downarrow 0$, we obtain that for any $t>0$, $\mathbb{P}_{(\Lambda,\mu)}\pr{ Z_t(U) =\infty}=1$.
		The desired result for this step follows.

		\nextstep
		Let $k\in \mathbb N$ be arbitrary, and let $T_k$ be the optional time given as in Step \ref{SF5}.
		Let $0<a < b < \infty$ and $x>0$ be arbitrary.
		From Step \ref{SF4}, we know that \eqref{eq:SF4} holds with $T$ being replaced by $T_k$, which, by taking $x\uparrow \infty$, implies that
		\begin{equation} \label{eq:SF7}
			\mathbb{P}_{(\Lambda,\mu)}\pr{T_k<\infty,   \sup_{ s\in [a,b]}  Z_{T_k+s}(U) < \infty}
			= \mathbb{E}_{(\Lambda,\mu)}\brk{\mathbf 1_{\{T_k<\infty\}} \mathbb{P}_{(\emptyset, Z_{T_k})} \pr{ \sup_{s\in [a,b]} Z_{s}(U) < \infty}}.
		\end{equation}
		From Step \ref{SF6}, we know that the left hand side of \eqref{eq:SF7} equals $0$.
		From Step \ref{SF5}, we know that the right hand side of \eqref{eq:SF7} equals $\mathbb P_{(\Lambda, \mu)}(T_k < \infty)$.
		Therefore, $\mathbb P_{(\Lambda, \mu)}(T_k < \infty) = 0$.
		As a consequence,
		\begin{equation}
			1 = \mathbb P_{(\Lambda,\mu)}(T_k = \infty) = \mathbb P_{(\Lambda,\mu)}(Z_t(U)>k, \forall t > 1/k).
		\end{equation}
		Taking the arbitrary $k\uparrow \infty$, by the monotone convergence theorem, we get the desired result (i) of Theorem \ref{Comming-down-finite-first-moment}.
	\end{proof}
	
	\begin{proof}[Proof of Theorem \ref{Comming-down-finite-first-moment} (ii)]
		Let the open interval $U_1:= \{x\in \mathbb R: \exists x_0\in U, |x-x_0|\leq 1\}$ be the unit enlargement of $U$, and let $g$ be a smooth function with bounded derivatives of all orders 	satisfying that $\mathbf 1_U \leq g \leq \mathbf 1_{U_1}$.
		Note that $g''$ is compactly supported.
		Note that, from the condition that $U\cap \operatorname{supp}(\Lambda, \mu)$ is bounded, we have $U_1 \cap F$ is bounded where $F$ is the smallest closed interval containing $\cup_{i\in \mathbb N}(x_i-1,x_i+1)$.
		Therefore, cf.~ Lemma \ref{lemma: super-martingale}, for any $a>0$,
		\begin{equation}\label{Def-of-M-g}
			M_g(t;a) := e^{\Phi'(0+)t}Z_t(g) - \int_a^t e^{\Phi'(0+)s} \frac{1}{2} Z_s(g'') \mathrm{d} s, \quad t\geq a
		\end{equation}
		is a super-martingale on the filtered probability space $(\Omega, \mathcal F, (\mathcal F_t)_{t\geq a}, \mathbb P_{(\Lambda, \mu)})$.
		From this and \cite{MR2152573}*{Theorem 1 of Section 1.4.}, we can verify that
		
		$\mathbb P_{(\Lambda, \mu)}$-a.s.,
		for any $b > a >0$,
		\begin{align}\label{sup-finite}
			\sup_{q\in [a,b]\cap \mathbb Q}Z_q(U)\leq \sup_{q\in [a,b]\cap \mathbb Q}Z_q(g)
			<\infty.
		\end{align}
		Let $(g_N)_{N\in \mathbb N}$ be an increasing sequence of compactly supported continuous functions on $\mathbb R$ approximating the indicator function $\mathbf 1_U$.
		Then,
		
		$\mathbb P_{(\Lambda, \mu)}$-almost surely,
		for every $t>0$,
		\begin{equation}\label{sup-finite2}
			Z_t(U)=\lim_{N\uparrow \infty}Z_t(g_N)
			= \lim_{N\uparrow \infty} \lim_{q\downarrow t, q\in \mathbb Q} Z_q(g_N)
			\leq \sup_{ q\in \mathbb Q\cap[t,t+1]} Z_q(U)
			< \infty
		\end{equation}
		as desired.
	\end{proof}
	
	\begin{proof}[Proof of Theorem \ref{Comming-down-finite-first-moment} (iii)--(iv)]
		\begin{old}
			(iii) follows from Proposition \ref{prop1}.
		\end{old}
		(iii) To apply Proposition \ref{prop1}, we recall the sequence $(x_i)_{i=1}^\infty$ associated with the initial trace and select $F$ as the smallest closed interval containing the union $\cup_{i=1}^\infty(x_i-1, x_i+1)$. Since the condition of the theorem assumes $U\cap \operatorname{supp}(\Lambda, \mu)$ is bounded, the interval $U$ intersects only a bounded accumulation of the points $x_i$. Therefore, the intersection $U\cap F$ is rigorously bounded, which satisfies the prerequisite of Proposition \ref{prop1}.
		
		From Proposition \ref{prop1}, for any $t > 0$ and $\gamma\in(0,1)$, taking $\gamma_0 := \gamma \Psi'(0+)/(2\beta_\mathrm c)$, we obtain the explicit upper bound:
		\[
		\mathbb{E}_{(\Lambda, \mu)}\brk{Z_t(U)} \leq \frac{e^{(\lambda_{\mathrm o} - \beta_{\mathrm o}) t}}{1-\gamma}\pr{\int_U v_{t}^{(\Lambda, \mu)}(y)\mathrm{d} y + \frac{\Psi'(0+)e^{\lambda_{\mathrm o} t}}{2}\mathcal{V}_t^{(\Lambda, \mu, F)}} + 3 e^{-\beta_{\mathrm o}t}\Cr{c:UFtg}(U,F,t,  \gamma_0).
		\]
		By Lemmas \ref{lemma:integral-of-v} (ii) and \ref{lem: Prob-w}, we have
		\[
		\limsup_{t\downarrow 0}\mathcal{V}_t^{(\Lambda, \mu, F)}<\infty \quad \text{and} \quad \limsup_{t\downarrow 0}\Cr{c:UFtg}(U,F,t,\gamma_0)<\infty.
		\]
		By Lemma \ref{lemma5} (ii), the condition $\bar{U} \cap \Lambda = \emptyset$ guarantees that
		\[
		\limsup_{t\downarrow 0} \int_U v_{t}^{(\Lambda, \mu)}(y)\mathrm{d} y < \infty.
		\]
		Therefore, taking the limit superior in the upper bound yields
		\[
		\limsup_{t\downarrow 0} \mathbb{E}_{(\Lambda, \mu)}\brk{Z_t(U)} < \infty.
		\]
		By Markov's inequality, for any $K>0$,
		\[
		\limsup_{t\downarrow 0} \mathbb{P}_{(\Lambda, \mu)}\pr{Z_t(U) > K} \leq \frac{1}{K} \limsup_{t\downarrow 0} \mathbb{E}_{(\Lambda, \mu)}\brk{Z_t(U)}.
		\]
		By taking $K$ large enough, the right-hand side becomes strictly less than $1$. This establishes that $\limsup_{t\downarrow 0} \mathbb{P}_{(\Lambda, \mu)}\pr{Z_t(U) > K} < 1$, mathematically preventing $Z_t(U)$ from converging to $\infty$ in probability.
		Thus, the CDI property does not hold.
		
		(iv) Since $U\cap \operatorname{supp}(\Lambda, \mu)$ is bounded, we know from (ii) that $Z_t(U) < \infty$ for all $t > 0$ almost surely. To establish the CDI property \eqref{eq:CDIP}, it suffices to prove that $Z_t(U)$ diverges to infinity in probability as $t\downarrow 0$. By the hypothesis $\bar{U} \cap \Lambda \neq \emptyset$, Lemma \ref{lemma5} (ii) guarantees that
		\[
		\lim_{t\downarrow 0} \int_U v_{t}^{(\Lambda, \mu)}(x)\mathrm{d} x = \infty.
		\]
		Concurrently, Theorem \ref{thm:CDIrate} states that the ratio of $Z_t(U)$ to this deterministic integral converges to $1$ in $L^1$ as $t\downarrow 0$. Since $L^1$ convergence implies convergence in probability, and the denominator diverges to infinity, it immediately follows that $Z_t(U)$ converges to infinity in probability as $t\downarrow 0$. 		Therefore, the CDI property holds.
	\end{proof}

	\begin{appendix}
		\section{Proof of Lemma \ref{lem:AE}}
		\label{sec:AA}
		\begin{proof}[Proof of Lemma \ref{lem:AE}]
			Let us define the transformation
			\begin{equation}
				w_t(x) := e^{-\Phi'(0+) t} \tilde v_t(x), \quad t>0, x\in \R.
			\end{equation}
			It is clear that $w$ is the unique non-negative solution to the reaction-diffusion equation with time-dependent reaction coefficient:
			\begin{equation} \label{eq:w_pde}
				\partial_t w_t(x) = \frac{1}{2} \Delta w_t(x) - \frac{\lambda(t) \Psi'(0+)}{2} w_t(x)^2,
			\end{equation}
			where $\lambda(t) := e^{\Phi'(0+) t}$. Note that $w$ is subject to the same
			initial trace
			condition as $\tilde v$ (and consequently $v$), because the factor $e^{-\Phi'(0+) t}$ tends to $1$ as $t \downarrow 0$.

			Fix $T>0$ and define the constants
			\begin{equation} \label{eq:lambda_bounds}
				\underline{\lambda}_T := \inf_{s \in (0, T]} e^{\Phi'(0+) s} \quad \text{and} \quad \overline{\lambda}_T := \sup_{s \in (0, T]} e^{\Phi'(0+) s}.
			\end{equation}
			Observe that on the interval $t\in (0, T]$, the time-dependent coefficient satisfies $\underline{\lambda}_T \Psi'(0+) \leq \lambda(t) \Psi'(0+) \leq \overline{\lambda}_T \Psi'(0+)$.
			
			For any parameter $\lambda > 0$, we define the scaled function
			\begin{equation} \label{eq:scaling}
				V^{(\lambda)}_t(x) := \frac{1}{\lambda} v_t(x), \quad t>0, x\in \mathbb R,
			\end{equation}
			where $v$ is the solution to the original equation \eqref{PDE}. A straightforward computation shows that $V^{(\lambda)}$ is the unique non-negative solution to the equation
			\begin{equation} \label{eq:V_pde}
				\partial_t V^{(\lambda)} = \frac{1}{2} \Delta V^{(\lambda)} - \frac{\lambda \Psi'(0+)}{2} (V^{(\lambda)})^2,
			\end{equation}
			with initial trace condition given by the scaled measure $(\tilde\Lambda, \tilde\mu/\lambda)$.
			
			Crucially, since $\underline{\lambda}_T \le 1 \le \overline{\lambda}_T$, the initial traces for the scaled functions are naturally ordered with respect to the initial trace of $w_t$, meaning $(\tilde\Lambda, \tilde\mu/\overline{\lambda}_T) \le (\tilde\Lambda, \tilde\mu) \le (\tilde\Lambda, \tilde\mu/\underline{\lambda}_T)$. Because both the initial data and the reaction coefficients follow the necessary ordering required for the parabolic comparison principle, we obtain:
			\begin{equation} \label{eq:comparison_ineq}
				V^{(\overline{\lambda}_T)}_t(x) \leq w_t(x) \leq V^{(\underline{\lambda}_T)}_t(x), \quad t\in (0,T], x\in \mathbb R.
			\end{equation}
			The comparison principle here for the non-negative solutions to the parabolic equations \eqref{eq:w_pde} and \eqref{eq:V_pde} is justified, in the presence of the very singular initial data $(\Lambda, \mu)$, because such solutions can be constructed as monotone limits of solutions with locally bounded measure initial data (c.f.~ \cite{MR1697494}*{Proposition 4.7}).
			Since the classical parabolic comparison principle holds for those approximating solutions, it extends to their limits.
			
			Using the scaling relation \eqref{eq:scaling} and the definition of $w$, the inequality \eqref{eq:comparison_ineq} becomes:
			\begin{equation}
				\frac{1}{\overline{\lambda}_T} v_t(x) \leq e^{-\Phi'(0+) t} \tilde v_t(x) \leq \frac{1}{\underline{\lambda}_T} v_t(x), \quad t\in (0,T], x \in \mathbb R.
			\end{equation}
			This implies the desired result for this lemma.
		\end{proof}

		\section{Proofs of Lemmas \ref{lem:CC}, \ref{eq:WC}, and Proposition \ref{prop:GI}} \label{append-A}
		\begin{proof}[Proof of Lemma \ref{lem:CC}]
			From the continuity of the function $\Psi(\cdot)$, $\Psi(1) = \beta_{\mathrm c} q_0 \geq 0$ and that
			\begin{equation}
				\Psi(2)= \beta_{\mathrm c}\pr{\sum_{k=0}^\infty q_k (-1)^k - 1} \leq \beta_{\mathrm c}\pr{\sum_{k=0}^\infty q_k - 1} = 0,
			\end{equation}
			we have $z^* \in [1,2]$ and $\Psi(z^*) = 0$.
			Observe that for any $k\in \mathbb Z_+$, $z^k - z = z(z^{k-1} - 1) \geq 0$ for every $z\in [-1,0]$.
			Therefore, for every $z\in [1,2]$, we have
			\begin{equation}
				\Phi(z)= \beta_{\mathrm o} \pr{p_0 z+ \sum_{k=1}^\infty p_k \pr{(1-z)^k - (1-z)} } \geq 0.
			\end{equation}
			In particular, $\Phi(z^*)\geq 0$.
			The fact $\Phi(0)=\Psi(0) = 0$ can be verified directly from their expressions.
			From the definition of $z^*$, the continuity of the function $\Psi(\cdot)$ and the fact that $\Psi(1)\geq 0$, we have $\Psi(z)\geq 0$ for every $z\in [1,z^*]$.
			Finally observe that,
			for every $z\in [0,1)$,
			since $x\mapsto z^x$ is a convex function on $\mathbb R$, by Jensen's inequality and \eqref{asp:A1},
			\begin{equation}
				\sum_{k=0}^\infty q_k z^k \geq z^{\sum_{k=0}^\infty kq_k} \geq z^2.
			\end{equation}
			This proves that $\Psi(z)\geq 0$ for every $z\in (0,1]$.
			
			If the additional assumption \eqref{asp:A2} holds, then just take an odd number $k_0$ with $q_{k_0}>0$, we see that
			\begin{align}
				\Psi(2)= \beta_{\mathrm c}\pr{\sum_{k=0}^\infty q_k (-1)^k - 1} \leq \beta_{\mathrm c}\pr{\sum_{k\neq k_0}^\infty q_k - q_{k_0}- 1} =-2\beta_{\mathrm c} q_{k_0}<0,
			\end{align}
			which implies that $z^*<2$. We are done.
		\end{proof}
		
		\begin{proof}[Proof of Lemma \ref{eq:WC}]
			Let $(\sigma_m)_{m=1}^\infty$ be a sequence of non-negative Lipschitz functions on $[0,z^*]$ converging to $\sqrt{\Psi}$ uniformly, satisfying that $\sigma_m(0) = \sigma_m(z^*) = 0$ for each $m\in \mathbb N$.
			For each $i\in \{1,2\}$ and $m\in \mathbb N$, let $u^{(i,m)}$ be a $\mathcal C(\mathbb R, [0,z^*])$-valued solution to the SPDE \eqref{eq:GSPDE} with $f$ and $\sqrt{\Psi}$ being replaced by $f^{(i)}$ and $\sigma_m$ respectively.
			For each $i\in \{1,2\}$ and $m\in \mathbb N$, the existence of the  $\mathcal C([0,\infty), \mathcal C(\mathbb R, [0,z^*]))$-valued random element $u^{(i,m)}$ is guaranteed by the standard theory of 1-d stochastic heat equation with Lipschitz coefficients \cite{MR1271224}.
			
			By the strong comparison principle \cite{MR1271224}*{Corollary 2.4}, we can assume without loss of generality that, for each $m \in \mathbb N$, $u^{(1,m)}$ and $u^{(2,m)}$ are defined on the same probability space satisfying that almost surely $u^{(1,m)}_t(x) \leq u^{(2,m)}_t(x)$ for every $t\geq 0$ and $x\in \mathbb R$.
			For each $i\in \{1,2\}$, it is also standard to argue, c.f.~\cite{MR1271224}*{Proof of Theorem 2.6}, that the family of $\mathcal C([0,\infty), \mathcal C(\mathbb R, [0,z^*]))$-valued  random elements $\{u^{(i,m)}:m\in \mathbb N\}$ is tight, and any sub-sequential convergence-in-distribution limit has the law $\mathscr L_{f^{(i)}}$.
			This implies that the family of $\mathcal C([0,\infty), \mathcal C(\mathbb R, [0,z^*]))^2$-valued random elements $\{(u^{(1,m)},u^{(2,m)}):m\in \mathbb N\}$ is tight.
			
			Let $(u^{(1)}, u^{(2)})$ be one of its sub-sequential convergence-in-distribution limit.
			Clearly, the law of $u^{(1)}$ and $u^{(2)}$ are given by $\mathscr L_{f^{(1)}}$ and $\mathscr L_{f^{(2)}}$ respectively; and almost surely, $u^{(1)}_t(x) \leq u^{(2)}_t(x)$ for every $t\geq 0$ and $x\in \mathbb R$.
			By the disintegration theorem \cite{MR4226142}*{Theorem 8.5},
			there exists a probability kernel $\mathscr K_{f^{(1)}, f^{(2)}}$ on $\mathcal C([0,\infty), \mathcal C(\mathbb R, [0,z^*]))$ such that  $\mathscr K_{f^{(1)}, f^{(2)}}(u^{(1)}, \cdot)$ is the regular conditional distribution of $u^{(2)}$ conditioned on given $u^{(1)}$.
			It is then straightforward to verify that $\mathscr K_{f^{(1)}, f^{(2)}}$ satisfies all the properties desired for this Lemma.
		\end{proof}
		
		\begin{proof}[Proof of Proposition \ref{prop:GI}]
			Construct a $\mathcal C([0,\infty),\mathcal C(\mathbb R, [0,z^*]))$-valued Markov chain $(u^{(m)})_{m\in \mathbb N}$, on a probability space with its probability measure denoted by $\tilde {\mathbb  P}_g$, such that the initial value $u^{(1)}$ has the law $\mathscr L_{f^{(1)}}$, and that, for each $m\in \mathbb N$, the transition kernel from steps $m$ to $m+1$ is $\mathscr K_{f^{(m)}, f^{(m+1)}}$.
			Here, $\mathscr L_{\cdot}$ and $\mathscr K_{\cdot, \cdot}$ are  the same as in Lemma \ref{eq:WC}.
			It is clear that $u^{(m)}$ has the law $\mathscr L_{f^{(m)}}$ for each $m\in \mathbb N$; and
			almost surely $u^{(m)}_t(x) \leq u^{(m+1)}_t(x)$ for every $t\geq 0, x\in \mathbb R$ and $m\in \mathbb N$.
			This allows us to define a random field $\bar u = (\bar u_t(x))_{t\geq 0, x\in \mathbb R}$ as the pointwisely non-decreasing limit of $u^{(m)}$ when $m\uparrow \infty$.
			
			It is standard (c.f.~\cite{MR4698025}*{p.~82}) to verify from the mild formulation \eqref{eq:mild}, Burkholder-Davis-Gundy inequality, Jensen's inequality, the property of the heat kernels, and the fact that $(u^{(m)})_{m\in \mathbb N}$ are bounded random fields, that, for any $T>0$ and $l > 2$,
			\begin{equation}
				\tilde {\mathbb E}_g\brk{\abs{U^{(m)}_t(x) - U^{(m)}_s(y)}^{2l}} \lesssim
				\pr{|x-y| + \sqrt{|t-s|}}^{l}
			\end{equation}
			uniformly in $m\in \mathbb N$, $x,y\in \mathbb R$ and $t,s\in [0,T]$.
			Here, for any  $t>0$ and  $x\in \mathbb R$,
			\begin{equation}
				U^{(m)}_t(x) :=
				u^{(m)}_t(x) - \int p_t(x-y) f^{(m)}(y) \mathrm{d} y + \iint_0^t p_{t-s}(x-y) \Phi\pr{u^{(m)}_s(y)} \mathrm{d} s\mathrm{d} y
			\end{equation}
			and $U^{(m)}_0(x) := 0$.
			Taking $m\uparrow \infty$, we know from the bounded convergence theorem that for any $T>0$ and $l>2$,
			\begin{equation} \label{eq:BU}
				\tilde{\mathbb E}_g\brk{\abs{\bar U_t(x) - \bar U_s(y)}^{2l}} \lesssim \pr{|x-y| + \sqrt{|t-s|}}^{l}
			\end{equation}
			uniformly in $x,y\in \mathbb R$ and $t,s\in [0,T]$, where, for each  $t> 0$  and $x\in \mathbb R$,
			\begin{equation}
				\bar U_t(x) := \bar u_t(x) - \int p_t(x-y) g(y) \mathrm{d} y + \iint_0^t p_{t-s}(x-y) \Phi(\bar u_s(y)) \mathrm{d} s\mathrm{d} y
			\end{equation}
			and $\bar U_0(x) := 0$.
			From \eqref{eq:BU}, we can verify (c.f.~\cite{MR1271224}*{Lemma 6.3 (i)}) that there exists a jointly continuous modification $(U_t(x))_{t\geq 0, x\in \mathbb R}$ of $(\bar U_t(x))_{t\geq 0, x\in \mathbb R}$.
			Define, for every $t>0$ and $x\in \mathbb R$,
			\begin{equation}
				u_t(x):= U_t(x) + \int p_t(x-y) g(y) \mathrm{d} y - \iint_0^t p_{t-s}(x-y) \Phi(\bar u_s(y)) \mathrm{d} s\mathrm{d} y
			\end{equation}
			and $u_0(x) := g(x).$
			It is then not hard to see that  $(u_t(x))_{t\geq 0, x\in \mathbb R}$ is a modification of $(\bar u_t(x))_{t\geq 0, x\in \mathbb R}$  and that  $(u_t)_{t> 0}$  is a $\mathcal C(\mathbb R, [0,z^*])$-valued continuous process.
			In particular, from Fubini's theorem, almost surely, for almost every $t>0$ and $x\in \mathbb R$ w.r.t.~the Lebesgue measure, $u_t(x) = \bar u_t(x)$.
			
			Notice from Proposition \ref{prop:D} that, for each $m\in \mathbb N$, \eqref{eq:Duality} holds with $f$ and $u$ being replaced by $f^{(m)}$ and $u^{(m)}$ respectively.
			Taking $m\uparrow \infty$, by the bounded convergence theorem, we have
			\begin{equation}
				\tilde {\mathbb E}_g\brk{\prod_{i=1}^n \pr{1-\bar u_t(x_i)}} = \mathbb E\brk{\prod_{\alpha \in I_t^{(n)}} \pr{1-g\pr{X_t^{(n,\alpha)}}}}, \quad  t\geq 0.
			\end{equation}
			Now, since $u$ is a modification of $\bar u$,   the desired statement \eqref{eq:S2} holds.
			
			Let us fix an arbitrary testing function $\phi \in \mathcal C_\mathrm c^\infty(\mathbb R)$.
			From \cite{MR1271224}*{Theorem 2.1}, almost surely for every $t\geq 0$ and $m\in \mathbb N$,
			\begin{align}\label{eq:MP}
				\int u^{(m)}_t(x)\phi(x)\mathrm{d} x
				& = \int f^{(m)}(x) \phi(x)\mathrm{d} x +  \iint_0^t u^{(m)}_s(y)\frac{\phi''(y)}{2} \mathrm{d} s\mathrm{d} y - {}
				\\&\qquad \iint_0^t \Phi\pr{u^{(m)}_s(y)}\phi(y)\mathrm{d} s\mathrm{d} y+ M^{(m,\phi)}_t
			\end{align}
			where $M^{(m,\phi)}_\cdot$ is a $(\mathcal G^{(m)}_t)_{t\geq 0}$-adapted
			continuous martingale with quadratic variation
			\begin{equation} \label{eq:QQV}
				\chv{M^{(m,\phi)}_{\cdot}}_t := \iint_0^t \Psi\pr{u^{(m)}_s(y)} \phi(y)^2 \mathrm{d} s \mathrm{d} y, \quad t\geq 0
			\end{equation}
			and $(\mathcal G_t^{(m)})_{t\ge 0}$ is the natural filtration of the process $(u^{(m)}_t)_{t\geq 0}$.
			In particular, for any $0\leq s\leq t,$ $N\in \mathbb N$, bounded continuous map $G$ from $\mathbb R^N$ to $\mathbb R$, and list $((s_i,y_i))_{i=1}^N$ in $[0,s] \times \mathbb R$, we have
			\begin{equation} \label{eq:AA}
				\tilde {\mathbb E}_g\brk{ G\pr{u^{(m)}_{s_1}(y_1), \dots, u^{(m)}_{s_N}(y_N)} M^{(m,\phi)}_t}
				=\tilde {\mathbb E}_g\brk{ G\pr{u^{(m)}_{s_1}(y_1), \dots, u^{(m)}_{s_N}(y_N)} M^{(m,\phi)}_s}
			\end{equation}
			and
			\begin{align}\label{eq:BB}
				& \tilde {\mathbb E}_g\brk{ G\pr{u^{(m)}_{s_1}(y_1), \dots, u^{(m)}_{s_N}(y_N)} \pr{\pr{M^{(m,\phi)}_t}^2 - \chv{M^{(m,\phi)}_\cdot}_t}}
				\\&=\tilde {\mathbb E}_g\brk{ G\pr{u^{(m)}_{s_1}(y_1), \dots, u^{(m)}_{s_N}(y_N)} \pr{\pr{M^{(m,\phi)}_s}^2 - \chv{M^{(m,\phi)}_\cdot}_s}}.
			\end{align}
			Taking $m\uparrow \infty$ in \eqref{eq:MP}, by the dominated convergence theorem, we know that $M_t^{(m,\phi)}$ converges to $M_t^{(\phi)}$ almost surely for every $t\geq 0$, where  the continuous process $(M^{(\phi)}_t)_{t\geq 0}$  is defined through
			\begin{align}\label{eq:RMP}
				\int u_t(x)\phi(x)\mathrm{d} x
				& = \int g(x) \phi(x)\mathrm{d} x + \iint_0^t u_s(y)\frac{\phi''(y)}{2} \mathrm{d} s\mathrm{d} y - {}
				\\&\qquad \iint_0^t \Phi\pr{u_s(y)}\phi(y)\mathrm{d} s\mathrm{d} y+ M^{(\phi)}_t, \quad t\geq 0.
			\end{align}
			Taking $m\uparrow \infty$ in \eqref{eq:QQV}, we know that $\langle M^{(m,\phi)}_\cdot\rangle_t$ converges to $\langle M^{(\phi)}_\cdot\rangle_t$ almost surely for every $t\geq 0$, where
			\[
			\chv{ M^{(\phi)}_\cdot}_t :=  \iint_0^t \Psi\pr{u_s(y)} \phi(y)^2 \mathrm{d} s \mathrm{d} y, \quad t\geq 0.
			\]
			Note that, for each $t\geq 0$, the families of random variables $(M_t^{(m,\phi)}:m\in \mathbb N)$ and $((M_t^{(m,\phi)})^2-\langle M_\cdot^{(m,\phi)}\rangle_t: m\in \mathbb N)$ are bounded by a deterministic constant depending only on $\Phi$, $\Psi$, $t$ and $\phi$.
			Therefore, taking $m\uparrow \infty$ in \eqref{eq:AA} and \eqref{eq:BB}, for any $0\leq s\leq t,$ $N\in \mathbb N$, bounded continuous map $G$ from $\mathbb R^N$ to $\mathbb R$, and list $((s_i,y_i))_{i=1}^N$ in $[0,s] \times \mathbb R$, we have
			\begin{equation}
				\tilde {\mathbb E}_g\brk{ G\pr{u_{s_1}(y_1), \dots, u_{s_N}(y_N)} M^{(\phi)}_t}
				=\tilde {\mathbb E}_g\brk{ G\pr{u_{s_1}(y_1), \dots, u_{s_N}(y_N)} M^{(\phi)}_s}
			\end{equation}
			and
			\begin{align}
				& \tilde {\mathbb E}_g\brk{ G\pr{u_{s_1}(y_1), \dots, u_{s_N}(y_N)} \pr{\pr{M^{(\phi)}_t}^2 - \chv{M^{(\phi)}_\cdot}_t}}
				\\&=\tilde {\mathbb E}_g\brk{ G\pr{u_{s_1}(y_1), \dots, u_{s_N}(y_N)} \pr{\pr{M^{(\phi)}_s}^2 - \chv{M^{(\phi)}_\cdot}_s}}.
			\end{align}
			These imply that $(M^{(\phi)}_t)_{t\geq 0}$ is an $(\mathcal G_t)_{t\geq 0}$-adapted continuous martingale with quadratic variation $(\langle M_\cdot^{(\phi)}\rangle_t)_{t\geq 0}$.
			We are done.
		\end{proof}
		
		\section{Proof of Lemma   \ref{lem:2}}\label{append-B}
		
		\begin{proof}[Proof of Lemma \ref{lem:2}]
			Define $\tau_{z} := \inf\{s\geq 0: \sup_{y\in F} u_s(y)> z\}$ for every $z \in (0,1)$.
			Let $t>0$.
			For every $\tilde \gamma \in [0,1)$,
			define
			\begin{align}\label{Def: Semi-martingale}
				M_s^{(\tilde \gamma, n)}:=
				\begin{cases}
					\displaystyle\frac{1}{1-\tilde \gamma}\int u_{s}(y)v_{t-s}^{(\emptyset, \mu^{(\tilde \gamma,n)})}(y)\mathrm{d}  y,\quad & s\in [0, t); \\
					\displaystyle\frac{1}{1-\tilde \gamma}\int u_{t}(y)\mu^{(\tilde \gamma, n)}\pr{\mathrm{d}  y},\quad                     & s=t,
				\end{cases}
			\end{align}
			with $\mu^{(\tilde \gamma,n)}:= (1-\tilde \gamma)\theta(\tilde \gamma)\sum_{i=1}^n \delta_{x_i}.$
			From stochastic Fubini's theorem \cite{MR3236753}*{Theorem 4.33}, one can verify that
			almost surely  for every $\tilde \gamma \in [0,1)$ and $s\in [0,t]$,
			\begin{align}\label{eq:SMM}
				& M_s^{(\tilde \gamma, n)}- M_0^{(\tilde \gamma,n)}                                                                                                                                \\
				& = \frac{1}{1-\tilde \gamma}\iint_0^s \pr{-v_{t-r}^{(\emptyset, \mu^{(\tilde \gamma,n)})}(y) \Phi(u_{r}(y))+\frac{\Psi'(0+)}{2}v_{t-r}^{(\emptyset, \mu^{(\tilde \gamma,n)})}(y)^2
					u_{r}(y)}\mathrm{d}  r\mathrm{d}  y \\&\quad +\frac{1}{1-\tilde \gamma}\iint_0^s
				v_{t-r}^{(\emptyset, \mu^{(\tilde \gamma,n)})}(y) \sqrt{\Psi(u_r(y))}W(\mathrm{d} r \mathrm{d} y).
			\end{align}
			In particular, $(M_s^{(\tilde \gamma,n)})_{s\in [0,t]}$ is a continuous semi-martingale.
			By \eqref{eq:SMM} and It\^{o}'s formula, it is easy to verify that almost surely for every $s\in [0,t],$
			\begin{align}\label{eq:HC}
				& \exp\brc{-e^{\lambda_{\mathrm o} (t-s)}M_s^{(\gamma,n)}} - \exp\brc{-e^{\lambda_{\mathrm o}t}M_0^{(\gamma,n)}}
				\\&
				=\iint_0^s \exp\brc{-e^{\lambda_0(t-r)}M^{(\gamma,n)}_r}\frac{e^{\lambda_0(t-r)}}{1-\gamma} \times {}
				\\& \qquad \qquad \Bigg( -v_{t-r}^{(\emptyset, \mu^{(\gamma,n)})}(y) \sqrt{\Psi(u_r(y))}W(\mathrm{d} r\mathrm{d} y) + {}
				\\& \qquad \qquad \qquad v_{t-r}^{(\emptyset, \mu^{(\gamma,n)})}(y) \pr{\Phi(u_{r}(y))+\lambda_{\mathrm o} u_{r}(y)} \mathrm{d}  r\mathrm{d}  y - {}
				\\& \qquad \qquad \qquad \frac{1}{2}v_{t-r}^{(\emptyset, \mu^{(\gamma,n)})}(y)^2 \pr{\Psi'(0+)u_r(y) - \frac{e^{\lambda_0(t-r)}}{1-\gamma} \Psi(u_r(y))}\mathrm{d} r \mathrm{d} y
				\Bigg).
			\end{align}
			Set
			\[
			\gamma_0:= \frac{\gamma}{2  \beta_{\mathrm c}} \Psi'(0+) =\frac{\gamma}{2}\pr{2-\sum_{k=0}^\infty kq_k} \in (0,\gamma).
			\]
			We see that for any $w\in (0, \gamma_0]$,
			\begin{align}\label{Lower-bound-of-Psi}
				& \Psi(w)= \int_0^w \Psi'(z)\mathrm{d}  z=
				\beta_{\mathrm c} \int_0^w \pr{2(1-z)-\sum_{k=1}^\infty k q_k(1-z)^{k-1}}\mathrm{d} z         \\
				& \geq \beta_{\mathrm c} w \pr{2(1-\gamma_0)- \sum_{k=1}^\infty kq_k} =(1-\gamma)\Psi'(0+)w.
			\end{align}
			Note the fact that almost surely for $r\leq \tau_{\gamma_0}\land t$ and $y\in F$,
			\[
			\Psi'(0+) u_r(y)-\frac{e^{\lambda_{\mathrm o}(t-r)}}{1-\gamma}\Psi(u_r(y)) \leq \Psi'(0+)u_r(y)-\frac{1}{1-\gamma}\Psi(u_r(y))\leq 0,
			\]
			where the last inequality follows from \eqref{Lower-bound-of-Psi}.
			Combining this with \eqref{eq: bounds-for-branching-mechanism}, $\lambda_{\mathrm o}>0$, and $M_{r}^{(\gamma, n)}\geq 0$,
			we get  by taking expectation for  \eqref{eq:HC} with $s = \tau_{\gamma_0} \wedge t$  that
			\begin{align}
				& \tilde{\mathbb{E}}_{\varepsilon \mathbf 1_U }\brk{\exp\brc{-e^{\lambda_{\mathrm o} (t-\tau_{\gamma_0}\land t)}M_{\tau_{\gamma_0}\land t}^{(\gamma,n)}} }
				\\&\geq \exp\brc{-e^{\lambda_{\mathrm o} t}M_0^{(\gamma,n)}} - {}
				\\& \qquad \frac{1}{2(1-\gamma)^2} \tilde{\mathbb{E}}_{\varepsilon \mathbf 1_U }  \bigg[ \int_0^{\tau_{\gamma_0} \land t} \int_{F^c} \exp\brc{-e^{\lambda_{\mathrm o} (t-r)}M_r^{(\gamma,n)}} e^{\lambda_{\mathrm o} (t-r)}v_{t-r}^{(\emptyset, \mu^{(\gamma,n)})}(y)^2 \times {} \\
				& \qquad\qquad \pr{(1-\gamma)\Psi'(0+)u_r(y)-\Psi(u_r(y))e^{\lambda_{\mathrm o} (t-r)}}  \mathrm{d} y \mathrm{d} r  \bigg]                                                                                                                                                         \\
				& \geq \exp\brc{-e^{\lambda_{\mathrm o} t}M_0^{(\gamma,n)}} - \frac{\Psi'(0+) e^{\lambda_{\mathrm o} t}}{2(1-\gamma)}\int_0^{ t} \int_{F^c}  v_{t-r}^{(\emptyset, \mu^{(\gamma,n)})}(y)^2 \tilde{\mathbb{E}}_{\varepsilon  \mathbf 1_U } \brk{u_r(y)}  \mathrm{d} y \mathrm{d} r .
			\end{align}
			From Lemma \ref{Upper-bound-w-t-x-2},  we get that
			\begin{align}
				& \tilde{\mathbb{E}}_{\varepsilon \mathbf 1_U }\brk{\exp\brc{-e^{\lambda_{\mathrm o}(t-\tau_{\gamma_0}\land t)}M_{\tau_{\gamma_0}\land t}^{(\gamma,n)}} }                                                                        \\
				& \geq  \exp\brc{-e^{\lambda_{\mathrm o} t}M_0^{(\gamma,n)}} - \frac{\varepsilon\Psi'(0+) e^{2\lambda_{\mathrm o}t} }{2(1-\gamma)}\int_0^{ t} \int_{F^c} v_{t-r}^{(\emptyset, \mu^{(\gamma,n)})}(y)^2  \mathrm{d} y\mathrm{d} r.
			\end{align}
			By the definition of $M_t^{(\gamma,n)}$  and that
			\[
			\tilde{\mathbb{E}}_{\varepsilon  \mathbf 1_U }\brk{e^{-M_t^{(\gamma,n)}} }
			\geq \tilde{\mathbb{E}}_{\varepsilon \mathbf 1_U}\brk{\exp\brc{-e^{\lambda_{\mathrm o}(t-\tau_{\gamma_0}\land t)}M_{\tau_{\gamma_0}\land t}^{(\gamma,n)}} } -\tilde{\mathbb{P}}_{\varepsilon  \mathbf 1_U }\pr{\tau_{\gamma_0}<t},
			\]
			we have
			\begin{align}\label{eq:1}
				\qquad & \tilde{\mathbb{E}}_{\varepsilon  \mathbf 1_U }\brk{\exp\brc{-\theta(\gamma)\sum_{i=1}^n u_{t}(x_i) }}
				\\& \geq \exp\brc{-\frac{\varepsilon e^{\lambda_{\mathrm o} t}}{1-\gamma}\int_U v_{t}^{(\emptyset,\mu^{(\gamma, n)})} (y)\mathrm{d}  y } - \frac{\varepsilon \Psi'(0+)e^{2\lambda_{\mathrm o} t}}{2(1-\gamma)} \int_0^t  \int_{F^c} v_{t-r}^{(\emptyset, \mu^{(\gamma,n)})}(y)^2 \mathrm{d}  r\mathrm{d}  y - {}
				\\&\qquad  \tilde{\mathbb{P}}_{\varepsilon  \mathbf 1_U }\pr{\sup_{s\leq t, y\in F} u_{s}(y)> \gamma_0}.
			\end{align}
			It has been explained in \cite{MR4698025}*{(2.19)} that
			\begin{align}
				& v_{r,y}^{(\emptyset, \mu^{(\gamma, n)})}
				\xrightarrow[n\uparrow \infty]{\text{increasingly}}
				v_{r,y}^{(\Lambda, (1-\gamma)\theta(\gamma)\mu)}\leq v_{r,y}^{(\Lambda, \mu)}, \quad r>0,y\in \mathbb R.
			\end{align}
			Now we get the first desired result \eqref{eq:LIL} by monotone convergence theorem while taking $n\uparrow\infty$ in \eqref{eq:1}.

			Let us now prove the upper bound \eqref{eq:LIU}.
			Similarly using \eqref{eq:SMM} and It\^{o}'s formula, we see that
			almost surely for every $s\in [0,t]$,
			\begin{align}
				& \exp\brc{-\kappa(\gamma) e^{-\beta_{\mathrm o} (t-s)}M_s^{(0,n)}} - \exp\brc{-\kappa(\gamma) e^{-\beta_{\mathrm o} t}M_0^{(0,n)}}
				\\&
				= \kappa(\gamma)\iint_0^s
				\exp\brc{- \kappa(\gamma)e^{\beta_0(t-r)}M^{(0,n)}_r}e^{-\beta_{\mathrm o} (t-r)}
				\times {}
				\\& \qquad \qquad \Bigg( -v_{t-r}^{(\emptyset, \mu^{(0,n)})}(y) \sqrt{\Psi(u_r(y))}W(\mathrm{d} r\mathrm{d} y)) + {}
				\\& \qquad \qquad \qquad v_{t-r}^{(\emptyset, \mu^{(0,n)})}(y) \pr{\Phi(u_{r}(y))-\beta_{\mathrm o}  u_{r}(y)} \mathrm{d} r\mathrm{d} y - {}
				\\& \qquad \qquad \qquad \frac{1}{2}v_{t-r}^{(\emptyset, \mu^{(\gamma,n)})}(y)^2 \pr{\Psi'(0+)u_r(y)-\kappa(\gamma) \Psi(u_r(y))e^{-\beta_{\mathrm o} (t-r)}} \mathrm{d} r\mathrm{d} y
				\Bigg).
			\end{align}
			Replacing $s$ with $\tau_{\gamma} \wedge t$ and taking the expectation of the above equation, combining \eqref{eq: bounds-for-branching-mechanism} and the fact that almost surely for all $r\leq t\wedge \tau_{\gamma}$ and $y\in F$,
			\[
			\Psi'(0+)u_r(y)-\kappa(\gamma) \Psi(u_r(y))e^{-\beta_{\mathrm o} (t-r)}\geq  \Psi'(0+)u_r(y)-\kappa(\gamma) \Psi(u_r(y)) \geq 0,
			\]
			we have
			\begin{align}
				& \tilde{\mathbb{E}}_{\varepsilon \mathbf 1_U }\brk{\exp\brc{-\kappa(\gamma) e^{-\beta_{\mathrm o} (t-t\land \tau_{\gamma})}M_{\tau_\gamma\land t}^{(0,n)}} }
				\\& \leq \exp\brc{-\kappa(\gamma) e^{-\beta_{\mathrm o} t}M_0^{(0,n)}} - {}
				\\& \qquad  \frac{\kappa(\gamma) }{2}\tilde{\mathbb{E}}_{\varepsilon \mathbf 1_U} \bigg[ \int_0^{\tau_\gamma \land t} \int_{F^c} \exp\brc{-\kappa(\gamma) e^{-\beta_{\mathrm o} (t-r)}M_r^{(0,n)}} e^{-\beta_{\mathrm o} (t-r)}v_{t-r}^{(\emptyset, \mu^{(0,n)})}(y)^2 \times {}
				\\&\qquad\qquad \pr{\Psi'(0+)u_r(y)-\kappa(\gamma) \Psi(u_r(y))e^{-\beta_{\mathrm o} (t-r)}} \mathrm{d} r\mathrm{d} y\bigg]
				\\&\leq  \exp\brc{-\kappa(\gamma) e^{-\beta_{\mathrm o} t}M_0^{(0,n)}} + \frac{1}{2} \tilde{\mathbb{E}}_{\varepsilon  \mathbf 1_U} \bigg[ \int_0^{t} \int_{F^c} v_{t-r,y}^{(\emptyset, \mu^{(0,n)})}(y)^2 \Psi(u_r(y)) \mathrm{d} r\mathrm{d} y\bigg],
			\end{align}
			where in the last inequality we used the fact that $\kappa(\gamma)\leq 1$.
			
			According to \eqref{eq: bounds-for-branching-mechanism} and Lemma \ref{Upper-bound-w-t-x-2},
			we concluded from above inequality that
			\begin{align}
				& \tilde{\mathbb{E}}_{\varepsilon  \mathbf 1_U}\brk{\exp\brc{-\kappa(\gamma) e^{-\beta_{\mathrm o} (t-t\land \tau_{\gamma})}M_{\tau_\gamma\land t}^{(0,n)}} }                                                            \\
				& \leq  \exp\brc{-\kappa(\gamma) e^{-\beta_{\mathrm o} t}M_0^{(0,n)}} + \varepsilon \beta_{\mathrm c} e^{\lambda_{\mathrm o} t}  \int_0^{t} \int_{F^c} v_{t-r}^{(\emptyset, \mu^{(0,n)})}(y)^2 \mathrm{d} r\mathrm{d} y.
			\end{align}
			Finally, noticing that
			\[0\leq
			\kappa(\gamma) e^{-\beta_{\mathrm o} (t-t\land \tau_{\gamma})}M_{\tau_\gamma\land t}^{(0,n)}\leq  e^{-\beta_{\mathrm o} (t-t\land \tau_{\gamma})}M_{\tau_\gamma\land t}^{(0,n)}
			\]
			and that
			\begin{align}
				\tilde{\mathbb{E}}_{\varepsilon  \mathbf 1_U}\brk{\exp\brc{- e^{-\beta_{\mathrm o} (t-t\land \tau_{\gamma})}M_{\tau_\gamma\land t}^{(0,n)}} }
				\geq  \tilde{\mathbb{E}}_{\varepsilon  \mathbf 1_U}\brk{\exp\brc{- M_{ t}^{(0,n)}} } -\tilde{\mathbb{P}}_{\varepsilon  \mathbf 1_U}\pr{\tau_\gamma \leq t},
			\end{align}
			we conclude that
			\begin{align}
				& \tilde{\mathbb{E}}_{\varepsilon  \mathbf 1_U }\brk{\exp\brc{- \sum_{i=1}^n u_t(x_i)} }
				=\tilde{\mathbb{E}}_{\varepsilon  \mathbf 1_U}\brk{\exp\brc{- M_{ t}^{(0,n)}} }
				\\&\leq \exp\brc{-\kappa(\gamma) e^{-\beta_{\mathrm o} t}M_0^{(0,n)}}  + \tilde{\mathbb{P}}_{\varepsilon  \mathbf 1_U}\pr{\tau_\gamma \leq t} + \varepsilon \beta_{\mathrm c}e^{\lambda_{\mathrm o}  t} \int_0^{t} \int_{F^c} v_{t-r}^{(\emptyset, \mu^{(0,n)})}(y)^2 \mathrm{d} r\mathrm{d} y.
			\end{align}
			Combining the above inequality  and the fact that $v_t^{(0,\mu^{(0,n)})}$ converges to $v_t^{(\Lambda, \mu)}$
			(c.f.~ the argument below \cite{MR4698025}*{(2.18)}), we get \eqref{eq:LIU}.
		\end{proof}
		
		\section{Proof of Lemma \ref{lem: Prob-w}}\label{append-C}
		Let $\beta_\mathrm o$, $(p_k)_{k=0}^\infty$, $\beta_{\mathrm c}$ and $(q_k)_{k=0}^\infty$ be given as in \eqref{eq:OBR}--\eqref{eq:CB}.
		Suppose that \eqref{asp:A3}, \eqref{asp:A1} and \eqref{asp:A2} hold.
		Let $\Phi$ and $\Psi$ be given as in \eqref{Def-Phi} and \eqref{Def-Psi} respectively.
		Let $f$ be a measurable function on $\mathbb R$ which can be approximated by the elements of $\mathcal C(\mathbb R, [0,z^*])$ monotonically from below.
		Let $(u_t)_{t>0}$
		be the continuous $\mathcal C(\mathbb R, [0,z^*])$-valued process given as in Proposition \ref{prop:GI}, with  initial value  $u_0 = f$,  on a probability space whose probability measure will be denoted by $\tilde {\mathbb P}_f$.
		
		In this section, we prove Lemma \ref{lem: Prob-w} following the standard strategy of \cite{MR1339735}.
		From \eqref{eq:S1}, we can assume without loss of generality, c.f.~\cite{MR0958288}*{Proof of Lemma 2.4}, that there exists a stochastic basis $(\tilde \Omega, \tilde {\mathcal F}, (\tilde {\mathcal F}_t)_{t\geq 0}, \tilde {\mathbb P}_f, W)$ such that $(u_t)_{t> 0}$ is an $(\tilde {\mathcal F}_t)_{t\geq 0}$-adapted $\mathcal C(\mathbb R, [0,z^*])$-valued continuous process, and that, for each $\phi \in \mathcal C_\mathrm c^\infty(\mathbb R)$, almost surely,
		\begin{align}
			\int u_t(x) \phi(x) \mathrm{d} x & = \int f(x) \phi(x) \mathrm{d} x + \iint_0^t u_s(y) \frac{\phi''(y)}{2} \mathrm{d} s \mathrm{d} y - {}
			\\& \qquad \iint_0^t \Phi(u_s(y)) \phi(y) \mathrm{d} s \mathrm{d} y + \iint_0^t \Psi(u_s(y)) W(\mathrm ds \mathrm dy), \quad t\geq 0.
		\end{align}
		It is then standard, c.f.~\cite{MR1271224}*{Theorem 2.1}, that the mild form \eqref{eq:mild} holds almost surely for every $(t,x)\in (0,\infty)\times \mathbb R$.
		Define continuous random fields $(M(s,y))_{s\geq 0,y\in \mathbb R}$ and $(N(s,y))_{s\geq 0,y\in \mathbb R}$ such that for any $s>0$ and $y\in \R$,
		\begin{align}
			& M(s,y) :=  \iint_0^s p_{s-r}(y-z)  \Phi(u_{r}(z))\mathrm{d} r\mathrm{d} z , \\
			& N(s,y)
			:= u_s(y) -\int p_{s}(y-z)f(z)\mathrm dz + \iint_0^t p_{s-r}(y-z) \Phi (u_r(z))\mathrm dr\mathrm dz
		\end{align}
		and $M_{0}(y):=N_0(y)=0$.
		Note that, for every $s>0$ and $y\in \mathbb R$, almost surely
		\begin{equation} \label{eq:RN}
			N(s,y)=\iint_0^s p_{s-r}(y-z) \sqrt{\Psi(u_r(z))}W(\mathrm{d} r\mathrm{d} z).
		\end{equation}
		In the following, we take $p_{s}(x):=0$ when  $s\leq 0$  and $x\in \R$ for the sake of notation conveniences.

		\begin{lemma}\label{Heat-Kernel-estimate}
			For every $t>0$, uniformly for every $(s,y), (s',y')\in [0,t]\times \mathbb R$,
			\begin{equation}\label{Ineq-2}
				\mathcal{K}_{s,y;s'y'}^{(1)}:=\iint_0^\infty  \left|p_{s-r}(y-z) - p_{s'-r}(y'-z)\right| \mathrm{d}  z \mathrm{d}  r
				\lesssim |y-y'| +\sqrt{|s-s'|}
			\end{equation}
			and
			\begin{equation}\label{Ineq-1}
				\mathcal{K}_{s,y;s'y'}^{(2)}:= \iint_0^\infty  \pr{p_{s-r}(y-z) - p_{s'-r}(y'-z)}^2 \mathrm{d}  z \mathrm{d}  r \lesssim |y-y'| +\sqrt{|s-s'|}.
			\end{equation}
		\end{lemma}
		
		\begin{proof}
			For \eqref{Ineq-1}, see \cite{MR1271224}*{Lemma 6.2 (i)}.
			We now prove \eqref{Ineq-2}.
			Let  $t>0$ and let $(s,y), (s',y')\in [0,t]\times \mathbb R$ be arbitrary.
			Without loss of generality, let us assume that  $s'\leq s$ and define  $\delta_s:=s-s'$    and $\delta_y := |y-y'|$.
			Note that
			$\mathcal{K}_{s,y;s'y'}^{(1)} \leq  I_1 + I_2
			$ where
			\begin{equation}
				I_1:=\iint_0^\infty \abs{p_{s-r}(y-z) - p_{s-r}(y'-z)} \mathrm{d}  r \mathrm{d}  z  =\iint_0^s  \abs{ p_r (z+ \delta_y)- p_r(z) } \mathrm{d}  r \mathrm{d} z
			\end{equation}
			and
			\begin{equation}
				I_2:=\iint_0^\infty \abs{p_{s-r}(y'-z) - p_{s'-r}(y'-z)} \mathrm{d}  r \mathrm{d}  z
				=\iint_{-\delta_s}^{s'} \left|p_{\delta_s+r}(z) - p_{r}(z)\right| \mathrm{d}  r \mathrm{d}  z.
			\end{equation}
			For $I_1$, we get that uniformly for the arbitrary $(s,y), (s',y')\in [0,t]\times \mathbb R$,
			\begin{align}\label{I-1}
				& I_1
				\leq\iint_0^s  \int_0^{\delta_y} \frac{\left|\xi +z\right| }{\sqrt{2\pi r^3}}e^{-(\xi +z)^2/(2r)} \mathrm{d}  \xi \mathrm{d}  r \mathrm{d} z
				= \delta_y  \int \frac{|z|}{\sqrt{2\pi}} e^{-z^2/2}\mathrm{d}  z \int_0^s \frac{1}{\sqrt{r}}\mathrm{d}  r
				\lesssim \delta_y.
			\end{align}
			For $I_2$, we decompose it as $I_2=I_{21} +  I_{22}$ where  uniformly for   $(s,y), (s',y')\in [0,t]\times \mathbb R$,
			\begin{equation}\label{I-2-1}
				I_{21}:=\iint_{-\delta_s}^{\sqrt{\delta_s}\wedge s'} \abs{p_{\delta_s+r}(z) - p_{r}(z)} \mathrm{d}  r \mathrm{d}  z
				\leq 2\delta_s +2\sqrt{\delta_s} \wedge s' \lesssim \sqrt{\delta_s}
			\end{equation}
			and
			\begin{align}\label{step_13}
				& I_{22}
				:=\iint_{\sqrt{\delta_s} \wedge s'}^{s'} \abs{p_{\delta_s+r}(z) - p_{r}(z)} \mathrm{d}  r \mathrm{d}  z
				=\iint_{\sqrt{\delta_s}}^{s'\vee \sqrt{\delta_s}} \abs{p_{\delta_s+r}(z) - p_{r}(z)} \mathrm{d}  r \mathrm{d}  z
				\\& \leq \iint_{\sqrt{\delta_s}}^{s'\vee \sqrt{\delta_s}} \int_r^{\delta_s+r} \abs{-\frac{1}{2a}+\frac{z^2}{2a^2}}p_a(z) \mathrm{d} a \mathrm{d}  r \mathrm{d}  z
				\leq \int_{\sqrt{\delta_s}}^{s'\vee \sqrt{\delta_s}} \int_r^{\delta_s+r} \frac{1}{a} \mathrm{d} a \mathrm{d}  r
				\\& = \int_{\sqrt{\delta_s}}^{s'\vee \sqrt{\delta_s}} \log\pr{\frac{\delta_s+r}{r}}\mathrm{d}  r
				\leq \int_{\sqrt{\delta_s}}^{s'\vee \sqrt{\delta_s}} \frac{\delta_s}{r}\mathrm{d}  r
				\leq  \int_{\sqrt{\delta_s}}^{s'\vee \sqrt{\delta_s}} \sqrt{\delta_s}\mathrm{d}  r
				\lesssim \sqrt{\delta_s}.
			\end{align}
			We are done.
		\end{proof}

		\begin{lemma}\label{Moment-of-M-N}
			Let $U$ be an open interval and $\varepsilon\in (0,1)$.
			Let the initial value $f$ of the process $(u_t)_{t\geq 0}$ be given by $f= \varepsilon  \mathbf 1_U $.
			Then for any $p>1$ and $t>0$, uniformly for the arbitrary open interval $U$, the arbitrary parameter $\varepsilon \in (0,1)$, $s,s' \in (0,t)$ and $y, y'\in \R$, we have
			\begin{equation}
				\tilde{\mathbb{E}}_{\varepsilon  \mathbf 1_U }\brk{\abs{M(s,y)- M(s',y')}^{p}}\lesssim \varepsilon \pr{|y-y'| +\sqrt{|s-s'|}}^{p-1} \pr{\mathbf{P}_y(B_s\in U)+ \mathbf{P}_{y'}(B_{s'}\in U)},
			\end{equation}
			and
			\begin{equation}
				\tilde{\mathbb{E}}_{\varepsilon  \mathbf 1_U }\brk{\abs{N(s,y)- N(s',y')}^{2p}}
				\lesssim \varepsilon \pr{|y-y'| +\sqrt{|s-s'|}}^{p-1} \pr{\mathbf{P}_y(B_s\in U)+ \mathbf{P}_{y'}(B_{s'}\in U)}.
			\end{equation}
		\end{lemma}
		\begin{proof}
			From \eqref{eq: bounds-for-branching-mechanism} we know that the random field $\Psi(u)$ is bounded by a deterministic constant.
			Therefore, uniformly for  the arbitrary open interval $U$, the arbitrary parameter $\varepsilon \in (0,1)$, $s,s' \in (0,t)$ and $y, y'\in \R$, we have
			\begin{equation}
				\left|M(s,y)- M(s', y')\right| \leq \iint_0^\infty \left| p_{s-r}(y-z) - p_{s'-r}(y'-z)\right| |\Phi(u_r(z))|\mathrm{d} r\mathrm{d} z 			\lesssim \mathcal{K}_{s,y;s',y'}^{(1)}.
			\end{equation}
			From \eqref{eq: bounds-for-branching-mechanism}, Lemma \ref{Upper-bound-w-t-x-2}, and the Markov property of the Brownian motion that, uniformly for  the arbitrary open interval $U$, the arbitrary parameter $\varepsilon \in (0,1)$, and $(s,y)\in [0,t]\times \mathbb R$,
			\begin{align}
				& \tilde{\mathbb{E}}_{\varepsilon  \mathbf 1_U } \brk{|M(s,y)|}
				\leq  \iint_0^s p_{s-r}(y-z)\tilde{\mathbb{E}}_{\varepsilon \mathbf 1_U } \brk{\abs{\Phi(u_r(z))} }\mathrm{d}  r \mathrm{d} z \\
				& \lesssim  \iint_0^s p_{s-r}(y-z)\tilde{\mathbb{E}}_{\varepsilon  \mathbf 1_U } \brk{u_r(z) }\mathrm{d} r \mathrm{d}  z
				\lesssim \varepsilon \iint_0^s p_{s-r}(y-z) \mathbf{P}_z(B_{r}\in U)\mathrm{d} r\mathrm{d} z
				\\& = \varepsilon \int_0^s \mathbf E_y \brk{\mathbf{P}_{B_{s-r}}(B_{r}\in U)} \mathrm{d} r = \varepsilon s\mathbf P_y(B_s\in U).
			\end{align}
			Therefore, for every $p> 1$, from Lemma \ref{Heat-Kernel-estimate}, uniformly for  the arbitrary open interval $U$, the arbitrary parameter $\varepsilon \in (0,1)$, and $(s,y), (s',y')\in [0,t]\times \mathbb R$,
			\begin{align}
				& \tilde{\mathbb{E}}_{\varepsilon \mathbf 1_U} \brk{\abs{ M(s,y)- M(s', y')}^p}
				\lesssim
				\pr{ \mathcal{K}_{s,y;s',y'}^{(1)}}^{p-1}
				\pr{ \tilde{\mathbb{E}}_{\varepsilon \mathbf 1_U } \brk{|M(s,y)|} + \tilde{\mathbb{E}}_{\varepsilon  \mathbf 1_U } \brk{|M(s', y')|} }
				\\&\lesssim \varepsilon \pr{|y-y'| +\sqrt{|s-s'|}}^{p-1}   \pr{\mathbf{P}_y (B_s \in U) + \mathbf{P}_{y'} (B_{s'} \in U) }
			\end{align}
			as desired for  the first inequality.
			
			For the second inequality, by \eqref{eq:RN}, Burkholder-Davis-Gundy's inequality, \eqref{eq: bounds-for-branching-mechanism}, the fact that the random field $u$ is bounded by 2, and the trivial inequality $(a-b)^2 \leq 2(a^2+b^2)$, we verify that, for every $p>1$, and uniformly for  the arbitrary open interval $U$, the arbitrary parameter $\varepsilon \in (0,1)$, and $(s,y), (s',y')\in [0,t]\times \mathbb R$,
			\begin{align}
				& \tilde{\mathbb{E}}_{\varepsilon  \mathbf 1_U}\brk{\abs{N(s, y)- N(s', y')}^{2p} }
				\\&= \tilde{\mathbb{E}}_{\varepsilon \mathbf 1_U }\brk{\abs{\iint_0^\infty  \pr{p_{s-r}(y-z)- p_{s'-r}(y'-z)} \sqrt{\Psi(u_r(z))}W(\mathrm{d} r\mathrm{d} z)}^{2p}} \\
				& \lesssim \tilde{\mathbb{E}}_{\varepsilon  \mathbf 1_U }\brk{ \pr{ \iint_0^\infty  \pr{p_{s-r}(y-z)- p_{s'-r}(y'-z)} ^2 \Psi(u_{r}(z)) \mathrm{d}  r\mathrm{d} z}^p}                           \\
				& \lesssim \tilde{\mathbb{E}}_{\varepsilon  \mathbf 1_U }\brk{ \pr{ \iint_0^\infty  \pr{p_{s-r}(y-z)- p_{s'-r}(y'-z)} ^2 u_{r}(z) \mathrm{d}  r\mathrm{d} z}^p}                                 \\
				& \lesssim \pr{\mathcal{K}_{s,y;s',y'}^{(2)}}^{p-1} \iint_0^\infty  \pr{p_{s-r}(y-z)- p_{s'-r}(y'-z)} ^2 \tilde{\mathbb{E}}_{\varepsilon  \mathbf 1_U}\brk{ u_{r}(z)} \mathrm{d}  r\mathrm{d} z
				\\& \lesssim \pr{\mathcal{K}_{s,y;s',y'}^{(2)}}^{p-1} \iint_0^\infty  \pr{p_{s-r}(y-z)^2+p_{s'-r}(y'-z)^2} \tilde{\mathbb{E}}_{\varepsilon  \mathbf 1_U}\brk{ u_{r}(z)} \mathrm{d}  r\mathrm{d} z.
			\end{align}
			From Lemma \ref{Upper-bound-w-t-x-2}, we have uniformly for  the arbitrary open interval $U$, the arbitrary parameter $\varepsilon \in (0,1)$, and $(s,y)\in [0,t]\times \mathbb R$,
			\begin{align}\label{Moment-of-N-2}
				& \iint_0^\infty  p_{s-r}(y-z) ^2 \tilde{\mathbb{E}}_{\varepsilon  \mathbf 1_U }\brk{u_{r}(z)} \mathrm{d}  r\mathrm{d} z
				\lesssim \varepsilon \iint_0^\infty  p_{s-r}(y-z) ^2 \mathbf P_z(B_r\in U) \mathrm{d}  r\mathrm{d} z
				\\&\leq \varepsilon \iint_0^\infty  \frac{1}{\sqrt{2\pi (s-r)}}p_{s-r}(y-z) \mathbf P_z(B_r\in U) \mathrm{d}  r\mathrm{d} z
				\\&= \varepsilon \int_0^s  \frac{1}{\sqrt{2\pi (s-r)}} \mathbf P_y(B_s\in U) \mathrm{d}  r
				\lesssim \varepsilon \mathbf P_y(B_s\in U).
			\end{align}
			Now, from above and Lemma \ref{Heat-Kernel-estimate}, for every $p>1$, uniformly for  the arbitrary open interval $U$, the arbitrary parameter $\varepsilon \in (0,1)$, and $(s,y), (s',y')\in [0,t]\times \mathbb R$, we have
			\begin{align}
				& \tilde{\mathbb{E}}_{\varepsilon  \mathbf 1_U}\brk{\abs{N(s, y)- N(s', y')}^{2p} }
				\lesssim \varepsilon \pr{\mathcal{K}_{s,y;s',y'}^{(2)}}^{p-1} \pr{\mathbf P_y(B_s\in U) + \mathbf P_{y'}(B_{s'}\in U)}
				\\& \lesssim \varepsilon \pr{|y-y'| + \sqrt{|s-s'|}}^{p-1} \pr{\mathbf P_y(B_s\in U) + \mathbf P_{y'}(B_{s'}\in U)},
			\end{align}
			which implies the second inequality.
			We are done.
		\end{proof}
		
		\begin{proof}[Proof of Lemma  \ref{lem: Prob-w}]
			
			Let us first show \eqref{eq:UFtg}.
			Let $\varepsilon \in (0,\gamma/2)$ be arbitrary, and assume that $f$, the initial value of the process $(u_t)_{t\geq 0}$, is given by $\varepsilon \mathbf 1_U$.   By
			\eqref{eq:mild},
			we see that
			\begin{align}
				u_{s}(y)\leq \varepsilon +
				\abs{M(s,y)}+ \abs{N(s,y)}, \quad (s,y)\in [0,\infty)\times \mathbb R, \text{a.s.}
			\end{align}
			Therefore, we only need to prove that
			uniformly for the arbitrary $\varepsilon\in (0, \gamma/2)$,
			\begin{align}\label{eq:DC}
				\tilde{\mathbb{P}}_{\varepsilon \mathbf 1_U }\pr{\sup_{s\leq t, y\in F}
					\pr{
						\abs{M(s,y)} + \abs{N(s,y)}}> \frac{\gamma}{2}}
				\lesssim \varepsilon.
			\end{align}
			
			Without loss of generality, we assume that $\tilde F=F \setminus \{\sup(F)\}$ is non-empty.
			In the following, we construct a dyadic approximation of the time-space region $[0,t)\times \tilde F$.
			Note that there exists an (at most) countable set $F_0$ and a constant $\delta \in (0,1]$ such that $\tilde F$ is the disjoint union $\cup_{y\in F_0 } [y,y+\delta).$
			Define a sequence of time-space lattices $(L_m)_{m\in \mathbb Z_+}$, so that
			$L_0:= \{(0,y):y\in F_0\}$,
			and inductively,
			\begin{equation}
				L_m:= \bigcup_{(s,y)\in L_{m-1}}\{(s,y),(s+ t2^{-m},y),(s,y+\delta 2^{-m}), (s+ t2^{-m},y+\delta 2^{-m})\}, \quad m\in \mathbb Z_+.
			\end{equation}
			For each $m\in \mathbb Z_+$, define a map $\Gamma_m$ from $[0,t)\times \tilde F$ to the lattice $L_m$ such that
			for every $(s,y)\in [0,t)\times \tilde F$,
			$(s_m,y_m)= \Gamma_m(s,y)$ is the unique element in $L_m$ satisfying that $s_m \leq s < s_m+t2^{-m}$ and $y_m \leq y < y_m + \delta 2^{-m}.$
			It is not hard to observe that
			\begin{itemize}
				\item[\eq\label{eq:LC}]
				$\Gamma_{m} \circ \Gamma_m = \Gamma_{m}$ and $\Gamma_{m-1} \circ \Gamma_m = \Gamma_{m-1}$ for each $m\in \mathbb N$;
			\end{itemize}
			and that
			\begin{itemize}
				\item[\eq\label{eq:GC}] $\Gamma_m(s,y)$ converges to $(s,y)$ when $m\uparrow \infty$ for every $(s,y)\in [0,t)\times \tilde F$.
			\end{itemize}
			Moreover,  for every $m\in \mathbb N,$ and $(s,y)\in [0,t)\times \mathbb R$,
			\begin{equation}\label{eq:ND}
				|y_m-y_{m-1}| + \sqrt{|s_m-s_{m-1}|} \leq t2^{-{(m-1)}} + \sqrt{\delta 2^{-(m-1)}}
				\leq (2t+\sqrt{2\delta}) 2^{-m/2}
			\end{equation}
			provided $(s_m,y_m)=\Gamma_{m}(s,y)$ and $(s_{m-1},y_{m-1})=\Gamma_{m-1}(s,y)$.
			
			Let us consider an event on which the fluctuations of the random fields $M$ and $N$ on the dyadic lattices $(L_m)_{m=0}^\infty$ are delicately controlled.
			That is, we consider the event
			\begin{equation} \label{eq:DfA}
				A := \bigcap_{m=1}^\infty \bigcap_{(s,y)\in L_m}  \bigcap_{k=1}^2 A^k_{m}(s,y)
			\end{equation}
			where, for any $(s,y)\in [0,t)\times \tilde F$,
			\begin{align}
				& A_{m}^1(s,y) := \brc{\abs{(M\circ \Gamma_m -M\circ \Gamma_{m-1})(s,y)}\leq \gamma_0 2^{-m/10}},
				\\&A_{m}^2(s,y) := \brc{\abs{(N\circ \Gamma_m-N\circ\Gamma_{m-1})(s,y)}\leq \gamma_0 2^{-m/10}}
			\end{align}
			and $\gamma_0>0$ is a constant determined so that $\gamma_0 \sum_{m=1}^\infty 2^{-m/10} = \gamma/4$.
			Now, almost surely on the event $A$, from \eqref{eq:GC}, the fact that $M$ and $N$ are continuous random fields, that $N\circ \Gamma_0 = M\circ \Gamma_0=0$ on $[0,t)\times \tilde F$, and \eqref{eq:LC}, we have, for every $(s,y)\in [0,t)\times \tilde F$,
			\begin{align}
				& |M(s,y)| = \abs{\sum_{m=1}^\infty \pr{M \circ \Gamma_{m}
						- M \circ \Gamma_{m-1}}(s,y)}
				\\&= \abs{\sum_{m=1}^\infty \pr{M \circ \Gamma_{m} - M \circ \Gamma_{m}} \circ \Gamma_{m-1} (s,y)}
				\leq \sum_{m=1}^\infty \gamma_0 2^{-m/10} = \gamma/4,
			\end{align}
			and similarly, $|N(s,y)| \leq \gamma/4$.
			In particular, the event in \eqref{eq:DC} is contained in $A^\mathrm c$.
			Therefore, to show \eqref{eq:DC}, it suffices to show that
			\begin{itemize}
				\item[\eq\label{eq:AE}]
				uniformly for the arbitrary $\varepsilon\in (0, \gamma/2)$, $
				\tilde{\mathbb{P}}_{\varepsilon \mathbf 1_U } (A^\mathrm c) \lesssim \varepsilon.$
			\end{itemize}
			
			Define $\Theta_U(s,y):=\mathbf P_y(B_s\in U)$ the probability that a Brownian motion initiated at location $y\in \mathbb R$ is in the interval $U$ at time $s\geq 0$.
			By the Markov inequality, \eqref{eq:ND}, and Lemma \ref{Moment-of-M-N}, uniformly for the arbitrary open interval $U$, the arbitrary closed interval $F$ satisfying that $U\cap F$ is bounded, the arbitrary parameter $\varepsilon \in (0,\gamma/2)$, $m\in \mathbb N$, and $(s,y)\in L_m$, we have
			\begin{align}
				& \tilde{\mathbb{P}}_{\varepsilon \mathbf 1_U }\pr{A_m^1(s,y)^{\mathrm c}}
				\leq \tilde{\mathbb E}_{\varepsilon \mathbf 1_U}\brk{\abs{(M\circ \Gamma_m-M\circ \Gamma_{m-1})(s,y)}^{20}} /  (\gamma_0 2^{-m/10})^{20}
				\\&\lesssim \varepsilon \pr{(2t + \sqrt{2\delta}) 2^{-m/2}}^{19} (\Theta_U\circ \Gamma_m + \Theta_U \circ \Gamma_{m-1})(s,y) /  (\gamma_0 2^{-m/10})^{20}
				\\ &\lesssim \varepsilon 2^{- 15m/2}  (\Theta_U \circ \Gamma_m + \Theta_U \circ \Gamma_{m-1})(s,y)
			\end{align}
			and
			\begin{align}
				& \tilde{\mathbb{P}}_{\varepsilon \mathbf 1_U }\pr{A_m^2(s,y)^{\mathrm c}}
				\leq \tilde{\mathbb E}_{\varepsilon \mathbf 1_U}\brk{\abs{(N \circ \Gamma_m-N\circ \Gamma_{m-1})(s,y)}^{40}} /  (\gamma_0 2^{-m/10})^{40}
				\\&\lesssim \varepsilon \pr{(2t + \sqrt{2\delta}) 2^{-m/2}}^{19} (\Theta_U \circ \Gamma_m + \Theta_U \circ \Gamma_{m-1})(s,y)
				/  (\gamma_0 2^{-m/10})^{40}
				\\ &\lesssim \varepsilon 2^{- 11m/2} (\Theta_U \circ \Gamma_m + \Theta_U \circ \Gamma_{m-1})(s,y).
			\end{align}
			Now, from \eqref{eq:DfA}, uniformly for the arbitrary open interval $U$, the arbitrary closed interval $F$ satisfying that $U\cap F$ is bounded, the arbitrary parameter $\varepsilon \in (0,\gamma/2)$, we have
			\begin{align}
				& \tilde{\mathbb P}_{\varepsilon \mathbf 1_{U}}(A^\mathrm c)
				\leq \sum_{m=1}^\infty \sum_{(s,y)\in L_m} \pr{\tilde{\mathbb P}_{\varepsilon \mathbf 1_{U}}\pr{A_m^1(s,y)^{\mathrm c}}+ \tilde{\mathbb P}_{\varepsilon \mathbf 1_{U}}\pr{A_m^2(s,y)^{\mathrm c}} }
				\\&\lesssim \varepsilon \sum_{m=1}^\infty 2^{- 11m/2}  \sum_{(s,y)\in L_m} (\Theta_U \circ \Gamma_m + \Theta_U \circ \Gamma_{m-1})(s,y)
				\\&\lesssim \varepsilon \sum_{m=1}^\infty 2^{- 7m/2}  \frac{1}{t2^{-m}\delta 2^{-m}}\sum_{(s,y)\in L_m} (\Theta_U \circ \Gamma_m+ \Theta_U \circ \Gamma_{m-1})(s,y)
				\\&= \varepsilon \sum_{m=1}^\infty 2^{- 7m/2} \int_{F}\int_0^t (\Theta_U \circ \Gamma_m+ \Theta_U \circ \Gamma_{m-1})(s,y) \mathrm{d} s \mathrm{d} y
				\\&\lesssim \varepsilon \sum_{m=0}^\infty 2^{- 7m/2} \int_{F}\int_0^t (\Theta_U \circ \Gamma_m)(s,y) \mathrm{d} s \mathrm{d} y. \label{eq:UI}
			\end{align}
			Since $U \cap F$ is bounded, it is not hard to verify the following analytic results:
			\[
			\int_{F}\int_0^t (\Theta_U \circ \Gamma_m)(s,y) \mathrm{d} s \mathrm{d} y<\infty, \quad m\in \mathbb Z_+
			\]
			and
			\begin{equation}
				\lim_{m\to \infty} \int_{F}\int_0^t (\Theta_U \circ \Gamma_m)(s,y) \mathrm{d} s \mathrm{d} y=\int_{F}\int_0^t \Theta_U(s,y) \mathrm{d} s \mathrm{d} y<\infty,
			\end{equation}
			which together imply that
			\begin{equation}
				\sum_{m=0}^\infty 2^{- 7m/2} \int_{F}\int_0^t (\Theta_U \circ \Gamma_m)(s,y) \mathrm{d} s \mathrm{d} y<\infty.
			\end{equation}
			Now, from \eqref{eq:UI} we have \eqref{eq:AE}, and therefore, the desired \eqref{eq:UFtg}.
			
			Note that $\Cr{c:UFtg}(U,F,t,\gamma)$ is increasing in $t>0$.
			Therefore, $\limsup_{t \downarrow 0} \Cr{c:UFtg}(U,F, t,\gamma) <\infty.$
			
			Finally, let us take $\tilde U$ to be an arbitrary open interval such that its intersection with $F_K:=[K,\infty)$ is bounded for every $K\in \mathbb R$.
			From \eqref{eq:UI}, we can verify that, uniformly for every $K\in \mathbb R$,
			\begin{equation}
				\Cr{c:UFtg}(\tilde U,F_K, t,\gamma)
				\lesssim \sum_{m=0}^\infty 2^{- 7m/2} \int_{K}^\infty \int_0^t (\Theta_{\tilde U} \circ \Gamma_m)(s,y) \mathrm{d} s \mathrm{d} y < \infty.
			\end{equation}
			By the monotone convergence theorem, we have $\lim_{K\uparrow \infty}\Cr{c:UFtg}(\tilde U,F_K, t,\gamma) = 0$ as desired. We are done.
		\end{proof}
		
		\section{Proofs of Lemmas  \ref{lemma:general-vague-convergence1},  \ref{lem:IM}, Propositions \ref{Prop:measurability-of-Z}, and Lemma \ref{lem:SRM}}\label{append-D}
		
		\begin{proof}[Proof of Lemma \ref{lemma:general-vague-convergence1}]
			It is clear that there exists $K>0$ such that $g(x) = 1$ for every $|x|>K$ and $\nu(\{-K,K\}) =0$.
			Therefore, cf.~ \cite{MR3642325}*{proof of Lemma 4.12}, we have $\lim_{m\to \infty}\nu_m((-K,K)) = \nu((-K,K))$. Since $\nu\in \mathcal N$,
			there exists a finite integer $\kappa \geq 0$, a finite list of distinct points $(z_i)_{i=1}^{\kappa}$ in $\mathbb R$, and a finite list $(l_i)_{i=1}^{\kappa}$ in $\mathbb N$, such that
			$\mathbf 1_{(-K,K)}\nu=\sum_{i=1}^{\kappa} l_i\delta_{z_i}$.
			Note that $\nu((-K,K)) = \sum_{i=1}^{\kappa} l_i$.
			Let $\epsilon >0$ be small enough so that, for every $i\neq j$ in $\{1,\dots, \kappa\}$, $[z_i-\epsilon, z_i+\epsilon]$ and $[z_j-\epsilon, z_j+\epsilon]$ are disjoint subsets of $(-K,K)$.
			Now for every $i\in \brc{1,\dots, \kappa}$, since $\nu(\{z_i-\epsilon, z_i+\epsilon\}) = 0$, we have, cf.~ \cite{MR3642325}*{proof of Lemma 4.12} again,
			\begin{align}\label{step_3}
				\lim_{m\to\infty} \nu_m([z_i-\epsilon, z_i+\epsilon]) = \nu([z_i-\epsilon, z_i+\epsilon]) = \nu(\{z_i\}) = l_i.
			\end{align}
			Denoting
			$A := (-K,K) \setminus \pr{\bigcup_{i=1}^\kappa [z_i-\epsilon, z_i+\epsilon]}.$
			Since for any subset $B\subset \R$ and $m\in \mathbb N$, $\nu_m(B)$ takes non-negative integer values, it follows that
			Therefore, there exists $M_0 > 0$  such that for all $m\geq M_0$, we have $\nu_m(A) = 0$ and
			\begin{align}\label{step_4}
				\nu_m([z_i-\epsilon, z_i+\epsilon])= l_i, \quad i\in \{1,\dots, \kappa\}.
			\end{align}
			So, for every $m\geq M_0$ and $i\in \{1,\dots, \kappa\}$, there is a finite list $(z^{(m)}_{i,j})_{j=1}^{l_i}$ in $[z_i - \epsilon, z_i+\epsilon]$ such that
			$
			\mathbf 1_{[z_i- \epsilon, z_i + \epsilon]} \nu_m = \sum_{j=1}^{l_i} \delta_{z^{(m)}_{i,j}}.
			$
			In particular, for every $m\geq M_0$, we have
			$
			\mathbf 1_{(-K,K)}\nu_m = \sum_{i=1}^\kappa \sum_{j=1}^{l_i} \delta_{z_{i,j}^{(m)}}.$
			Now, for every $m\geq M_0$,
			\begin{align}
				& \abs{  \prod_{z\in \mathbb R} g(z)^{\nu_m(\{z\})}- \prod_{z\in \mathbb R} g(z)^{\nu(\{z\})}}
				=   \abs{ \prod_{i=1}^{\kappa} \prod_{j=1}^{l_i} g(z_{i,j}^{(m)})- \prod_{i=1}^{\kappa} \prod_{j=1}^{l_i} g(z_{i})}
				\\&\leq \sum_{i=1}^\kappa \sum_{j=1}^{l_{i}}\abs{g(z_{i,j}^{(m)})-g(z_{i})}
				\leq \sum_{i=1}^\kappa \sum_{j=1}^{l_{i}}\sup \brc{\abs{g(y)-g(z_{i})}:y\in [z_i-\epsilon, z_i+\epsilon]},
			\end{align}
			where the first inequality follows from  \cite{MR3930614}*{Lemma 3.4.3}.
			Taking $m\to\infty$ in the above inequality, we get
			\begin{equation}
				\limsup_{m\to \infty} \abs{  \prod_{z\in \mathbb R} g(z)^{\nu_m(\{z\})}- \prod_{z\in \mathbb R} g(z)^{\nu(\{z\})}}
				\leq \sum_{i=1}^\kappa \sum_{j=1}^{l_{i}}\sup \brc{\abs{g(y)-g(z_{i})}:y\in [z_i-\epsilon, z_i+\epsilon]}.
			\end{equation}
			From the fact that $\epsilon>0$ can be taken arbitrarily small, $\kappa$ and $(l_i)_{i=1}^\kappa$ do not depend on $\epsilon$ when $\epsilon \to 0$, and that $g$ is a continuous function, we have
			\begin{equation}
				\limsup_{m\to \infty} \abs{  \prod_{z\in \mathbb R} g(z)^{\nu_m(\{z\})}- \prod_{z\in \mathbb R } g(z)^{\nu(\{z\})}}= 0.
			\end{equation}
			Now the desired result  \eqref{step_6} holds.
		\end{proof}

		\begin{proof}[Proof of Lemma \ref{lem:IM}]
			Recall that $z^*\in [1,2)$.
			Define $I:=\{i \in \mathbb N: z_i \in (1,z^*]\}$ and $J:= \{i\in \mathbb N: z_i = 1\}$.
			
			\firststep
			Suppose that $J \neq \emptyset$.
			Without loss of generality, we assume that $z_1 = 1$.
			Clearly, the right hand side of \eqref{eq:UM} is $0$.
			Since $\abs{\prod_{i=1}^k \pr{1-z_i^{(m)}}}
			\leq |1-z_1^{(m)}|$, taking  $k\to \infty$ first, and then $m\to \infty$, we obtain that
			$\lim_{m\to \infty} \abs{\prod_{i=1}^\infty \pr{1-z_i^{(m)}}}  = 0.$
			This implies the desired \eqref{eq:UM} in this case.

			\nextstep
			Suppose that $|I| = \infty$.
			It is clear that the right hand side of \eqref{eq:UM} is $0$.
			In this case, for any $k\in \mathbb N$, there exists $N_k \in \mathbb N$ such that $\abs{I_k} = k$ where $I_k:= \{i\in \mathbb N: i\leq N_k, z_i\in (1,z^*]\}$. Note that, for each $k\in \mathbb N$, there exists an $M_k\in \mathbb N$, such that for any $m\geq M_k$ and $i\in I_k$, $z_i^{(m)} \in (1,z^*]$.
			So for any $k\in \mathbb N$ and $m\geq M_k$,
			\[
			\abs{\prod_{i=1}^\infty (1-z_i^{(m)})} \leq \prod_{i\in I_k} \abs{1-z_i^{(m)}}
			\leq |1-z^*|^k.
			\]
			Taking $m\to \infty$ first, and then $k\to \infty$, we obtain that
			$\lim_{m\to \infty} \abs{\prod_{i=1}^\infty \pr{1-z_i^{(m)}}}  = 0$,
			which implies \eqref{eq:UM} in this case.
			
			\nextstep Suppose that $|J| = 0$ and $|I| = K$ for some finite integer $K$.
			Without loss of generality, we assume that $I = \{1,\dots, K\}$. Then, for any $m\in \mathbb N$ and integer $i >K$, $0\leq z_i^{(m)} \leq z_i < 1$.
			Therefore,
			\begin{align}
				\prod_{i=1}^{K+l} \pr{1-z_i^{(m)}} = \pr{\prod_{i=1}^K \pr{1-z_i^{(m)}}  }\exp\brc{ -\sum_{i=K+1}^{K+l}-\log \pr{1-z_{i}^{(m)}}}.
			\end{align}
			Taking $l\to \infty$ first, and then $m\to \infty$, by the monotone convergence theorem, we have
			\[
			\lim_{m\to \infty} \prod_{i=1}^{\infty} \pr{1-z_i^{(m)}}
			= \pr{\prod_{i=1}^K \pr{1-z_i}  }\exp\brc{ -\sum_{i=K+1}^{\infty}-\log (1-z_{i})}.
			\]
			So $\eqref{eq:UM}$ also holds in this case.
			We are done.
		\end{proof}
		
		\begin{proof}[Proof of Proposition \ref{Prop:measurability-of-Z}]
			Without loss of generality, we assume that $g$ is non-negative.
			Also, it suffices to show that the process $(\tilde Z_t(g))_{t\in [a,b)}$, restricted to a fixed arbitrary interval $[a,b) \subset (0,\infty)$, has a measurable version.
			
			\firststep \label{S4}
			Let us consider the dyadic discretization of the interval $[a,b)$.
			That is, for each $t\in [a,b)$ and $m\in \mathbb N$, we define
			\[
			t_m : = \inf\brc{s > t: \exists k\in \mathbb N, s = a + \frac{k}{2^m}(b-a)}.
			\]
			In particular, $0 < t_m - t \leq 1/2^m$. Define $X_t^{(m)}:= \tilde Z_{t_m}(g)$ for every $t\in [a,b)$ and $m\in \mathbb N$.
			Clearly, for every $w\in \Omega$ and $m\in \mathbb N$, the map $t\mapsto X_t^{(m)}(\omega)$ is measurable on $[a,b)$ w.r.t.~the Borel $\sigma$-field $\mathcal B_{[a,b)}$.
			This allows us to define, for each $m\in \mathbb N$, a measurable map $X^{(m)}: (t,\omega) \mapsto X_t^{(m)}(\omega)$ on the product measurable space $([a,b)\times \Omega, \mathcal B_{[a,b)} \otimes \mathcal F)$, which will be equipped with a product probability measure
			\[M_{[a,b)}(\mathrm{d} x,\mathrm{d} \omega):= \frac{\mathbf{1}_{[a,b)}(x)\mathrm{d} x}{b-a} \otimes \mathbb{P}_{(\Lambda, \mu)}(\mathrm{d} \omega),\quad (x,\omega) \in [a,b)\times \Omega.\]

			\nextstep
			We investigate the limit of the measurable map $X^{(m)}$ when $m\uparrow \infty$.
			Fix an arbitrary $\epsilon > 0$ and define
			$
			G_{l,m}(t):= \mathbb{P}_{(\Lambda, \mu)}( | X_t^{(l)} -X_t^{(m)}| >\epsilon )$ for every $t\in [a,b)$ and $l,m\in \mathbb{N}.
			$
			By Lemma \ref{lem:SCg}, we get that for each $t\in [a,b)$,
			\begin{align}\label{Conv-in-prob-3}
				& \sup_{l,m \geq N} G_{l,m}(t)                                                                                                                                                                         \\
				& \leq \sup_{l\geq N}\mathbb{P}_{(\Lambda,\mu)}\pr{ | X_t^{(l)} -\tilde{Z}_t(g)| >\frac{\epsilon}{2} }+\sup_{m\geq N}\mathbb{P}_{(\Lambda,\mu)}\pr{ | X_t^{(m)} -\tilde{Z}_t(g)| >\frac{\epsilon}{2} }
				\xrightarrow[N\to \infty]{}0.
			\end{align}
			From Fubini's theorem and bounded convergence theorem, we see that
			\begin{align}
				& \sup_{l,m\geq N} M_{[a,b)}\pr{ | X^{(l)} -X^{(m)}| >\epsilon }
				=  \sup_{l,m\geq N} \frac{1}{b-a} \int_a^b G_{l,n}(t)\mathrm{d} t
				\\ &\leq \frac{1}{b-a} \int_a^b \sup_{l,m\geq N} G_{l,m}(t)\mathrm{d} t
				\xrightarrow[N\to \infty]{} 0.
			\end{align}
			Therefore,
			\begin{align}
				& \limsup_{N\to \infty}\sup_{l,m\geq N}\int \pr{\abs{X^{(l)}_t(\omega) - X^{(m)}_t(\omega)} \wedge 1 } M_{[a,b)}(\mathrm{d} t, \mathrm{d} \omega)
				\\& \leq \epsilon + \limsup_{N\to \infty}\sup_{l,m\geq N} M_{[a,b)}\pr{ | X^{(l)} -X^{(m)}| >\epsilon } = \epsilon.
			\end{align}
			Since $\epsilon > 0$ is arbitrarily chosen, we have
			\begin{equation}
				\sup_{l,m\geq N}\int \pr{\abs{X^{(l)}_t(\omega) - X^{(m)}_t(\omega)} \wedge 1 } M_{[a,b)}(\mathrm{d} t, \mathrm{d} \omega)  \xrightarrow[N\to \infty]{} 0.
			\end{equation}
			Therefore, $(X^{(m)})_{m\in \mathbb N}$ is a Cauchy sequence in probability w.r.t.~$M_{[a,b)}$ in the sense of \cite{MR4226142}*{p.~104}.
			The limit, denoted by $\hat Y^g$, is clearly a measurable function on $([a,b)\times \Omega, \mathcal B_{[a,b)}\otimes \mathcal F)$.

			\nextstep
			By Step \ref{S4}  and \cite{MR4226142}*{Lemma 5.2}, there exists a strictly increasing sequence $(m_k)_{k=1}^\infty$ in $\mathbb N$ such that $(X^{(m_k)})_{k\in \mathbb N}$ converges almost surely to $\hat Y^g$ w.r.t.~$M_{[a,b)}$.
			Let $\Xi$ be a measurable null subset of the probability space $([a,b)\times \Omega, \mathcal B_{[a,b)}\otimes \mathcal F, M_{[a,b)})$ which contains all elements $(t,\omega)$ such that $X^{(m_k)}_t(\omega)$ does not converge to $\hat Y^g(t,\omega)$ when $k\uparrow \infty$.
			By Fubini's theorem, the Lebesgue measure of the Borel measurable set $K:= \{t\in [a,b): \int \mathbf 1_{\Xi}(t,\omega) \mathbb P_{\Lambda, \mu}(\mathrm{d} \omega)>0\}$ is $0$.
			Now for each $t\in [a,b)$, define a random variable $\tilde Y^g_t$ on $\Omega$ such that $\tilde Y^g_t(\omega) = \tilde Z_t(g)(\omega) \mathbf 1_{K}(t) + \hat Y^g(t,\omega)\mathbf 1_{K^c}(t)$ for every $\omega \in \Omega$.
			It is clear that $(\tilde Y^g_t)_{t\in [a,b)}$ is a measurable process.
			
			\nextstep
			We finish the proof by showing that $(\tilde Y_t^{g})_{t\in [a,b)}$ is a version of $(\tilde Z_t(g))_{t\in [a,b)}$.
			Note that for any $t\in K$, we already have $\mathbb P_{(\Lambda, \mu)}(\tilde Y_t^{g}=\tilde Z_t(g)) = 1$ according to how $\tilde Y_t^{g}$ is defined.
			Therefore, we only have to consider the case when $t\in [a,b)\setminus K$.
			In this case, we have $\int \mathbf 1_{\Xi}(t,\omega)\mathbb P_{(\Lambda, \mu)}(\mathrm{d} \omega) = 0$, which implies that $X_t^{(m_k)}$ converges to $\tilde Y_t^g$ almost surely when $k\uparrow \infty$ w.r.t.~$\mathbb P_{(\Lambda, \mu)}$.
			From Lemma \ref{lem:SCg}, we have that $X_t^{(m_k)}$ converges to $\tilde Z_t(g)$ in probability.
			So we must have  $\mathbb P_{(\Lambda, \mu)}(\tilde Y_t^{g}=\tilde Z_t(g)) = 1$ as desired in this case. We are done.
		\end{proof}
		
		\begin{proof}[Proof of Lemma \ref{lem:SRM}]
			
			\firststep
			\label{step:Zt}
			We will show in this step that there exists an $\mathcal N$-valued process $(\hat X_t)_{t\geq t_0}$ such that almost surely, for every $t\geq t_0$ and every strictly decreasing sequence $(q_m)_{m\in \mathbb N}$ in $\mathbb Q$ which converges to $t$, $(\tilde X_{q_m})_{m\in \mathbb N}$ converges to $\hat X_t$ in $\mathcal N$.
			We will also show that, almost surely, for every strictly increasing sequence $(q_m)_{m\in \mathbb N}$ in $\mathbb Q\cap [t_0,\infty)$,  $(\tilde X_{q_m})_{m\in \mathbb N}$ converges in $\mathcal N$.
			
			Recall that $\mathcal N$ is equipped with the complete metric $d_{\mathcal N}$ which is defined using a sequence $(h_i)_{i\in \mathbb N}$ in $\mathcal C_{\mathrm c}^\infty (\mathbb R)$.
			For each $i\in \mathbb N$, by hypothesis, the real-valued process $(\tilde X_t(h_i))_{t\geq t_0}$ admits a c\`adl\`ag modification; let $(Y^{h_i}_t)_{t\geq t_0}$ denote such a modification, so that almost surely $Y^{h_i}_q = \tilde X_q(h_i)$ for every $q\in \mathbb Q\cap [t_0,\infty)$.
			Therefore, almost surely, for every strictly decreasing, or strictly increasing, sequence $(q_m)_{m\in \mathbb N}$ in $\mathbb Q\cap [t_0,\infty)$,
			\begin{align}
				& \sup_{m,l\geq N} d_\mathcal N\pr{\tilde X_{q_m},\tilde X_{q_l}}
				= \sup_{m,l\geq N}  \sum_{i=1}^\infty \frac{1}{2^i}\pr{1\land \abs{\tilde X_{q_m} (h_i)- \tilde X_{q_l}(h_i) }}
				\\&= \sup_{m,l\geq N}  \sum_{i=1}^\infty \frac{1}{2^i}\pr{1\land \abs{Y^{h_i}_{q_m}- Y^{h_i}_{q_l} }}
				\leq  \sum_{i=1}^\infty \frac{1}{2^i}\pr{1\land \sup_{m,l\geq N}  \abs{Y^{h_i}_{q_m}- Y^{h_i}_{q_l} }}
				\xrightarrow[]{N\to \infty} 0.
			\end{align}
			Here, in the last step, we used the fact that almost surely for every $i\in \mathbb N$, $Y_s^{h_i}$ is c\`adl\`ag in $s>0$.
			Now, the event
			\begin{itemize}
				\item
				for every strictly decreasing, or strictly increasing, sequence $(q_m)_{m\in \mathbb N}$ in $\mathbb Q\cap [t_0,\infty)$,  $(\tilde X_{q_m})_{m\in \mathbb N}$ is a Cauchy sequence with respect to the metric $d_\mathcal N$
			\end{itemize}
			has probability 1.
			Since $(\mathcal N, d_{\mathcal N})$ is complete, on the above full-measure event, every strictly decreasing (respectively, strictly increasing) rational sequence converges in $\mathcal N$.
			Moreover, on this full-measure event, any two strictly decreasing rational sequences converging to the same $t$ have a common Cauchy subsequence, and hence share the same limit.
			The desired result of this step follows.
			
			\nextstep \label{step:cadlag}
			We argue that $(\hat X_t)_{t\geq t_0}$ is an $\mathcal N$-valued c\`adl\`ag process.
			On one hand, from Step \ref{step:Zt},   we know that almost surely, for every $t\geq t_0$ and every strictly decreasing sequence $(t_m)_{m\in \mathbb N}$ in $\mathbb R$ which converges to $t$, since there exists a strictly decreasing sequence $(q_m)_{m\in \mathbb N}$ in $\mathbb Q$ such that
			$q_1 > t_1 > q_2 > t_2 > \dots$ and that $d_\mathcal N (\tilde X_{q_m}, \hat X_{t_m}) \leq 1/m$ for every $m\in \mathbb N$, we have \[d_\mathcal N(\hat X_{t_m}, \hat X_t) \leq d_\mathcal N(\hat X_{t_m}, \tilde X_{q_m}) + d_\mathcal N(\tilde X_{q_m}, \hat X_t) \xrightarrow[]{m\to \infty} 0.\]
			On the other hand, from Step \ref{step:Zt} again, almost surely, for every strictly increasing sequence $(t_m)_{m\in \mathbb N}$ in $[t_0,\infty)$, since there exists a strictly increasing sequence $(q_m)_{m\in \mathbb N}$ in $\mathbb Q$ such that $t_1 < q_1 < t_2 < q_2 < \dots$ and that $d_\mathcal N(\tilde X_{q_m}, \hat X_{t_m}) \leq 1/m$ for every $m\in \mathbb N$, we have
			\begin{align}
				& \sup_{l,m\geq N} d_{\mathcal N}(\hat X_{t_m}, \hat X_{t_l})
				\leq \sup_{l,m\geq N} \pr{ d_{\mathcal N}(\hat X_{t_m}, \tilde X_{q_m}) + d_{\mathcal N}(\tilde X_{q_m}, \tilde X_{q_l}) + d_{\mathcal N}( \tilde X_{q_l}, \hat X_{t_l}) }
				\\&\leq \frac{2}{N} + \sup_{l,m \geq N} d_{\mathcal N}(\tilde X_{q_m}, \tilde X_{q_l}) \xrightarrow[]{N \to \infty} 0.
			\end{align}
			Now, the desired result for this step holds.
			
			\nextstep \label{step:modification}
			We argue that the process $(\hat X_t)_{t\geq t_0}$ is a modification of $(\tilde X_t)_{t\geq t_0}$.
			Let us fix an arbitrary $t\geq t_0$.
			Let $(q_m)_{m\in \mathbb N}$ be a strictly decreasing sequence in $\mathbb Q$ which converges to $t$.
			Note that almost surely for each $m,i\in \mathbb N$, $Y^{h_i}_{q_m} = \tilde X_{q_m}(h_i)$.
			Combined with Step \ref{step:Zt},   almost surely, for each $i\in \mathbb N$, we have
			\begin{equation}
				\hat X_t(h_i) = \lim_{m\to \infty} \tilde X_{q_m}(h_i) = \lim_{m\to \infty} Y^{h_i}_{q_m} = Y^{h_i}_t
			\end{equation}
			where in the last step we used the fact that almost surely $Y^{h_i}_s$ is c\`adl\`ag in $s>0$.
			Therefore, we have almost surely $\hat X_t(h_i) = \tilde X_t(h_i)$ for each $i\in \mathbb N$, which further implies that almost surely $d_\mathcal N (\hat X_t, \tilde X_t) = 0$.
			The desired result for this step now follows.
			
			Combining Steps \ref{step:cadlag} and \ref{step:modification}, we conclude that $(\tilde X_t)_{t\geq t_0}$ admits a c\`adl\`ag modification.
		\end{proof}
		
	\end{appendix}

\end{document}